\documentclass[11pt]{amsart}

\usepackage[divide={2.5cm,*,2.5cm}]{geometry}

\input{Macros.sty}

\allowdisplaybreaks

\begin{document}
\title{Serre weight conjectures for $p$-adic unitary groups of rank 2}

\author{Karol Kozio\l}
\address{Department of Mathematics, University of Michigan, 2074 East Hall, 530 Church Street, Ann Arbor, MI 48109-1043 USA} 
\email{kkoziol@umich.edu}

\author{Stefano Morra}
\address{Universit\'e Paris 8, Laboratoire d'Analyse, G\'eom\'etrie et Applications, Universit\'e Sorbonne Paris Nord, CNRS, UMR 7539,  F-93430, Villetaneuse, France}
\email{morra@math.univ-paris13.fr}

\subjclass[2010]{11F80 (primary), 11F33, 11F55, 20C33 (secondary)}

\begin{abstract}
We prove a version of the weight part of Serre's conjecture for mod $p$ Galois representations attached to automorphic forms on rank 2 unitary groups which are non-split at $p$.  More precisely, let $F/F^+$ denote a CM extension of a totally real field such that every place of $F^+$ above $p$ is unramified and inert in $F$, and let $\overline{r}: \textnormal{Gal}(\overline{F^+}/F^+) \longrightarrow {}^C\mathbf{U}_2(\overline{\mathbb{F}}_p)$ be a Galois parameter valued in the $C$-group of a rank 2 unitary group attached to $F/F^+$.  We assume that $\overline{r}$ is semisimple and sufficiently generic at all places above $p$.  Using base change techniques and (a strengthened version of) the Taylor--Wiles--Kisin conditions, we prove that the set of Serre weights in which $\overline{r}$ is modular agrees with the set of Serre weights predicted by Gee--Herzig--Savitt.  
\end{abstract}

\maketitle
\tableofcontents

\section{Introduction}

Let $p$ be a prime number.  The mod $p$ local Langlands program (cf. \cite{breuil-ICM}, \cite{berger-bourbaki}, \cite{breuilmezard}) predicts a correspondence between continuous Galois representations $\rhobar: \Gal(\overline{\bbQ}_p/\bbQ_p)\longrightarrow \bG\bL_n(\Fpbar)$ and smooth admissible $\bG\bL_n(\qp)$-representations on $\Fpbar$-vector spaces.  It is expected to be compatible with the classical local Langlands correspondence over $\bbC$, its geometric realization in the torsion cohomology of Shimura varieties, and classical local/global compatibility.

The case when $n=2$ has been most extensively studied, and such a correspondence has now been established (see \cite{colmez}, \cite{CDP}, \cite{emerton-lgc}, and the above references).   However, the picture for $n > 2$ (or more general $p$-adic fields) still remains highly conjectural, and evidence suggests that such a correspondence will be much more intricate (see, for example, \cite{breuilherzig}).  Despite this deficiency, there has been substantial progress on several expected consequences of this conjecture:  the weight part of Serre's conjecture, the Breuil--M\'ezard conjecture, and Breuil's lattice conjecture (\cite{BDJ}, \cite{GLS}, \cite{breuilmezard}, \cite{gee-kisin}, \cite{breuil-buzzati}, \cite{EGH}, \cite{EGS}).

In a different direction, one is also interested in the possibility of enlarging the conjectural correspondence to include more general groups.  The works \cite{ramla-sl2} and \cite{Karol-U11} give some preliminary indication that a Langlands-type correspondence might be expected to hold for the groups $\bS\bL_2(\bbQ_p)$ and $\bU_2(\bbQ_p)$, and reveal some new phenomena (e.g., the existence of $L$-packets in the mod $p$ setting).  In general, the work of Buzzard--Gee \cite{buzzardgee} lays out precise statements of Langlands-type conjectures for general reductive groups by making use of an enhancement of the Langlands dual group (this will figure prominently in our considerations below).  This framework reconciles the classical local Langlands correspondence with its geometric realization.  These developments are also related to recent work of Gee--Herzig--Savitt: the article \cite{ghs} gives a formulation of the weight part of Serre's modularity conjectures for a large class of non-classical reductive groups.

Classical Langlands correspondences (i.e., with $\bbC$-coefficients) for various reductive groups, and the relations among them, are at the core of the Langlands functoriality principle.  In the specific example of unitary groups, this principle predicts that a correspondence between (packets of) automorphic representations of unitary groups on the one side and $L$-group valued Galois parameters on the other side is obtained from a correspondence on general linear groups.  When the unitary group has low rank, this is studied in \cite[\S 15.1]{rogawski}.

The goal of the present work is to give evidence for a mod $p$ Langlands correspondence for rank 2 unitary groups.  Specifically, given a Galois parameter $\rbar$ with values in the $C$-dual of our unitary group, we prove that the Serre weights for $\rbar$ predicted by \cite{ghs} (which are representations of finite unitary groups) are exactly equal to the Serre weights in which $\rbar$ is modular (we give a precise statement below).  In order to do this, we use known instances of functoriality (in the form of classical base change results) and local/global compatibility.  Thus, our methods hint at a mod $p$ principle of unitary base change.

We now introduce some notation and setup in order to state our main result.  Let $K_2/K/\qp$ be unramified extensions, with $K_2/K$ quadratic.  We let $\bU_2$ denote the unramified unitary group in two variables defined over the ring of integers $\cO_K$ of $K$.  Note that $\bU_2$ splits over $K_2$.  We let ${}^C \bU_2$ denote the $C$-group of $\bU_2$, in the terminology of \cite{buzzardgee} (${}^C\bU_2$ is the usual Langlands $L$-group of a canonical central extension of $\bU_2$).  An $L$-parameter is a continuous homomorphism $\overline{\rho}: \Gal(\qpb/K) \longrightarrow {}^C{\bU_2}(\Fpbar)$, compatible with the projection ${}^C{\bU_2}(\Fpbar) \longtwoheadrightarrow \Gal(K_2/K)$.  The $C$-group also comes equipped with a canonical map ${}^C\bU_2 \longrightarrow \bG_m$, and we assume that the composite character $\Gal(\qpb/K) \stackrel{\rhobar}{\longrightarrow} {}^C{\bU_2}(\Fpbar) \longrightarrow \Fpbar^\times$ (called the multiplier of $\rhobar$) is equal to the mod $p$ cyclotomic character.

Inspired by the conjectures of \cite{buzzardgee} and the prospect of a mod $p$ Langlands program for unitary groups, we would like to infer that the $L$-parameter $\rhobar$ is associated to an $L$-packet of smooth representations of $\bU_2(K)$ over $\Fpbar$.  Unfortunately, such representations are poorly understood beyond the case $K = \qp$ (cf. \cite{Karol-U11}).  A possible first step in understanding such a correspondence would be to study this question in a global context, that is, to study local/global compatibility for an $L$-parameter $\overline{r}: \Gal(\overline{\bbQ}/F^+) \longrightarrow {}^C{\bU_2}(\Fpbar)$, where $F^+/\bbQ$ is a totally real field.  We assume furthermore that $\rbar$ is associated to a non-zero Hecke eigenclass in the mod $p$ cohomology with infinite level at $p$ of a definite unitary group $\bG_{/\cO_{F^+}}$ which is \textbf{non-split} at places of $F^+$ above $p$.  We would like to stress that our setting differs quite markedly from the body of work related to Serre weights for unitary groups (e.g., \cite{GLS}, \cite{BLGG13}), wherein the group $\bG$ is split at places above $p$.  In particular, our Serre weights are representations of finite unitary groups, not general linear groups.  We define $\textnormal{W}_{\textnormal{mod}}(\overline{r})$ to be the set consisting of the $\prod_{v|p}\bU_2(\cO_{F^+_v})$-representations appearing in the socle of the Hecke isotypic component attached to $\rbar$ of the mod $p$ cohomology of $\bG$.

According to the conjectures of \cite{ghs}, the set $\textnormal{W}_{\textnormal{mod}}(\overline{r})$ should be described in an explicit way by $(\overline{r}|_{\Gal(\qpb/F^+_v)})_{v|p}$ using purely representation-theoretic constructions.  Let us denote $\textnormal{W}^?(\overline{r}) \defeq \bigotimes_{v|p} \textnormal{W}^?(\overline{r}|_{\Gal(\qpb/F^+_v)})$, where $\textnormal{W}^?(\overline{r}|_{\Gal(\qpb/F^+_v)})$ is the set described combinatorially in \cite{ghs} (thus $\textnormal{W}^?(\overline{r})$ is again a set of representations of the group $\prod_{v|p}\bU_2(\cO_{F^+_v})$).

The main theorem of this paper is the following (we refer the reader to the bulk of the paper for any unfamiliar terminology).

\begin{theo}[Corollary \ref{cor:SWC}]
\label{thm:1:intro}
Let $F/F^+$ be a CM field extension of $F^+$ which is unramified at all finite places, suppose that $p$ is unramified in $F^+$ and that every place of $F^+$ above $p$ is inert in $F$.  Let $\rbar:\Gal(\ovl{\bbQ}/{F^+})\longrightarrow {}^C\bU_2(\overline{\bbF}_p)$ be an $L$-parameter with cyclotomic multiplier. Assume that: 
\begin{itemize}
	\item $\rbar^{-1}({}^C\bU_2^\circ(\overline{\bbF}_p)) = \Gal(\ovl{\bbQ}/{F})$;
	\item $\rbar$ is modular;
	\item $\rbar$ is unramified outside $p$;
	\item $\rbar$ is semisimple and $4$-generic at places above $p$;
	\item $\overline{\bbQ}^{\ker(\textnormal{ad}^0(\rbar))}$ does not contain $F(\zeta_p)$; and
	\item $\BC(\rbar)(\Gal(\overline{\bbQ}/F)) \supseteq \bG\bL_2(\bbF')$ for some subfield $\bbF' \subseteq \Fpbar$ with $|\bbF'| > 6$.
\end{itemize}
Then
$$\textnormal{W}^?(\overline{r}) = \textnormal{W}_{\textnormal{mod}}(\rbar).$$
\end{theo}

In the $\bG\bL_2$ setting, the results of \cite{BP} and \cite{EGS} imply that, for a $\bG\bL_2(\overline{\bbF}_p)$-valued Galois representation $\rhobar'$, the set $\textnormal{W}^?(\rhobar')$ of modular Serre weights should be equal to the set of representations appearing in the $\bG\bL_2(\cO_K)$-socle of the $\bG\bL_2(K)$-representation associated to $\rhobar'$ via some sort of mod $p$ local Langlands correspondence.  For $\bU_2$, the supersingular representations of $\bU_2(\bbQ_p)$ constructed in \cite{Karol-U11} all have simple $\bU_2(\bbZ_p)$-socle, while the set $\textnormal{W}^?(\rhobar)$ (for generic semisimple $\rhobar$) has size $2^{[K:\bbQ_p]}$.  {Thus, in the $K = \bbQ_p$ case, the global evidence provided by Theorem \ref{thm:1:intro} suggests that $\textnormal{W}^?(\rhobar)$ (for appropriate $\rhobar$) takes into account the $\bU_2(\bbZ_p)$-socles of all $\bU_2(\bbQ_p)$-representations in a supersingular $L$-packet. }

We obtain Theorem \ref{thm:1:intro} by following the strategy of \cite{gee-kisin}.  We first prove the containment $\textnormal{W}^?(\rbar) \supseteq \textnormal{W}_{\textnormal{mod}}(\rbar)$ by using a global base change argument and applying results of \cite{gee-hmf}.  The opposite containment follows by using a modified version of the patching functor constructed in \cite{CEGGPS} and the explicit description of ${}^C\bU_2$-valued local deformation rings.  We explain these arguments with more details presently.

The main novelty in the unitary group setting is that for both inclusions we make use of the analogous results for ${\bG\bL_2}_{/K_2}$.  Firstly, we establish a compatibility between classical local base change of automorphic types (as may be deduced from work of Rogawski \cite{rogawski}) and the set of predicted Serre weights $\textnormal{W}^?(\rhobar)$ (for which we introduce a notion of base change of weights).  In this direction our results give the following proposition, which may be thought of as evidence towards a notion of mod $p$ base change.  Recall that a tame $\bU_2(\cO_K)$-type is the inflation of an irreducible $\bU_2(\bbF_q)$-representation over $\overline{\bbQ}_p$, where $\bbF_q$ denotes the residue field of $K$.

\begin{prop}[Lemma \ref{JHunderBC}, Theorem \ref{main:thm:local}]
\label{prop:2:intro}
Let $\sigma$ denote a $1$-generic tame type for $\bU_2(\cO_K)$, and let $V$ denote a Serre weight for $\bU_2(\cO_K)$.  Let $\BC(\sigma)$ denote the base change of $\sigma$ \emph{(}as defined in Subsection \ref{autBC}\emph{)}.  Then
$$V \in \JH(\overline{\sigma}) \quad \Longleftrightarrow \quad \BC(V) \in \JH\big(\overline{\BC(\sigma)}\big),$$
where $\BC(V)$ is the base change of the Serre weight $V$ (as defined in Subsection \ref{subsec:BC:WT}) and $\JH(\overline{W})$ denotes the set of Jordan-H\"older factors of the mod $p$ reduction of a $\zpb$-lattice in $W$.

In particular, if $\rhobar:\textnormal{Gal}(\overline{\bbQ}_p/K) \longrightarrow {}^C\bU_2(\Fpbar)$ is a $1$-generic tame $L$-parameter with cyclotomic multiplier, then the set of predicted local Serre weights $\textnormal{W}^?(\rhobar)$ is of the form $\JH(\overline{\sigma})$, and we obtain
$$V \in \textnormal{W}^?(\rhobar) \quad \Longleftrightarrow \quad \BC(V) \in \textnormal{W}^?\big({\BC(\rhobar)}\big).$$
Here $\BC(\rhobar):\textnormal{Gal}(\overline{\bbQ}_p/K_2) \longrightarrow \bG\bL_2(\Fpbar)$ denotes the Galois representation obtained by restricting $\rhobar$ to the absolute Galois group of $K_2$ and projecting onto the $\bG\bL_2$ factor.
\end{prop}

The tame $\bG\bL_2(\cO_{K_2})$-type $\BC(\sigma)$ of the proposition is characterized by the property that $\BC(\sigma)\otimes \bbC \longhookrightarrow \BC(\pi)$, where $\pi$ is any smooth irreducible complex representation of $\bU_2(K)$ containing $\sigma \otimes \bbC$, and where $\BC(\pi)$ denotes the stable base change of the $L$-packet containing $\pi$ (\cite[\S 11]{rogawski}).  Using the above proposition, we prove in Theorem \ref{thm:WE} the inclusion $\textnormal{W}^?(\rbar) \supseteq \textnormal{W}_{\textnormal{mod}}(\rbar)$ by base changing to $\bG\bL_2$, and using results of Gee (\cite{gee-hmf}) on the set $\textnormal{W}^?(\BC(\rhobar))$.

In order to prove the inclusion $\textnormal{W}^?(\rbar) \subseteq \textnormal{W}_{\textnormal{mod}}(\rbar)$, we would like to employ a patching argument, which requires information regarding certain deformation rings.  More precisely, let us suppose that $\rhobar: \Gal(\overline{\bbQ}_p/K) \longrightarrow {}^C\bU_2(\bbF)$ is an $L$-parameter with $\bbF$ a finite extension of $\bbF_p$, and let $\cO$ denote the ring of integers in some {sufficiently large} finite extension of $\bbQ_p$ with residue field $\bbF$.  We let $R_{\rhobar}^{(1,0,1),\tau'}$ denote the deformation ring parametrizing potentially crystalline framed deformations of $\rhobar$ to $\cO$-algebras with (parallel) $p$-adic Hodge type $(1,0,1)$, inertial type $\tau'$, and cyclotomic multiplier.  In order to study the ring $R^{(1,0,1),\tau'}_{\rhobar}$, we introduce the notion of {Frobenius twist self-dual} Kisin modules.  Given this, we are able to describe the structure of $R_{\rhobar}^{(1,0,1),\tau'}$ in terms of the ``base changed'' deformation ring $R_{\BC(\rhobar)}^{(1,0), \tau'}$.  Combining these calculations with Proposition \ref{prop:2:intro}, along with the analogous results of \cite{gee-kisin} for $\bG\bL_2$, we obtain the following result, which may be viewed as a ``Breuil--M\'ezard-type'' result for unitary groups.

\begin{prop}
\label{intro:BM}
Let $\rhobar: \Gal(\qpb/K) \longrightarrow {}^C\bU_2(\bbF)$ be a $3$-generic tame $L$-parameter with cyclotomic multiplier.  Let $\tau'$ be a ${}^C\bU_2$-valued, $2$-generic inertial type for $I_K$ and $\sigma(\tau')$ the tame $\bU_2(\cO_K)$-type associated to $\tau'$ via the inertial local Langlands correspondence of Theorem \ref{ILLC}.  Then
$$\big|\textnormal{W}^?(\rhobar) \cap \JH\big(\overline{\sigma(\tau')}\big)\big| = e\big(R_{\rhobar}^{(1,0,1),\tau'}\otimes_{\cO}\bbF\big),$$
where $e(-)$ denotes the Hilbert--Samuel multiplicity.  
\end{prop}

To conclude, we employ a variant of the construction of \cite{CEGGPS} in order to produce a patching functor $M_\infty(-)$ on the category of $\cO$-modules with an action of $\prod_{v|p} \bU_2(\cO_{F^+_v})$.  Using the explicit structure of the rings $R_{\rhobar}^{(1,0,1),\tau'}$ (namely their integrality), the properties of the patching functor $M_\infty(-)$, and Proposition \ref{intro:BM}, we obtain the inclusion $\textnormal{W}^?(\rbar) \subseteq \textnormal{W}_{\textnormal{mod}}(\rbar)$ in Theorem \ref{thm:WExt}. This is enough to prove the main Theorem \ref{thm:1:intro}.

Our results on the geometry of $R_{\rhobar}^{(1,0,1),\tau'}$ in \S \ref{subsec:Def:thy} can also be used to deduce new cases of automorphy lifting phenomena for unitary groups which are non-split at $p$.
Indeed, the integrality of $R_{\rhobar}^{(1,0,1),\tau'}$ (cf. \S \ref{subsubsec:integrality} and Table \ref{Table3}) together standard Taylor--Wiles--Kisin arguments give the following Theorem (again, we refer the reader to the bulk of the paper for unfamiliar terminology):

\begin{theo}
\label{thm:mod:lift}
Let $F/F^+$ be a CM field extension of $F^+$ which is unramified at all finite places, suppose that $p$ is unramified in $F^+$ and that every place of $F^+$ above $p$ is inert in $F$.  

Let $r':\Gal(\ovl{\bbQ}/{F})\longrightarrow \bG\bL_2(\overline{\mathbb{Z}}_p)$ be a continuous Galois representation, and let $\rbar': \Gal(\overline{\mathbb{Q}}/F) \longrightarrow \bG\bL_2(\Fpbar)$ denote the associated residual representation.
Assume that
\begin{itemize}
	\item $r'$ is unramified at all but finitely many places;
	\item we have $r'^c \cong r'^\vee \otimes \varepsilon^{-1}$, where $c\in \textnormal{Gal}(F/F^+)$ is the complex conjugation;
	\item for all places $v$ of $F$ above $p$, the local representation $r'|_{\Gal(\ovl{\bbQ}_p/{F_v})}$ is potentially crystalline, with parallel Hodge type $(-1,0)$ and $4$-generic tame inertial type $\tau'_v$;
	\item for all places $v$ of $F$ above $p$, the local representation $\rbar'|_{\Gal(\ovl{\bbQ}_p/{F_v})}$ is semisimple and $4$-generic;
	\item $\rbar'$ is unramified outside places above $p$;
	\item $\rbar'\cong \rbar_{\imath}(\pi)$ where $\pi$ is a cuspidal automorphic representation of $\bG(\bbA_{F^+})$, such that $\pi_\infty$ is trivial and for all places $v$ of $F^+$ above $p$, the local component $\pi_v$ contains the tame $\bU_2(\cO_{F^+_v})$-representation associated to $\tau'_v$ by the inertial local Langlands correspondence \emph{(}cf. Theorem \ref{ILLC}\emph{)};
	\item $\overline{\bbQ}^{\ker(\textnormal{ad}(\rbar'))}$ does not contain $F(\zeta_p)$; and
	\item $\rbar'(\Gal(\overline{\bbQ}/F)) \supseteq \bG\bL_2(\bbF')$ for some subfield $\bbF' \subseteq \Fpbar$ with $|\bbF'| > 6$.	
\end{itemize}
Then
$r'$ is automorphic.
\end{theo}

(Recall that $r'$ is automorphic if $r'\otimes_{\overline{\bbZ}_p}\overline{\bbQ}_p$ is isomorphic to $r_\imath(\pi')$ for some cuspidal automorphic representation $\pi'$ of $\bG(\bbA_{F^+})$, where $r_\imath(\pi')$ is the continuous Galois representation associated to $\pi'$ as in Theorem \ref{auttogal}.)

We conclude this introduction with a few remarks on natural questions which arise from the results in this paper.

In Theorems \ref{thm:1:intro} and \ref{thm:mod:lift}, the assumption that $\rbar$ is unramified outside $p$ is used to simplify our arguments, and it should be possible to remove it.  
On the other hand removing the condition that the $L$-parameter is residually tame at places above $p$ requires further analysis of the possible set of modular weights $\textnormal{W}_{\textnormal{mod}}(\rbar|_{\Gal(\ovl{\bbQ}_p/F^+_v)})\subseteq \textnormal{W}^?(\rbar|_{\Gal(\ovl{\bbQ}_p/F^+_v)}^{\textnormal{ss}})$, and will depend in a subtle way on the geometry of $R_{\rbar|_{\Gal(\ovl{\bbQ}_p/F^+_v)}}^{(1,0,1),\tau'}$.

In the case where $\rbar|_{\Gal(\ovl{\bbQ}_p/F^+_v)}$ is semisimple, the combinatorics of the set $\textnormal{W}^?(\rbar|_{\Gal(\ovl{\bbQ}_p/F^+_v)})$ and the set of Jordan--H\"older consitutents of tame types for $\bU_2(\cO_{F^+_v})$ suggest that tame $\bU_2(\cO_{F^+_v})$-representations will play the role of Breuil--Pa\v{s}k\={u}nas diagrams for non-split unitary groups.  We expect these representations to be useful in constructing, by a purely local procedure, some mod-$p$ representations of $\bU_2(K)$ which naturally appear in the cohomology of Shimura curves with tame level at $p$.  We hope to come back to these questions in future work.

\vspace{20pt}

The paper is organized as follows.  In Section \ref{sec:UnitaryGPS}, we discuss the unitary groups over $\cO_K$ which are relevant for this paper, namely the unramified unitary group in two variables $\bU_{2/\cO_K}$.  In fact, in order to speak about Serre weight conjectures, we must work with a certain central extension $\widetilde{\bU}_2$ of $\bU_2$ constructed by Buzzard--Gee in \cite{buzzardgee}.  We also define the $C$-group ${}^C\bU_2$, which is the ``classical'' Langlands $L$-group of $\widetilde{\bU}_2$.  We give explicit descriptions of the Galois actions on these groups, their character groups, and their $\Fp$-structures.  Since the groups appearing are slightly non-standard, we have attempted to give a detailed account.

Section \ref{sec:RepThy} is devoted to the theory of types, that is, absolutely irreducible $\bU_2(\bbF_q)$-representations over $\textnormal{Frac}(\cO)$, and their reductions over $\bbF$. In Subsection \ref{autBC}, we recall the notion of base change for types and compare it with local automorphic base change of smooth $\bU_2(K)$-representations over $\bbC$.  Then, in Subsections \ref{subsec:cmb:types} and \ref{subsec:BC:WT}, we analyze the Jordan--H\"older constituents of the mod $p$ reductions of types vis-\`a-vis the constituents of the mod $p$ reductions of their base changes.  This allows us to establish several useful properties of base change of Serre weights.

In Section \ref{subsec:L:pmts} we study $L$-parameters of the form $\rhobar:\Gal(\overline{\bbQ}_p/K) \longrightarrow {}^C\bU_2(\bbF)$.  We relate these parameters to $\bU_2(\cO_K)$-representations to produce the set of (local) predicted weights $\textnormal{W}^?(\rhobar)$, as defined in \cite{ghs}.  The core of this section is Subsection \ref{subsec:BCandW}, which examines the compatibility between Serre weights of $L$-parameters and their base changes.  To conclude, we establish Theorem \ref{main:thm:local}, which figures in subsequent base change results.

Section \ref{sec:Loc:Def} deals with local deformation theory of $C$-group valued $L$-parameters.  We introduce the notion of {Frobenius twist self-dual} Kisin modules over $\fS_R = (\cO_{K_2}\otimes_{\bbZ_p}R)[\![u]\!]$ in Subsection \ref{sec:dual:KM}, which are Kisin modules equipped with an isomorphism between their Frobenius pullback and their dual.  Using this definition, we deduce the deformation theory of {Frobenius twist self-dual} Kisin modules from that of Kisin modules over $\fS_R$ by means of base change (as in \cite{LLLM1}).  The precise relation between deformation theory of {Frobenius twist self-dual} Kisin modules and $C$-group valued $L$-parameters is achieved in Subsection \ref{subsec:Def:thy}.  In particular, we obtain an explicit presentation for the deformation rings $R^{(1,0,1), \tau'}_{\rhobar}$.

Sections \ref{sec:global} and \ref{glob2} contain the main global applications, and the proof of the main theorem.  In Subsections \ref{subsec:UGGlob} -- \ref{subsec:Galois-Automorphic}, we provide the background on algebraic automorphic forms on unitary groups which are quasi-split (but \textbf{not} split) at $p$, and the Galois representations associated to them, by generalizing the usual results in the literature for groups which are split at $p$ (see Theorem \ref{galrepheckealg}).  We remark that the compatibility of base change of types as recalled in \ref{autBC} and classical base change are integral to these generalizations.  The main result of Section \ref{sec:global} is Theorem \ref{thm:WE}, which is the ``weight elimination'' statement.

In Subsections \ref{sub:setup} -- \ref{sub:patching} we generalize the patching construction of \cite{CEGGPS} to our unitary groups (cf. Proposition \ref{patchingprops}).  The modifications are largely formal, using as input the results from Subsection \ref{subsec:Galois-Automorphic}.  The main result on ``weight existence'' is then obtained in Subsection \ref{subsec:WExt}, following the patching techniques of \cite{gee-kisin}.  The main result on automorphy lifting follows in Subsection \ref{autlift}.

\vspace{10pt}

\textbf{Acknowledgements:}
The authors would like to thank Patrick Allen, {Rapha\"el Beuzart--Plessis, John Enns,} Florian Herzig, {Tasho Kaletha,} {Bao Le Hung,} and Sug-Woo Shin for numerous enlightening discussions about this project.  Parts of this work were carried out during stays at the IH\'ES, the Max Planck Institute for Mathematics, Universit\'e de Montpellier, and the University of Toronto; we would like to thank each of these institutions for their support (financial, logistical, and otherwise).  During the preparation of this article the first author was supported by NSF grant DMS-1400779 and an EPDI fellowship, while the second author was supported by the A.N.R.\ project CLap-CLap ANR-18-CE40-0026.  
{Finally, we would like to heartily thank the referee(s) for his/her meticulous reading and comments on several versions of the present article, which greatly improved its clarity and precision.}

\subsection{Notation}

Let $p$ denote an odd prime number, and fix an algebraic closure $\qpb$ of $\qp$.  We denote its ring of integers by $\overline{\bbZ}_p$ and its residue field by $\overline{\bbF}_p$, and we assume that all field extensions of $\qp$ are contained in $\qpb$.  Given a $p$-adic field $F$ and an element $x$ in its residue field, we define $\tilde{x}$ to be its Teichm\"uller lift.  Throughout we will work with a finite extension $E$ of $\qp$ which will serve as our field of coefficients. We let $\cO$ denote the ring of integers of $E$, $\varpi$ its uniformizer, and $\bbF$ its residue field.  We will assume $E$ and $\bbF$ are sufficiently large as necessary.

For any field $F$, we let $\Gamma_F \defeq \textnormal{Gal}(\overline{F}/F)$ denote the absolute Galois group of $F$, where $\overline{F}$ is a fixed separable closure of $F$.  If $F$ is a number field and $v$ is a place of $F$, we let $F_v$ denote the completion of $F$ at $v$, and use the notation $\textnormal{Frob}_v$ to denote a geometric Frobenius element of $\Gamma_{F_v}$.  If $F$ is a $p$-adic field, we let $I_F$ denote the inertia subgroup of $\Gamma_F$.

For $F$ either a number field or a $p$-adic field, we let $\varepsilon:\Gamma_F\longrightarrow \zp^\times$ denote the $p$-adic cyclotomic character, and let $\overline{\varepsilon}$ or $\omega$ denote its reduction mod $p$.

If $F$ is a $p$-adic field, $V$ a de Rham representation of $\Gamma_F$ over $E$, and $\kappa:F \longhookrightarrow E$ an embedding, then we define $\textnormal{HT}_\kappa(V)$ to be the multiset of Hodge--Tate weights with respect to $\kappa$.  Thus, $\textnormal{HT}_\kappa(V)$ contains $i$ with multiplicity $\dim_E(V\otimes_{F,\kappa}\widehat{\overline{F}}(i))^{\Gamma_F}$.  In particular, $\textnormal{HT}_\kappa(\varepsilon) = \{-1\}$.  Further, we let $\textnormal{WD}(V)$ denote the Weil--Deligne representation associated to $V$, normalized so that $V \longmapsto \textnormal{WD}(V)$ is a covariant functor.

Let $F$ be a $p$-adic field.  We let $\textnormal{Art}_F:F^\times \longrightarrow \Gamma_F^{\textnormal{ab}}$ denote the Artin map, which sends uniformizers to geometric Frobenius elements.  Let $\textnormal{rec}_{\bbC}$ denote the Local Langlands correspondence of \cite{HT}, from isomorphism classes of smooth irreducible representations of $\bG\bL_n(F)$ over $\bbC$ to isomorphism classes of $n$-dimensional, Frobenius-semisimple Weil--Deligne representations of the Weil group of $F$ (normalized to agree $\textnormal{Art}_F$ in dimension 1).  For a choice of isomorphism $\imath: \overline{E} \stackrel{\sim}{\longrightarrow} \bbC$, we define $\textnormal{rec}_{\overline{E}} \defeq \imath^{-1}\circ \textnormal{rec}_{\bbC} \circ \imath$ to be the Local Langlands correspondence over $\overline{E}$.

All representations will live on vector spaces over $E$ or $\bbF$, or on $\cO$-modules, unless otherwise indicated.  
By abuse of notation, we will generally not distinguish between a representation and its isomorphism class.  
If $G$ is a group, $H \trianglelefteq G$ a normal subgroup, $\rho$ an $H$-representation and $g\in G$, we write $\rho^g$ to denote the $H$-representation given by $h \longmapsto \rho(ghg^{-1})$.

Given a finite length representation $V$ of some group, we let $\JH(V)$ denote its set of Jordan--H\"older factors.  If $V$ denotes a representation of a (pro)finite group $G$ on a finite-dimensional $E$-vector space, then we may choose a $G$-stable $\cO$-lattice $V^\circ$ inside $V$, and we write $\overline{V^\circ}$ for its reduction mod $\varpi$.  By \cite[Thm. 32]{serre-reptheory}, the set of Jordan--H\"older factors of $\overline{V^\circ}$ is independent of the choice of lattice $V^\circ$.  We write $\JH(\overline{V})$ for $\JH(\overline{V^\circ})$.  
{We denote by $V\longmapsto V^\vee$ the duality functor defined on the category of finite dimensional $E$-vector spaces (resp.~finite dimensional $\bbF$-vector spaces).}

We write matrix transposes \emph{on the right}, so that $A^\top$ denotes the transpose of a matrix $A$.  Given an (anti)automorphism $\theta$ of $\bG\bL_n(R)$ which commutes with the transpose, we write $A^{\theta\top}$ for $(A^\theta)^{\top}$; in particular, we write $A^{-\top}$ for $(A^{-1})^\top$.

\section{Group-theoretic constructions}
\label{sec:UnitaryGPS}

Our first task will be to introduce the groups which will be relevant to arithmetic applications.  After defining unitary groups and certain central extensions in Subsections \ref{unitarygps:local} and \ref{unitarygps:qp}, we construct the dual groups with which we will be working in Subsection \ref{dualgps}.  For the sake of thoroughness, we also give explicit descriptions of the Galois actions and $\bbF_p$-structures.  We mostly follow \cite{buzzardgee} and \cite[\S 9]{ghs}.

\subsection{Unitary groups over $p$-adic fields}\label{unitarygps:local}

\subsubsection{}

Let $f\geq 1$, and let $K$ denote the unramified extension of $\qp$ of degree $f$.  We let $\cO_K$ denote its ring of integers, with canonical uniformizer $p$, and identify its residue field with $\bbF_q =  \bbF_{p^f}$.  We let $\varphi \in \Gamma_{\bbQ_p}$ denote a fixed lift of $\textnormal{Art}_{\bbQ_p}(p)\in \Gamma_{\bbQ_p}^{\textnormal{ab}}$; in particular, $\varphi$ is a geometric Frobenius element and we have $\varepsilon(\varphi) = 1$.  The group $\Gamma_K$ is topologically generated by $\varphi^f$ and $I_K$ ( $= I_{\qp}$).

We let $K_2$ denote the unique unramified quadratic extension of $K$, and $\cO_{K_2}$ its ring of integers.  The group $\bU_1(K)\subseteq \cO_{K_2}^\times$ is defined as the kernel of the norm map $K_2^\times \longrightarrow K^\times$.

Fix a choice of root $\pi \defeq (-p)^{1/(p^{2f} - 1)} \in \overline{\bbQ}_p$.  We define a character $\widetilde{\omega}_{\pi}:\Gamma_{K_2} \longrightarrow \cO_{K_2}^\times$ by
$$\gamma \longmapsto \frac{\pi^\gamma}{\pi}.$$
We fix once and for all an embedding $\varsigma_0: K_2 \longhookrightarrow E$, and define 
$$\widetilde{\omega}_{2f} \defeq \varsigma_0 \circ \widetilde{\omega}_\pi : \Gamma_{K_2} \longrightarrow \cO^\times.$$
We denote by $\omega_{2f}$ the mod $p$ reduction of $\widetilde{\omega}_{2f}$.  Note that $\omega_{2f}^{(p^{2f} - 1)/(p - 1)} = \omega$.

\subsubsection{}

Let $\bU_2$ denote the algebraic group over $\cO_K$ given by
$$\bU_2(R) = \left\{g\in \bG\bL_2(\cO_{K_2}\otimes_{\cO_K}R): g^{(\varphi^f\otimes 1)\top}\Phi_2g = \Phi_2\right\}\footnote{{This group is quasi-split, and is customarily denoted $\bU_{1,1}$ in the literature.}},$$
where $R$ is an $\cO_K$-algebra, and $\Phi_2 \defeq \sm{0}{1}{-1}{0}$.

{Recall that the field $K_2$ is considered as a subfield of $\qpb$.  The projection $\cO_{K_2} \otimes_{\cO_K} \qpb \longrightarrow \qpb$ defined by $x \otimes y \longmapsto xy$ induces an isomorphism $\bU_2(\qpb) \stackrel{\sim}{\longrightarrow} \bG\bL_2(\qpb)$, and via this isomorphism $\bG\bL_2(\qpb)$ obtains a $\Gamma_K$-action given by}
$$\gamma\cdot g = \begin{cases}g^\gamma & \textnormal{if}~\gamma\in \Gamma_{K_2}, \\ \left(\Phi_2g^{-\top}\Phi_2^{-1}\right)^\gamma & \textnormal{if}~\gamma\in \Gamma_K\smallsetminus\Gamma_{K_2}.\end{cases}$$

\subsubsection{}

Following \cite[\S 5.3]{buzzardgee}, we set $\bH \defeq \widetilde{\bU}_2$, so that $\bH$ is a canonical central extension
$$1\longrightarrow\bG_m\longrightarrow \bH \longrightarrow \bU_2\longrightarrow 1$$
of algebraic groups over $\cO_{K}$.  {(To be precise, the construction of \cite{buzzardgee} which we outline below is done over $K$.  The integral model for $\bU_2$ above gives rise to a hyperspecial point in the semisimple Bruhat--Tits building of $\bU_2(K)$, which is identified with the semisimple Bruhat--Tits building of $\bH(K)$, since the extension defining $\bH$ is central.  We therefore obtain a hyperspecial point and the desired integral model for $\bH$.)}  {We will often abuse notation and conflate algebraic groups over $\cO_K$ with their generic fibers.} The group $\bH$ possesses a twisting element, in the terminology of \emph{op. cit.}.  We now recall the explicit construction of $\bH$.

We proceed as follows.  The group $\bH$ is defined as a pushout followed by a pullback:

\begin{center}
\begin{tikzcd}
1 \ar[r] & \ar[dr, phantom, "\square"] \boldsymbol{\mu}_2 \ar[r] \ar[d, hook] & \bS\bL_2 \ar[r]\ar[d] & \bP\bG\bL_2 \ar[r] \ar[d, equal] & 1\\
1 \ar[r] & \bG_m \ar[r] \ar[d, equal] & \ar[dr, phantom, "\square"] \bG\bL_2 \ar[r] & \bP\bG\bL_2 \ar[r] & 1 \\
1 \ar[r] & \bG_m \ar[r,"\imath"] & \bH \ar[u] \ar[r] & \bU_2 \ar[u, two heads] \ar[r] & 1
\end{tikzcd}
\end{center}

\noindent Concretely, $\bH$ is the set of all pairs $(h,h')$, with $h\in \bU_2, h'\in \bG\bL_2$, subject to the condition that $h$ and $h'$ have the same image in $\bP\bG\bL_2$.  The maps $\bH \longrightarrow \bU_2$ and $\bH\longrightarrow \bG\bL_2$ are the projections onto the corresponding factors, and the map $\imath:\bG_m\longrightarrow \bH$ is $\lambda\longmapsto (1, \sm{\lambda}{0}{0}{\lambda})$.

Note that {the $\qpb$-points of} the top two rows of the diagram above carry the standard (i.e., split) action {of $\Gamma_K$}.  In particular, the action {of $\Gamma_K$} on the first factor of ${\bH(\qpb)}$ is the one induced from ${\bU_2(\qpb)}$, while the action on the second factor is the standard one.

\subsubsection{}
\label{tori}

Let $\bT_{\bU}$ denote the diagonal maximal torus of $\bU_2$, and $\bT_\bH$ its preimage in $\bH$.  Furthermore, let $\bT_\bG, \bT_\bS$, and $\bT_\bP$ denote the diagonal maximal tori of $\bG\bL_2$, $\bS\bL_2$ and $\bP\bG\bL_2$, respectively.  The character groups of these tori fit into a diagram:

\begin{center}
\begin{tikzcd}[row sep=large, column sep=10ex]
0 \ar[r] & X^*(\bT_\bP)\cong \bbZ \ar[r, "a\mapsto 2a"] & X^*(\bT_\bS) \cong \bbZ \ar[r] \ar[dr, phantom, "\square"] & X^*(\boldsymbol{\mu}_2)\cong \bbZ/2\bbZ \ar[r] & 0 \\
0 \ar[r] &  X^*(\bT_\bP) \cong \bbZ  \ar[r, "a \mapsto (a{,}-a)"] \ar[u, equal] \ar[d, hook, "a\mapsto (a{,}-a)"] \ar[dr, phantom, "\square"] & X^*(\bT_\bG)\cong \bbZ^2  \ar[u, "(a{,}b)\mapsto a - b"'] \ar[r, "(a{,}b)\mapsto a + b"] \ar[d] & X^*(\bG_m) \cong \bbZ \ar[u, two heads] \ar[r] \ar[d, equal] & 0 \\
0 \ar[r] & X^*(\bT_\bU) \cong \bbZ^2 \ar[r] & X^*(\bT_{\bH}) \ar[r] & X^*(\bG_m)\cong \bbZ \ar[r] & 0
\end{tikzcd}
\end{center}

\noindent The isomorphisms appearing are the canonical ones.  (The notation $X^*(\bT_{\bullet})$, for $\bullet\in\{\bP,\, \bS,\, \bU,\, \bG,\, \bH\}$, stands for the character group of the torus $\bT_{\bullet}$ over $\overline{\mathbb{Q}}_p$.)

We describe the remaining character group.  The group $X^*(\bT_{\bH})$ is a pushout, so we may identify it as
$$X^*(\bT_{\bH}) = \left\{\vecf{a}{b}{c}{d}\in X^*(\bT_\bU)\oplus X^*(\bT_\bG)\cong \bbZ^4 \right\}\big/\sim$$
where
$$\vecf{a}{b}{c}{d}\sim \vecf{a + z}{b - z}{c - z}{d + z}$$
for $z\in \bbZ$.  
The maps $X^*(\bT_\bU)\longrightarrow X^*(\bT_{\bH}), X^*(\bT_\bG)\longrightarrow X^*(\bT_{\bH})$ are the inclusions into the corresponding factors, and the projection $X^*(\bT_{\bH})\longrightarrow X^*(\bG_m)\cong \bbZ$ is 
$$\vecf{a}{b}{c}{d}\longmapsto c + d.$$

\subsubsection{}

We now consider cocharacter groups.  The bottom two rows of the diagram above give the following commutative diagram:

\begin{center}
\begin{tikzcd}[row sep=large, column sep=10ex]
0 \ar[r] & X_*(\bG_m) \cong  \bbZ \ar[r, "a'\mapsto (a'{,}a')"] & X_*(\bT_\bG) \cong \bbZ^2 \ar[r, "(a'{,}b')\mapsto a' - b'" ] \ar[dr, phantom, "\square"] & X_*(\bT_\bP) \cong \bbZ \ar[r] & 0\\
0 \ar[r] & X_*(\bG_m) \cong  \bbZ \ar[r]\ar[u, equal] & X_*(\bT_{\bH}) \ar[r] \ar[u, two heads] & X_*(\bT_{\bU}) \cong \bbZ^2 \ar[r] \ar[u, two heads, "(a'{,}b')\mapsto a' - b'"'] & 0
\end{tikzcd}
\end{center}

\noindent The isomorphisms are again the canonical ones, {and we again consider the cocharacter groups of the tori over $\qpb$.}

We describe the remaining cocharacter group.  The group $X_*(\bT_{\bH})$ is a pullback, so we may identify it as
$$X_*(\bT_{\bH}) = \left\{\vecf{a'}{b'}{c'}{d'}\in X_*(\bT_\bU)\oplus X_*(\bT_\bG)\cong \bbZ^4 :a' - b' = c'- d'\right\}.$$
The maps $X_*(\bT_{\bH})\longrightarrow X_*(\bT_\bU), X_*(\bT_{\bH})\longrightarrow X_*(\bT_\bG)$ are the projections onto the corresponding factors, and the map $X_*(\bG_m)\cong \bbZ\longrightarrow X_*(\bT_{\bH})$ is 
$$a'\longmapsto \vecf{0}{0}{a'}{a'}.$$

\subsubsection{}

The actions of $\Gamma_K$ on $X^*(\bT_{\bH})$ and $X_*(\bT_{\bH})$ are the ones induced from $X^*(\bT_\bU)$ and $X_*(\bT_\bU)$: they are both unramified, and we have
$$\varphi^f\cdot\vecf{a}{b}{c}{d} =  \vecf{-b}{-a}{c}{d}$$
for $\vecf{a}{b}{c}{d}\in X^*(\bT_{\bH})$ and 
$$\varphi^f\cdot\vecf{a'}{b'}{c'}{d'} = \vecf{-b'}{-a'}{c'}{d'}$$
for $\vecf{a'}{b'}{c'}{d'}\in X_*(\bT_{\bH})$.

The pairing $\langle-,-\rangle:X^*(\bT_\bH)\times X_*(\bT_\bH)\longrightarrow \bbZ$ between characters and cocharacters is given by
$$\Big\langle \vecf{a}{b}{c}{d}, \vecf{a'}{b'}{c'}{d'}\Big\rangle = aa' + bb' + cc' + dd';$$
this is well-defined and Galois-invariant.  The roots $\Phi_\bH \subseteq X^*(\bT_{\bH})$ are given by $\{\pm\alpha_\bH\}$, where 
$$\alpha_\bH \defeq \vecf{1}{-1}{0}{0}.$$ 
Likewise, the coroots $\Phi_\bH^\vee \subseteq X_*(\bT_{\bH})$ are given by $\{\pm\alpha_{\bH}^\vee\}$ where 
$$\alpha_\bH^\vee \defeq \vecf{1}{-1}{1}{-1}.$$ 
We define the set of simple roots as $\Delta_\bH \defeq \{\alpha_\bH\}$, and let $\bB_\bH$ denote the corresponding Borel subgroup of $\bH$.  We therefore have $\Delta_\bH^\vee = \{\alpha_\bH^\vee\}$.

The group $\bH$ has a twisting element, in the sense of \cite{buzzardgee}: tracing through the construction in \emph{op. cit.}, we obtain
$$\eta_\bH \defeq \vecf{0}{0}{1}{0}\in X^*(\bT_\bH).$$
This element is Galois-invariant, and $\langle\eta_\bH,\alpha_\bH^\vee\rangle = 1$.

The Weyl group of $\bH$ with respect to $\bT_{\bH}$ is denoted $W_{\bH}$; it is a cyclic group of order 2.  We denote by $s$ the unique simple reflection, which generates $W_{\bH}$.

\subsection{Unitary groups over $\qp$}
\label{unitarygps:qp}

\subsubsection{}
\label{sub:sub:unitary:loc}
We now consider unitary groups over $\qp$.  We set 
$$(\bG,~\bB,~\bT) \defeq  \textnormal{Res}_{\cO_K/\zp}\left(\bH,~ \bB_{\bH},~\bT_{\bH} \right),$$
all group schemes over $\zp$.  We have
$$\bG(\qpb)\cong \Ind_{\Gamma_K}^{\Gamma_{\qp}}\left(\bH(\qpb)\right)$$ 
as $\Gamma_{\qp}$-groups, and via the evaluation maps, we have
{\begin{eqnarray*}
(\textnormal{ev}_1,\textnormal{ev}_{\varphi},\ldots, \textnormal{ev}_{\varphi^{f - 1}}):\bG(\qpb) & \stackrel{\sim}{\longrightarrow} & \prod_{i = 0}^{f - 1}\bH(\qpb)\\
f & \longmapsto & \left(f(\varphi^i)\right)_{0\leq i \leq f - 1}.
\end{eqnarray*}}

Recall that $\varphi^f$ acts on ${\bH(\qpb)}$ by
$$\varphi^f\cdot (h_1,h_2) = \left((\Phi_2h_1^{-\top}\Phi_2^{-1})^{\varphi^f}, h_2^{\varphi^f}\right).$$
Tracing through the isomorphisms above, the action of $\Gamma_{\qp}$ on the right-hand-side product is given as follows:
$${\varphi\cdot\big((h_{0,1}, h_{0,2}),h_1,\ldots, h_{f - 1}\big) = \left(h_1,\ldots, h_{f - 1},(\Phi_2h_{0,1}^{-\top}\Phi_2^{-1}, h_{0,2})^{\varphi^f}\right)}$$
with inertia acting in the standard, diagonal way.  In particular,
$$\textnormal{ev}_1:\bG(\qp) = \bG(\qpb)^{\Gamma_{\qp}}\stackrel{\sim}{\longrightarrow}\bH(\qpb)^{\Gamma_K} = \bH(K) \cong \widetilde{\bU}_2(K).$$

\subsubsection{}
\label{tori-direct-sum}

The character and cocharacter groups of the torus $\bT$ are given by 
$${X^*(\bT) \cong \textnormal{Ind}_{\Gamma_K}^{\Gamma_{\qp}}\big(X^*(\bT_{\bH})\big),\qquad X_*(\bT) \cong \textnormal{Ind}_{\Gamma_K}^{\Gamma_{\qp}}\big(X_*(\bT_{\bH})\big)}$$
{(cf.~ \cite[\S 9.4]{ghs}).  Using the evaluation maps as above (with the same ordering), we identify}
$${X^*(\bT) \cong \bigoplus_{i = 0}^{f - 1}X^*(\bT_{\bH}),\qquad X_*(\bT) \cong \bigoplus_{i = 0}^{f - 1}X_*(\bT_{\bH}).}$$
We will write elements of $X^*(\bT)$ as
$$\mu = \vecf{\un{a}}{\un{b}}{\un{c}}{\un{d}} = \vecf{a_0}{b_0}{c_0}{d_0}\vecf{a_1}{b_1}{c_1}{d_1}\cdots\vecf{a_{f - 1}}{b_{f - 1}}{c_{f - 1}}{d_{f - 1}}$$
(and similarly for $X_*(\bT)$).

The perfect pairing $\langle -,-\rangle: X^*(\bT)\times X_*(\bT)\longrightarrow \bbZ$ is given by
$$\Big\langle\vecf{\un{a}}{\un{b}}{\un{c}}{\un{d}},~\vecf{\un{a}'}{\un{b}'}{\un{c}'}{\un{d}'}\Big\rangle = \sum_{i = 0}^{f - 1} a_ia_i' + b_ib_i' + c_ic_i' + d_id_i',$$
and the action of $\Gamma_{\qp}$ on $X^*(\bT)$ is given by 
$$\varphi\cdot \vecf{a_0}{b_0}{c_0}{d_0}\vecf{a_1}{b_1}{c_1}{d_1}\cdots\vecf{a_{f - 1}}{b_{f - 1}}{c_{f - 1}}{d_{f - 1}} = \vecf{a_1}{b_1}{c_1}{d_1}\cdots\vecf{a_{f - 1}}{b_{f - 1}}{c_{f - 1}}{d_{f - 1}}\vecf{-b_0}{-a_0}{c_0}{d_0}.$$
An analogous action (i.e., with a ``shift left'') holds for $X_*(\bT)$.

\subsubsection{}

We define the simple roots $\Delta$ as those functions $f$ in $\Ind_{\Gamma_K}^{\Gamma_{\qp}}(X^*(\bT_{\bH}))$ with image in $\{0\}\cup\Delta_{\bH}$, and such that $f(\gamma) = 0$ for all but a single coset.  Explicitly, we have $\Delta = \{\alpha_i\}_{0\leq i \leq f- 1}$, where
$$\alpha_i \defeq \vecf{0}{0}{0}{0}\cdots \underbrace{\vecf{1}{-1}{0}{0}}_{i^{\textnormal{th}}~\textnormal{entry}}\cdots\vecf{0}{0}{0}{0}\in X^*(\bT).$$  
We define $\Delta^{\vee}$ analogously, and obtain $\Delta^\vee = \{\alpha_i^\vee\}_{0\leq i \leq f- 1}$, where
$$\alpha_i^\vee \defeq \vecf{0}{0}{0}{0}\cdots \underbrace{\vecf{1}{-1}{1}{-1}}_{i^{\textnormal{th}}~\textnormal{entry}}\cdots\vecf{0}{0}{0}{0}\in X_*(\bT).$$

The Weyl group $W$ of $\bG$ with respect to $\bT$ is equal to $W_{\bH}^{f}$.  We shall write elements of $W$ as $w = (w_0,w_1,\ldots, w_{f - 1})$.  The group $W$ has a nontrivial Galois action given by
$$\varphi\cdot (w_0,w_1,\ldots w_{f - 1}) = (w_1,\ldots w_{f - 1}, w_0).$$
Finally, we define $\un{1} \defeq (1,1,\ldots, 1)$ and $\un{s} \defeq (s,s,\ldots, s)$.

The map $\textnormal{ev}_1$ induces a bijection $X^*(\bT)^{\Gamma_{\qp}} \stackrel{\sim}{\longrightarrow} X^*(\bT_{\bH})^{\Gamma_K}$.  In particular, the twisting element $\eta_{\bH}\in X^*(\bT_{\bH})^{\Gamma_K}$ corresponds to the twisting element 
$$\eta \defeq \vecf{\underline{0}}{\underline{0}}{\underline{1}}{\underline{0}} = \vecf{0}{0}{1}{0}\vecf{0}{0}{1}{0}\cdots\vecf{0}{0}{1}{0}\in X^*(\bT)^{\Gamma_{\qp}}.$$

\subsection{Dual groups}
\label{dualgps}

We now define the relevant Langlands dual groups.

\subsubsection{}

The based root datum of $\bU_2$ (with respect to the upper-triangular Borel subgroup) is given by 
$$\left(X^*(\bT_\bU)\cong \bbZ^2 ,~\{(1,-1)\},~ X_*(\bT_\bU)\cong \bbZ^2 ,~ \{(1,-1)\}\right).$$
Therefore, we may take $\widehat{\bU}_2 \defeq \bG\bL_2$ as the dual group, which we consider as a split group scheme over $\zp$, along with its diagonal maximal torus, upper-triangular Borel subgroup, and the fixed isomorphism between $\bG_a$ and the unipotent radical of the Borel given by $x\longmapsto \sm{1}{x}{0}{1}$.  We equip this data with the canonical isomorphism between the based root datum of $\widehat{\bU}_2$ and the dual based root datum of $\bU_2$.  In choosing this isomorphism, we obtain an induced action of $\Gamma_K$ on $\widehat{\bU}_2$ given by
$$\gamma\cdot \widehat{g} = \begin{cases} \widehat{g} & \textnormal{if}~ \gamma\in \Gamma_{K_2},\\ \Phi_2\widehat{g}^{-\top}\Phi_2^{-1} = \begin{pmatrix}\det(\widehat{g})^{-1}& 0 \\ 0 & \det(\widehat{g})^{-1}\end{pmatrix}\widehat{g} & \textnormal{if}~ \gamma\in \Gamma_K\smallsetminus\Gamma_{K_2}.\end{cases}$$

\subsubsection{}
\label{subsub:dual:root}
Consider now the group $\bH = \widetilde{\bU}_2$.  The based root datum of $\bH$ is given by 
$$\Psi_\bH \defeq \left(X^*(\bT_\bH),~ \Delta_\bH,~ X_*(\bT_\bH),~ \Delta_\bH^\vee\right),$$
and therefore the dual based root datum is 
$$\Psi_\bH^\vee = \left(X_*(\bT_\bH),~ \Delta_\bH^\vee,~ X^*(\bT_\bH),~ \Delta_\bH\right).$$
We let $\widehat{\bH}$ denote the dual group of $\bH$, with maximal torus $\widehat{\bT}_{\bH}$ and Borel $\widehat{\bB}_{\bH}$ which contains $\widehat{\bT}_{\bH}$.  By \cite[Prop. 5.39]{buzzardgee}, we have
$$\widehat{\bH} \cong (\widehat{\bU}_2\times\bG_m)\bigg/\left\langle\left(\begin{pmatrix}-1 & 0 \\ 0 & -1\end{pmatrix},-1\right)\right\rangle = \bG\bL_2\times^{{\boldsymbol{\mu}_2}}\bG_m,$$
where the Galois action on $\widehat{\bH}$ is the one induced from $\widehat{\bU}_2$.  We have an isomorphism 
\begin{eqnarray*}
\widehat{\bH} = \bG\bL_2\times^{{\boldsymbol{\mu}_2}}\bG_m & \stackrel{\sim}{\longrightarrow} & \bG\bL_2\times\bG_m\\
 {[}\widehat{h},a] & \longmapsto & \left(\begin{pmatrix} a & 0 \\ 0 & a \end{pmatrix}\widehat{h}, a^2\right)
\end{eqnarray*}
and we will identify $\widehat{\bH}$ with $\bG\bL_2\times \bG_m$ via this isomorphism.  The Galois action is then given by
$$\gamma\cdot (\widehat{h},a) = \begin{cases}(\widehat{h},a) & \textnormal{if}~\gamma\in \Gamma_{K_2},\\ \left(\begin{pmatrix}a & 0 \\ 0 & a\end{pmatrix}\Phi_2\widehat{h}^{-\top}\Phi_2^{-1}, a\right) = \left(\begin{pmatrix}a\det(\widehat{h})^{-1}& 0 \\ 0 & a\det(\widehat{h})^{-1}\end{pmatrix}\widehat{h},a\right) & \textnormal{if}~\gamma\in \Gamma_K\smallsetminus\Gamma_{K_2},\end{cases}$$
for $(\widehat{h},a)\in\bG\bL_2\times\bG_m$.

Thus, we obtain the based root datum for $\widehat{\bH}$
\begin{eqnarray*}
\Psi_{\widehat{\bH}} & \defeq & \left(X^*(\widehat{\bT}_{\bH}),~ \widehat{\Delta},~ X_*(\widehat{\bT}_{\bH}),~ \widehat{\Delta}^{\vee}\right)\\
 & = & \left(\bbZ^3,~ \{(1,-1,0)\},~ \bbZ^3,~ \{(1,-1,0)\}\right),
 \end{eqnarray*}
equipped with an action of $\Gamma_K$.  Moreover, we obtain an isomorphism of based root data $\phi:\Psi_\bH^\vee\stackrel{\sim}{\longrightarrow}\Psi_{\widehat{\bH}}$:
\begin{eqnarray*}
 \phi:X_*(\bT_\bH) & \stackrel{\sim}{\longrightarrow} & X^*(\widehat{\bT}_{\bH})\\
 (a',b',c',d') & \longmapsto & (a',b',c'-a')\\
 (\phi^\vee)^{-1}:X^*(\bT_\bH) & \stackrel{\sim}{\longrightarrow} & X_*(\widehat{\bT}_{\bH})\\
 (a,b,c,d) & \longmapsto & (a + c,b + d,c + d)
\end{eqnarray*}
{where the last coordinate in the character (resp.~cocharacter) group of $\widehat{\bT}_{\bH}$ corresponds to the $\bG_m$ factor of $\widehat{\bH}$.}  Note that this exchanges the roots and coroots.  We use this isomorphism {to} identify the Weyl group of $\widehat{\bT}{_{\bH}}$ with $W_{\bH}$.

\subsubsection{}

Finally, we define
$${}^C\bU_2 \defeq {}^L\bH = \widehat{\bH}\rtimes \textnormal{Gal}(K_2/K) = (\bG\bL_2\times\bG_m)\rtimes\textnormal{Gal}(K_2/K),$$
with the Galois group acting on $\widehat{\bH}$ as above.  The injection $\imath:\bG_m\longrightarrow \bH$ induces a dual map $\widehat{\imath}:{}^L\bH\longrightarrow\bG_m$, which is given by $(\widehat{h},a)\rtimes\gamma\longmapsto a$.

\begin{rmk}
\label{cgroupglobal}

We will need to make use of the above construction in a global setting as follows.  Suppose $F/F^+$ is a quadratic extension of global fields, and let $v$ denote a place of $F^+$ which is unramified and inert in $F$, and such that $F^+_v \cong K$ and $F_v \cong K_2$.  We then identify ${}^C\bU_2$ with 
$$\widehat{\bH}\rtimes \textnormal{Gal}(F/F^+)$$ 
via the isomorphism $\textnormal{Gal}(F/F^+) \cong \textnormal{Gal}(F_v/F^+_v) \cong \textnormal{Gal}(K_2/K)$.   
\end{rmk}

\subsubsection{}

We set
$$\left(\widehat{\bG},~\widehat{\bB},~\widehat{\bT}\right) \defeq \Ind_{\Gamma_K}^{\Gamma_{\qp}}\left(\widehat{\bH},~\widehat{\bB}_{\bH},~\widehat{\bT}_{\bH}\right),$$
all group schemes over $\zp$, equipped with the induced $\Gamma_{\qp}$-action.  Using the (induced versions of the) isomorphisms above, we consider $\widehat{\bG}$ as the dual group of $\bG$, and set 
$${}^L\bG \defeq \widehat{\bG}\rtimes \textnormal{Gal}(K_2/\qp).$$

\subsection{An isomorphism}\label{isomsect}

We briefly digress to recall a construction of ${}^C\bU_2$ from \cite{CHT} (see also \cite[\S 8.3]{buzzardgee}).

Let $\cG_2$ denote the group scheme over $\zp$ which is a semidirect product of $\bG\bL_2\times\bG_m$ by $\textnormal{Gal}(K_2/K)$, with $\varphi^{f}\in \textnormal{Gal}(K_2/K)$ acting by 
$$\varphi^{f}\cdot (\widehat{h},a) = \left(\begin{pmatrix}a & 0 \\ 0 & a \end{pmatrix}\widehat{h}^{-\top}, a\right).$$
There is an isomorphism between our model ${}^C\bU_2$ and $\cG_2$ given as follows:
\begin{eqnarray*}
{}^C\bU_2 & \stackrel{\sim}{\longrightarrow} & \cG_2\\
(\widehat{h},a)\rtimes 1 & \longmapsto & \left(\begin{pmatrix}a^{-1} & 0 \\ 0 &  a^{-1}\end{pmatrix}\widehat{h}, a^{-1}\right)\rtimes 1\\
(\widehat{h},a)\rtimes \varphi^{f} & \longmapsto & \left(\begin{pmatrix}a^{-1} & 0 \\ 0 &  a^{-1}\end{pmatrix}\widehat{h}\Phi_2, -a^{-1}\right)\rtimes \varphi^{f}\\
\left(\begin{pmatrix}a^{-1} & 0 \\ 0 & a^{-1}\end{pmatrix}\widehat{h},a^{-1}\right) \rtimes 1 & \longmapsfrom & (\widehat{h},a)\rtimes 1\\
\left(\begin{pmatrix}a^{-1} & 0 \\ 0 & a^{-1}\end{pmatrix}\widehat{h}\Phi_2, -a^{-1}\right)\rtimes\varphi^{f} & \longmapsfrom & (\widehat{h},a)\rtimes\varphi^{f}.
\end{eqnarray*}

The group $\cG_2$ also possesses a map $\nu:\cG_2\longrightarrow \bG_m$, given by $(\widehat{h},a)\rtimes (\varphi^{f})^i\longmapsto (-1)^ia$.  Under the isomorphism above, this corresponds to the map ${(-)^{-1}\circ\widehat{\imath}}:{}^C\bU_2\longrightarrow \bG_m$.

As in Remark \ref{cgroupglobal}, we will often identify $\cG_2$ with $(\bG\bL_2\times\bG_m)\rtimes\textnormal{Gal}(F/F^+)$.

\subsection{$\mathbb{F}_p$-structures}

\subsubsection{}

Viewing $\bG$ and $\widehat{\bG}$ as group schemes over $\zp$, we can form the $\Fpbar$-group schemes:
$$(\underline{\bG}, \underline{\bB}, \underline{\bT}) \defeq (\bG, \bB, \bT)\times_{\zp}\Fpbar,\qquad (\underline{\bG}^*, \underline{\bB}^*, \underline{\bT}^*) \defeq (\widehat{\bG}, \widehat{\bB}, \widehat{\bT})\times_{\zp}\Fpbar.$$
{We denote by $\biF$ the relative Frobenius on $\underline{\bG}$, and denote by $\biF^*$ the composite $\textnormal{Fr} \circ \varphi$, where $\textnormal{Fr}$ is the relative Frobenius on the split group $\underline{\bG}^*$.  In particular, we have $\underline{\bG}^{\biF} = \bG(\Fp) = \bH(\Fq)$. }

{The action of $\biF$ on $X^*(\underline{\bT})$ is defined by $\biF(\chi) = \chi\circ\biF$, so that $\biF = p\varphi$ on $X^*(\underline{\bT})$.  Identifying $X^*(\underline{\bT}) \cong X^*(\bT)$ with $\bigoplus_{i = 0}^{f - 1} X^*(\bT_\bH)$ as in Subsubsection \ref{tori-direct-sum}, this action is explicitly given by
$$\biF\vecf{a_0}{b_0}{c_0}{d_0}\vecf{a_1}{b_1}{c_1}{d_1}\cdots \vecf{a_{f - 1}}{b_{f - 1}}{c_{f - 1}}{d_{f - 1}} =  \vecf{pa_1}{pb_1}{pc_1}{pd_1}\cdots \vecf{pa_{f - 1}}{pb_{f - 1}}{pc_{f - 1}}{pd_{f - 1}}\vecf{-pb_0}{-pa_0}{pc_0}{pd_0}$$}

{Similarly, the action of $\biF^*$ on $X_*(\underline{\bT}^*)$ is given by $\biF^*(\lambda) = \biF^*\circ\lambda$, so that $\biF^* = p\varphi$ on $X_*(\underline{\bT}^*)$.  Therefore, after chasing through the isomorphisms of root data of Subsubsection \ref{subsub:dual:root} and using the identification $X_*(\underline{\bT}^*) \cong X_*(\widehat{\bT}) \cong  \bigoplus_{i = 0}^{f - 1} X_*(\widehat{\bT}_\bH)$ similar to above, this map is explicitly given by
$$\biF^*\Big(\begin{smallmatrix}a_0 \\ b_0\\ c_0\end{smallmatrix}\Big)\Big(\begin{smallmatrix}a_1 \\ b_1\\ c_1\end{smallmatrix}\Big)\cdots \Big(\begin{smallmatrix}a_{f - 1} \\ b_{f - 1}\\ c_{f - 1}\end{smallmatrix}\Big) = \Big(\begin{smallmatrix}pa_1 \\ pb_1\\ pc_1\end{smallmatrix}\Big)\cdots \Big(\begin{smallmatrix}pa_{f - 1} \\ pb_{f - 1}\\ pc_{f - 1}\end{smallmatrix}\Big)\Big(\begin{smallmatrix}p(c_0 - b_0) \\ p(c_0 - a_0)\\ pc_0\end{smallmatrix}\Big).$$}

\section{Representation theory}
\label{sec:RepThy}

We now collect various results we will use regarding types and weights for the groups $\widetilde{\bU}_2(\bbF_q)$ and $\bG\bL_2(\bbF_{q^2})$.  We give definitions of base change for both types and weights in Subsections \ref{autBC} and \ref{subsec:BC:WT}, respectively, and relate the former to automorphic base change.  Subsection \ref{subsec:cmb:types} discusses various compatibilities between types and weights, and contains useful combinatorial properties which will be employed extensively in the applications which follow.

\subsection{The group $\bG$}
\label{subsec:thegroupG}

\subsubsection{}
\label{rep-theory-G}

Let $X_+(\bT), X_1(\bT)$ and $X^0(\bT)$ denote respectively the subsets of $X^*(\bT)$ consisting of dominant, $p$-restricted, and inner-product-zero elements:
\begin{eqnarray*}
X_+(\bT) & \defeq & \{\mu \in X^*(\bT):0 \leq \langle \mu, \alpha_i^\vee\rangle ~\textnormal{for all}~0\leq i \leq f - 1\}\\
X_1(\bT) & \defeq & \{\mu \in X^*(\bT):0 \leq \langle \mu, \alpha_i^\vee\rangle \leq p - 1 ~\textnormal{for all}~0\leq i \leq f - 1\}\\
X^0(\bT) & \defeq & \{\mu \in X^*(\bT):\langle \mu, \alpha_i^\vee\rangle = 0~\textnormal{for all}~0\leq i \leq f - 1\}.
\end{eqnarray*}

\subsubsection{}

Recall that a \emph{Serre weight of $\bG(\bbF_p)$} is an irreducible representation of $\bG(\bbF_p)$ on an $\overline{\bbF}_p$-vector space.  Given $\mu \in X_+(\bT)$, we let $F(\mu)$ denote the restriction to $\bG(\bbF_p)$ of the algebraic $\bG$-representation of highest weight $\mu$.  We then have the following result.

\begin{prop}[\cite{ghs}, Lemma 9.2.4] \label{serrewtparam}
The map
\begin{eqnarray*}
\frac{X_1(\bT)}{(\biF - 1)X^0(\bT)} & \longrightarrow & \{\textnormal{Serre weights of}~\bG(\bbF_p)\}_{/\cong}\\
\mu & \longmapsto & F(\mu)
\end{eqnarray*}
is a well-defined bijection.  
\end{prop}

\noindent We will always assume that the coefficient field $\bbF$ is large enough so that the representations $F(\mu)$ may be realized over $\bbF$.

\begin{df}
Given a character $\mu\in X^*(\bT)$, we say $\mu$ lies \emph{$n$-deep} in the fundamental alcove if we have
$$n < \langle\mu + \eta,\alpha_i^\vee \rangle < p - n$$
for all $0 \leq i \leq f - 1$.  We say a Serre weight $F$ is \emph{$n$-deep} if we can write $F \cong F(\mu)$ for some $n$-deep character $\mu$.  (Note that this notion is independent of the choice of $\mu$.)
\end{df}

\subsubsection{}

We likewise consider Deligne--Lusztig representations for the group $\bG(\bbF_p)$, as in \cite[\S 9.2]{ghs}.  In particular, for $w \in W$ and $\mu\in X^*(\bT)$ such that $(\bT_{w}, \theta_{w, \mu})$ is maximally split, we let $R_{w}(\mu)$ denote the associated Deligne--Lusztig representation, a representation of $\bG(\bbF_p)$ over $\overline{\bbQ}_p$.  
{Note that if $\mu-\eta$ is 0-deep, then $(\bT_{w}, \theta_{w, \mu})$ is maximally split for any choice of $w \in W$, cf.~\cite[Lem. 2.2.3]{LLL}.}  We again assume the coefficient field $E$ is large enough so that $R_w(\mu)$ may be realized over $E$.  Using the surjection $\bG(\bbZ_p) \longtwoheadrightarrow \bG(\bbF_p)$, we will occasionally view Serre weights and Deligne--Lusztig representations as representations of the compact group $\bG(\bbZ_p) \cong \widetilde{\bU}_2(\cO_K)$.

{By \cite[\S 4.1]{herzig-duke}, if $(w, \mu)\in W\times X^*(\bT)$ with $\mu - \eta$ being 0-deep, and if $(\nu,v) \in X^*(\bT) \rtimes W$, then we have an isomorphism
\begin{equation}
\label{DLequiv}
R_w(\mu) \cong R_{v w\biF(v)^{-1}}\left(v(\mu) + \biF(\nu) - v w\biF(v)^{-1}(\nu)\right).
\end{equation}}
{Moreover, by \cite[Lem. 4.2]{herzig-duke}, if $(w,\mu), (w', \mu') \in W\times X^*(\bT)$ are two pairs with $\mu - \eta$ and $\mu' - \eta$ being 0-deep, and we have an isomorphism $R_w(\mu)\cong R_{w'}(\mu')$, then $(w,\mu)$ and $(w',\mu')$ lie in the same $X^*(\bT)\rtimes W$-orbit.}

\begin{df}
Let $\sigma$ denote a Deligne--Lusztig representation.  We say $\sigma$ is \emph{$n$-generic} if there is an isomorphism $\sigma \cong R_w(\mu + \eta)$, where $\mu$ lies $n$-deep in the fundamental alcove.  
\end{df}

\subsubsection{}

We shall also need to know how the representations $R_{w}(\mu)$ decompose upon reduction mod $p$.  To this end, we define the following elements of $X^*(\bT)$.  Fix $w = (w_0,w_1,\ldots, w_{f - 1})\in W$, and set
\begin{align*}
\rho_{w}  \defeq  \cdots\underbrace{\vecf{0}{0}{0}{0}}_{w_i = 1}\cdots \underbrace{\vecf{1}{0}{0}{0}}_{w_i = s}\cdots,& \qquad \varepsilon_w  \defeq  \cdots\underbrace{\vecf{0}{0}{0}{0}}_{w_i = 1}\cdots \underbrace{\vecf{0}{1}{0}{0}}_{w_i = s}\cdots,\\
\gamma_w  \defeq   \cdots\underbrace{\vecf{1}{1}{0}{0}}_{w_i = 1}\cdots \underbrace{\vecf{0}{0}{0}{0}}_{w_i = s}\cdots, & \qquad \rho  \defeq  \vecf{\un{1}}{\un{0}}{\un{0}}{\un{0}}.
\end{align*}

Suppose that $\mu\in X^*(\bT)$ is such that $\mu - \eta$ is $1$-deep.  By the main theorem in the appendix of \cite{herzig-duke}, we have
\begin{equation}
\label{jantzen}
\JH\big(\overline{R_w(\mu)}\big) = \left\{F_{w'}(R_w(\mu))\right\}_{w'\in W},
\end{equation}
where
\begin{equation}
\label{def:JHcomps}
F_{w'}(R_w(\mu))  \defeq  F\left(p\gamma_{w'} + w'(\mu - w\pi(\varepsilon_{\un{s}w'})) + p\rho_{w'} - \pi(\rho)\right),
\end{equation}
and where $\pi$ denotes the action of $\varphi^{-1}$ on $X^*(\bT)$.

\begin{df}
{Let $\sigma$ be a $1$-generic Deligne--Lusztig representation and fix a presentation $\sigma\cong R_w(\mu)$ with $w \in W$ and $\mu - \eta$ being $1$-deep.  We define the Deligne--Lusztig representation $\beta(\sigma)$ by 
$$\beta(\sigma) \defeq R_{\un{s}w}\left(\un{s}(\mu - \eta)+ (p - 1)\eta\right).$$
By \cite[Prop. 2.2.15]{LLL} and equation (\ref{DLequiv}) one easily checks that $\beta(\sigma)$ does not depend on the choice of presentation $\sigma \cong R_w(\mu)$.  Moreover, note that if $n < \langle \mu, \alpha_i^\vee\rangle < p - n$, then $\un{s}(\mu - \eta)+ (p - 1)\eta$ will satisfy the same set of inequalities.  Therefore if $R_w(\mu)$ is $n$-generic, then $\beta(R_w(\mu))$ will also be $n$-generic.}
\end{df}

\subsubsection{}\label{trivactiondescent}

Finally, suppose $\mu\in X^*(\bT)$ is a character of the form
$$\mu = \vecf{a_0}{b_0}{0}{0}\vecf{a_1}{b_1}{0}{0}\cdots\vecf{a_{f - 1}}{b_{f - 1}}{0}{0}$$
(that is, suppose $\mu$ is in the image of $X^*(\textnormal{Res}_{\cO_K/\bbZ_p}(\bT_\bU)) \longhookrightarrow X^*(\bT)$).  Then $R_w(\mu)$ is a representation of $\bG(\zp) = \widetilde{\bU}_2(\cO_K)$ on which $\imath(\cO_K^\times)$ acts trivially, and therefore we view it as a representation of $\widetilde{\bU}_2(\cO_K)/\imath(\cO_K^\times) = \bU_2(\cO_K)$.  Conversely, if $R_w(\mu)$ is a representation of $\bG(\zp)$ on which $\imath(\cO_K^\times)$ acts trivially, then $\mu$ is a character of the form
$$\mu = \vecf{a_0}{b_0}{c_0}{0}\vecf{a_1}{b_1}{c_1}{0}\cdots\vecf{a_{f - 1}}{b_{f - 1}}{c_{f - 1}}{0}$$
with $\sum_{i = 0}^{f - 1} c_ip^i \equiv 0~ (\textnormal{mod}~p^f - 1)$.  By applying the equivalence $R_w(\mu) \cong R_w(\mu + (\biF - w)\mu')$ for an appropriately defined element $\mu' \in X^*(\bT)$ and using the equivalence relation on $X^*(\bT_{\bH})$, we may assume $\mu$ is of the form
$$\mu = \vecf{a_0'}{b_0'}{0}{0}\vecf{a_1'}{b_1'}{0}{0}\cdots\vecf{a_{f - 1}'}{b_{f - 1}'}{0}{0}$$
(one can even take $\mu' \in X^0(\bT)$).

\subsection{The group $\bG\bL_2$}

\subsubsection{}\label{subsubsec:groupG'}

Set 
$$(\bG', \bT') \defeq \Res_{\cO_{K_2}/\zp}(\bG\bL_{2/\cO_{K_2}}, \bT_{\bG/\cO_{K_2}}),$$ 
so that $\bG'(\bbF_p) = \bG\bL_2(\bbF_{q^2})$.  We identify the maximal torus $\bT_{\bG/\cO_{K_2}}$ of $\bG\bL_{2/\cO_{K_2}}$ with $\bT_{\bU}\times_{\cO_K}\cO_{K_2}$.  {This gives isomorphisms
\begin{eqnarray*}
\bT'\times_{\bbZ_p}\qpb & \cong & \prod_{i = 0}^{2f - 1}(\bT_{\bU} \times_{\cO_K}\cO_{K_2}) \times_{\cO_{K_2},\varphi^i} \qpb\\
& \cong & \Bigg(\prod_{i = 0}^{f - 1}\bT_{\bU}  \times_{\cO_{K},\varphi^i} \qpb \Bigg) \times \Bigg(\prod_{i = f}^{2f - 1}\bT_{\bU}  \times_{\cO_{K},\varphi^i} \qpb\Bigg)\\
& \cong & \left(\textnormal{Res}_{\cO_K/\bbZ_p}(\bT_{\bU}) \times_{\bbZ_p} \qpb\right) \times \left(\textnormal{Res}_{\cO_K/\bbZ_p}(\bT_{\bU}) \times_{\bbZ_p} \qpb\right),
\end{eqnarray*}
where we conflate $\textnormal{Gal}(K_2/\bbQ_p)$ with the set of embeddings of $K_2$ into $\qpb$.  In this way}, we identify $X^*(\bT')$ with two copies of $X^*(\Res_{\cO_K/\zp}(\bT_\bU))$.  We write elements of $X^*(\bT')$ as $(\mu,\mu')$ with $\mu,\mu'\in X^*(\Res_{\cO_K/\zp}(\bT_\bU))$.  In particular, if $\mu$ is an element of $X^*(\bT)$ of the form 
$$\mu = \vecf{a_0}{b_0}{0}{0}\vecf{a_1}{b_1}{0}{0}\cdots\vecf{a_{f - 1}}{b_{f - 1}}{0}{0},$$
we will identify it with 
$$\vect{a_0}{b_0}\vect{a_1}{b_1}\cdots\vect{a_{f - 1}}{b_{f - 1}} \in X^*(\Res_{\cO_K/\zp}(\bT_\bU)),$$
and consider expressions such as $(\mu,\mu)$ or $(\mu,-\un{s}(\mu))$ in $X^*(\bT')$.

Given some object (homomorphism, character, etc.) associated to $\bG$, we will denote by a prime the analogous object associated to $\bG'$.  For example, the notation $\biF'$ will be used to denote the Frobenius map on $\bG'\times_{\zp}\overline{\bbF}_p$, and on the character lattice $X^*(\bT')$:
$$\biF'\vect{a_0}{b_0}\vect{a_1}{b_1}\cdots \vect{a_{2f - 1}}{b_{2f - 1}} = \vect{pa_1}{pb_1}\cdots \vect{pa_{2f - 1}}{pb_{2f - 1}}\vect{pa_0}{pb_0}$$

We identify the Weyl group $W'$ of $\bT'$ with two copies of $W$, and we will sometimes write elements of $W'$ as $(w, w')$ where $w,w'\in W$.

\subsubsection{}

We define the subsets $X_+(\bT'), X_1(\bT')$ and $X^0(\bT')$ as above, and denote by $F'(\mu)$ the restriction to $\bG'(\bbF_p) \cong \bG\bL_2(\bbF_{q^2})$ of the algebraic $\bG'$-representation of highest weight $\mu\in X_+(\bT')$.  In particular, we have the following classification result.

\begin{prop}[\cite{ghs}, Lemma 9.2.4]
\label{Serre-wt-classn-GL2}
The map
\begin{eqnarray*}
\frac{X_1(\bT')}{(\biF' - 1)X^0(\bT')} & \longrightarrow & \{\textnormal{Serre weights of}~\bG\bL_2(\bbF_{q^2})\}_{/\cong}\\
\mu & \longmapsto & F'(\mu)
\end{eqnarray*}
is a well-defined bijection.  
\end{prop}

Once again, we will assume that $\bbF$ is large enough so that $F'(\mu)$ may be realized over $\bbF$.

\subsubsection{}

We define Deligne--Lusztig representations $R'_{w}(\mu)$ for $w\in W', \mu\in X^*(\bT')$ analogously to the above.  We again assume that $R'_w(\mu)$ may be realized over $E$.  Furthermore, an analog of equation \eqref{DLequiv} holds.  We will often view $F'(\mu)$ and $R'_w(\mu)$ as representations of $\bG'(\bbZ_p) \cong \bG\bL_2(\cO_{K_2})$ by inflation.

\subsubsection{}
\label{def-of-pi-prime}

Given $w = (w_0,w_1,\ldots, w_{2f - 1})\in W'$, we define the following elements of $X^*(\bT')$:
\begin{align*}
\rho_{w}'  \defeq \cdots\underbrace{\vect{0}{0}}_{w_i = 1}\cdots \underbrace{\vect{1}{0}}_{w_i = s}\cdots,& \qquad \varepsilon_{w}'  \defeq  \cdots\underbrace{\vect{0}{0}}_{w_i = 1}\cdots \underbrace{\vect{0}{1}}_{w_i = s}\cdots,\\
\gamma_{w}'  \defeq   \cdots\underbrace{\vect{1}{1}}_{w_i = 1}\cdots \underbrace{\vect{0}{0}}_{w_i = s}\cdots, & \qquad \rho'  \defeq  \vect{\un{1}}{\un{0}}.\\
\end{align*}

Suppose $\mu \in X^*(\bT')$ is such that $\mu - \rho'$ is $1$-deep.  The analog of \eqref{jantzen} takes the following form: 
\begin{equation}
\label{jantzenGL2}
\JH\big(\overline{R'_w(\mu)}\big) = \left\{F'_{w'}(R'_w(\mu))\right\}_{w'\in W'},
\end{equation}
where 
$$F'_{w'}(R'_w(\mu)) \defeq F'\left(p\gamma_{w'}' + w'\big(\mu  - w\pi'(\varepsilon_{(\un{s},\un{s})w'}')\big)+ p\rho_{w'}' - \rho'\right),$$
and where $\pi'$ is the automorphism of $X^*(\bT')$ such that $\biF' = p\pi'^{-1}$.

\begin{df} 
{Let $\sigma'$ be a 1-generic Deligne–Lusztig representation and fix a presentation $\sigma \cong R'_{w}(\mu)$ with $w \in W'$ and $\mu - \rho'$ being 1-deep. We define the Deligne–Lusztig representation $\beta'(\sigma')$ by
$$\beta'(\sigma') \defeq R'_{(\un{s},\un{s})w}\left((\un{s},\un{s})(\mu - \rho')+ (p - 1)\rho'\right).$$}
{As above for the map $\beta$, the above expression is well-defined, and if $R'_{w}(\mu)$ is $n$-generic, then $\beta'(R'_w(\mu))$ will also be $n$-generic.}
\end{df}

\subsection{Base change of types}\label{autBC}
Our next task will be to define a notion of base change for tame types of $\bU_2(\cO_K)$. We note that this is essentially the Shintani lifting considered in \cite{kawanaka}.

\subsubsection{}

We first recall the classification of irreducible representations of $\bU_2(\bbF_q)$ in characteristic zero (cf. \cite{ennola}).

Fix a character 
$$\psi: \bbF_{q^2}^\times  \longrightarrow \cO^\times,$$
which we also view as a character of $\bB_\bU(\bbF_q)$ via
$$\psi\left(\begin{pmatrix} x & y \\ 0 & x^{-q}\end{pmatrix}\right) = \psi(x),$$
where $x\in \bbF_{q^2}^\times, y\in \bbF_{q^2}$, and $xy^q = yx^q$.  (Here $\bB_{\bU}$ denotes the upper triangular Borel subgroup of $\bU_2$.)  We let $\Ind_{\bB_\bU(\bbF_q)}^{\bU_2(\bbF_q)}(\psi)$ denote the induced representation.  If $\psi^{-q}\neq \psi$, then $\Ind_{\bB_\bU(\bbF_q)}^{\bU_2(\bbF_q)}(\psi)$ is irreducible.  On the other hand, if $\psi^{-q} = \psi$, then $\psi$ extends to a character of $\bU_2(\bbF_q)$, and we have
$$\Ind_{\bB_\bU(\bbF_q)}^{\bU_2(\bbF_q)}(\psi) \cong \psi~ \oplus~ \big(\psi\otimes_E\textnormal{St}\big),$$
where $\textnormal{St}$ denotes the irreducible representation $\Ind_{\bB_\bU(\bbF_q)}^{\bU_2(\bbF_q)}(\mathbf{1})/\mathbf{1}$.

Consider now the group $\bJ_{\mathrm{end}} \defeq \bU_1\times \bU_1$ over $\cO_K$, which is the unique elliptic endoscopic group of $\bU_2$.  Fix a character 
\begin{eqnarray*}
\theta = \theta_1\otimes\theta_2: \bJ_{\mathrm{end}}(\bbF_q) = \bU_1(\bbF_q)\times\bU_1(\bbF_q) & \longrightarrow & \cO^{\times}\\
(x,y) & \longmapsto & \theta_1(x)\theta_2(y).
\end{eqnarray*}
We suppose that $\theta_1\neq \theta_2$, and let $\sigma(\theta)$ denote the associated irreducible cuspidal representation of $\bU_2(\bbF_q)$, as in \cite[\S 3.1(b)]{blasco}.

We have the following classification theorem.

\begin{theo}[\cite{ennola}]
\label{thm:ennola}
Any irreducible representation of $\bU_2(\bbF_q)$ over $E$ is isomorphic to one of the following:
\begin{itemize}
\item $\psi$, where $\psi$ is a character of $\bU_2(\bbF_q)$;
\item $\psi\otimes_E\textnormal{St}$, where $\psi$ is a character of $\bU_2(\bbF_q)$;
\item $\textnormal{Ind}_{\bB_{\bU}(\bbF_q)}^{\bU_2(\bbF_q)}(\psi)$, where $\psi$ is a character of $\bbF_{q^2}^\times$ which satisfies $\psi^{-q} \neq \psi$;
\item $\sigma(\theta)$, where $\theta = \theta_1 \otimes \theta_2$ is a character of $\bU_1(\bbF_q)\times\bU_1(\bbF_q)$ with $\theta_1 \neq \theta_2$.
\end{itemize}
The only isomorphisms among these representations are $\textnormal{Ind}_{\bB_{\bU}(\bbF_q)}^{\bU_2(\bbF_q)}(\psi) \cong \textnormal{Ind}_{\bB_{\bU}(\bbF_q)}^{\bU_2(\bbF_q)}(\psi^{-q})$ and $\sigma(\theta_1\otimes\theta_2) \cong \sigma(\theta_2\otimes\theta_1)$.  
\end{theo}

\begin{df}
We define a \emph{tame type} $\sigma$ to be an irreducible $\bU_2(\cO_K)$-representation over $E$ which arises by inflation from an irreducible $\bU_2(\bbF_q)$-representation over $E$.  Likewise, we define a \emph{tame type over $\cO$} to be a representation $\sigma$ of $\bU_2(\cO_K)$ on a finite-free $\cO$-module, such that $\sigma\otimes_{\cO}E$ is a tame type over $E$.  {We make similar definitions for the group $\bG\bL_2(\cO_{K_2})$.  }
\end{df}

\subsubsection{The principal series case}

Consider again a character
$$\psi: \bbF_{q^2}^\times  \longrightarrow \cO^\times$$
which satisfies $\psi^{-q} \neq \psi$, and let $\Ind_{\bB_\bU(\bbF_q)}^{\bU_2(\bbF_q)}(\psi)$ denote the (irreducible) principal series representation.  We may extend the character $\psi$ to a character $\psi\otimes\psi^{-q}$ of $\bB_\bU(\bbF_{q^2})$ as follows:
\begin{eqnarray*}
\psi\otimes\psi^{-q}:\bB_\bU(\bbF_{q^2}) & \longrightarrow & \cO^\times\\
\begin{pmatrix} x & y \\ 0 & z\end{pmatrix} & \longmapsto & \psi(x)\psi(z)^{-q},
\end{eqnarray*}
where $x,z\in \bbF_{q^2}^\times, y\in \bbF_{q^2}$.  We consider the (irreducible) induced representation $\Ind_{\bB_\bU(\bbF_{q^2})}^{\bG\bL_2(\bbF_{q^2})}(\psi\otimes\psi^{-q})$ of $\bU_2(\bbF_{q^2}) = \bG\bL_2(\bbF_{q^2})$, and view it as a tame type of $\bG\bL_2(\cO_{K_2})$ by inflation.

\begin{df}
\label{df:PS:BC}
Let $\psi:\bbF_{q^2}^\times\longrightarrow \cO^\times$ be a character such that $\psi^{-q}\neq \psi$.  We define the \emph{base change of $\Ind_{\bB_\bU(\bbF_q)}^{\bU_2(\bbF_q)}(\psi)$} to be the $\bG\bL_2(\cO_{K_2})$-type given by
$$\BC\left(\Ind_{\bB_\bU(\bbF_q)}^{\bU_2(\bbF_q)}(\psi)\right) \defeq \Ind_{\bB_\bU(\bbF_{q^2})}^{\bG\bL_2(\bbF_{q^2})}(\psi\otimes\psi^{-q}).$$
\end{df}

There is a compatibility of this definition with automorphic base change, as follows.  Let $\sigma = \Ind_{\bB_\bU(\bbF_q)}^{\bU_2(\bbF_q)}(\psi)$, and suppose $\pi$ is a smooth irreducible representation of $\bU_2(K)$ over $\bbC$ such that $\sigma\otimes_E\bbC\subseteq \pi|_{\bU_2(\cO_K)}$ (for some choice of morphism $E\longhookrightarrow \bbC$).  {This implies that $\pi^{\textnormal{Iw}_1} \neq 0$, where $\textnormal{Iw}_1$ denotes the upper-triangular pro-$p$-Iwahori subgroup of $\bU_2(\cO_K)$.  Consequently, $\pi$ cannot be supercuspidal, and is therefore a subquotient of a principal series representation.  Since the character $\psi$ is regular, this subquotient must in fact be an irreducible principal series (see \cite[\S 11.1]{rogawski} for a classification of nonsupercuspidal representations of $\bU_2(K)$).}  We let $\BC(\pi)$ denote the stable base change of $\pi$ to a representation of $\bG\bL_2(K_2)$ (cf. \cite[\S 11.4]{rogawski}).  Then $\BC(\pi)$ contains a unique tame type, which is isomorphic to $\BC(\sigma)\otimes_{E}\bbC$.

\subsubsection{}
\label{subsub:PSequiv}

We now wish to compute the base change map on Deligne--Lusztig representations.  Let $\mu \in X^*(\bT)$ be such that
$$\mu = \vecf{a_0}{b_0}{0}{0}\cdots \vecf{a_{f - 1}}{b_{f - 1}}{0}{0},$$
and suppose 
$$\sum_{i = 0}^{f - 1} a_ip^i - p^f\sum_{i = 0}^{f - 1} b_ip^i \not\equiv \sum_{i = 0}^{f - 1} b_ip^i - p^f\sum_{i = 0}^{f - 1} a_ip^i~(\textnormal{mod}~p^{2f} - 1).$$  
By \cite[Prop. 8.2]{DL} we have an isomorphism of $\bU_2(\cO_K)$-representations
$$R_{\un{1}}(\mu) \cong \Ind_{\bB_{\bU}(\bbF_q)}^{\bU_2(\bbF_q)}(\theta_{\mu}),$$
where 
$$\theta_{\mu}\left(\begin{pmatrix}x & y \\ 0 & x^{-q}\end{pmatrix}\right) = \varsigma_0\left(\tilde{x}^{\sum_{i = 0}^{f - 1} a_ip^i - p^f\sum_{i = 0}^{f - 1} b_ip^i}\right)$$
(recall that we identify representations of $\widetilde{\bU}_2(\cO_K)$ trivial on $\imath(\cO_K^\times)$ and representations of $\bU_2(\cO_K)$).  Further, the assumption on $\mu$ and Proposition 7.4 of \emph{op. cit.} imply that $R_{\un{1}}(\mu)$ is irreducible.  Consequently, the base change map becomes
$$\BC\left(R_{\un{1}}(\mu)\right) = R'_{(\un{1},\un{1})}\left(\mu, -\un{s}(\mu)\right).$$

Now let $w\in W$ be an element in the $\biF$-conjugacy class of $\un{1}$, and choose $w'\in W$ such that $w'w\biF(w')^{-1} = \un{1}$.  Applying first the equivalence \eqref{DLequiv} for the element $v = w'$ {(and $\nu = 0$)}, then the above equation, then the equivalence induced by $(w'^{-1}, w'^{-1})$, we obtain
\begin{equation}\label{BCDLPS}
\BC\left(R_{w}(\mu)\right)   \cong R'_{(w, w)}\left(\mu, -\un{s}(\mu)\right).
\end{equation}

\subsubsection{The cuspidal case}

Consider again the character $\theta = \theta_1\otimes\theta_2$ of $\bJ_{\textnormal{end}}(\bbF_q)$ which satisfies $\theta_1 \neq \theta_2$.  By base change we obtain the character 
\begin{eqnarray*}
\tld{\theta}:\bJ_{\textnormal{end}}(\bbF_{q^2}) = \bbF_{q^2}^\times\times\bbF_{q^2}^\times & \longrightarrow & \cO^{\times}\\
(x,y) & \longmapsto & \theta_1(x^{1-q})\theta_2\left(y^{1-q}\right).
\end{eqnarray*}
By inflation, we view this as a character of upper triangular Borel subgroup $\bB_{\bU}(\bbF_{q^2})$ of $\bU_2(\bbF_{q^2}) \cong \bG\bL_2(\bbF_{q^2})$, and view the (irreducible) induced representation $\Ind_{\bB_{\bU}(\bbF_{q^2})}^{\bG\bL_2(\bbF_{q^2})}(\widetilde{\theta})$ as a tame type of $\bG\bL_2(\cO_{K_2})$.

\begin{df}
\label{df:END:BC}
Let $\theta = \theta_1\otimes\theta_2:\bJ_{\textnormal{end}}(\bbF_q)\longrightarrow \cO^\times$ be a character such that $\theta_1 \neq \theta_2$.  We define the \emph{base change of $\sigma(\theta)$} to be the $\bG\bL_2(\cO_{K_2})$-type given by
$$\BC\left(\sigma(\theta)\right) \defeq \Ind_{\bB_{\bU}(\bbF_{q^2})}^{\bG\bL_2(\bbF_{q^2})}(\widetilde{\theta}).$$
\end{df}

We again have a compatibility of this definition with automorphic base change.  Let $\sigma = \sigma(\theta)$, and suppose $\pi$ is a smooth irreducible representation of $\bU_2(K)$ over $\bbC$ such that $\sigma\otimes_{E}\bbC \subseteq \pi|_{\bU_2(\cO_K)}$ (for some choice of morphism $E \longhookrightarrow \bbC$).  This implies that $\pi$ is a level 0 supercuspidal representation, and we let $\BC(\pi)$ denote the stable base change of (the $L$-packet containing) $\pi$.  Then $\BC(\pi)$ contains a unique tame type, which is isomorphic to $\BC(\sigma)\otimes_E \bbC$ (see \cite[Cor. 3.6]{blasco}).

\subsubsection{}
\label{subsub:cuspequiv}

We now wish to compute the base change map on cuspidal Deligne--Lusztig representations.  Let $\mu\in X^*(\bT)$ be such that
$$\mu = \vecf{a_0}{b_0}{0}{0}\cdots\vecf{a_{f - 1}}{b_{f -1}}{0}{0},$$
and suppose that $\sum_{i = 0}^{f - 1}a_ip^i \not\equiv \sum_{i = 0}^{f - 1} b_ip^i~(\textnormal{mod}~p^f + 1)$.  By \cite[Thm. 8.3]{DL}, this assumption guarantees that $R_{(s,1,\ldots, 1)}(\mu)$ is an irreducible, cuspidal $\bU_2(\cO_K)$-representation.  
Then a straighforward character computation using \cite[\S 6]{ennola} and \cite[Cor. 7.2]{DL} gives 
$$R_{(s,1,\ldots, 1)}(\mu) \cong \sigma(\theta_{\mu}),$$
where 
$$\theta_{\mu}(x,y) = \varsigma_0\left(\tilde{x}^{\sum_{i = 0}^{f - 1}a_ip^i}\tilde{y}^{\sum_{i = 0}^{f - 1}b_ip^i}\right).$$
Consequently, the base change map becomes
$$\BC\left(R_{(s,1,\ldots, 1)}(\mu)\right) = R'_{(\un{1},\un{1})}\left(\mu, -\mu\right).$$

Now let $w\in W$ be an element in the $\biF$-conjugacy class of $(s,1,\ldots,1)$, and choose $w'\in W$ such that $w'w\biF(w')^{-1} = (s,1,\ldots, 1)$.  Applying first the equivalence \eqref{DLequiv} for the element {$v = w'$ (and $\nu = 0$)}, then the above equation, then the equivalence induced by $(w'^{-1}, \un{s}w'^{-1})$, we obtain
\begin{equation}\label{BCDLcusp}
\BC\left(R_{w}(\mu)\right)  \cong R'_{(w, w)}\left(\mu, -\un{s}(\mu)\right).
\end{equation}

\subsubsection{}

We define a base change map on the remaining irreducible representations of $\bU_2(\bbF_q)$.  Given a character $\psi_0:\bU_1(\bbF_q) \longrightarrow \cO^\times$, we let $\widetilde{\psi}_0$ denote the character
\begin{eqnarray*}
\tld{\psi}_0:\bbF_{q^2}^\times & \longrightarrow & \cO^\times\\
x  & \longmapsto & \psi_0(x^{1 - q}).
\end{eqnarray*}

\begin{df}
Let $\psi_0: \bU_1(\bbF_q) \longrightarrow \cO^\times$ denote a character of $\bU_1(\bbF_q)$.  We define
\begin{eqnarray*}
\BC\left(\psi_0\circ\det\right) & \defeq & \widetilde{\psi}_0\circ\det,\\
\BC\left(\psi_0\circ\det\otimes_E\textnormal{St}\right) & \defeq & \widetilde{\psi}_0\circ\det\otimes_E\textnormal{St}',
\end{eqnarray*}
where $\textnormal{St}'$ denotes the Steinberg representation of $\bG\bL_2(\bbF_{q^2})$, inflated to $\bG\bL_2(\cO_{K_2})$.  
\end{df}

Taken together, these definitions give a base change map on isomorphism classes of tame types.  One further checks that the association $\sigma \longmapsto \BC(\sigma)$ is injective on isomorphism classes.

\subsubsection{}\label{epsilondef}

We define an involution $\epsilon$ on (isomorphism classes of) representations of $\bG\bL_2(\bbF_{q^2})$ by twisting a representation by the automorphism
$$g\longmapsto \left(\Phi_2 g^{-\top}\Phi_2^{-1}\right)^{(q)}$$
(note that the fixed points in $\bG\bL_2(\bbF_{q^2})$ of this automorphism are exactly $\bU_2(\bbF_q)$).  On Deligne--Lusztig representations, this becomes
$$\epsilon\left(R_{(w,w')}'(\mu, \mu')\right) = R_{(w',w)}'\left(-\un{s}(\mu'), -\un{s}(\mu)\right).$$
{The above can be checked using the equivalences of Subsubsections \ref{subsub:PSequiv} and \ref{subsub:cuspequiv}, and character tables for $\bG\bL_2(\bbF_{q^2})$ (see, e.g., \cite[\S 1]{diamond}).  Note that by dimension reasons $\epsilon(R_{(w,w')}'(\mu, \mu'))$ is an irreducible principal series, resp.~an irreducible cuspidal representation, if and only if $R_{(w,w')}'(\mu, \mu')$ is such a representation.}

The following lemma {is one of the main results of \cite{kawanaka}}.

\begin{lem}\label{invforBC}
Let $\sigma'$ denote a tame $\bG\bL_2(\cO_{K_2})$-type over $E$.  Then we have $\epsilon(\sigma') \cong \sigma'$ if and only if $\sigma'$ is of the form $\BC(\sigma)$ for a tame $\bU_2(\cO_K)$-type $\sigma$.  
\end{lem}

\subsection{Combinatorics of types and weights}
\label{subsec:cmb:types}

For future applications to weight elimination and weight existence results, we now analyze the combinatorial properties of the set $\JH(\overline{\sigma})$ for a tame type $\sigma$.

\subsubsection{}

Before proceeding, we make some definitions to simplify the discussion below.  

\begin{df}
\begin{enumerate}
\item We define $\tW \defeq  X^*(\bT) \rtimes W$ to be the extended affine Weyl group.  It acts on $X^*(\bT)$ in the natural way, and we write elements of $\tW$ as $t_\mu w$, with $\mu\in X^*(\bT)$, $w\in W$, to underscore this action.  
\item An \emph{alcove} is a connected component of 
$$X^*(\bT) \otimes_{\bbZ}\bbR - \Big(\bigcup_{\alpha,n}\big\{\langle\mu + \eta, \alpha^\vee\rangle = np\big\}\Big).$$
We let $C_0$ denote the dominant base alcove
$$\{\mu\in X^*(\bT)\otimes_{\bbZ}\bbR: 0 < \langle\mu + \eta, \alpha_i^\vee\rangle < p~\textnormal{for all}~0 \leq i \leq f - 1\}.$$
\item The group $pX^*(\bT) \rtimes W \subseteq \tW$ acts on the set of alcoves via the dot action $\sbullet$ centered at $-\eta$.  We define 
$${\Omega} \defeq \{\tw \in pX^*(\bT) \rtimes W :\tw \sbullet C_0 = C_0\}.$$
\end{enumerate}
\end{df}

\begin{rmk}
One easily checks that if $\tw = wt_{-p\nu} = (w_i t_{-p\nu_i})_i\in {\Omega}$, then we must have 
$$(w_i, \nu_i) \in \{1\}\times X^0(\bT_\bH)\quad \textnormal{or}\quad (w_i, \nu_i) \in \{s\}\times (\eta_{\bH} + X^0(\bT_{\bH}))$$
for all $0\leq i \leq f - 1$.  
\end{rmk}

\begin{lem}
\label{sub:char:intsct}
Let $R_w(\mu + \eta)$ denote a Deligne--Lusztig representation of $\bG(\bbZ_p)$, and suppose $\mu$ is a $1$-deep character.  Let $\lambda\in X_1(\bT)$.  Then $F(\lambda) \in \JH(\overline{R_w(\mu + \eta)})$ if and only if there exists $z t_{-p\nu}\in {\Omega}$ such that 
$$z t_{-p\nu}\sbullet (\mu + w\pi(\nu)) = \un{s}\sbullet(\lambda - p\eta).$$
{Moreover, the element $z \in W$ is unique, and every choice of $z$ arises.  Consequently we have $|\!\JH(\overline{R_w(\mu + \eta)})|=2^f$. \(Compare with \eqref{jantzen}.\)}
\end{lem}

\begin{proof}
{First, note that by our depth assumption on $\mu$, we may apply \cite[Prop. 10.1.2]{ghs} (which is based on \cite[Satz 4.3]{Jantzen-red}) and \cite[Prop. 10.1.8]{ghs}.}

Suppose $F(\lambda)\in \textnormal{JH}(\overline{R_{w}(\mu + \eta)})$ for some $\lambda\in X_1(\bT)$.  By \cite[Prop. 10.1.8]{ghs}, this holds if and only if there exists $\nu\in X^*(\bT)$ such that 
$$z'\sbullet(\mu + (w\pi - p)\nu) \uparrow \un{s}\sbullet(\lambda - p\eta)$$
for all $z'\in W$.  (We refer to \cite[II.6.4]{RAGS} for the definition of $\uparrow$; since the root system is of type $A_1 \times \cdots \times A_1$, the condition $\mu' \uparrow \lambda'$ is equivalent to $\mu' \leq \lambda'$ and $\mu' \in (p{\Lambda_R}\rtimes W)\sbullet \lambda'$, {where $\Lambda_R$ denotes the root lattice of $\bf{G}$}.)  Select $z\in W$ such that 
$$z(\mu + (w\pi - p)\nu + \eta)\in X_+(\bT).$$  
Since $z\sbullet (\mu + (w\pi - p)\nu)$ lies below $\un{s}\sbullet (\lambda - p\eta)$ in the $\uparrow$ ordering, since $z(\mu + (w\pi - p)\nu + \eta)$ is dominant, and since $\un{s}(\lambda - (p - 1)\eta)$ is $p$-restricted, we must have 
$$z\sbullet (\mu + (w\pi - p)\nu) = \un{s}\sbullet (\lambda - p\eta).$$
The proof of \cite[Prop. 4.1.3]{LLL} shows that {$|\langle\nu, \alpha_i^\vee\rangle| \leq 1$ for every $i$, from which we deduce} $z t_{-p\nu}\in {\Omega}$.  Furthermore, we deduce \emph{a posteriori} that the choice of $z$ is unique.


Conversely, if $z\sbullet (\mu + (w\pi - p)\nu) = \un{s}\sbullet (\lambda - p\eta)$ for some ${zt_{-p\nu} \in \Omega}$, then $z(\mu + (w\pi - p)\nu + \eta)\in X_+(\bT)$, and \cite[II.6.4(5)]{RAGS} implies 
$$(z'z)\sbullet (\mu + (w\pi - p)\nu) \uparrow z\sbullet (\mu + (w\pi - p)\nu) = \un{s}\sbullet (\lambda - p\eta)$$
for all $z'\in W$, so that $F(\lambda)\in \JH(\overline{R_w(\mu + \eta)})$ by \cite[Prop. 10.1.8]{ghs}.

{To show that every choice of $z$ arises, choose any $\nu\in X^*(\bT)$ such that $zt_{-p\nu} \in \Omega$, and define
$$\lambda_z \defeq \un{s}z(\mu + (w\pi - p)\nu + \eta) + (p - 1)\eta.$$
The depth assumption on $\mu$ implies that $\lambda_z \in X_1(\bT)$, and by definition we have
$$z t_{-p\nu}\sbullet (\mu + w\pi(\nu)) = \un{s}\sbullet(\lambda_z - p\eta),$$
so $F(\lambda_z) \in \JH(\overline{R_w(\mu + \eta)})$.  Finally, we note that different choices of $\nu$ will alter $\lambda_z$ by an element of $(p - \pi)X^0(\bT)$, which will give an isomorphic Serre weight.}
\end{proof}

\begin{prop}
\label{char:intsct}
Let $\sigma_1 = R_{w_1}(\mu_1 + \eta),~ \sigma_2 = R_{w_2}(\mu_2 + \eta)$ be two Deligne--Lusztig representation of $\bG(\zp)$.  Suppose that {$\mu_1$ is $3$-deep and $\mu_2$ is $1$-deep.  }
\begin{enumerate}
\item\label{it1:char:intsct} 
We have $\JH(\overline{\sigma_1}) \cap \JH(\overline{\sigma_2}) \neq \emptyset$ if and only if there exists a pair $(w_2', \mu_2')\in W\times X^*(\bT)$ such that $R_{w_2}(\mu_2 + \eta) \cong R_{w_2'}(\mu_2' + \eta)$ and 
$$t_{\mu_2'}w_2' = t_{\mu_1}w_1  \tld{w}$$ 
in $\tW$, where $\tld{w}_i\in\{1,~ s,~ t_{\alpha_\bH}s\}$ for all $0\leq i \leq f - 1$.
\item\label{it2:char:intsct} 
Suppose $\JH(\overline{\sigma_1})\cap\JH(\overline{\sigma_2})\neq \emptyset$, and let $(w_2',\mu_2')$ and $\tld{w} = (\tld{w}_i)_i$ be as in item \ref{it1:char:intsct}.
Let $\lambda\in X_1(\bT)$.  
Then $F(\lambda)\in \JH(\overline{\sigma_1}) \cap \JH(\overline{\sigma_2})$ if and only if
$$\un{s}\sbullet(\lambda - p\eta) = \tld{w}_\lambda\sbullet (\mu_1 + w_1\pi (\nu))= \tld{w}_\lambda\sbullet (\mu_2' + w_2'\pi (\nu))$$
for an element $\tld{w}_\lambda = wt_{-p\nu}\in  \Omega$ satisfying the following conditions:
\begin{enumerate}[label={\upshape(\alph*)}]
\item\label{it2:1:char:intsct} if $\tld{w}_i = s$ then {$(\tld{w}_\lambda)_{i - 1} \equiv 1 ~\textnormal{mod}~ X^0(\bT_\bH)$}; and
\item\label{it2:2:char:intsct} if $\tld{w}_i = t_{\alpha_\bH}s$ then {$(\tld{w}_\lambda)_{i - 1} \equiv st_{-p\eta_{\bH}} ~\textnormal{mod}~ X^0(\bT_\bH)$}.
\end{enumerate}
{In particular, since adding $(p - \pi)X^0(\bT)$ to $\lambda$ does not affect the isomorphism class of a Serre weight, we obtain}
$$\left|\JH(\overline{\sigma_1})\cap\JH(\overline{\sigma_2})\right| = 2^{\left|\{i:\tld{w}_i = 1\}\right|}.$$
\end{enumerate}
\end{prop}

\begin{proof}
We begin with item \ref{it1:char:intsct}. 

Assume $\JH(\overline{\sigma_1})\cap \JH(\overline{\sigma_2})\neq \emptyset$; by Lemma \ref{sub:char:intsct} we have
\begin{equation}
\label{eq:equality:characters}
\mu_2 + (w_2\pi - p)\nu^{(2)} + \eta = {(z^{(2)})^{-1}}z^{(1)} (\mu_1 + (w_1\pi - p)\nu^{(1)} + \eta)
\end{equation}
where $z^{(j)} t_{-p\nu^{(j)}}\in {\Omega}$.  This gives
\begin{eqnarray*}
R_{w_2}\left(\mu_2 + \eta\right) & \cong & R_{w_2}\left(\mu_2 + \eta + (w_2\pi - p)\nu^{(2)}\right)\\
 & \cong & R_{w_2}\left({(z^{(2)})^{-1}}z^{(1)} (\mu_1 + (w_1\pi - p)\nu^{(1)} + \eta)\right)\\
 & \cong & R_{{(z^{(1)})^{-1}}z^{(2)}w_2\biF({(z^{(1)})^{-1}}z^{(2)})^{-1}}\left(\mu_1 + (w_1\pi - p)\nu^{(1)} + \eta\right)\\
 & \cong & R_{{(z^{(1)})^{-1}}z^{(2)}w_2\biF({(z^{(1)})^{-1}}z^{(2)})^{-1}}\left(\mu_1 + w_1\pi(\nu^{(1)})  - w_2'\pi(\nu^{(1)}) + \eta  \right)\\
 & \cong & R_{w_2'}\left(\mu_2' + \eta  \right)
\end{eqnarray*}
where the first isomorphism comes from {equation (\ref{DLequiv}) by} adding $(w_2\pi - p)\nu^{(2)}$, the second from (\ref{eq:equality:characters}), the third from {equation (\ref{DLequiv}) by} conjugation by ${(z^{(1)})^{-1}}z^{(2)}$, and the fourth {again from equation (\ref{DLequiv})} by adding $(p - w_2'\pi)\nu^{(1)}$.  Here, we define $w_2' \defeq {(z^{(1)})^{-1}}z^{(2)}w_2\biF({(z^{(1)})^{-1}}z^{(2)})^{-1}$ and $\mu_2' \defeq \mu_1 + w_1\pi(\nu^{(1)})  - w_2'\pi(\nu^{(1)})$.

We now proceed entrywise:
\begin{itemize}
\item if $w_{2,i}' = w_{1,i}$, then by definition we have $\mu_{2,i}' = \mu_{1,i}$;
\item if $w_{2,i}' = w_{1,i}s$, then $\mu_{2,i}' = \mu_{1,i} + w_{1,i}(\pi(\nu^{(1)})_i - s\pi(\nu^{(1)})_i)$.  Since $\pi(\nu^{(1)})_i\in\{0,~\eta_{\bH}\} + X^0(\bT_{\bH})$, we have $\pi(\nu^{(1)})_i - s\pi(\nu^{(1)})_i \in\{0,~ \alpha_{\bH}\}.$
\end{itemize}
This means exactly that $t_{\mu_2'}w_2' = t_{\mu_1}w_1\tld{w}$ in $\tW$, with $\tld{w}_i\in \{1,~ s,~ t_{\alpha_{\bH}} s\}$.

For the converse, suppose that $(w_2',\mu_2')$ satisfies $\sigma_2 = R_{w_2}(\mu_2 + \eta) \cong R_{w_2'}(\mu_2' + \eta)$ and 
$$t_{\mu_2'}w_2' = t_{\mu_1}w_1\tld{w}$$ 
with $\tld{w}_i\in \{1,~ s,~ t_{\alpha_{\bH}} s\}$.   In particular, this implies $\mu_2'$ is 1-deep.  Let $\tld{w}_\lambda = wt_{-p\nu}\in {\Omega}$ be any element which satisfies conditions \ref{it2:1:char:intsct}, \ref{it2:2:char:intsct} in the statement of the lemma.  
\begin{itemize}
\item If $\tld{w}_i = s$, then $w_{2,i}' = w_{1,i}s$ and 
$$\mu_{2,i}' = \mu_{1,i} = \mu_{1,i} + w_{1,i}(\nu_{i - 1} - s(\nu_{i - 1})) = \mu_{1,i} + w_{1,i}\pi(\nu)_i - w_{2,i}^{{\prime}}\pi(\nu)_i.$$
\item If $\tld{w}_i = t_{\alpha_{\bH}}s$, then $w_{2,i}' = w_{1,i}s$ and 
$$\mu_{2,i}' = \mu_{1,i} + w_{1,i}(\alpha_{\bH}) = \mu_{1,i} + w_{1,i}(\nu_{i - 1} - s(\nu_{i - 1})) = \mu_{1,i} + w_{1,i}\pi(\nu)_i - w_{2,i}^{{\prime}}\pi(\nu)_i.$$
\item Finally, if $\tld{w}_i = 1$, then $w_{2,i}' = w_{1,i}$ and
$$\mu_{2,i}' = \mu_{1,i} = \mu_{1,i} + w_{1,i}(\nu_{i - 1}) - w_{2,i}^{{\prime}}(\nu_{i - 1}) = \mu_{1,i} + w_{1,i}\pi(\nu)_i - w_{2,i}^{{\prime}}\pi(\nu)_i.$$
\end{itemize}
Collecting these, we obtain $\mu_1 + w_1\pi(\nu) = \mu_2' + w_2'\pi(\nu)$, i.e.
\begin{equation*}
\tld{w}_\lambda\sbullet\big(\mu_1 + w_1\pi(\nu)\big) = \tld{w}_\lambda\sbullet\big(\mu_2' + w_2'\pi(\nu)\big).
\end{equation*}
By Lemma \ref{sub:char:intsct}, we conclude that $F(\lambda)\in \JH(\overline{\sigma_1})\cap \JH(\overline{\sigma_2})$, where $\lambda$ is defined by
$$\un{s}\sbullet(\lambda - p\eta) = \tld{w}_\lambda\sbullet\big(\mu_1 + w_1\pi(\nu)\big) = \tld{w}_\lambda\sbullet\big(\mu_2' + w_2'\pi(\nu)\big).$$
This completes the proof of item \ref{it1:char:intsct} and of the ``if'' direction in item \ref{it2:char:intsct}, and shows that 
$$\left|\JH(\overline{\sigma_1})\cap\JH(\overline{\sigma_2})\right| \geq 2^{\left|\{i:\tld{w}_i = 1\}\right|}.$$

We now conclude the proof of item \ref{it2:char:intsct}.  

Suppose there exists some $F(\lambda)\in \JH(\overline{\sigma_1})\cap \JH(\overline{\sigma_2})$, and let $(w_2',\mu_2')$ be as in item \ref{it1:char:intsct}.  By Lemma \ref{sub:char:intsct} there exist $z^{(j)}t_{-p\nu^{(j)}}\in {\Omega}$ such that
\begin{equation}\label{JHint}
\un{s}\sbullet(\lambda - p\eta) = {z^{(1)}t_{-p\nu^{(1)}}}\sbullet\left(\mu_1 + w_1\pi(\nu^{(1)})\right) = {z^{(2)}t_{-p\nu^{(2)}}}\sbullet\left(\mu_2' + w_2'\pi(\nu^{(2)})\right).
\end{equation}
{Pairing the middle expression with $\alpha_i^\vee$ and reducing modulo 2 gives
\begin{eqnarray*}
\left\langle z^{(1)}t_{-p\nu^{(1)}}\sbullet(\mu_1 + w_1\pi(\nu^{(1)})), \alpha_i^\vee\right\rangle & = & \left\langle z^{(1)}(\mu_1 + (w_1 \pi - p)\nu^{(1)} + \eta), \alpha_i^\vee\right\rangle - 1 \\
 & = & (-1)^{\delta_{z_i^{(1)}, s}}\left\langle \mu_1 + (w_1 \pi - p)\nu^{(1)} + \eta, \alpha_i^\vee\right\rangle - 1 \\
 & \equiv & \left\langle \mu_1 + (w_1 \pi - p)\nu^{(1)} + \eta, \alpha_i^\vee\right\rangle - 1 \\
 & \equiv & \left\langle \mu_1 , \alpha_i^\vee\right\rangle + (-1)^{\delta_{w_{1,i},s}} \left\langle \pi(\nu^{(1)}), \alpha_i^\vee\right\rangle - p\left\langle\nu^{(1)}, \alpha_i^\vee\right\rangle \\
 & \equiv & \left\langle \mu_1 , \alpha_i^\vee\right\rangle +  \left\langle \pi(\nu^{(1)}), \alpha_i^\vee\right\rangle + \left\langle\nu^{(1)}, \alpha_i^\vee\right\rangle \\
 & \equiv & \left\langle \mu_1 , \alpha_i^\vee\right\rangle +  \delta_{z^{(1)}_{i - 1}, s} + \delta_{z^{(1)}_i, s} ~(\textnormal{mod}~2).
\end{eqnarray*}
We have a similar calculation for the rightmost expression.  Recalling how $\mu_2'$ and $\mu_1$ are related (cf. item \ref{it1:char:intsct}), we see that $\langle \mu_2', \alpha_i^\vee\rangle \equiv \langle \mu_1, \alpha_i^\vee\rangle ~(\textnormal{mod}~2)$.  Consequently, the last equality in \eqref{JHint} gives
}
$$\delta_{z^{(1)}_i, s} + \delta_{z^{(2)}_i,s} \equiv \delta_{z^{(1)}_{i - 1},s} + \delta_{z^{(2)}_{i - 1},s}~(\textnormal{mod}~2).$$

{Suppose by contradiction that $z^{(1)} \neq z^{(2)}$, so that $z^{(1)}_i\neq z^{(2)}_i$ for some $i$.  The above equation implies that} this inequality holds for all $i$, i.e., $z^{(2)} = \un{s}z^{(1)}$ and $\nu^{(1)} + \nu^{(2)} = \eta {+ \beta}$ {for some $\beta \in X^0(\bT)$}.  Substituting this into equation \eqref{JHint} and {cancelling $z^{(1)}$ yields
\begin{equation}
\label{intersection-contradiction}
\mu_1 + (w_1\pi - p)\nu^{(1)} + \eta = \un{s}\left(\mu_2' + (w_2'\pi - p)(\eta + \beta - \nu^{(1)}) + \eta\right).
\end{equation}
Recalling how $(\mu_1, w_1)$ is related to $(\mu_2', w_2')$ via $\widetilde{w}$ (cf. item \ref{it1:char:intsct}), by pairing the above equation with $\alpha_i^\vee$ we see that
\begin{equation}
\label{intersection-contradiction-2}
\langle \mu_1, \alpha_i^\vee\rangle = \begin{cases} \frac{p - 1}{2} - \delta_{w_{1,i},1} & \textnormal{if}~ \widetilde{w}_i = 1,\\  \frac{p - 1}{2} - \delta_{w_{1,i},sz_{i - 1}^{(1)}} & \textnormal{if}~ \widetilde{w}_i = s,\\ \frac{p - 1}{2} - \delta_{w_{1,i},sz_{i - 1}^{(1)}}  + (-1)^{\delta_{w_{1,i}, 1}} & \textnormal{if}~ \widetilde{w}_i = t_{\alpha_{\bH}}s.  \end{cases}
\end{equation}}

{To proceed further, let us write $w_2' = w_1v$ and $\mu_2' = \mu_1 + w_1(\xi)$, where $v \in W$ and $\xi = \sum_{i = 0}^{f - 1} a_i \alpha_i$ with $a_i \in \{0,1\}$ and $a_i = 1$ only if $v_i = s$.  We wish to evaluate the expression
\begin{equation}
\label{intersection-contradiction-3}
- \un{s}w_1(\xi) + \sum_{i = 0}^{f - 1}\big(\langle \mu_1, \alpha_i^\vee\rangle + 1\big)\alpha_i + \Big(w_1\pi(\nu^{(1)}) - \pi(\nu^{(1)})\Big) + \Big(- \un{s}w_1 v\pi(\eta) + \un{s}\pi(\eta)\Big) + \Big(\un{s}w_1v\pi(\nu^{(1)}) - \un{s}\pi(\nu^{(1)})\Big),
\end{equation}
which lies in $\Lambda_R$.  By working entrywise and considering all possibilities for $w_{1,i}, z_{i - 1}^{(1)}, v_i$ and $a_i$, and using equation \eqref{intersection-contradiction-2}, we see that \eqref{intersection-contradiction-3} is equal to 
$$\sum_{i = 0}^{f - 1}\frac{p - 1}{2} \alpha_i.$$
On the other hand, rearranging the expression \eqref{intersection-contradiction-3} gives
$$\Big( \mu_1 - \un{s}(\mu_1) - \un{s}w_1(\xi) + \eta - \un{s}(\eta) + w_1\pi(\nu^{(1)}) - \un{s}w_1 v\pi(\eta) + \un{s}w_1v\pi(\nu^{(1)}) \Big) - \pi(\nu^{(1)}) + \un{s}\pi(\eta) - \un{s}\pi(\nu^{(1)}),$$
and using equation \eqref{intersection-contradiction} to further simplify the parenthesized term above, we get
\begin{flushleft}
$\displaystyle{ \left( p\nu^{(1)} + \pi(\beta) - p\un{s}(\eta) - p\beta + p\un{s}(\nu^{(1)})\right) - \pi(\nu^{(1)}) + \un{s}\pi(\eta) - \un{s}\pi(\nu^{(1)})}$
\end{flushleft}
\begin{flushright}
$\displaystyle{ = (p - \pi)\left(\nu^{(1)} - \beta - \un{s}(\eta) + \un{s}(\nu^{(1)})\right).}$
\end{flushright}
Combining these two calculations, we see that $\sum_{i = 0}^{f - 1}\frac{p - 1}{2} \alpha_i$ lies in $\Lambda_R \cap (p - \pi)X^*(\bT) = (p - \pi)\Lambda_R$, which yields the desired contradiction.
}

The above argument shows we must have {$z^{(2)} = z^{(1)}$ and $\nu^{(2)} = \nu^{(1)} + \beta$ for some $\beta \in X^0(\bT)$.  Thus, equation \eqref{JHint} reduces to 
$$\mu_1 + w_1\pi(\nu^{(1)}) = \mu_2' + w_2'\pi(\nu^{(1)}) + (\pi - p)\beta.$$
Writing $\mu_2' = \mu_1 + w_1(\xi)$ and $w_2' = w_1v$ as above, this equation becomes
$$\mu_1 + w_1\pi(\nu^{(1)}) = \mu_1 + w_1(\xi) + w_1 v\pi(\nu^{(1)}) + (\pi - p)\beta.$$
Cancelling $\mu_1$ and applying $w_1^{-1}$ gives the equation
$$\xi + v\pi(\nu^{(1)}) - \pi(\nu^{(1)}) = (p - \pi)\beta;$$
since the intersection $\Lambda_R \cap X^0(\bT)$ is trivial, we conclude that $\beta = 0$ and 
$$\pi(\nu^{(1)}) - v\pi(\nu^{(1)}) = \xi.$$
This equation and the condition that $z^{(1)}t_{-p\nu^{(1)}}\in \Omega$ determines $\nu^{(1)}$ up to an element in $X^0(\bT)$, so that $z^{(1)}t_{-p\nu^{(1)}}$ must exactly be one of the $\tld{w}_\lambda$ of the statement of the lemma.  This shows $\left|\JH(\overline{\sigma_1})\cap\JH(\overline{\sigma_2})\right| = 2^{\left|\{i:\tld{w}_i = 1\}\right|}$.}
\end{proof}

%
%

{\begin{rmk}
\label{rmk:bij}
The above results hold \emph{mutatis mutandis} for the group $\bG' = \textnormal{Res}_{\cO_{K_2}/\zp}(\bG\bL_{2/\cO_{K_2}})$. More precisely:
\begin{enumerate}
\item
\label{it1:rmk:bij}
The statement of Lemma \ref{sub:char:intsct} holds with $\bG$ (and related objects, e.g., $\bT, W, R_w(\mu + \eta)$, etc.) replaced by $\bG'$ (resp., the relevant primed objects).  Moreover, the quantity $2^f$ is replaced by $2^{2f}$, and the element $\eta\in X^*(\bT)$ is replaced by the character $\rho' \in X^*(\bT')$ (corresponding to $\vect{\un{1}}{\un{0}}\in(\bbZ^2)^{\oplus 2f} \cong X^*(\bT')$).
\item
\label{it2:rmk:bij}
The statement of Proposition \ref{char:intsct}\ref{it1:char:intsct} holds with $\bG$ (and related objects) replaced by $\bG'$ (resp., the relevant primed objects), and the element $\alpha_{\bH}\in X^*(\bT_{\bH})$ replaced by the character of $\bG\bL_2$ corresponding to $\vect{1}{-1}\in\bbZ^2 \cong X^*(\bT_\bG)$ (recall from Subsubsection \ref{tori} that $\bT_\bG$ is the diagonal maximal torus of $\bG\bL_{2}$).  
\end{enumerate}
\end{rmk}}

\subsubsection{}

We are now in a position to compare how intersection of Jordan--H\"older factors behaves under base change.

\begin{prop}
\label{prop:double:weights}
Let $\sigma_1$, $\sigma_2$ be two $3$-generic Deligne--Lusztig representations of $\bG(\zp)$ on which $\imath(\cO_K^\times)$ acts trivially.  Then
$$\big|\JH\big(\overline{\BC(\sigma_1)}\big) \cap \JH\big(\overline{\BC(\sigma_2)}\big)\big| = |\JH(\overline{\sigma_1})\cap\JH(\overline{\sigma_2})|^2.$$
\end{prop}

\begin{proof}
Let us write $\sigma_j \cong R_{w_j}(\mu_j)$ for $j = 1,2$, with $\mu_j - \eta$ being $3$-deep.  By the discussion in Subsubsection \ref{trivactiondescent}, we may assume that the last two entries of $\mu_j$ in each embedding are equal to 0.

Suppose first that $\JH(\overline{\sigma_1}) \cap \JH(\overline{\sigma_2}) \neq \emptyset$.  Let $\tld{w} = t_{\xi}v \in \tW$ be as in {Proposition} \ref{char:intsct}, so that $\sigma_2 \cong R_{w_1v}(\mu_1 + w_1(\xi))$ and $|\JH(\overline{\sigma_1}) \cap \JH(\overline{\sigma_2})| = 2^{|\{i:v_i = 1\}|}$.  Using \eqref{BCDLPS} or \eqref{BCDLcusp}, we get
$$\BC(\sigma_1) \cong R'_{(w_1,w_1)}\big(\mu_1,-\un{s}(\mu_1)\big),$$
and
\begin{eqnarray*}
\BC(\sigma_2) & \cong & R'_{(w_1v,w_1v)}\big(\mu_1 + w_1(\xi),~ -\un{s}(\mu_1) - \un{s}w_1(\xi)\big)\\
 & \cong & R'_{(w_1v,w_1v)}\big((\mu_1, -\un{s}(\mu_1)) +(w_1(\xi), w_1(\xi))\big)
\end{eqnarray*}
(note that $\xi_i \in \{0,~\alpha_{\bH}\}$ for each $0 \leq i \leq f - 1$, so $-\un{s}(\xi) = \xi$).  
By the $\bG\bL_2$ analog of Proposition \ref{char:intsct}\ref{it1:char:intsct}, we may obtain $\BC(\sigma_2)$ from $\BC(\sigma_1)$ via the element $t_{(\xi,\xi)}(v,v) \in \tW'$, and therefore
$$\big|\JH\big(\overline{\BC(\sigma_1)}\big) \cap \JH\big(\overline{\BC(\sigma_2)}\big)\big| = 2^{|\{i: (v,v)_i = 1\}|} = 2^{2|\{i: v_i = 1\}|} = \left|\JH(\overline{\sigma_1}) \cap \JH(\overline{\sigma_2})\right|^2.$$
{(We write $\tW'$ for the affine Weyl group of $\bG' = \textnormal{Res}_{\cO_{K_2}/\zp}(\bG\bL_{2/\cO_{K_2}})$, and use similar primed notation below for the automorphism $\pi'$ acting on the character lattice of $\bG'$ (see Subsubsection \ref{def-of-pi-prime}).)}

To conclude, it suffices to prove that $\JH(\overline{\BC(\sigma_1)}) \cap \JH(\overline{\BC(\sigma_2)}) \neq \emptyset$ implies $\JH(\overline{\sigma_1}) \cap \JH(\overline{\sigma_2}) \neq \emptyset$.  Assume the former.  
{By Remark \ref{rmk:bij}\ref{it2:rmk:bij}}
we may obtain $\BC(\sigma_2)$ from $\BC(\sigma_1)$ via an element $t_{(\xi,\xi')}(v,v')\in \tW'$ with $t_{\xi_i}v_i,~ t_{\xi'_i}v_i'\in \{1,~ s,~ t_{\alpha_\bH} s\}$.  
That is, we have
$$\BC(\sigma_2) \cong R_{(w_2,w_2)}'\big(\mu_2, -\un{s}(\mu_2)\big) \cong R_{(w_1v,w_1v')}'\big((\mu_1, -\un{s}(\mu_1)) + (w_1,w_1)(\xi,\xi')\big).$$
By Lemma \ref{invforBC}, the isomorphism class of the representation on the right is invariant under $\epsilon$, {which implies
$$R_{(w_1v,w_1v')}'\big((\mu_1, -\un{s}(\mu_1)) + (w_1,w_1)(\xi,\xi')\big) \cong R_{(w_1v',w_1v)}'\big((\mu_1, -\un{s}(\mu_1)) + (w_1,w_1)(\xi',\xi)\big).$$
By \cite[Prop. 2.2.15]{LLL} (which can be used by the depth assumption on $\mu_1 - \eta, \mu_2 - \eta$ and the fact that $\xi_i, \xi_i' \in \{0,\alpha_{\bH}\}$ for all $i$) there exists $zt_{-p\nu} \in \widetilde{W}'$ such that 
\begin{itemize}
\item if $z_i = 1$, then $\nu_i \in X^0(\bT_\bG)$; 
\item if $z_i = s$, then $\nu_i \in \vect{1}{0} + X^0(\bT_\bG)$;
\item we have $(w_1v', w_1v) = z(w_1v,w_1v')\pi'(z)^{-1}$; and
\item we have
\begin{flushleft}
$(\mu_1, -\un{s}(\mu_1)) + (w_1,w_1)(\xi',\xi)$
\end{flushleft}
\begin{flushright}
$ = z(\mu_1, -\un{s}(\mu_1)) + z(w_1,w_1)(\xi,\xi') + \big(p -  z(w_1v,w_1v')\pi'(z)^{-1}\big)\pi'(\nu).$
\end{flushright}
\end{itemize}
Rearranging the equation in the last item, we obtain
$$\big(p -  z(w_1v,w_1v')\pi'(z)^{-1}\big)\pi'(\nu) = (\mu_1, -\un{s}(\mu_1)) - z(\mu_1, -\un{s}(\mu_1)) + (w_1,w_1)(\xi',\xi) - z(w_1,w_1)(\xi,\xi'),$$
and the right-hand term lies in the root lattice of $\bG'$; consequently, the same is true for the element $\nu$.  Combining this with the first two items implies that $\nu = 0$, and thus $z = 1$.  Finally, the third and fourth items imply $v' = v$ and $\xi' = \xi$.}

{The above argument shows}
$$\BC(\sigma_2) \cong R_{(w_1v,w_1v)}'\big((\mu_1, -\un{s}(\mu_1)) + (w_1,w_1)(\xi,\xi)\big) \cong \BC(R_{w_1v}(\mu_1 + w_1(\xi))).$$
Since the base change map is injective on isomorphism classes of tame types, we get 
$$\sigma_2 \cong R_{w_1v}(\mu_1 + w_1(\xi))$$
and consequently $\JH(\overline{\sigma_1}) \cap \JH(\overline{\sigma_2}) \neq \emptyset$ by Proposition \ref{char:intsct}.  
\end{proof}

\subsubsection{}

We introduce a metric on the set of Serre weights contained in a sufficiently generic tame type.  This will turn out to be useful in the proof of Theorem \ref{thm:WExt}.

\begin{df}
\label{def:label}
Let $R_w(\mu)$ denote a Deligne--Lusztig representation of $\bG(\bbZ_p)$, and suppose $\mu - \eta$ is $1$-deep.  Let $F(\lambda)\in \JH(\overline{R_w(\mu)})$.  By Lemma \ref{sub:char:intsct}, there exists an element $z t_{-p\nu} \in {\Omega}$ defined by
$$\un{s}\sbullet(\lambda - p\eta) = z t_{-p\nu}\sbullet(\mu - \eta + w\pi(\nu)).$$
We say $z\in W$ is the \emph{label} of $F(\lambda)$ with respect to $(w,\mu)$.
\end{df}

\begin{rmk}
\label{rmk:well-defined}
Maintain the setting of Definition \ref{def:label}.  If $(w',\mu')$ is another pair such that $R_w(\mu) \cong R_{w'}(\mu')$ with $\mu' - \eta$ being $1$-deep, then by equation \eqref{DLequiv} we have 
$$(w',\mu') = (vw\biF(v)^{-1}, v(\mu) + \biF(\nu) - vw\biF(v)^{-1}(\nu))$$ 
for some pair $(\nu, v)\in  X^*(\bT) \rtimes W$.  It is easily checked that if the label of $F(\lambda)$ with respect to $(w,\mu)$ is $z$, then the label of $F(\lambda)$ with respect to $(w',\mu')$ is given by by $z v^{-1}$.  
\end{rmk}

\begin{df}
Let $\sigma$ denote a $1$-generic Deligne--Lusztig representation of $\bG(\zp)$, and let $F, F'\in \JH(\overline{\sigma})$.  Choose an isomorphism $\sigma \cong R_w(\mu)$, with $\mu - \eta$ being $1$-deep, and suppose that the labels of $F$ and $F'$ with respect to $(w,\mu)$ are $z$ and $z'$, respectively.  We define the \emph{graph distance} $\dgr{F}{F'}$ as the number of $i$ for which $z_i \neq z_i'$ (i.e., $\dgr{F}{F'}$ is the length {$\ell(z'z^{-1})$} of $z'z^{-1}$).  By Remark \ref{rmk:well-defined} the graph distance is well-defined.
\end{df}

\begin{rmk}
Suppose that $\sigma_1$ and $\sigma_2$ are two $3$-generic Deligne--Lusztig representations of $\bG(\zp)$, and suppose $F,F'\in \JH(\overline{\sigma_1})\cap \JH(\overline{\sigma_2})$.  Then the graph distance between $F$ and $F'$, computed using $\sigma_1$, agrees with the graph distance between $F$ and $F'$, computed using $\sigma_2$ (this follows from Lemma \ref{sub:char:intsct} and Proposition \ref{char:intsct}).
\end{rmk}

\begin{lem}
\label{lem:comb:typ1}
Let $\sigma$ be a $4$-generic Deligne--Lusztig representation of $\bG(\zp)$, and let $F,F'\in \JH(\overline{\sigma})$.  Then there exists a tame type $\sigma'$ such that:
\begin{itemize}
	\item $F,\, F'\in \JH(\overline{\sigma'})$; and
	\item for any $F''\in \JH(\overline{\sigma})\cap\JH(\overline{\sigma'})$ satisfying $F''\neq F'$, we have 
	$$\dgr{F}{F''} < \dgr{F}{F'}.$$
\end{itemize}
{Specifically}, $\sigma$ and $\sigma'$ can be {written} so that $\sigma \cong R_w(\mu)$, $\sigma' \cong R_{w'}(\mu')$ with $\mu - \eta$ being $3$-deep, and $t_{\mu' - \eta}w' = t_{\mu - \eta}wt_{\alpha_{\biF(z)}}\biF(z)$ for an element $z \in W$ which satisfies $\ell(\un{s}z)=\dgr{F}{F'}$.  \emph{(}{For} $v\in W$ we denote $\alpha_v \defeq \underset{i\,:\,v_i = s}{\sum}\alpha_i$.\emph{)}  In this case, 
$$\left|\JH(\overline{\sigma}) \cap \JH(\overline{\sigma'})\right| = 2^{\ell(\un{s}z)} = 2^{\dgr{F}{F'}}.$$
\end{lem}

\begin{proof}
Let us write $\sigma \cong R_w(\mu)$ with $\mu - \eta$ being $4$-deep.  By applying the equivalence \eqref{DLequiv}, we may assume that the label of $F$ with respect to $(w,\mu)$ is $\un{s}$ at the cost of assuming $\mu - \eta$ is only $3$-deep.  Suppose that the label of $F'$ with respect to $(w,\mu)$ is $z$.  By definition, we have
\begin{eqnarray*}
F & \cong & F\big(\mu + w(\eta) - \eta\big) = F(\mu - \alpha_w)\\
F' & \cong & F\big(\un{s}z(\mu + (w\pi - p)\nu) + (p - 1)\eta\big),
\end{eqnarray*}
where $z t_{-p\nu}\in {\Omega}$.

We define $\sigma' \defeq R_{w\biF(z)}(\mu + w\pi(\alpha_{z}))$.  We easily see that $F$ and $F'$ are Jordan--H\"older factors of $\overline{\sigma'}$, whose labels with respect to $(w\biF(z),\mu + w\pi(\alpha_{z}))$ are $\un{s}$ and $z$, respectively.  Moreover, by the explicit description of $\JH(\overline{\sigma}) \cap \JH(\overline{\sigma'})$ of Proposition \ref{char:intsct}\ref{it2:char:intsct}, we see that any element $F''\neq F'$ of the intersection satisfies $\dgr{F}{F''} < \dgr{F}{F'}$.
{The final part of the aforementioned proposition gives the size of the intersection.}  
\end{proof}

\subsection{Base change of weights}
\label{subsec:BC:WT}

We now define a notion of base change for weights, and show that it is compatible with the notion of base change of types defined above.

\subsubsection{}

\begin{df}
\label{BC-Serre-wt}
Let $\mu\in X_1(\bT)$ and let $F(\mu)$ denote a Serre weight of $\bG(\bbZ_p)$ on which $\imath(\cO_K^\times)$ acts trivially.  As in Subsection \ref{trivactiondescent}, we may assume $\mu$ is of the form
$$\mu = \vecf{a_0}{b_0}{0}{0}\vecf{a_1}{b_1}{0}{0}\cdots \vecf{a_{f - 1}}{b_{f - 1}}{0}{0}.$$
We define the \emph{base change of $F(\mu)$} as
$$\BC(F(\mu)) \defeq F'(\mu, -\un{s}(\mu)).$$
\end{df}

One easily checks that the map $F \longmapsto \BC(F)$ is well-defined and injective on isomorphism classes of Serre weights.

\subsubsection{}

Recall the automorphism $\epsilon$ of $\bG'(\bbF_p) \cong \bG\bL_2(\bbF_{q^2})$ defined in Subsection \ref{epsilondef}.  On Serre weights, this automorphism gives
$$\epsilon\left(F'(\mu,\mu')\right) = F'\big(-\un{s}(\mu'),-\un{s}(\mu)\big).$$
{We have the following result:}

\begin{lem}\label{invforBCsw}
Let $F'$ denote a Serre weight of $\bG\bL_2(\bbF_{q^2})$.  Then we have $\epsilon(F') \cong F'$ if and only if $F'$ is of the form $\BC(F)$ for a Serre weight $F$ of $\bU_2(\bbF_q)$.  
\end{lem}

\begin{proof}
{The backwards implication is clear.  We prove the forward implication.  Thus, suppose $\mu, \mu' \in X_1(\bT)$ are as in Definition \ref{BC-Serre-wt}, and suppose we have an isomorphism
$$F'(\mu, \mu') \cong F'(-\un{s}(\mu'), -\un{s}(\mu)).$$
By Proposition \ref{Serre-wt-classn-GL2}, there exists $\beta' \in X^0(\bT')$ such that
\begin{equation}
\label{eq:stable}(\mu + \un{s}(\mu'), \mu' + \un{s}(\mu)) = (\biF' - 1)\beta'.
\end{equation}
Since the right-hand side of \eqref{eq:stable} is fixed by $(\un{s},\un{s})$, the same is true of the left-hand side.  This implies that the left-hand side is also fixed by $\pi'^f$, from which we obtain
$$(\biF' - 1)\pi'^f(\beta') = (\biF' - 1)\beta'.$$
Since $\biF' - 1$ is injective on $X^0(\bT')$, we see that $\beta'$ lies in $\ker(\pi'^f - 1)$, which in turn is equal to $\textnormal{im}(\pi'^f + 1)$. Therefore, we can write $\beta' = (\beta,\beta)$ with $\beta \in X^0(\Res_{\cO_K/\zp}(\bT_\bU))$, and equation \eqref{eq:stable} becomes
\begin{equation}
\label{invforBCsw-eqn}
{(\mu + \un{s}(\mu'), \mu' + \un{s}(\mu)) = (\biF' - 1)(\beta,\beta).}
\end{equation}}

{Applying Proposition \ref{Serre-wt-classn-GL2} again, we get an isomorphism
$$F'(\mu,\mu') \cong F'\big((\mu, \mu') + (\biF' - 1)(-\beta,0)\big).$$
The equation \eqref{invforBCsw-eqn} then implies that the term on the right-hand side above is of the form $F'(\mu'', -\un{s}(\mu''))$, and the result follows.}
\end{proof}

\subsubsection{}

We now wish to relate base change of types with base change of weights.  The relevant result is the following.

\begin{lem}
\label{JHunderBC}
Let $\sigma$ denote a $1$-generic Deligne--Lusztig representation of $\bG(\bbZ_p)$ on which $\imath(\cO_K^\times)$ acts trivially, and let $F$ denote a Serre weight on which $\imath(\cO_K^\times)$ acts trivially.  We then have
$$F \in \JH(\overline{\sigma}) \quad \Longleftrightarrow \quad \BC(F) \in \JH\big(\overline{\BC(\sigma)}\big).$$
\end{lem}

\begin{proof}
Let us write $\sigma \cong R_w(\mu)$ where $\mu$ is of the form
$$\mu = \vecf{a_0}{b_0}{0}{0}\vecf{a_1}{b_1}{0}{0}\cdots \vecf{a_{f - 1}}{b_{f - 1}}{0}{0}$$
and $\mu - \eta$ is $1$-deep.  Thus $\BC(\sigma) \cong R'_{(w,w)}(\mu,-\un{s}(\mu))$.

Suppose first that $F \in \JH(\overline{\sigma})$.  By equation \eqref{jantzen}, $F$ is of the form
$$F \cong F_{w'}(R_{w}(\mu)) = F\big(p\gamma_{w'} + w'(\mu - w\pi(\varepsilon_{\un{s}w'})) + p\rho_{w'} - \pi(\rho)\big)$$
for some $w' \in W$.  Note that the parenthesized character has its last two entries equal to 0 in each embedding.  Thus, we have
\begin{flushleft}
$\BC(F)  \cong  F'\Big(p(\gamma_{w'}, - \gamma_{w'}) + (w',w')\big((\mu, -\un{s}(\mu)) - (w,w)(\pi(\varepsilon_{\un{s}w'}), -\un{s}\pi(\varepsilon_{\un{s}w'}))\big)$
\end{flushleft}
\begin{flushright}
$  + p(\rho_{w'},-\un{s}(\rho_{w'})) - (\pi(\rho),-\un{s}\pi(\rho))\Big).$
\end{flushright}
A straightforward calculation shows that adding $(p - \pi')(0,2\gamma_{w'} + \rho_{w'} + \un{s}(\rho_{w'}))\in (\biF' - 1)X^0(\bT')$ to the parenthesized character gives
$$\BC(F)  \cong  F'\Big(p\gamma'_{(w',w')} + (w',w')\big((\mu, -\un{s}(\mu)) - (w,w)\pi'(\varepsilon'_{(\un{s}w',\un{s}w')})\big) + p\rho'_{(w',w')} - \rho'\Big).$$
Hence, we obtain
$$\BC(F) \cong F'_{(w',w')}\big(R'_{(w,w)}(\mu,-\un{s}(\mu))\big) \in \JH\big(\overline{R'_{(w,w)}(\mu,-\un{s}(\mu))}\big) = \JH\big(\overline{\BC(\sigma)}\big).$$

To prove the converse we begin with an observation.  Let $F'_{(v,v')}(\BC(\sigma))$ be a Jordan--H\"older factor of ${\overline{\BC(\sigma)} = \overline{R'_{(w,w)}(\mu,-\un{s}(\mu))}}$ as in \eqref{jantzenGL2}.  (Note that $F'_{(v,v')}(\BC(\sigma))$ depends on the pair $(w,\mu)$.  For readability, we fix the presentation $(w,\mu)$ and omit this dependence from the notation.)  Since $\epsilon(\BC(\sigma)) \cong \BC(\sigma)$ by Lemma \ref{invforBC}, we obtain $\epsilon(F'_{(v,v')}(\BC(\sigma))) \in \JH(\overline{\BC(\sigma)})$.  A similar argument to the one above shows that
$$\epsilon\big(F'_{(v,v')}(\BC(\sigma))\big) \cong F'_{(v',v)}(\BC(\sigma)).$$

Suppose now that $\BC(F) \in \JH(\overline{\BC(\sigma)})$.  Then there exists $(v,v')\in W'$ such that $\BC(F) \cong F'_{(v,v')}(\BC(\sigma))$.  Since $\BC(F)$ is a base change, Lemma \ref{invforBCsw} and the above equation imply
\begin{eqnarray*}
F'_{(v',v)}(\BC(\sigma)) & \cong & \epsilon\big(F'_{(v,v')}(\BC(\sigma))\big) \\
 & \cong & \epsilon\big(\BC(F)\big) \\
 & \cong & \BC(F) \\
 & \cong & F'_{(v,v')}(\BC(\sigma)).
\end{eqnarray*}
Since {$\mu - \eta$ is 1-deep, equation \eqref{jantzenGL2} {and Remark \ref{rmk:bij}\ref{it1:rmk:bij}} implies that we have a bijection between $W'$ and the (all distinct) Jordan--H\"older factors of $\overline{\BC(\sigma)}$.  We obtain $v' = v$, and thus}
$$\BC(F) \cong F'_{(v,v)}(\BC(\sigma)) \cong \BC(F_v(R_w(\mu))).$$
Since the base change map is injective on Serre weights, we conclude that 
$$F \cong F_v(R_w(\mu)) \in \JH\big(\overline{R_w(\mu)}\big) = \JH(\overline{\sigma}).$$
\end{proof}

\subsubsection{}

The following lemma will be useful in the proof of Theorem \ref{thm:WE}.

\begin{lem}
\label{emptyint}
{Let $\sigma$ be a $2$-generic Deligne--Lusztig representation of $\bG(\bbZ_p)$ on which $\imath(\cO_K^{{\times}})$ acts trivially, and let $F$ denote a $3$-deep Serre weight with trivial action of $\imath(\cO_K^{{\times}})$ such that $F \not\in \JH(\overline{\sigma})$.  Then there exists another Deligne--Lusztig representation $\sigma'$ of $\bG(\bbZ_p)$ such that $F \in \JH(\overline{\sigma'})$ and $\JH(\overline{\sigma})\cap \JH(\overline{\sigma'}) = \emptyset$.}
\end{lem}

\begin{proof}
If $\overline{\sigma}$ and $F$ have different central characters, then any $\sigma'$ for which $F \in \JH(\overline{\sigma'})$ works.  We may therefore assume that $\overline{\sigma}$ and $F$ have the same central character.  The remainder of the proof will be based on the combinatorics of the \emph{extension graph} for Serre weights for $\bG\bL_2$, as defined in \cite[\S 2]{LMS}.  We recall some of the definitions and constructions of \emph{op. cit.} (and use similar notation for convenience of comparison).

Define ${\Lambda}'_{W}$ to be the weight lattice for $\bG'^{\textnormal{der}}$, the derived subgroup of $\bG'$, and let $\Lambda'_R$ denote the root lattice, so that $\Lambda'_R \subseteq \Lambda'_W$.  Note that ${\Lambda}'_W\cong \bbZ^f\times \bbZ^f$ and we fix such an identification in what follows.  Recall {from Subsubsection \ref{subsubsec:groupG'}} that $X^*(\bT')$ denotes the weight lattice for the group $\bG'$; we have $\Lambda'_R \subseteq X^*(\bT')$ and $X^*(\bT') \longtwoheadrightarrow \Lambda_W$.  We further {write} $\tld{W}'_a$ (resp.~$\tld{W}'$) 
{for} the affine (resp.~extended affine) Weyl group of $\bG'$, which admits a factorization $\tld{W}'_a \cong W'\ltimes \Lambda'_R$ (resp.~$\tld{W}' \cong W'\ltimes X^*(\bT')$).  The group $\widetilde{W}'_a$ (resp.~$\widetilde{W}'$) is canonically isomorphic to two copies of $W \ltimes {\Lambda_R}$ (resp.~$W \ltimes X^*(\textnormal{Res}_{\cO_K/\bbZ_p}(\bT_\bU))$)  (compare with Subsubsection \ref{subsubsec:groupG'}).

We let ${\Omega'}$ denote the set of elements of $\widetilde{W}'$ which stabilize the fundamental alcove $C_0'$ of $\bG'$ under the \textit{$p$-dilated dot action} $\sbulletp$:
$$w't_{\lambda'}  \sbulletp  \mu' = w'(\mu' + \rho' + p\lambda') - \rho',$$
where $w't_{\lambda'} \in \widetilde{W}'$ and $\mu' \in X^*(\bT')$.   Thus, ${\Omega'}$ is the analog of ${\Omega}$ of Subsection \ref{subsec:cmb:types}, except that translations have been scaled by a factor of $p^{-1}$.  (We apologize for this inconsistency of notation.)

Let $\mu' \in X^*(\bT')$ satisfy $0 \leq \langle \mu', \alpha'^\vee\rangle < p - 1$ for every positive coroot $\alpha'^\vee$ of $\bT'$.  We call such characters \textit{$p$-regular}.  We then define the map $\Trns'_{\mu'}$ by
\begin{eqnarray*}
\Trns'_{\mu'}:{\Lambda}'_W & \longrightarrow & \frac{X^*(\bT')}{(\biF'-1)X^0(\bT')}\\
\omega' & \longmapsto & \widehat{\tld{w}}' \sbulletp  (\mu'+\widehat{\omega}'-\rho')
\end{eqnarray*}
where $\widehat{\omega}'\in X^*(\bT')$ is a lift of $\omega'\in \Lambda'_W$, and $\widehat{\tld{w}}'$ is the unique element in ${\Omega'}$ such that the class of $-\pi'^{-1}(\widehat{\omega}')$ corresponds to the class of $\widehat{\tld{w}}'$ via the isomorphism $X^*(\bT')/{\Lambda}'_{R}\stackrel{\sim}{\longrightarrow} \tld{W}'/\tld{W}'_a$.  Note that this is well defined.
Define furthermore 
\begin{equation}
\label{def:LaMu}
{\Lambda'}_W^{\mu'}\defeq\{\omega'\in \Lambda'_W:\, \omega' + \mu' - \rho' \in C_0'\}
\end{equation}
(where we consider the image of $\mu' - \rho'$ and $C_0'$ in $\Lambda'_W$), and let $\Trns_{\mu'}$ be the restriction of $\Trns_{\mu'}'$ to ${\Lambda'}_W^{\mu'}$.
Then, as in \cite[\S 2.1]{LLLM2}, one checks that:
\begin{enumerate}
\item 
\label{it:JH:graph:item1}
The image of $\Trns_{\mu'}$ is contained in the set of $p$-regular characters.  Further, the map $\omega' \longmapsto F'(\Trns_{\mu'}(\omega'))$ defines a bijection between ${\Lambda'}_W^{\mu'}$ and the set of $p$-regular Serre weights with the same central character as $F'(\mu' - \rho')$ (see the discussion preceding \cite[Prop. 2.9]{LMS}). 
\item 
\label{it:JH:graph:prime}
Suppose $\mu'\in X^*(\bT')$ is such that $\mu' - \rho'$ is 2-deep, and consider the Deligne--Lusztig representation $R'_{w'}(\mu')$.  Applying the analog of equation \eqref{DLequiv} for $\bG'(\bbZ_p)$, we obtain an isomorphism
$$R'_{w'}(\mu') \cong R'_{w'}\big((\un{s},\un{s})(\mu')  + p\rho' - w'(\rho')\big),$$
where the character $(\un{s},\un{s})(\mu')  + p\rho' - w'(\rho') - \rho'$ is 1-deep.  Combining this isomorphism with \cite[Props. {2.5 and} 2.11]{LMS}, Proposition \ref{prop:weights:type} and Remark \ref{rmk:weights:type} below, we obtain
$$\JH(\overline{R'_{w'}(\mu')}) = \left\{F'\Big(\Trns_{\mu' + \rho'}\big(t_{-\alpha_{w'}} (\un{s}, \un{s})w' (\Sigma')\big)\Big)\right\},$$
where $\Sigma'\subseteq \Lambda'_{W}$ is the subset consisting of (images of) elements of the form $\vect{1}{0}$ or $\vect{0}{0}$ in each embedding.  (That is, $\Sigma'$ is the image in $\Lambda'_W$ of $\{\rho'_{w'}\}_{w'\in W'}$.)
\item 
\label{it:JH:graph:item3}
Let $\mu\in X^*(\textnormal{Res}_{\cO_K/\bbZ_p}(\bT_{\bU})) \subseteq X^*(\bT)$ satisfy $0 \leq \langle \mu+\rho, \alpha^\vee\rangle < p - 1$ for every positive coroot $\alpha^\vee$ of $\bT$.  We let $\Lambda_W$ denote the weight lattice of $\bG^{\textnormal{der}}$ (which is a quotient of $X^*(\textnormal{Res}_{\cO_K/\bbZ_p}(\bT_{\bU}))$~), and define 
\begin{equation*}
{\Lambda_W^{\mu+\rho}\defeq\{\omega\in \Lambda_W:\, \omega + \mu \in C^{\textnormal{der}}_0\}}
\end{equation*}
{where $C^{\textnormal{der}}_0$ is the fundamental alcove of $\bG^{\textnormal{der}}$.}
Consider the map
\begin{eqnarray*}
{\Trns^{\bG}_{\mu+\rho}}:\Lambda^{\mu + \rho}_W & \longrightarrow & \frac{X_1(\bT')}{(\biF'-1)X^0(\bT')}\\
\omega & \longmapsto & \Trns_{(\mu+\rho,-\underline{s}(\mu)+\rho)}(\omega,-\underline{s}(\omega)).
\end{eqnarray*}
Using Lemma \ref{invforBCsw} and item \ref{it:JH:graph:item1}, one checks that $\omega \longmapsto F'({\Trns^{\bG}_{\mu+\rho}}(\omega))$ defines a bijection between $\Lambda^{\mu+\rho}_W$ and $p$-regular Serre weights of $\bG'(\bbZ_p)$ which are in the image of the base change map and have the same central character as $F'(\mu, -\underline{s}(\mu))$.  
\end{enumerate}

We now proceed with the proof.  Let us write $\sigma = R_w(\mu)$, with $\mu$ chosen as in Lemma \ref{JHunderBC} and $\mu - \rho$ being $2$-deep.  By assumption and Lemma \ref{JHunderBC}, we have 
$$\BC(F)\not\in \JH(\ovl{\BC(\sigma)}) = \JH\big(\overline{R'_{(w,w)}(\mu, -\underline{s}(\mu))}\big).$$
Since the character $(\mu, -\underline{s}(\mu)) - \rho'$ is 2-deep, item \ref{it:JH:graph:prime} above implies
$$\BC(F) \not\in \left\{F'\Big(\Trns_{(\mu, -\underline{s}(\mu)) + \rho'}\big(t_{-\alpha_{(w,w)}} (\un{s}w, \un{s}w) (\Sigma')\big)\Big)\right\}.$$
Therefore, since $F$ and $\overline{\sigma}$ have the same central character, item \ref{it:JH:graph:item3} implies that $\BC(F) = F'({\Trns^{\bG}_{\mu+\rho}}(\omega))$ for some $\omega \in \Lambda_W^{\mu + \rho} \smallsetminus t_{-\alpha_{w}} \un{s}w(\Sigma)$. (Here, $\Sigma$ is the image in $\Lambda_W$ of $\{\rho_w\}_{w\in W}$.)  Since $\Sigma'$ is a fundamental domain for the translation action of $\Lambda'_R$ on $\Lambda'_W$, there exists an element  $t_{(\nu,-\underline{s}(\nu))} \in \Lambda_R' \subseteq \tld{W}'$ such that 
\begin{equation}\label{translate}
(\omega, -\underline{s}(\omega)) \in t_{(\nu, -\un{s}(\nu))}t_{-\alpha_{(w,w)}} (\un{s}w, \un{s}w) (\Sigma').
\end{equation}
{(Since the left-hand side of the above containment is fixed by $\epsilon$, the translation element of $\Lambda_R'$ must be of the form $t_{(\nu,-\underline{s}(\nu))}$.)}  Note that this implies $\nu \neq 0$ and consequently 
\begin{equation}\label{emptyintersec}
t_{-\alpha_{(w,w)}} (\un{s}w, \un{s}w) (\Sigma') \cap t_{(\nu, -\un{s}(\nu))}\left(t_{-\alpha_{(w,w)}} (\un{s}w, \un{s}w) (\Sigma')\right) = \emptyset.
\end{equation}

Recall that we have assumed $F$ is 3-deep.  Therefore the same is true of $\BC(F)$.  Using the relation $\BC(F) = F'({\Trns^{\bG}_{\mu+\rho}}(\omega))$, and the fact that ${\Omega'}$ preserves $C_0'$ under $\sbulletp$, we get that the character $\mu + \omega$ is 3-deep.  On the other hand, the relation \eqref{translate} implies that we have 
$$\omega = \nu - \alpha_w + \un{s}w(\rho_v)$$
for some $v\in W$.  Since $0 \leq \langle \alpha_w - \un{s}w(\rho_v), \alpha_i^\vee \rangle \leq 2$ for all $0 \leq i \leq f - 1$, the relation 
$$\mu + \omega + \alpha_w - \un{s}w(\rho_v) = \mu + \nu $$
implies that $2 < \langle \mu + \nu, \alpha_i^\vee \rangle < p - 2$ for all $0 \leq i \leq f - 1$.  That is, we have that $(\mu, -\un{s}(\mu)) + (\nu, -\un{s}(\nu)) - \rho'$ is 2-deep.

Now set $\sigma' \defeq R_w(\mu + \nu)$.  By the previous paragraph and item \ref{it:JH:graph:prime}, we have
\begin{eqnarray*}
\JH(\ovl{\BC(\sigma')}) & = & \JH\big(\overline{R'_{(w,w)}(\mu + \nu, -\underline{s}(\mu + \nu))}\big) \\
 & = & \left\{F'\Big(\Trns_{(\mu, -\underline{s}(\mu)) + (\nu, -\un{s}(\nu))+ \rho'}\big(t_{-\alpha_{(w,w)}} (\un{s}w, \un{s}w) (\Sigma')\big)\Big)\right\} \\
 & = & \left\{F'\Big(\Trns_{(\mu, -\underline{s}(\mu)) + \rho'}\big(t_{ (\nu, -\un{s}(\nu))}t_{-\alpha_{(w,w)}} (\un{s}w, \un{s}w) (\Sigma')\big)\Big)\right\},
\end{eqnarray*}
where the last equality follows from the definition of $\Trns_{\mu'}$ and the fact that $(\nu,-\un{s}(\nu)) \in \Lambda_R'$.  Thus, the relation $\BC(F) = F'({\Trns^{\bG}_{\mu+\rho}}(\omega))$ and equation \ref{translate} imply that $\BC(F) \in \JH(\ovl{\BC(\sigma')})$, and the injectivity of $\Trns_{\mu'}$ and equation \ref{emptyintersec} imply $\JH(\overline{\BC(\sigma)}) \cap \JH(\overline{\BC(\sigma')}) = \emptyset$.  We conclude by using Lemma \ref{JHunderBC} and Proposition \ref{prop:double:weights}.
\end{proof}

\section{Predicted Serre weights}
\label{sec:Lpar}

In this section we discuss the conjectural set of weights attached to Galois parameters and their relation with base change.  We give the relevant definitions in Subsection \ref{subsec:L:pmts}, along with a classification of mod $p$ tamely ramified $L$-parameters.  We then define the set $\textnormal{W}^?(\rhobar)$ in Subsection \ref{subsec:predwts}.  The main result is Theorem \ref{main:thm:local}, which relates the sets $\textnormal{W}^?(\rhobar)$ and $\textnormal{W}^?(\BC(\rhobar))$.  Finally, we state in Subsection \ref{ILLC-sect} a version of the inertial local Langlands correspondence that we will require for local/global compatibility.  Our discussion is based on \cite[\S 9]{ghs}.

\subsection{$L$-parameters}
\label{subsec:L:pmts}

\subsubsection{}

We first define the Galois representations we shall consider.

\begin{df}
Let $R$ be a topological $\zp$-algebra.
An \emph{$L$-parameter (with $R$-coefficients)} is a continuous homomorphism $\Gamma_\qp \longrightarrow {}^L\bG(R)$, which is compatible with the projection to $\textnormal{Gal}(K_2/\qp)$.  
Likewise, we define an \emph{inertial $L$-parameter} (or an \emph{inertial type}) to be a continuous homomorphism $I_{\qp} \longrightarrow \widehat{\bG}(R)$ which admits an extension to an $L$-parameter $\Gamma_{\qp} \longrightarrow {}^L\bG(R)$.  We say two (inertial) $L$-parameters are \emph{equivalent} if they are $\widehat{\bG}(R)$-conjugate.

We make similar definitions for homomorphisms valued in $\cG_2(R)$.  
\end{df}

By \cite[Lem. 9.4.1]{ghs}, the $\widehat{\bG}(R)$-conjugacy classes of $L$-parameters $\Gamma_{\qp} \longrightarrow {}^L\bG(R)$ are in bijection with $\widehat{\bH}(R)$-conjugacy classes of $L$-parameters $\Gamma_K \longrightarrow {}^L\bH(R) = {}^C\bU_2(R)$.  A similar statement holds for inertial $L$-parameters (cf. \emph{op. cit.}, Lemma 9.4.5).

We make similar definitions of $L$-parameters $\Gamma_{F^+} \longrightarrow {}^C\bU_2(R)$ if $F^+$ is a global field with a place $v$ satisfying $F^+_v \cong K$ (cf. Remark \ref{cgroupglobal}).

\subsubsection{}

The following lemma is easily checked.

\begin{lem}\label{galreps}
Let $\rhobar:\Gamma_K\longrightarrow {}^C\bU_2(\bbF)$ denote an $L$-parameter such that $\rhobar|_{\Gamma_{K_2}}$ is semisimple \emph{(}or, equivalently, tamely ramified\emph{)}.  Then, up to equivalence, $\rhobar$ is of one of the following two forms:
\begin{enumerate}
\item \begin{eqnarray*}
\rhobar(h) & = & \left(\begin{pmatrix}\omega_{2f}^r\textnormal{nr}_{2f,\nu^{-1}\lambda}(h) & 0 \\ 0 & \omega_{2f}^{-qr + (q + 1)s}\textnormal{nr}_{2f,\nu\lambda}(h)\end{pmatrix}, \omega_f^s\textnormal{nr}_{2f,\lambda^2}(h)\right)\rtimes 1,\\
\rhobar(\varphi^{-f}) & = & \left(\begin{pmatrix}1 & 0 \\ 0 & \nu\end{pmatrix},\lambda\right)\rtimes\varphi^{-f},
\end{eqnarray*}
where $h\in \Gamma_{K_2}$, $0\leq r < q^2 - 1$, $0 \leq s < q - 1$, and $\lambda,\nu\in \bbF^\times$.  
\item \begin{eqnarray*}
\rhobar(h) & = & \left(\begin{pmatrix}\omega_{2f}^{s + (1 - q)k}\textnormal{nr}_{2f,-\lambda}(h) & 0 \\ 0 & \omega_{2f}^{s + (1 - q)\ell}\textnormal{nr}_{2f,-\lambda}(h)\end{pmatrix}, \omega_f^s\textnormal{nr}_{2f,\lambda^2}(h)\right)\rtimes 1,\\
\rhobar(\varphi^{-f}) & = & \left(\begin{pmatrix}0 & -1 \\ 1 & 0\end{pmatrix},\lambda\right)\rtimes\varphi^{-f},
\end{eqnarray*}
where $h\in \Gamma_{K_2}$, $0\leq k,\ell < q + 1$, $0 \leq s < q - 1$, and $\lambda\in \bbF^\times$.  
\end{enumerate}
In both cases $\textnormal{nr}_{2f,x}$ denotes the unramified character of $\Gamma_{K_2}$ sending $\varphi^{-2f}$ to $x$.  
\end{lem}

\subsubsection{}

\begin{df}
Let $R$ denote a topological $\bbZ_p$-algebra.  
\begin{enumerate}
\item  Let $\rho:\Gamma_K\longrightarrow {}^C\bU_2(R)$ denote an $L$-parameter, and write $\rho|_{\Gamma_{K_2}} = \rho_2\oplus \rho_1$, where $\rho_2:\Gamma_{K_2}\longrightarrow \bG\bL_2(R), \rho_1:\Gamma_{K_2}\longrightarrow \bG_m(R) = R^\times$.  We define the \emph{base change of $\rho$} to be
$$\BC(\rho) \defeq \rho_2.$$
\item Let $\rho:\Gamma_K\longrightarrow \cG_2(R)$ denote an $L$-parameter valued in $\cG_2$, and write $\rho|_{\Gamma_{K_2}} = \rho_2\oplus \rho_1$, where $\rho_2:\Gamma_{K_2}\longrightarrow \bG\bL_2(R), \rho_1:\Gamma_{K_2}\longrightarrow \bG_m(R) = R^\times$.  We define the \emph{base change of $\rho$} to be
$$\BC'(\rho) \defeq \rho_2.$$
\end{enumerate}
\end{df}

We make similar definitions if $F^+$ is a global field with a place $v$ satisfying $F^+_v \cong K$ (cf. Remark \ref{cgroupglobal}).

The two notions of base change are related as follows.  Let $\rho:\Gamma_K \longrightarrow {}^C\bU_2(R)$ denote an $L$-parameter, and let $\theta$ denote the continuous character $\widehat{\imath}\circ\rho: \Gamma_K \longrightarrow R^\times$.  Using the isomorphism of Subsection \ref{isomsect}, we get an isomorphism of $\bG\bL_2(R)$-valued Galois representations
\begin{equation}\label{twoBC}
\BC'(\rho) \cong \BC(\rho)\otimes\theta^{-1}.
\end{equation}

\subsubsection{}

Recall from Subsubsection \ref{subsub:dual:root} that we have a map $(\phi^\vee)^{-1}:X^*(\bT_\bH) \stackrel{\sim}{\longrightarrow} X_*(\widehat{\bT}_\bH)$, which induces an isomorphism $X^*(\bT) \stackrel{\sim}{\longrightarrow} X_*(\widehat{\bT})$.  Given $\mu\in X^*(\bT)$ (viewed as an element of $X_*(\widehat{\bT})$) and $w\in W$, we define a tamely ramified inertial $L$-parameter $\tau(w,\mu): I_K \longrightarrow \widehat{\bH}(\bbF)$ by
$$\tau(w,\mu) \defeq \prod_{i = 0}^{2f - 1}(\biF^*\circ w^{-1})^i\big(\mu(\omega_{2f})\big)$$
{(compare with \cite[\S 9.2]{ghs} and note that $(\biF^*\circ w^{-1})^{2f}=p^{2f}$).}

We define $\BC(\phi)$ to be the canonical identification of the dual root datum of the split group $\bG\bL_{2/\cO_{K_2}}$ with the root datum of its dual group.  Given this, we make an analogous definition of tamely ramified inertial $L$-parameters $\tau'((w,w'), (\mu,\mu')):I_{K_2} \longrightarrow \bG\bL_2(\bbF)$.

\begin{lem}\label{BCinertial}
Suppose $\rhobar:\Gamma_K \longrightarrow {}^C\bU_2(\bbF)$ is a tamely ramified $L$-parameter which satisfies $\widehat{\imath}\circ \rhobar = \omega$. 
Via the identification of \cite[Lem. 9.4.5]{ghs} we have
$$\rhobar|_{I_K} \cong \tau(w,\mu + \eta)$$
with $w\in W$ and $\mu\in X^*(\bT)$ of the form
$$\mu = \vecf{a_0}{b_0}{0}{0}\cdots \vecf{a_{f - 1}}{b_{f - 1}}{0}{0}.$$
Furthermore, we have
$$\BC(\rhobar)|_{I_{K_2}} \cong \tau'\big((w,w),~(\mu,-\un{s}(\mu)) + \rho'\big).$$
\end{lem}

\begin{proof}
The proof is a straightforward exercise using the definitions.
\end{proof}

\subsubsection{}

We will also need a definition of \emph{genericity} to study the relation between $L$-parameters, the set of conjectural associated weights and local deformations.

\begin{df}
\label{df:n:generic}
Suppose $\rhobar: \Gamma_{K} \longrightarrow {}^C\bU_2(\bbF)$ is a tamely ramified $L$-parameter.  
We say $\rhobar$ is \emph{$n$-generic} if, via the identification of \cite[Lem. 9.4.5]{ghs}, we can write
$$\rhobar|_{I_K} \cong \tau(w,\mu + \eta)$$
where $w\in W$ and $\mu\in X^*(\bT)$ is $n$-deep.  
\end{df}

\subsection{The set $\textnormal{W}^?$}
\label{subsec:predwts}

We now give a description of the set $\textnormal{W}^?$. We refer to \cite[\S 9]{ghs} for the definition, and to \emph{op. cit.}, Proposition 9.2.1 for the definition of $V_\phi$.

\begin{prop}
\label{prop:weights:type}
Let $\rhobar: \Gamma_K\longrightarrow {}^{C}\bU_2(\bbF)$ be a $1$-generic tamely ramified $L$-parameter, and write $\rhobar|_{I_K} \cong \tau(w,\mu + \eta)$ as in Definition \ref{df:n:generic}, with $\mu$ being $1$-deep.  Let $V_\phi(\rhobar) = R_{w}(\mu + \eta)$ be the associated Deligne--Lusztig representation of $\bG(\Fp)$ as in \cite[Props. 9.2.1 and 9.2.2]{ghs}.
Then
$$\textnormal{W}^?(\rhobar) = \JH\left(\overline{\beta(R_{w}(\mu + \eta))}\right).$$
\end{prop}

\begin{proof}
By definition of $\textnormal{W}^?(\rhobar)$, we must prove that
$$\mathcal{R}\left(\JH\big(\overline{R_{w}(\mu + \eta)}\big)\right) = \JH\left(\overline{\beta(R_{w}(\mu + \eta))}\right),$$
where $\cR$ is the reflection operator defined in \cite[\S 9.2]{ghs}.  We use Equation \eqref{jantzen}.  We claim that
\begin{equation}
\label{match-nose}
\mathcal{R}\left(F_{w'}(R_{w}(\mu + \eta))\right) \cong F_{w'}\left(\beta(R_{w}(\mu + \eta))\right)
\end{equation}
for all $w'\in W$.  Note first that $\pi$ and $\un{s}$ commute as operators on $X^*(\bT)$, and the group $W$ is commutative.  Therefore, in order to prove \eqref{match-nose}, it suffices to show
\begin{flushleft}
$p\un{s}(\gamma_{w'}) + \un{s}w'(\mu + \eta) - \un{s}w'w\pi(\varepsilon_{\un{s}{w'}}) + p\un{s}(\rho_{w'}) - \pi(\un{s}(\rho)) + \un{s}(\rho) - p\un{s}(\eta) - \rho \equiv$
\end{flushleft}
\begin{flushright}
$p\gamma_{w'} + w'\left(\un{s}(\mu + \eta) - \un{s}(\eta)+(p - 1)\eta - \un{s}w\pi(\varepsilon_{\un{s}{w'}})\right) + p\rho_{w'} - \pi(\rho),$
\end{flushright}
the equivalence being taken modulo $(\biF - 1)X^0(\bT)$.

One easily checks that $\un{s}(\gamma_{w'}) = \gamma_{w'}$ and $-\pi(\un{s}(\rho)) + \un{s}(\rho) - \rho = -\pi(\rho)$, and hence (\ref{match-nose}) will be satisfied if we show that
\begin{equation}
\label{match-nose1}
p\un{s}(\rho_{w'}) - p\un{s}(\eta) \equiv -w'\un{s}(\eta) + (p - 1)w'(\eta) + p\rho_{w'}
\end{equation}
modulo $(\biF - 1)X^0(\bT)$.

Expanding the left-hand side gives 
$$p\un{s}(\rho_{w'}) - p\un{s}(\eta) = \cdots\underbrace{\vecf{0}{0}{0}{-p}}_{w_i' = 1}\cdots \underbrace{\vecf{0}{p}{0}{-p}}_{w_i' = s}\cdots,$$
while expanding the right-hand side gives
$$-w'\un{s}(\eta) + (p - 1)w'(\eta) + p\rho_{w'} = \cdots\underbrace{\vecf{0}{0}{p - 1}{-1}}_{w_i' = 1}\cdots \underbrace{\vecf{p}{0}{-1}{p - 1}}_{w_i' = s}\cdots.$$
In particular, adding 
$$(\biF - 1)\vecf{\un{0}}{\un{0}}{\un{1}}{\un{1}} =  \vecf{\un{0}}{\un{0}}{\un{p - 1}}{\un{p - 1}}$$
to the left-hand side of (\ref{match-nose1}) gives
$$\cdots\underbrace{\vecf{0}{0}{p - 1}{-1}}_{w_i' = 1}\cdots \underbrace{\vecf{0}{p}{p - 1}{-1}}_{w_i' = s}\cdots = \cdots\underbrace{\vecf{0}{0}{p - 1}{-1}}_{w_i' = 1}\cdots \underbrace{\vecf{p}{0}{-1}{p - 1}}_{w_i' = s}\cdots,$$
where the equality follows form the equivalence relation on $X^*(\bT)$.  This gives the claim.  
\end{proof}

\begin{rmk}
\label{rmk:weights:type}
The above proposition and its proof carry over mutatis mutandis to the group $\bG\bL_2(\cO_{K_2})$ and a tamely ramified Galois parameter $\Gamma_{K_2}\longrightarrow \bG\bL_2(\bbF)$ (cf. \cite{diamond}).  
\end{rmk}

\subsection{Base change and $\textnormal{W}^?$}
\label{subsec:BCandW}

This section contains the main result on compatibility between the set $\textnormal{W}^{?}$ and base change of $L$-parameters (Theorem \ref{main:thm:local}).

\subsubsection{}

\begin{prop}
\label{prop:compatibilityBC}
Let $\rhobar:\Gamma_{K}\longrightarrow {}^{C}\bU_2(\bbF)$ be a $1$-generic tamely ramified $L$-parameter which satisfies $\widehat{\imath}\circ\rhobar = \omega$.  Then the subgroup $\imath(\cO_K^\times)$ acts trivially on $\beta(V_{\phi}(\rhobar))$, and 
$$\beta'\left(V_{\BC(\phi)}(\BC(\rhobar))\right) \cong \BC\big(\beta(V_{\phi}(\rhobar))\big).$$
\end{prop}

\begin{proof}
By Lemma \ref{BCinertial}, we may write 
$$\rhobar|_{I_K} \cong \tau(w,\mu + \eta),$$
with $\mu$ being $1$-deep and of the form
$$\mu = \vecf{a_0}{b_0}{0}{0}\cdots \vecf{a_{f - 1}}{b_{f - 1}}{0}{0}.$$
Applying the map $\beta$ to $V_\phi(\rhobar) \cong R_w(\mu + \eta)$ gives
$$\beta(V_\phi(\rhobar)) \cong R_{\un{s}w}(\un{s}(\mu) + (p - 1)\eta).$$
Notice that $\imath(\cO_K^\times)$ acts trivially on this representation.  In order to apply the base change map, the character appearing inside the Deligne--Lusztig representation must have its last two entries equal to zero.  Using the equivalence given by adding the element $-(\biF - \un{s}w)(\eta)$, we get
$$\beta(V_\phi(\rhobar)) \cong R_{\un{s}w}\big(\un{s}(\mu) - \eta + \un{s}w(\eta)\big) = R_{\un{s}w}\Big(\un{s}(\mu) - \sum_{w_i = 1}\alpha_i\Big),$$
and by \eqref{BCDLPS} or \eqref{BCDLcusp}, we obtain
$$\BC\big(\beta(V_\phi(\rhobar))\big) \cong R_{(\un{s}w, \un{s}w)}'\Big(\un{s}(\mu) - \sum_{w_i = 1}\alpha_i,~ -\mu - \sum_{w_i = 1}\alpha_i\Big).$$

On the other hand, Lemma \ref{BCinertial} gives
$$\BC(\rhobar)|_{I_{K_2}} \cong \tau'\big((w,w),~(\mu,-\un{s}(\mu)) + \rho'\big),$$
and therefore
$$V_{\BC(\phi)}(\BC(\rhobar)) \cong R_{(w,w)}'\big((\mu,-\un{s}(\mu)) + \rho'\big).$$
Applying the map $\beta'$ gives
$$\beta'\big(V_{\BC(\phi)}(\BC(\rhobar))\big) \cong R_{(\un{s}w,\un{s}w)}'\big((\un{s}(\mu),-\mu) + (p - 1)\rho'\big).$$
Finally, using the equivalence given by adding $-(\biF' - (\un{s}w,\un{s}w))(\rho')$ we obtain
\begin{eqnarray*}
\beta'\big(V_{\BC(\phi)}(\BC(\rhobar))\big) & \cong & R_{(\un{s}w,\un{s}w)}'\big((\un{s}(\mu),-\mu) - \rho' + (\un{s}w,\un{s}w)(\rho')\big) \\
& \cong & R_{(\un{s}w, \un{s}w)}'\Big(\un{s}(\mu) - \sum_{w_i = 1}\alpha_i,~ -\mu - \sum_{w_i = 1}\alpha_i\Big).
\end{eqnarray*}
\end{proof}

\subsubsection{}

The main result of this section concerns local functoriality of predicted Serre weights.

\begin{theo}
\label{main:thm:local}
Let $\rhobar:\Gamma_{K}\longrightarrow {}^{C}\bU_2(\bbF)$ be a $1$-generic tamely ramified $L$-parameter which satisfies $\widehat{\imath}\circ\rhobar = \omega$, and let $F$ denote a Serre weight of $\bG(\bbZ_p)$ on which $\imath(\cO_K^\times)$ acts trivially.  Then
$$F\in \textnormal{W}^?(\rhobar) \quad \Longleftrightarrow \quad \BC(F)\in \textnormal{W}^?(\BC(\rhobar)).$$
\end{theo}

\begin{proof}
This follows by combining Lemma \ref{JHunderBC} and Propositions \ref{prop:weights:type} and \ref{prop:compatibilityBC}.
\end{proof}

\subsection{Inertial Local Langlands}
\label{ILLC-sect}

In this subsection we discuss the inertial local Langlands correspondence which will be used in the rest of the paper.
Recall that a \emph{tame inertial type $\tau'$} is a homomorphism $\tau':I_{K_2} \longrightarrow \bG\bL_2(\cO)$ with open kernel, which is tamely ramified, and such that $\tau'$ extends to a representation of the Weil group of $K_2$.

We set $L \defeq K_2((-p)^{1/(p^{2f} - 1)})${. As $\tau'$ is tame, it factors as} 
$$\tau': I_{K_2} \longtwoheadrightarrow \Gal(L/K_2) \longrightarrow \bG\bL_2(\cO).$$
This implies that $\tau'$ is of the form
$$\tau' \cong \widetilde{\omega}_{2f}^a \oplus \widetilde{\omega}_{2f}^b.$$
If $a \not\equiv b~(\textnormal{mod}~p^{2f} - 1)$, we call such a type a \emph{principal series tame (inertial) type}.  By Henniart's appendix to \cite{breuilmezard}, the inertial type $\tau'$ is associated to the tame type
$$\sigma'(\tau') \defeq \Ind_{\bB_{\bU}(\bbF_{q^2})}^{\bG\bL_2(\bbF_{q^2})}\big(\theta_a \otimes \theta_b\big)$$
if $a \not\equiv b~(\textnormal{mod}~p^{2f} - 1)$, and 
$$\sigma'(\tau') \defeq \theta_a\circ \det$$
if $a \equiv b~(\textnormal{mod}~p^{2f} - 1)$, where we use the notation $\theta_z(x) = \varsigma_0(\tilde{x}^z)$.  We view $\sigma'(\tau')$ as a representation of $\bG\bL_2(\cO_{K_2})$ by inflation.  {(According to the appendix of \cite{breuilmezard}, the $a\equiv b ~(\textnormal{mod}~p^{2f} - 1)$ case corresponds to a twist of the Bernstein component denoted $s_0$ in \emph{op. cit.}, and consequently we have two options for $\sigma'(\tau')$.  We choose $\sigma'(\tau')$ to be one-dimensional in order to guarantee that we are in the $N = 0$ case in Theorem \ref{ILLC} below.)}

Suppose now that $(\tau')^{\varphi^{-f}} \cong \tau'^\vee$, where $\tau'^\vee$ denotes the dual type{, i.e.~the type $\widetilde{\omega}_{2f}^{-a} \oplus \widetilde{\omega}_{2f}^{-b}$ if $\tau' \cong\widetilde{\omega}_{2f}^a \oplus \widetilde{\omega}_{2f}^b$.} 
{(Note that the condition $(\tau')^{\varphi^{-f}} \cong \tau'^\vee$ means exactly that $\tau'$ extends to a map $\rho:\Gamma_K \longrightarrow {}^C\bU_2(\cO)$ such that $\BC(\rho)|_{I_{K_2}} \cong \tau'$ and $(\widehat{\imath}\circ\rho)|_{I_K}$ is the trivial character.)}  
In this case $\tau'$ is of the form 
$$\widetilde{\omega}_{2f}^{c} \oplus \widetilde{\omega}_{2f}^{-q c}\quad \textnormal{or}\quad \widetilde{\omega}_{2f}^{(1 - q)a} \oplus \widetilde{\omega}_{2f}^{(1 - q)b},$$
so that $\sigma'(\tau')$ is of the form
$$\Ind_{\bB_{\bU}(\bbF_{q^2})}^{\bG\bL_2(\bbF_{q^2})}\big(\theta_{c} \otimes \theta_{-qc}\big),\quad  \Ind_{\bB_{\bU}(\bbF_{q^2})}^{\bG\bL_2(\bbF_{q^2})}\big(\theta_{(1 - q)a} \otimes \theta_{(1 - q)b}\big),\quad \textnormal{or}\quad \theta_{(1 - q)a}\circ \det.$$
In particular, these tame types come via base change from tame types of $\bU_2(\cO_K)$.  We therefore make the following definition.

\begin{df}
\label{inertialtype}
Let $\tau'{: I_{K_2} \longrightarrow \bG\bL_2(\cO)}$ denote a tame inertial type which factors through $\Gal(L/K_2)$, and suppose furthermore that $(\tau')^{\varphi^{-f}} \cong \tau'^\vee$.  
\begin{enumerate}
\item If $\tau' \cong \widetilde{\omega}_{2f}^{c} \oplus \widetilde{\omega}_{2f}^{-q c}$ with $c\not\equiv -qc ~(\textnormal{mod}~p^{2f} - 1)$, we set
$$\sigma(\tau') \defeq \Ind_{\bB_{\bU}(\bbF_q)}^{\bU_2(\bbF_q)}(\theta_c),$$
which we view as a representation of $\bU_2(\cO_K)$ via inflation.
\item If $\tau' \cong \widetilde{\omega}_{2f}^{(1 - q)a} \oplus \widetilde{\omega}_{2f}^{(1 - q)b}$ with $a \not\equiv b~(\textnormal{mod}~p^{2f} - 1)$, we set
$$\sigma(\tau') \defeq \sigma(\theta_a\otimes\theta_b),$$
where we view $\theta_a$ and $\theta_b$ as characters of $\bU_1(\bbF_q)$ by restriction, and where we view $\sigma(\tau')$ as a representation of $\bU_2(\cO_K)$ via inflation.
\item If $\tau' \cong \widetilde{\omega}_{2f}^{(1 - q)a} \oplus \widetilde{\omega}_{2f}^{(1 - q)a}$, we set
$$\sigma(\tau') \defeq \theta_a\circ\det,$$
where we view $\theta_a$ as a character of $\bU_1(\bbF_q)$ by restriction, and where we view $\sigma(\tau')$ as a representation of $\bU_2(\cO_K)$ via inflation.
\end{enumerate}
{Note that the representations $\sigma(\tau')$ are all irreducible by Theorem \ref{thm:ennola}, and} by construction we have $\BC(\sigma(\tau')) \cong \sigma'(\tau')$.  
\end{df}

We may now state a version of the inertial Local Langlands correspondence.

\begin{theo}
\label{ILLC}
Let $\tau': I_{K_2} \longrightarrow \bG\bL_2(\cO)$ be a tame inertial type as in Definition \ref{inertialtype}, so that in particular $(\tau')^{\varphi^{-f}} \cong \tau'^\vee$.  Let $\pi$ denote a smooth irreducible representation of $\bU_2(K)$ over $\overline{E}$, and let $\pi^\oplus$ denote the direct sum of all representations appearing in the $L$-packet containing $\pi$.  Let $\BC(\pi)$ denote the stable base change of the $L$-packet containing $\pi$.  Then $\pi^\oplus|_{\bU_2(\cO_K)}$ contains $\sigma(\tau')$ if and only if $\textnormal{rec}_{\overline{E}}(\BC(\pi))|_{I_{K_2}} \cong \tau'$ and $N = 0$ on $\textnormal{rec}_{\overline{E}}(\BC(\pi))$.  In this case, we have $\dim_{\overline{E}}\Hom_{\bU_2(\cO_K)}(\sigma(\tau'), \pi^\oplus|_{\bU_2(\cO_K)}) = 1$.
\end{theo}

\begin{proof}
This follows from Henniart's inertial local Langlands correspondence (\cite{breuilmezard}; see also \cite[3.7 Thm.]{CEGGPS}) and the properties of the stable base change map (\cite[\S 11.4]{rogawski}).

{To verify the claim about multiplicities, suppose that $\Hom_{\bU_2(\cO_K)}(\sigma(\tau'), \pi^\oplus|_{\bU_2(\cO_K)}) \neq 0$, so that in particular $\pi^{\oplus}$ has an irreducible summand of depth zero.  By the classification of depth zero $L$-packets (see \cite[\S 11.1]{rogawski} and \cite[Prop. 2.1(ii)]{blasco}, or \cite[\S~ 3.1]{adlerlansky}), the (semisimple) representation $\pi^\oplus|_{\bU_2(\cO_K)}$ is either a subrepresentation of $\Ind_{\bB_{\bU}(K)}^{\bU_2(K)}(\chi)|_{\bU_2(\cO_K)}$, or a direct sum $(\pi_1 \oplus \pi_2)|_{\bU_2(\cO_K)}$, where $\chi:\bB_{\bU}(K) \longrightarrow \overline{E}^\times$ is a smooth tame character, and where $\pi_1, \pi_2$ are irreducible supercuspidal representations of $\bU_2(K)$ which are conjugate under the action of $\bG\bU_2(K)$. }

{Suppose that $\pi^\oplus|_{\bU_2(\cO_K)}$ is a subrepresentation of $\Ind_{\bB_{\bU}(K)}^{\bU_2(K)}(\chi)|_{\bU_2(\cO_K)}$, and let $\bU_2(\cO_K)_1$ denote the principal congruence subgroup of $\bU_2(\cO_K)$.  Using the Mackey formula and the Iwasawa decomposition, we have
\begin{eqnarray*}
\Hom_{\bU_2(\cO_K)}(\sigma(\tau'), \pi^\oplus|_{\bU_2(\cO_K)}) & \subseteq & \Hom_{\bU_2(\cO_K)}\left(\sigma(\tau'), \Ind_{\bB_{\bU}(K)}^{\bU_2(K)}(\chi)|_{\bU_2(\cO_K)}\right) \\
 & = & \Hom_{\bU_2(\cO_K)}\left(\sigma(\tau'), \Ind_{\bB_{\bU}(K)}^{\bU_2(K)}(\chi)^{\bU_2(\cO_K)_1}\right) \\
 & \cong & \Hom_{\bU_2(\cO_K)}\left(\sigma(\tau'), \Ind_{\bB_{\bU}(\bbF_q)}^{\bU_2(\bbF_q)}(\chi|_{\bB_{\bU}(\bbF_q)})\right).
\end{eqnarray*}
Since $\sigma(\tau')$ is irreducible and $\Ind_{\bB_{\bU}(\bbF_q)}^{\bU_2(\bbF_q)}(\chi|_{\bB_{\bU}(\bbF_q)})$ is multiplicity-free (cf. \cite[\S~6]{ennola}), the result follows in this case.
}

{Suppose now that $\pi^{\oplus}|_{\bU_2(\cO_K)} =  (\pi_1 \oplus \pi_2)|_{\bU_2(\cO_K)}$.  We may label the supercuspidal representations such that $\pi_1^{\bU_2(\cO_K)_1} \neq 0$ and $\pi_2^{\bU_2(\cO_K)_1} = 0$.  This gives
\begin{eqnarray*}
\Hom_{\bU_2(\cO_K)}(\sigma(\tau'), \pi^\oplus|_{\bU_2(\cO_K)}) & = & \Hom_{\bU_2(\cO_K)}\left(\sigma(\tau'), (\pi_1 \oplus \pi_2)|_{\bU_2(\cO_K)}\right) \\
 & = & \Hom_{\bU_2(\cO_K)}\left(\sigma(\tau'), \pi_1^{\bU_2(\cO_K)_1}\right).
\end{eqnarray*}
As in \cite[\S~ 3.1]{adlerlansky}, we may write $\pi_1 \cong \textnormal{c-Ind}_{\bU_2(\cO_K)}^{\bU_2(K)}(\sigma)$, where $\sigma$ denotes an irreducible cuspidal representation of $\bU_2(\bbF_q)$, inflated to $\bU_2(\cO_K)$.  Applying the Mackey formula and the Cartan decomposition, and using cuspidality of $\sigma$, we obtain $\pi_1^{\bU_2(\cO_K)_1} \cong \sigma$.  Again using the irreducibility of $\sigma(\tau')$, we obtain the desired multiplicity result.}
\end{proof}

\section{Local deformations}
\label{sec:Loc:Def}

In this section we compute potentially crystalline deformation rings for certain $L$-parameters $\rhobar: \Gamma_K \longrightarrow {}^{C}\bU_2(\bbF)$.  The main result is Corollary \ref{cor:irr:cmpts} which relates Hilbert--Samuel multiplicities of such rings with the set $\textnormal{W}^?(\rhobar)$.  This will be used to prove the ``weight existence'' direction of Corollary \ref{cor:SWC}.

We follow \cite[\S 6]{LLLM1}, adapting the base change techniques to our setting (see also \cite{CDM}).  Subsection \ref{subsec:KisinModules} contains the background on Kisin modules for $\bG\bL_2$, together with their classification by shapes.  In Subsection \ref{sec:dual:KM} we introduce the notion of polarized (or {Frobenius twist self-dual}) Kisin modules and use a base change technique to compute their deformations.  We then relate the deformation problems of polarized Kisin modules and of $L$-parameters to obtain the desired description of the potentially crystalline deformations rings.

\subsection{Kisin modules}
\label{subsec:KisinModules}

Throughout this subsection, we let $R$ denote a complete local Noetherian $\cO$-algebra with residue field $\bbF$.  We start by defining the relevant categories of Kisin modules with tame descent data $Y^{\mu,\tau'}(R) \subseteq Y^{[0,1],\tau'}(R)$ (\cite[\S 5]{CL}, see also \cite[\S 3]{Dan-mult1})

\subsubsection{}

The ring $\fS_R \defeq (\cO_{K_2}\otimes_{\bbZ_p}R)[\![u]\!]$ is equipped with a Frobenius map $\overline{\varphi}:\fS_R \longrightarrow \fS_R$ which is the \emph{arithmetic} Frobenius on $\cO_{K_2}$ (i.e., $\overline{\varphi} = \varphi^{-1}$ on $\cO_{K_2}$), which is trivial on $R$, and which sends $u$ to $u^p$.

\begin{df}
A \emph{Kisin module} with height in [0,1] over $R$ is a finitely generated projective $\fS_R$-module $\fM$ together with an $\fS_R$-linear map $\phi_\fM: \overline{\varphi}^*\fM \defeq \fS_R \otimes_{\overline{\varphi},\fS_R}\fM \longrightarrow \fM$ such that
$$E(u)\fM \subseteq \phi_{\fM}\big(\overline{\varphi}^*\fM\big) \subseteq \fM,$$
where $E(u)$ denotes the Eisenstein polynomial of $(-p)^{1/(p^{2f} - 1)}$ over $K_2$, i.e. $E(u) = u^{p^{2f} - 1} + p$.  
\end{df}

We often write $\fM$ for a Kisin module, the Frobenius map $\phi_{\fM}$ being implicit.

\subsubsection{}

Recall that $\pi = (-p)^{1/(p^{2f} - 1)}\in \overline{\bbQ}_p$, and set $L \defeq K_2(\pi)$.  For $g\in \Gal(L/K_2)$, we have defined
$$\widetilde{\omega}_\pi(g) = \frac{\pi^g}{\pi}\in \cO_{K_2}^\times.$$
(Note that reducing $\widetilde{\omega}_\pi$ mod $p$ induces an isomorphism $\Gal(L/K_2) \stackrel{\sim}{\longrightarrow} \bbF_{p^{2f}}^\times$.)  Given $g\in \Gal(L/K_2)$, we let $\widehat{g}$ denote the $\cO_{K_2}\otimes_{\bbZ_p}R$-linear automorphism of $\fS_R$ given by $u \longmapsto (\widetilde{\omega}_\pi(g)\otimes 1)u$.  Note that $\overline{\varphi}\circ\widehat{g} = \widehat{g}\circ\overline{\varphi}$.

\begin{df}
Let $\fM$ denote a Kisin module over $R$.
\begin{enumerate}
\item A \emph{semilinear action of $\Gal(L/K_2)$ on $\fM$} is a collection $\{\widehat{g}\}_{g\in \Gal(L/K_2)}$ of $\widehat{g}$-semilinear additive bijections $\widehat{g}:\fM \longrightarrow \fM$ such that $\widehat{g}\circ\widehat{h} = \widehat{gh}$ for all $g,h\in \Gal(L/K_2)$.  
\item A \emph{Kisin module with descent datum over $R$} is a Kisin module together with a semilinear action of $\Gal(L/K_2)$ given by $\{\widehat{g}\}_{g\in \Gal(L/K_2)}$ which commutes with $\phi_{\fM}$, i.e., we have
$$\widehat{g}\circ \phi_{\fM} = \phi_{\fM} \circ \overline{\varphi}^*\widehat{g}$$
for all $g\in \Gal(L/K_2)$.  
\end{enumerate}
\end{df}

\subsubsection{}
\label{kisindecomp}

Any Kisin module $\fM$ admits a decomposition 
$$\fM = \bigoplus_{i = 0}^{2f - 1}\fM^{(i)},$$ 
where $\fM^{(i)}$ is the $R[\![u]\!]$-submodule of $\fM$ such that $(x\otimes 1)m = (1\otimes \varsigma_0\circ \varphi^{i}(x))m$ for $m\in \fM^{(i)}$ and $x\in \cO_{K_2}$.

We let 
$$\tau':I_{K_2} \longtwoheadrightarrow \Gal(L/K_2) \longrightarrow \bG\bL_2(\cO)$$ 
denote a tamely ramified inertial type which factors through $\Gal(L/K_2)$.  Recall that this implies $\tau'$ can be written $\tau' = \widetilde{\omega}_{2f}^a \oplus \widetilde{\omega}_{2f}^b$.

\begin{df}
\label{df:KM:tau}
Suppose $\fM$ is a Kisin module with descent datum over $R$.  We say the descent datum is \emph{of type $\tau'$} if we have $\fM^{(i)}/u\fM^{(i)} \cong {\tau'^\vee \otimes_{\cO} R}$ as representations of $\Gal(L/K_2)$ for every $0 \leq i \leq 2f - 1$, {where $\tau'^\vee$ denotes the dual type}.\footnote{We impose the condition $\fM^{(i)}/u\fM^{(i)} \cong \tau'^\vee \otimes_\cO R$ because our functors to Galois representations in later sections are \emph{contravariant}.  In particular, if $\fM$ is a Kisin module over $\cO$ with height in $[0,1]$ and descent datum of type $\tau'$ (as defined in Definition \ref{df:KM:tau}), then the $\Gamma_{K_2}$-representation $T_{\textnormal{dd}}^*(\fM)[1/p]$ will have inertial type $\tau'$.  (See below for undefined notation and terminology.)}
\end{df}


\subsubsection{}

We now define the categories of Kisin modules that will be relevant for us.  
Let $\mu \defeq \vect{\un{1}}{\un{0}}$ denote the standard minuscule cocharacter of $\bG' = \Res_{\cO_{K_2}/\zp}\bG\bL_{2/\cO_{K_2}}$.

\begin{df}
Fix a principal series tame type $\tau'$.
\begin{enumerate}
\item We define $Y^{[0,1],\tau'}(R)$ to be the {groupoid} of Kisin modules over $R$ of rank $2$, with height in [0,1], and descent datum of type $\tau'$.  
\item We define $Y^{\mu,\tau'}(R)$ to be the {(full) subgroupoid} of $Y^{[0,1],\tau'}(R)$ consisting of Kisin modules such that 
\begin{equation}
\label{eq:df:Kisin}
{E(u)\det\fM = \phi_{\fM}\big(\overline{\varphi}^*(\det \fM)\big).}
\end{equation}
\end{enumerate}
\end{df}
{Note that the definition of $Y^{\mu,\tau'}(R)$ above is consistent with the construction of \cite[\S 5]{CL}, thanks to Theorem 5.13 and Corollary 5.12 of \emph{op.~cit.}.  See also \cite[Thm. 4.18]{LLLM1}.}

\subsubsection{}
\label{subsub:gen:cond:typs}

We fix some notation, following \cite[\S 2.1]{LLLM1}.  Let $\tau'$ be a principal series tame type of $I_{K_2}$.  We may write
$${\tau'^\vee} = \eta_1 \oplus \eta_2 = \widetilde{\omega}_{2f}^{-\sum_{i = 0}^{2f - 1}a_{1,i}p^i} \oplus \widetilde{\omega}_{2f}^{-\sum_{i = 0}^{2f - 1}a_{2,i}p^i},$$
with $0 \leq a_{k,i} \leq p - 1$ for all $i$.  {By Remark \ref{twistremark} below, we may assume without loss of generality} that neither $\eta_1$ nor $\eta_2$ are trivial, i.e.~$(a_{k,i})_i\notin\{(p-1,\dots,p-1), (0\dots,0)\}$ for $k=1,2$.

{\begin{rmk}\label{twistremark}
The goal of Section \ref{sec:Loc:Def} is to compute the deformation rings $R_{\rhobar}^{\tau'}$ (described in Subsubsection \ref{localgalrepdefrings} below), where $\rhobar: \Gamma_K \longrightarrow {}^C\bU_2(\bbF)$ is a tamely ramified $L$-parameter.  Given an integer $0 \leq k < p^f + 1$, we define $\rhobar \otimes \omega_{2f}^{(1 - p^f)k}$, the twist of $\rhobar$ by $\omega_{2f}^{(1 - p^f)k}$, by the rules
\begin{eqnarray*}
(\rhobar \otimes \omega_{2f}^{(1 - p^f)k})(h) & = & \rhobar(h) \cdot \left(\begin{pmatrix} \omega_{2f}^{(1 - p^f)k}(h) & 0 \\ 0 & \omega_{2f}^{(1 - p^f)k}(h) \end{pmatrix}, 1\right) \rtimes 1 \\
(\rhobar \otimes \omega_{2f}^{(1 - p^f)k})(\varphi^{-f}) & = & \rhobar(\varphi^{-f})
\end{eqnarray*}
where $h \in \Gamma_{K_2}$.  One checks that these rules give a well-defined tamely ramified $L$-parameter which satisfies $\widehat{\imath}\circ (\rhobar \otimes \omega_{2f}^{(1 - p^f)k}) = \widehat{\imath}\circ \rhobar$.  Using this twisting procedure, the proof of \cite[Lem.~2.1.2]{gee-kisin} shows that we have an isomorphism of deformation rings
$$R_{\rhobar}^{\tau'} \cong R_{\rhobar \otimes \omega_{2f}^{(1 - p^f)k}}^{\tau'\otimes \widetilde{\omega}_{2f}^{(1 - p^f)k}}.$$
Consequently, we may assume that $\tau'$ does not contain the trivial character.  
\end{rmk}}

Set $\mathbf{a}_1 \defeq (a_{1,i})_i,~\mathbf{a}_2 \defeq (a_{2,i})_i$, and given $0\leq j \leq 2f - 1$, define the shifted sums
$$\mathbf{a}_1^{(j)}  \defeq  \sum_{i = 0}^{2f - 1} a_{1,i - j}p^i,\qquad \mathbf{a}_2^{(j)}  \defeq  \sum_{i = 0}^{2f - 1} a_{2,i - j}p^i,$$
so that, in particular, $\eta_1 = \widetilde{\omega}_{2f}^{-\mathbf{a}_1^{(0)}}, \eta_2 =  \widetilde{\omega}_{2f}^{-\mathbf{a}_2^{(0)}}$.  

\begin{df}
Let $n \geq 0$.  We say the pair $(\mathbf{a}_1,\mathbf{a}_2)$ is \emph{$n$-generic } if 
$$n < |a_{1,i} - a_{2,i}| < p - n$$
for every $0\leq i \leq 2f - 1$.  If $\tau'$ is associated to $(\mathbf{a}_1,\mathbf{a}_2)$ as above, we say \emph{$\tau'$ is $n$-generic} if the pair $(\mathbf{a}_1,\mathbf{a}_2)$ is.
\end{df}

This agrees with the notion of genericity given in Definition \ref{df:n:generic}.

\subsubsection{}

\begin{df}
Let $\tau'$ denote a principal series tame type of $I_{K_2}$, and let $(\mathbf{a}_1,\mathbf{a}_2)$ denote the associated pair.  Suppose $\tau'$ is $2$-generic.  An \emph{orientation} of $\tau'$ is an element $w = (w_i)_i \in {S_2^{2f}}$ such that
$$\mathbf{a}^{(i)}_{w_i(1)} \geq \mathbf{a}^{(i)}_{w_i(2)}$$
for all $0\leq i \leq 2f - 1$.  
\end{df}

{(We view $S_2$ as a subgroup of $\bG\bL_2(\bbZ)$ via the standard embedding as permutation matrices.  Since $S_2^{2f} \cong W'$, we also view orientations as elements of $W'$ when convenient.)}  We note that an orientation depends on the ordering of the characters $\eta_1, \eta_2$.  Since we take $\tau'$ to be $2$-generic, the orientation is unique, and $w_i$ depends only on the pair $(a_{1, 2f-1-i}, a_{2, 2f-1-i})$.

\subsubsection{}
In what follows, we use the notation $v\defeq u^{p^{2f} - 1}$.

\begin{df}
\label{def:eigenbasis}
Let $\tau'$ denote a $2$-generic principal series tame type, {and write $\tau'^\vee = \eta_1 \oplus \eta_2$}.  Let $\fM\in Y^{[0,1],\tau'}(R)$, and let $\fM = \bigoplus_{i = 0}^{2f - 1} \fM^{(i)}$ be the decomposition of $\fM$ as in Subsection \ref{kisindecomp}.  
\begin{enumerate}
\item We let $\fM^{(i)}_1$ (resp.~$\fM^{(i)}_2$) denote the $R[\![v]\!]$-submodule of $\fM^{(i)}$ on which $\Gal(L/K_2)$ acts by $\eta_1$ (resp.~$\eta_2$).  
\item We define ${}^{\overline{\varphi}}\fM^{(i)}_1$ (resp.~${}^{\overline{\varphi}}\fM^{(i)}_2$) to be the $R[\![v]\!]$-submodule of $\overline{\varphi}^*(\fM^{(i)}) = (\overline{\varphi}^*\fM)^{(i + 1)}$ on which $\Gal(L/K_2)$ acts by $\eta_1$ (resp.~$\eta_2$).  
\item We define an \emph{eigenbasis} $\beta \defeq \{\beta^{(i)}\}_{i}$ of $\fM$ to be a collection of ordered bases $\beta^{(i)} = (f_1^{(i)}, f_2^{(i)})$ of each $\fM^{(i)}$ such that $f_1^{(i)}\in \fM_1^{(i)}$ and $f_2^{(i)}\in \fM_2^{(i)}$.  
\end{enumerate}
\end{df}

Now let $\tau'$ be a $2$-generic principal series tame type, with orientation $w = (w_i)_i$.  We have a commutative diagram
\begin{center}
\begin{tikzcd}[row sep=large, column sep=20ex]
{}^{\overline{\varphi}}\fM^{(i - 1)}_{w_i(2)} \ar[r, "\scriptstyle{u^{p^{2f} - 1 - (\mathbf{a}_{w_i(1)}^{(i)} - \mathbf{a}_{w_i(2)}^{(i)})}}"] \ar[d, "\phi^{(i - 1)}_{\fM,w_i(2)}"] & {}^{\overline{\varphi}}\fM^{(i - 1)}_{w_i(1)} \ar[r, "u^{\mathbf{a}_{w_i(1)}^{(i)} - \mathbf{a}_{w_i(2)}^{(i)}}"] \ar[d, "\phi^{(i - 1)}_{\fM,w_i(1)}"] & {}^{\overline{\varphi}}\fM^{(i - 1)}_{w_i(2)} \ar[d, "\phi^{(i - 1)}_{\fM,w_i(2)}"] \\
\fM^{(i)}_{w_i(2)} \ar[r, "u^{p^{2f} - 1 - (\mathbf{a}_{w_i(1)}^{(i)} - \mathbf{a}_{w_i(2)}^{(i)})}"] & \fM^{(i)}_{w_i(1)} \ar[r, "u^{\mathbf{a}_{w_i(1)}^{(i)} - \mathbf{a}_{w_i(2)}^{(i)}}"] & \fM^{(i)}_{w_i(2)}
\end{tikzcd}
\end{center}
\noindent Here, $\phi_{\fM,k}^{(i - 1)}$ denotes the restriction of $\phi_\fM$ to ${}^{\overline{\varphi}}\fM^{(i - 1)}_{k}$.

\subsubsection{}
\label{Amatrix}

Fix a principal series $2$-generic tame type $\tau'$ and $\fM \in Y^{[0,1],\tau'}(R)$.  Let $w = (w_i)_i$ denote the orientation of $\tau'$, and let $\beta = \{\beta^{(i)}\}_i$ denote an eigenbasis for $\fM$.  We define
\begin{eqnarray*}
\beta^{(i)}_{w_i(2)} & \defeq & \left(u^{\mathbf{a}_{w_i(1)}^{(i)} - \mathbf{a}_{w_i(2)}^{(i)}}f^{(i)}_{w_i(1)},~f^{(i)}_{w_i(2)}\right),\\
{}^{\overline{\varphi}}\beta^{(i - 1)}_{w_i(2)} & \defeq & \left(u^{\mathbf{a}_{w_i(1)}^{(i)} - \mathbf{a}_{w_i(2)}^{(i)}}\otimes f^{(i - 1)}_{w_i(1)},~1\otimes f^{(i - 1)}_{w_i(2)}\right);
\end{eqnarray*}
the first is an $R[\![v]\!]$-basis for $\fM^{(i)}_{w_i(2)}$, the second is an $R[\![v]\!]$-basis for ${}^{\overline{\varphi}}\fM^{(i - 1)}_{w_i(2)}$.  We then define the matrix $A^{(i)}\in \textnormal{Mat}_{2\times 2}(R[\![v]\!])$ by the condition
\begin{equation}\label{partialfrob}
\phi_{\fM,w_{i + 1}(2)}^{(i)}\left({}^{\overline{\varphi}}\beta^{(i)}_{w_{i + 1}(2)}\right) = \beta^{(i + 1)}_{w_{i + 1}(2)}A^{(i)}.
\end{equation}
We say that $A^{(i)}$ is the \emph{matrix of the partial Frobenius} of $\fM$ (at embedding $i$, with respect to $\beta$).

\subsubsection{}
We now find a more convenient expression for the data of the matrices $(A^{(i)})_i$.

We define the extended affine Weyl group of $\bG\bL_2$ as 
$$\widetilde{\cW} \defeq N_{\bG\bL_2}(\widehat{\bT}_\bG)(\bbF(\!(v)\!))/\widehat{\bT}_{\bG}(\bbF[\![v]\!]),$$
where $\widehat{\bT}_{\bG}$ denotes the torus dual to $\bT_{\bG/\cO_{K_2}}$.  We have an exact sequence
$$0 \longrightarrow X_*(\widehat{\bT}_\bG) \longrightarrow \widetilde{\cW} \longrightarrow S_2 \longrightarrow 0,$$
where the first nontrivial map sends a cocharacter to its value on $v$.  Furthermore, we have a Bruhat decomposition
$$\bG\bL_2(\bbF(\!(v)\!)) = \bigsqcup_{\tw\in \widetilde{\cW}}\cI \tw\cI,$$
where $\cI$ denotes the standard Iwahori subgroup of $\bG\bL_2(\bbF[\![v]\!])$, that is, the set of matrices which are upper triangular mod $v$.

Using the canonical identification $X_*(\widehat{\bT}_{\bG}) \cong X^*(\bT_{\bG/\cO_{K_2}})$, we identify $\widetilde{\cW}^{2f}$ with the extended affine Weyl group $\tW'$ of $\bG'$.

\begin{df}
Let $\tw = (\tw_i)_i\in \tW'$, let $\tau'$ be a principal series $2$-generic tame type, and let $w = (w_i)_i\in W'$ denote the orientation of $\tau'$.  Let $\overline{\fM} \in Y^{[0,1],\tau'}(\bbF)$.  
\begin{enumerate}
\item We say $\overline{\fM}$ has \emph{shape $\tw$} if for some eigenbasis $\overline{\beta}$, the matrices $({A}^{(i)})_i$ (defined by \eqref{partialfrob}, with respect to $\overline{\beta}$) have the property that ${A}^{(i)} \in \cI \tw_i \cI$.  
\item As in the discussion following \cite[Def. 2.17]{LLLM1}, the notion of shape does not depend on the choice of eigenbasis.  We define $Y^{\mu,\tau'}_{\tw}(\bbF)$ to be the full subcategory of $Y^{\mu,\tau'}(\bbF)$ consisting of Kisin modules of shape $\tw$.  
\end{enumerate}
\end{df}

\subsubsection{}

Upon choosing the dominant chamber corresponding to $\cI$ in $X_*(\widehat{\bT}_{\bG})\otimes_{\bbZ}\bbR$, we obtain a Bruhat order $\leq$ on $\widetilde{\cW}$.  Given a cocharacter $\lambda\in X_*(\widehat{\bT}_{\bG})$, we define the $\lambda$-admissible set as
$$\textnormal{Adm}(\lambda) \defeq \left\{\tw \in \widetilde{\cW}: \tw \leq t_{w(\lambda)}~\textnormal{for some}~w\in S_2\right\}.$$
In particular, we have
$$\textnormal{Adm}(\vect{1}{0}) = \left\{\begin{pmatrix}v & 0 \\ 0 & 1\end{pmatrix},\quad\begin{pmatrix} 1 & 0 \\ 0 & v \end{pmatrix},\quad \begin{pmatrix}0 & 1 \\ v & 0\end{pmatrix}\right\}.$$
We denote these elements by $\ft, \ft'$ and $\fw$, respectively.  Given $\mu = \vect{\un{1}}{\un{0}}$, we define
$$\textnormal{Adm}(\mu) \defeq \prod_{i = 0}^{2f - 1}\textnormal{Adm}(\vect{1}{0}),$$
which we call the $\mu$-admissible set.  As in \cite[Cor. 2.19]{LLLM1}, we have that $Y^{\mu,\tau'}_{\tw}(\bbF)$ is nonempty if and only if $\tw\in \textnormal{Adm}(\mu)$.

We now have the analog of \cite[Thm. 2.21]{LLLM1}, using \emph{op. cit.}, Lemma 2.20.

\begin{lemdef}
Suppose $\tw = (\tw_i)_i\in \tW'$ is $\mu$-admissible and $\tau'$ is a $2$-generic principal series tame type.  Let $\overline{\fM} \in Y^{\mu,\tau'}_{\tw}(\bbF)$.  Then there is an eigenbasis $\overline{\beta}$ for $\overline{\fM}$ such that the matrix of partial Frobenius ${A}^{(i)}$ has the form given in Table \ref{Table1}.  We call such an eigenbasis a \emph{gauge basis}.  
\end{lemdef}

\begin{table}[h]
\centering
\caption{\textbf{Shapes of Kisin modules over $\bbF$} }
\vspace{-10pt}
\caption*{ \footnotesize{Here we have $\overline{c}_{j,k}\in \bbF$ and $\overline{c}_{j,k}^*\in \bbF^\times$.}}
\label{Table1}
\centering
\begin{tabular}{| c || c | c | c | }
\hline
$\tld{w}_i$ & $\ft$ & $\ft'$ & $\fw$ \\ 
\hline
${A}^{(i)}$ & $\begin{pmatrix} v \ovl{c}_{1,1}^* & 0\\ v \ovl{c}_{2,1} & \ovl{c}_{2,2}^*\end{pmatrix}$ & $\begin{pmatrix} \ovl{c}_{1,1}^* & \ovl{c}_{1,2}\\ 0 & v\ovl{c}_{2,2}^*\end{pmatrix}$ & $\begin{pmatrix} 0& \ovl{c}_{1,2}^*\\ v \ovl{c}^*_{2,1} & 0\end{pmatrix}$\\
\hline
\end{tabular}
\end{table}

\subsubsection{}

Now fix $\overline{\fM}\in Y^{\mu,\tau'}_{\tw}(\bbF)$, and fix a gauge basis $\overline{\beta}$ for $\overline{\fM}$.  {We denote by $Y^{\mu,\tau'}_{\overline{\fM}}(R)$ the category of pairs $(\fM, \jmath)$, where $\fM \in Y^{\mu, \tau'}(R)$ and $\jmath$ is an isomorphism $\jmath:\fM\otimes_R \bbF \stackrel{\sim}{\longrightarrow} \overline{\fM}$.   }

\begin{df}
Let ${(\fM,\jmath)} \in Y^{\mu,\tau'}_{\overline{\fM}}(R)$.  A \emph{gauge basis of ${(\fM,\jmath)}$} is an eigenbasis $\beta$ lifting $\overline{\beta}$ {via $\jmath$} such that the matrix of partial Frobenius ${A}^{(i)}$ satisfies the degree conditions given in Table \ref{Table2}.
\end{df}

Note that a gauge basis for ${(\fM,\jmath)} \in Y_{\overline{\fM}}^{\mu, \tau'}(R)$ exists by the analog of \cite[Thm 4.1]{LLLM1}, and the set of gauge bases for ${(\fM,\jmath)}$ is in bijection with the set of eigenbases of $\fM/u\fM$ lifting $\overline{\beta}\mod u$ by the analog of \cite[Thm. 4.16]{LLLM1}. ({See also the cases $\ovl{A}_1,~\ovl{A}_2$ of \cite[Thm. 3.3]{Dan-mult1}, where a detailed proof of the cases $\ft$ and $\fw$ above is given}.)

\begin{table}[h]
\centering
\caption{\textbf{Deforming Kisin modules by shape}}
\vspace{-10pt}
\caption*{\footnotesize{Here, $\textnormal{deg}(A^{(i)})$ denotes the degree of the polynomial in each entry.  We write $n^*$ to denote a polynomial entry of degree $n$ whose leading coefficient is a unit. We have $c_{j,k}\in R$ and $c_{j,k}^*\in R^\times$.  Row $3$ is deduced from row $2$ by imposing condition \eqref{eq:df:Kisin}}}
\label{Table2}
\centering
\begin{tabular}{| c || c | c | c | }
\hline
$\tw_i$ & $\ft$ & $\ft'$ & $\fw$\\
\hline
$\textnormal{deg}(A^{(i)})$ & $\begin{pmatrix} 1^* & -\infty\\ v (\leq0) & 0^*\end{pmatrix}$ & $\begin{pmatrix} 0^* & \leq 0\\  -\infty& 1^*\end{pmatrix}$ & $\begin{pmatrix} \leq 0& 0^*\\ v(0^*) & \leq0\end{pmatrix}$\\
\hline
$A^{(i)}$& $\begin{pmatrix} (v+p)c_{1,1}^* & 0\\ v c_{2,1} & c_{2,2}^*\end{pmatrix}$ & $\begin{pmatrix} c_{1,1}^* & c_{1,2}\\  0 & c_{2,2}^*(v+p)\end{pmatrix}$ & $\begin{matrix} \begin{pmatrix}  c_{1,1}& c_{1,2}^*\\ v c_{2,1}^* & c_{2,2}\end{pmatrix}\\
c_{1,1}c_{2,2} = -pc_{1,2}^*c_{2,1}^*
\end{matrix}$\\ 
\hline
\end{tabular}
\end{table}

\subsection{Duality}
\label{sec:dual:KM}
We introduce the notion of {Frobenius twist self-dual} Kisin modules over $K$ and study their relation with usual Kisin modules over $K_2$ via the theory of base change (as in \cite[\S 6]{LLLM1}).  The main result of this section (Lemma \ref{lem:BC:shape}) describes the matrix of partial Frobenius on {Frobenius twist self-dual} Kisin modules.

\subsubsection{}

We now collect the relevant properties of Cartier duality which we will need.

\begin{defn}{(cf.~ \cite[\S~3.4.1]{broshi})}
Suppose $\tau'$ is a tame principal series type, $R$ is a local Artinian $\cO$-algebra with residue field $\bbF$, and let $\fM \in Y^{\mu,\tau'}(R)$.  We define the \emph{Cartier dual of $\fM$} to be 
$$\fM^\vee \defeq \textnormal{Hom}_{\fS_R}(\fM, \fS_R),$$ 
which we equip with a Frobenius map by
$$1 \otimes f \longmapsto \phi_{\fS_R} \circ (1 \otimes f) \circ \phi_{\fM}^{-1} \circ E(u),$$
where $1 \otimes f\in \overline{\varphi}^* \textnormal{Hom}_{\fS_R}(\fM, \fS_R) \cong \textnormal{Hom}_{\fS_R}(\overline{\varphi}^*\fM, \overline{\varphi}^*\fS_R)$.  (Note that the map $\phi_\fM$ is injective by \cite[Lem. 1.2.2(1)]{kisin-annals}.)  {We also equip $\fM^\vee$ with a descent datum, given by 
\begin{eqnarray*}
\widehat{g}: \fM^\vee & \longrightarrow &  \fM^\vee \\
f & \longmapsto & \widehat{g} \circ f \circ \widehat{g^{-1}}
\end{eqnarray*}
(the right-hand $\widehat{g^{-1}}$ denotes the semilinear action of $\Gal(L/K_2)$ on $\fM$, while the left-hand $\widehat{g}$ denotes the semilinear action on $\fS_R$).  With this definition, one easily checks that the descent datum of $\fM^\vee$ is of type $\tau'^\vee$, where $\tau'^\vee$ is the type dual to $\tau'$, so that $\fM^\vee \in Y^{\mu,\tau'^\vee}(R)$.  }
\end{defn}

Before proceeding with the proof of the proposition below, we introduce some notation.  We define $\{p_n\}_{n\geq 0}$ to be a sequence of elements of $\overline{\bbQ}_p$ which satisfy $p_{n + 1}^p = p_{n}$ and $p_0 = -p$, and define $K_\infty \defeq \bigcup_{n\geq 0} K(p_n)$ and $K_{2,\infty} \defeq \bigcup_{n\geq 0} K_2(p_n)$.  Note that $\Gal(K_{2,\infty}/K_\infty) \cong \Gal(K_2/K)$.

\begin{prop}
\label{cartier}
Suppose $R$ is a local Artinian $\cO$-algebra with residue field $\bbF$.  Let $\tau':I_{K_2} \longrightarrow \bG\bL_2(\cO)$ be a principal series tame type.  Then $\fM \longmapsto \fM^\vee$ 
defines an involutive functor $Y^{\mu,\tau'}(R) \longrightarrow Y^{\mu,\tau'^\vee}(R)$, which enjoys the following properties:
\begin{itemize}
\item We have $T^*_{\textnormal{dd}}(\fM^\vee) \cong T^*_{\textnormal{dd}}(\fM)^\vee \otimes \varepsilon$ as $\Gamma_{K_{2,\infty}}$-representations, where the functor $T^*_{\textnormal{dd}}$ is as defined in \cite[\S 2.3]{LLLM1}.
\item Let $\beta = \{\beta^{(i)}\}_i$ be an eigenbasis of $\fM$ as in Definition \ref{def:eigenbasis}.  Let $C^{(i)}\defeq \Mat_{\beta}(\phi_{\fM}^{(i)}) \in \Mat_{2\times 2}(R[\![u]\!])$ denote the matrix of the Frobenius on $\overline{\varphi}^*(\fM^{(i)})$, defined by
$$\phi_{\fM}^{(i)}\left(1\otimes f^{(i)}_1,~1\otimes f^{(i)}_2\right) = \left(f^{(i+1)}_1,~f^{(i+1)}_2\right)C^{(i)}.$$
Then the matrix of Frobenius on $\fM^{\vee,(i)}$ with respect to the dual basis $\beta^\vee$ is given by
\begin{equation}
\label{dualfrobmatrix}
\Mat_{\beta^\vee}(\phi_{\fM^\vee}^{(i)}) = E(u)(C^{(i)})^{-\top}.
\end{equation}
\end{itemize}
\end{prop}

\begin{proof}
The first point follows from \cite[Prop. 3.4.1.7]{broshi}, while the second point follows from an explicit calculation.  
\end{proof}

\subsubsection{}
\label{dualshape}

We explain how orientations and shapes change under duality.  Suppose $\tau'$ is a $2$-generic principal series tame type, {and write $\tau'^\vee = \eta_1 \oplus \eta_2$}.  {We then have the associated} pair $(\mathbf{a}_1,\mathbf{a}_2)$ and orientation $w = (w_i)_i\in {S_2^{2f}}$.  We fix an ordering on the characters of {$\tau'$} so that {$\tau' = \eta_1^{-1} \oplus \eta_2^{-1}$} is associated to the pair $(\un{p} - \un{1} - \mathbf{a}_1, \un{p} - \un{1} - \mathbf{a}_2)$ and orientation $(\un{s},\un{s})w$ {(recall that we view elements of $S_2^{2f} \cong W'$ as pairs of elements of $S_2^f \cong W$ as in \S \ref{subsubsec:groupG'})}.  Note that $\tau'$ is $n$-generic if and only if $\tau'^\vee$ is $n$-generic.

Assume that $R$ is a local Artinian $\cO$-algebra with residue field $\bbF$, $\fM \in Y^{\mu,\tau'}(R)$, and let $\beta$ denote an eigenbasis of $\fM$.  The matrix $C^{(i)}$ of Frobenius on $\overline{\varphi}^*(\fM^{(i)})$ (as in the above proposition) and the matrix $A^{(i)}$ of the partial Frobenius (as in Subsubsection \ref{Amatrix}) are related by the equation
$$C^{(i)} = w_{i + 1}\begin{pmatrix}u^{\mathbf{a}^{(i + 1)}_{w_{i + 1}(1)}} & 0 \\ 0 & u^{\mathbf{a}^{(i + 1)}_{w_{i + 1}(2)}} \end{pmatrix}A^{(i)}\begin{pmatrix}u^{-\mathbf{a}^{(i + 1)}_{w_{i + 1}(1)}} & 0 \\ 0 & u^{-\mathbf{a}^{(i + 1)}_{w_{i + 1}(2)}} \end{pmatrix}w_{i + 1}^{-1}$$
(for the proof, see \cite[Prop 2.13]{LLLM1}).  Using this relation for the dual Kisin module $\fM^\vee$ and dual type $\tau'^\vee$ (ordered as in the previous paragraph), along with Proposition \ref{cartier}, we conclude that the matrix of partial Frobenius on $\fM^\vee$, with respect to $\beta^\vee$ at embedding $i$, is equal to 
$$E(u)s(A^{(i)})^{-\top}s = \begin{pmatrix} v + p & 0 \\ 0 & v + p \end{pmatrix}s(A^{(i)})^{-\top}s.$$

Now suppose $\overline{\fM}\in Y^{\mu,\tau'}(\bbF)$.  The above relation shows that $\overline{\fM}$ has shape $\tw_i$ at embedding $i$ if and only if $\overline{\fM}^\vee$ has shape $\sm{v}{0}{0}{v}s\tw_i^{-\top}s$ at embedding $i$.  In particular this involution on $\tW'$ fixes $\textnormal{Adm}(\mu)$ pointwise, and thus Cartier duality induces an involutive functor
$$Y^{\mu,\tau'}_{\tw}(\bbF) \longrightarrow Y^{\mu,\tau'^\vee}_{\tw}(\bbF).$$
Furthermore, equation \eqref{dualfrobmatrix} shows that $\overline{\beta}$ is a gauge basis for $\overline{\fM}$ if and only if $\overline{\beta}^\vee$ is a gauge basis for $\overline{\fM}^\vee$.  Similarly, if ${(\fM,\jmath)}\in Y^{\mu,\tau'}_{\overline{\fM}}(R)$, then $\beta$ is a gauge basis for ${(\fM,\jmath)}$ if and only if $\beta^\vee$ is a gauge basis for ${(\fM^\vee,(\jmath^\vee)^{-1})}$.

\subsubsection{}

In the following, we use the notation $\sigma$ to denote the \emph{automorphism} of $\fS_R$ which is the arithmetic Frobenius on $\cO_{K_2}$ and which acts trivially on $R$ and the variable $u$.  Thus, given a Kisin module $\fM$, we may form the pullback $\sigma^*\fM \defeq \fS_R \otimes_{\sigma, \fS_R} \fM$ along $\sigma$, equipped with Frobenius $\phi_{\sigma^*\fM} \defeq \sigma^*\phi_{\fM}$.  {One easily checks that $(\sigma^*\fM)^{(i)} = \sigma^*(\fM^{(i - 1)})$.}  If $\fM$ comes equipped with a descent datum, $\sigma^*\fM$ obtains a descent datum via the canonical identification
$$\widehat{g^p}^*\big(\sigma^*\fM\big) \stackrel{\sim}{\longrightarrow} \sigma^*\big(\widehat{g}^*\fM\big)$$
(here $\widehat{g}^*\fM$ denotes the pullback of $\fM$ along the automorphism $\widehat{g}$ of $\fS_R$, and similarly for $\widehat{g^p}$).  
We note that if {$\fM$ has type $\tau'$ and $\tau'^\vee = \eta_1 \oplus \eta_2$}, then $\sigma^*\fM$ has type {$(\tau')^{{\varphi}}$, where $(\tau')^{\varphi, \vee} = (\tau')^{\vee, \varphi} = \eta_1^{p^{-1}} \oplus \eta_2^{p^{-1}}$}.  Thus, Frobenius twisting gives a functor
$$\sigma^*: Y^{\mu,\tau'}(R) \longrightarrow Y^{\mu, (\tau')^{{\varphi}}}(R).$$
We make similar definitions for iterates of $\sigma$.

We briefly describe how the Frobenius twist transforms certain objects associated to Kisin modules.  Twisting changes the principal series tame type $\tau'$ into $ (\tau')^{\varphi}$.  Thus, it also transforms the associated pair 
$$(\mathbf{a}_1, \mathbf{a}_2) = \big((a_{1,0}, a_{1,1}, \ldots, a_{1,2f - 1}),~ (a_{2,0}, a_{2,1}, \ldots, a_{2,2f - 1})\big)$$ 
into 
$$\big((a_{1,1}, a_{1,2}, \ldots, a_{1,2f - 1}, a_{1,0}),~ (a_{2,1}, a_{2,2}, \ldots, a_{2,2f - 1}, a_{2,0})\big),$$ 
and transforms the orientation $w = (w_0, w_1, \ldots, w_{2f - 1})$ into $(w_{2f - 1}, w_0, \ldots, w_{2f - 2})$.  Further, given an eigenbasis $\beta = \{(f^{(i)}_1, f^{(i)}_2)\}_i$ for $\fM$, the elements $\sigma^*\beta \defeq \{(1 \otimes f^{(i)}_1, 1\otimes f_2^{(i)})\}_i$ form an eigenbasis of $\sigma^*\fM$.  Therefore, by their definition, the Frobenius twist transforms the matrices $(A^{(0)}, A^{(1)}, \ldots, A^{(2f - 1)})$ of partial Frobenius (with respect to $\beta$) into $(A^{(2f - 1)}, A^{(0)}, \ldots, A^{(2f - 2)})$, and if $\fM \in Y^{\mu, \tau'}(\bbF)$ has shape $\widetilde{w} = (\widetilde{w}_0, \widetilde{w}_1, \ldots, \widetilde{w}_{2f - 1})$, then $\sigma^*\fM \in Y^{\mu,(\tau')^{{\varphi}}}(\bbF)$ will have shape $(\widetilde{w}_{2f - 1}, \widetilde{w}_0, \ldots, \widetilde{w}_{2f - 2})$.  Finally, we obtain an isomorphism on the associated $\Gamma_{K_{2,\infty}}$-representation
$$T^*_{\textnormal{dd}}(\sigma^*\fM) \cong T^*_{\textnormal{dd}}(\fM)^{\varphi},$$
where we recall that the superscript $\varphi$ denotes the twist of the representation by $\varphi$.

%
%
%
%

\subsubsection{}

Suppose now that $\tau'$ is a $2$-generic principal series tame type which satisfies $(\tau')^{\varphi^{-f}}  \cong \tau'^\vee$ (and note that $(\tau')^{\varphi^{-f}} = (\tau')^{\varphi^f}$).  As in Subsection \ref{ILLC-sect}, this implies that $\tau'$ is of the form
$$\tau' = \widetilde{\omega}_{2f}^{-c} \oplus \widetilde{\omega}_{2f}^{p^fc} \quad \textnormal{or} \quad \tau' = \widetilde{\omega}_{2f}^{(-1 + p^f)a} \oplus \widetilde{\omega}_{2f}^{(-1 + p^f)b}.$$
If $\tau'$ is $2$-generic, the orientation on $\tau'$ has the form $(z,z)$ for $z\in W$ in the first case, while in the second case the orientation has the form $(z,z\un{s})$ for $z\in W$.

The discussion above gives the following:
\begin{lem}
\label{lem:prel:BC:shape}
Assume $\tau'$ is a $2$-generic principal series tame type, {write $\tau'^\vee = \eta_1 \oplus \eta_2$}, and let $w = (w_i)_i\in W'$ denote its orientation. Suppose that $(\tau')^{\varphi^{-f}} \cong \tau'^\vee$.  Then 
$$\eta_{w_{i + f}(2)}^{p^f} = \eta^{-1}_{w_is(2)}$$
for every $0\leq i \leq 2f - 1$.  
\end{lem}

\begin{proof}
This may be proved casewise, using the possible orientations on $\tau'$.  
\end{proof}

\subsubsection{}

As in \cite[\S 6.1]{LLLM1}, we define Kisin modules which are \emph{Frobenius-twist self-dual}:
\begin{df} 
\label{df:fixed}
Let $R$ denote a local Artinian $\cO$-algebra with residue field $\bbF$, and let $\tau'$ denote a principal series tame type which satisfies $(\tau')^{\varphi^{-f}} \cong \tau'^\vee$.  We define
$$Y^{\mu, \tau'}_{\pol}(R) \defeq \left\{ (\fM, \iota) : \fM \in Y^{\mu, \tau'}(R),\quad \iota: (\sigma^f)^*\fM \stackrel{\sim}{\longrightarrow} \fM^\vee \right\},$$
where $\iota$ is a map of Kisin modules with descent data, such that the composite morphism
$$\fM \xrightarrow{\mathrm{can}} (\sigma^f)^{*} \big((\sigma^f)^{*}\fM\big) \xrightarrow{(\sigma^f)^{*}\iota} (\sigma^f)^{*}(\fM^\vee) \xrightarrow{\mathrm{can}} \big((\sigma^f)^{*}\fM\big)^\vee \xrightarrow{(\iota^\vee)^{-1}} \fM$$
is $-1$ on $\fM$.  We call $\iota$ a \emph{polarization} of $\fM$.  A morphism $(\fM_1,\iota_1) \longrightarrow (\fM_2,\iota_2)$ in $Y^{\mu,\tau'}_{\pol}(R)$ is a morphism $\alpha: \fM_1 \longrightarrow \fM_2$ in $Y^{\mu,\tau'}(R)$ such that the following diagram commutes:
\begin{center}
\begin{tikzcd}
(\sigma^f)^*\fM_1 \ar[d, "\iota_1"] \ar[r, "(\sigma^f)^*\alpha"] & (\sigma^f)^*\fM_2 \ar[d, "\iota_2"] \\
\fM_1^\vee  & \fM_2^\vee \ar[l, "\alpha^\vee"'] \\
\end{tikzcd}
\end{center}
\end{df}

\begin{df} 
\label{gauge BC} Let $R$ be a local Artinian $\cO$-algebra with residue field $\bbF$, suppose $\tau'$ is a $2$-generic principal series tame type, and let $(\fM,\iota) \in Y^{\mu, \tau'}_{\pol}(R)$ . A \emph{gauge basis of $(\fM,\iota)$} is a gauge basis $\beta$ of $\fM\in Y^{\mu, \tau'}(R)$ which is compatible with $\iota$, meaning $\iota((\sigma^f)^*\beta) = {(\un{1}, -\un{1})}\beta^\vee$.
\end{df}

We now discuss the effect of adding a gauge basis.  

\begin{prop}
\label{gaugeuniqueBC} Let $R$ be a local Artinian $\cO$-algebra with residue field $\bbF$, and let $\tau'$ be a $2$-generic principal series tame type.  Let $(\fM,\iota) \in Y^{\mu, \tau'}_{\pol}(R)$.  Then the set of gauge bases of $(\fM,\iota)$ is a torsor for $\widehat{\bT}_{\bG}(\cO_{K_2}\otimes_{\bbZ_p}R)^{\sigma^f = \textnormal{inv}}$,
where $\textnormal{inv}$ denotes the homomorphism $t \longmapsto t^{-1}$.  
\end{prop}

\begin{proof} 
The proof follows the argument of \cite[Prop. 6.12]{LLLM1}.  Let $\beta$ be a gauge basis of $\fM \in Y^{\mu, \tau'}(R)$. Then $\iota((\sigma^f)^*\beta)$ is a gauge basis of $\fM^\vee$ and by \cite[Thm. 4.16]{LLLM1}, the set of gauge bases of $\fM^\vee$ are uniquely determined up to scaling and are exactly $\widehat{\bT}_{\bG}(\cO_{K_2}\otimes_{\bbZ_p}R)\beta^\vee$.  Thus $\iota((\sigma^f)^*\beta) = c\beta^\vee$ for a unique $c \in \widehat{\bT}_{\bG}(\cO_{K_2}\otimes_{\bbZ_p}R)$, and the cocycle condition satisfied by $\iota$ is equivalent to $c^{-1}\sigma^f(c) = {-1}$.  Further, given $t\in \widehat{\bT}_{\bG}(\cO_{K_2}\otimes_{\bbZ_p}R)$, we have $\iota((\sigma^f)^*(t\beta)) = \sigma^f(t)\iota((\sigma^f)^*\beta) = \sigma^f(t)c\beta^\vee$.  Since the basis on $\fM^\vee$ dual to $t\beta$ is $t^{-1}\beta^\vee$, we conclude that the set of gauge bases of $(\fM,\iota)$ is exactly the set of solutions $t \in \widehat{\bT}_{\bG}(\cO_{K_2}\otimes_{\bbZ_p}R)$ to the equation
${(\un{1}, -\un{1})}t^{-1} = \sigma^f(t)c$.  The conclusion follows as in \cite[Prop. 6.12]{LLLM1}: using that $\Res_{\cO_{K_2}/\bbZ_p}$ splits over $\cO$, we have that the equation has a solution, and the solution set is a $\widehat{\bT}_{\bG}(\cO_{K_2}\otimes_{\bbZ_p}R)^{\sigma^f = \textnormal{inv}}$-torsor.
\end{proof}

\subsubsection{}

\begin{lem}
\label{lem:BC:shape}
Let $R$ be a local Artinian $\cO$-algebra with residue field $\bbF$, and $\tau'$ a $2$-generic principal series tame type which satisfies $(\tau')^{\varphi^{-f}} \cong \tau'^\vee$.  Let $w = (w_i)_i\in W'$ denote the orientation of $\tau'$.  

\begin{enumerate}
\item 
\label{it:1:lem:BC:shape}
Let $(\fM,\iota)\in Y^{\mu, \tau'}_{\pol}(R)$ and let $\beta$ denote a gauge basis for $(\fM,\iota)$.  Let $A^{(i)}$ be the matrix of partial Frobenius of $\fM \in Y^{\mu, \tau'}(R)$ with respect to $\beta$.  We then have
\begin{equation}
\label{eq:CSD:Frob}
A^{(i - f)} = {\begin{cases} E(u)s(A^{(i)})^{-\top}s & \textnormal{if}~ i \neq f - 1, 2f - 1, \\ -E(u)s(A^{(i)})^{-\top}s & \textnormal{if}~ i = f - 1, 2f - 1.\end{cases} }
\end{equation}
In particular, if $R = \bbF, (\overline{\fM},\overline{\iota})\in Y^{\mu,\tau'}_{\pol}(\bbF)$, and $\overline{\fM}$ has shape $\tw = (\tw_i)_i \in \tW'$, then 
$$\tw_{i - f} = \tw_i.$$

\item Conversely if $\fM\in Y^{\mu, \tau'}(R)$ and the matrices $A^{(i)}$ of partial Frobenius satisfy the condition \eqref{eq:CSD:Frob} for a gauge basis $\beta$ of $\fM$, then there exists a polarization $\iota$ on $\fM$ such that $(\fM,\iota)\in Y^{\mu, \tau'}_{\pol}(R)$, and such that $\beta$ is a gauge basis for $(\fM,\iota)$.  
\end{enumerate}

\end{lem}

\begin{proof}
(i)  We follow \cite[\S 2.1, 6.1]{LLLM1}.  Let $\jmath: (\sigma^f)^*\fM \longrightarrow \fM$ denote the $\sigma^{-f}$-semilinear bijection sending $s\otimes m$ to $\sigma^{-f}(s)m$.  We have a commutative diagram of $R[\![v]\!]$-modules:
\begin{center}
\begin{tikzcd}[row sep = large, column sep = 8ex]
{}^{\overline{\varphi}}\fM^{(i - f)}_{{\eta}_{{w_{i - f + 1}}(2)}}  \ar[d, "\phi^{(i - f)}_{\fM}"] & {}^{\overline{\varphi}}\big((\sigma^f)^*\fM\big)^{(i)}_{\eta^{p^f}_{{w_{i - f + 1} }(2)}} \ar[l, "\overline{\varphi}^*\jmath"', "\sim"] \ar[d, "\phi^{(i)}_{(\sigma^f)^*\fM}"] \ar[r, equal] & {}^{\overline{\varphi}}\big((\sigma^f)^*\fM\big)^{(i)}_{\eta^{-1}_{{w_{i + 1} } s(2)}}  \ar[r, "\overline{\varphi}^*\iota", "\sim"'] \ar[d, "\phi^{(i)}_{(\sigma^f)^*\fM}"]& {}^{\overline{\varphi}}(\fM^\vee)^{(i)}_{\eta^{-1}_{{ w_{i + 1} } s(2)}} \ar[d, "\phi^{(i)}_{\fM^\vee}"]\\
\fM^{(i - f + 1)}_{{\eta}_{{w_{i - f + 1} } (2)}}  & \big((\sigma^f)^*\fM\big)^{(i + 1)}_{\eta^{p^f}_{{w_{i - f + 1} } (2)}}  \ar[l, "\jmath"', "\sim"] \ar[r, equal] & \big((\sigma^f)^*\fM\big)^{(i + 1)}_{\eta^{-1}_{{w_{i + 1} } s(2)}} \ar[r, "\iota", "\sim"'] & (\fM^\vee)^{(i + 1)}_{\eta^{-1}_{{w_{i + 1} } s(2)}}
\end{tikzcd}
\end{center}
(here the subscripts denote isotypic components).  The left square commutes by \cite[Lem. 6.2]{LLLM1}, the center square commutes by Lemma \ref{lem:prel:BC:shape}, and the right square commutes by definition of polarization.  By Subsubsection \ref{dualshape}, we see that the matrix of partial Frobenius on $\fM^\vee$ at embedding $i$ is $E(u)s(A^{(i)})^{-\top}s$.  Since $\beta$ is a gauge basis which is compatible with the polarization, the above commutative diagram implies that {$A^{(i - f)}$ is of the form stated above.}

\vspace{5pt}

\noindent (ii) We may define $\iota:(\sigma^f)^*\fM \stackrel{\sim}{\longrightarrow}\fM^\vee$ by the condition 
$${\iota(1\otimes f^{(i)}_{w_i(k)}) = \begin{cases} - f^{(i - f),\vee}_{w_{i - f}s(k)} & \textnormal{if}~ 0\leq i \leq f - 1, \\ f^{(i - f),\vee}_{w_{i - f}s(k)} & \textnormal{if}~ f \leq i \leq 2f - 1, \end{cases} }$$ 
where $k = 1,2$, and where $f^{(i - f),\vee}_{w_{i - f}s(k)}$ denotes the basis vector of $\fM^\vee$ dual to $f^{(i - f)}_{w_{i - f}s(k)}$.  The relation \eqref{eq:CSD:Frob} guarantees that $\iota$ is a morphism of Kisin modules.  
\end{proof}

\subsection{Deformations}
\label{subsec:Def:thy}

In this subsection we describe deformations of {Frobenius twist self-dual} Kisin modules and relate them to deformations of local $L$-parameters.  The main result is Corollary \ref{cor:irr:cmpts}, giving a description of the special fiber of the Galois deformation ring in terms of Serre weights.

Throughout the discussion, we fix a tamely ramified $L$-parameter $\rhobar:\Gamma_K \longrightarrow {}^C\bU_2(\bbF)$ such that $\widehat{\imath} \circ \rhobar = \overline{\varepsilon}$, and let $\tau':I_{K_2} \longrightarrow \bG\bL_2(\cO)$ be a principal series tame type satisfying $(\tau')^{\varphi^{-f}}  \cong \tau'^\vee$.

\subsubsection{}

We begin with Kisin modules.  Fix $(\overline{\fM},\overline{\iota})\in Y^{\mu,\tau'}_{\pol}(\bbF)$ and let $\tw = (\tw_i)_i\in \tW'$ denote the shape of $\overline{\fM}$.  We also fix a compatible gauge basis $\overline{\beta}$, and assume that $\tau'$ is $2$-generic.   Given a local Artinian $\cO$-algebra $R$ with residue field $\bbF$, we let 
$$Y^{\mu,\tau'}_{\overline{\fM}, \pol}(R) \defeq \left\{(\fM_R,\iota_R{,\jmath_R}): \begin{array}{l}
\diamond~{(\fM_R, \iota_R) \in Y^{\mu,\tau'}_{\pol}(R)} \\
\diamond~{\jmath_R}:\fM_R \otimes_R \bbF \stackrel{\sim}{\longrightarrow} \overline{\fM}\\
\diamond~ {(\jmath_R^\vee)^{-1}\circ (\iota_R \otimes_R \bbF) = \overline{\iota} \circ (\sigma^f)^*\jmath_R}
\end{array}\right\}$$
and
$$D^{\tau',\overline{\beta}}_{\overline{\fM}, \pol}(R) \defeq \left\{(\fM_R,\iota_R,{\jmath_R}, \beta_R): \begin{array}{l}
\diamond~(\fM_R, \iota_R,{\jmath_R}) \in Y^{\mu,\tau'}_{\overline{\fM},\pol}(R),\\ 
\diamond~\beta_R ~\textnormal{is a gauge basis of $(\fM_R,\iota_R)$ lifting $\overline{\beta}$}
\end{array}\right\}.$$
Using \cite[Thms. 4.16, 4.17]{LLLM1} along with Lemma \ref{lem:BC:shape} and Proposition \ref{gaugeuniqueBC}, we see that $D^{\tau',\overline{\beta}}_{\overline{\fM}, \pol} \longrightarrow Y^{\mu,\tau'}_{\overline{\fM},\pol}$ is a $\widehat{\bG}_m^{2f}$-torsor{, and in particular is representable by a formal Artin stack, since $Y^{\mu,\tau'}_{\overline{\fM},\pol}$ is. 
As $D^{\tau',\overline{\beta}}_{\overline{\fM}, \pol}$ has no nontrivial automorphisms we conclude that} $D^{\tau',\overline{\beta}}_{\overline{\fM}, \pol}$ is representable by a complete local Noetherian $\cO$-algebra $R_{\overline{\fM},\pol}^{\tau',\overline{\beta}}$.  The act of deforming a polarized Kisin module $(\ovl{\fM}, \overline{\iota})$ with $\overline{\fM}\in Y^{\mu,\tau'}_{\tld{w}}(\bbF)$ and a gauge basis on it is equivalent to deforming the collection of associated matrices ${(A^{(i)})_{0 \leq i \leq 2f - 1} }$ subject to the degree conditions of Table \ref{Table2} and \eqref{eq:CSD:Frob}.  We conclude that:
\begin{theo} 
\label{thm:dringBC}  
Let $\tau'$ be a $2$-generic principal series tame type which satisfies $(\tau')^{\varphi^{-f}} \cong \tau'^\vee$, and let $\overline{\fM}, \tw, \overline{\iota},$ and $\overline{\beta}$ be as above.  Then
$$R^{\tau', \overline{\beta}}_{\overline{\fM},\pol} \cong \widehat{\bigotimes}_{i \in \{0, \ldots, f - 1\}} R_{\tw_i}^{\textnormal{expl}}$$
where $R_{\tw_i}^{\textnormal{expl}}$ is as in Table \ref{Table3}, and the completed tensor product is taken over $\cO$.  
In particular $R_{\overline{\fM},\pol}^{\tau',\overline{\beta}}$ is an integral domain.
\end{theo}

\begin{table}[h]
\centering
\caption{\textbf{Deformation rings by shape}}
\vspace{-10pt}
\caption*{\footnotesize{{The variables $x_{i,j}^*$ appearing in the power series rings below correspond to the coefficients $c^*_{i,j}-[\ovl{c}_{i,j}^*]$ of the universal matrices appearing in Tables \ref{Table1}, \ref{Table2}.}}}
\label{Table3}
\begin{tabular}{| c || c | c | c | }
\hline
$\tw_i$ & $\ft$ & $\ft'$ & $\fw$\\
\hline
$R_{\widetilde{w}_i}^{\textnormal{expl}}$ & $\cO[\![c_{2,1},~x^*_{1,1},~x^*_{2,2}]\!]$ & $\cO[\![c_{1,2},~x^*_{1,1},~x^*_{2,2}]\!]$ & $\cO[\![x_{1,1},~y_{2,2},~x^*_{1,2},~x^*_{2,1}]\!]/(x_{1,1}y_{2,2} + p)$\\
\hline
\end{tabular}
\end{table}

\subsubsection{}
\label{poln-equiv}
We now discuss deformations of $L$-parameters.

We recall a result from \cite{CHT} in a language more suited for our purposes.  Let $R$ be a topological $\bbZ_p$-algebra.  By Lemma 2.1.1 of \emph{op. cit.}, there is a bijection between
\begin{itemize}
\item $L$-parameters $\rho: \Gamma_K \longrightarrow {}^C\bU_2(R)$; and
\item triples $(\rho', \theta, \alpha)$, where $\rho':\Gamma_{K_2} \longrightarrow \bG\bL_2(R)$ is a continuous homomorphism, $\theta: \Gamma_{K} \longrightarrow R^\times$ is a continuous character, and $\alpha$ is a compatible polarization, that is, $\alpha: (\rho')^{\varphi^{-f}} \stackrel{\sim}{\longrightarrow} \rho'^\vee \otimes \theta$ such that the composite map
$$\rho' \xrightarrow{v \mapsto \rho'(\varphi^{-2f})v} (\rho')^{\varphi^{-2f}} \xrightarrow{\alpha^{\varphi^{-f}}} (\rho'^\vee \otimes \theta)^{\varphi^{-f}} \xrightarrow{\textnormal{can}} (\rho'\otimes\theta^{-1})^{\varphi^{-f},\vee} \xrightarrow{(\alpha^\vee)^{-1}} \rho'$$
is equal to multiplication by $-\theta(\varphi^{-f})$. 
\end{itemize}
The correspondence is given by sending $\rho:\Gamma_K \longrightarrow {}^C\bU_2(R)$ to $(\BC(\rho), \widehat{\imath}\circ\rho, \alpha)$, where $\rho(\varphi^{-f}) = (A, \theta(\varphi^{-f})) \rtimes \varphi^{-f}$ and $\alpha(v) = \Phi_2^{-1}A^{-1}v$.

In what follows, we will usually fix $\theta = \varepsilon$ {(hence $-\theta(\varphi^{-f})=-1$)}, so that $L$-parameters $\rho:\Gamma_K \longrightarrow {}^C\bU_2(R)$ with $\widehat{\imath}\circ \rho = \varepsilon$ correspond bijectively to pairs $(\rho', \alpha)$ where $\rho': \Gamma_{K_2} \longrightarrow \bG\bL_2(R)$ is a continuous homomorphism and $\alpha$ is a compatible polarization.  In particular, our fixed $\rhobar$ is associated to $(\BC(\rhobar), \overline{\alpha})$.

\subsubsection{}
\label{localgalrepdefrings}

We introduce several deformation problems for Galois representations.  Let $R_{\rhobar}^{\Box}$ denote the universal framed deformation ring of $\rhobar$.  By \cite[\S\S 3.2--3.3]{bellovin-gee}, there exists a unique {$\cO$-flat} quotient $R_{\rhobar}^{\tau'}$ of $R_{\rhobar}^{\Box}$ with the property that if $B$ is a finite local $E$-algebra, then a morphism $x: R_{\rhobar}^{\Box} \longrightarrow B$ factors through $R_{\rhobar}^{\tau'}$ if and only if the corresponding $L$-parameter $\rho_x: \Gamma_{K} \longrightarrow {}^C\bU_2(B)$ is potentially crystalline with $p$-adic Hodge type $(\un{1}, \un{0}, \un{1})\in X_*(\widehat{\bT})$, inertial type $\tau'$ and cyclotomic multiplier $\widehat{\imath}\circ \rho_x = \varepsilon$.  We recall the terminology used above:
\begin{itemize}
\item An $L$-parameter $\Gamma_K \longrightarrow {}^C\bU_2(B)$ is potentially crystalline if and only if it is so after composition with any faithful algebraic representation ${}^C\bU_2 \longhookrightarrow \bG\bL_n$.  
\item Suppose the $L$-parameter $\rho:\Gamma_K \longrightarrow {}^C\bU_2(B)$ has cyclotomic multiplier $\widehat{\imath}\circ\rho = \varepsilon$.  Then $\rho$ has $p$-adic Hodge type $(\un{1}, \un{0}, \un{1})$ if and only if $\BC(\rho)$ has $p$-adic Hodge type $\mu = \vect{\un{1}}{\un{0}}$, that is, if $\BC(\rho)$ has Hodge--Tate weights $\{-1, 0\}$.  
\item An $L$-parameter $\rho:\Gamma_K \longrightarrow {}^C\bU_2(B)$ has inertial type $\tau'$ if $\textnormal{WD}(\rho)|_{I_K} \cong { (\tau' \oplus \mathbf{1}_{I_K}) \otimes_E \overline{E} }$ ({by $\textnormal{WD}(\rho)$ we mean the $\overline{E}$-points of the {${}^C\bU_2$-torsor} whose construction is contained in \cite{bellovin-gee}, Section 2.8, Lemma 2.6.6, and Definition 2.1.1}).  Assuming $\rho$ has cyclotomic multiplier, this is equivalent to $\textnormal{WD}(\BC(\rho))|_{I_{K_2}} \cong \tau'$.  
\end{itemize}
(In this section, we will always be working with framed deformations with $p$-adic Hodge type $(\un{1}, \un{0}, \un{1})$ and cyclotomic multiplier, so we omit $(\un{1}, \un{0}, \un{1})$, $\varepsilon$ and $\Box$ from the notation.)  We write $D_{\rhobar}^{\tau'} = \Spf R_{\rhobar}^{\tau'}$.

Similarly, we let $R_{\BC(\rhobar)}^{\tau'}$ be the framed potentially crystalline deformation ring parametrizing lifts of $\BC(\rhobar)$ with $p$-adic Hodge type $\mu$ and inertial type $\tau'$.  We write $D_{\BC(\rhobar)}^{\tau'} = \Spf R_{\BC(\rhobar)}^{\tau'}$.

\subsubsection{}

Let $R$ denote a local Artinian $\cO$-algebra with residue field $\bbF$. We define
\begin{align*}
{D}^{\tau'}_{\BC(\rhobar),\pol}(R)&\defeq \left\{ (\rho'_R,\alpha_R): \begin{array}{l}
\diamond~ \rho'_R \in {D}_{\BC(\rhobar)}^{\tau'}(R) \\ 
\diamond~ \alpha_R~\textnormal{is a compatible polarization}\\
\quad \textnormal{of $\rho'_R$ lifting $\overline{\alpha}$}
\end{array} \right\}
\end{align*}
We have natural maps
$$D_{\rhobar}^{\tau'} \stackrel{\sim}{\longrightarrow} D_{\BC(\rhobar),\pol}^{\tau'} \longrightarrow {D}^{\tau'}_{\BC(\rhobar)}$$
where the first isomorphism follows from Subsubsection \ref{poln-equiv}.

\subsubsection{}

Our next task will be to relate deformations of $L$-parameters to deformations of Kisin modules.  Before considering further deformation problems we record the following result.

\begin{lem}
\label{lem:Kis:var:triv}
Let $\rhobar:\Gamma_K \longrightarrow {}^C\bU_2(\bbF)$ be a tamely ramified $L$-parameter satisfying $\widehat{\imath}\circ\rhobar = \overline{\varepsilon}$ and $\tau'$ a $2$-generic principal series tame type satisfying $(\tau')^{\varphi^{-f}} \cong \tau'^\vee$.  Then there exists at most one Kisin module $\overline{\fM}\in Y^{\mu, \tau'}(\bbF)$ such that $T_{\textnormal{dd}}^*(\overline{\fM}) \cong \BC(\rhobar)|_{\Gamma_{K_{2,\infty}}}$.  If such an $\overline{\fM}$ exists, then there is a unique polarization $\overline{\iota}$ on $\overline{\fM}$ such that $(\overline{\fM}, \overline{\iota}) \in Y_{\pol}^{\mu, \tau'}(\bbF)$, and such that $\overline{\iota}$ is compatible with the polarization $\overline{\alpha}$ on $\BC(\rhobar)|_{\Gamma_{K_{2,\infty}}}$ via $T_{\textnormal{dd}}^*$.
\end{lem}

\begin{proof}
The first part of the Lemma is \cite[Thm. 3.2]{LLLM1}.  Assume that $\overline{\fM} \in Y^{\mu, \tau'}(\bbF)$ satisfies $T_{\textnormal{dd}}^*(\overline{\fM})\cong \BC(\rhobar)|_{\Gamma_{K_{2,\infty}}}$, and let $\mathcal{M}\defeq \overline{\fM}\otimes_{\fS_\bbF}\cO_{\cE,\bbF}$ denote the associated \'etale $\varphi$-module (where $\cO_{\cE,\bbF} \defeq \cO_{\cE}\otimes_{\bbZ_p}\bbF$ and $\cO_{\cE}$ is the $p$-adic completion of $\cO_{K_2}[\![u]\!][1/u]$).  Since the category of $\Gamma_{K_{2,\infty}}$-representations is equivalent to the category of \'etale $\varphi$-modules, and since $\BC(\rhobar)$ is essentially conjugate self dual, we have an isomorphism
$$\iota:(\sigma^f)^*\, \mathcal{M}\stackrel{\sim}{\longrightarrow} \mathcal{M}^\vee$$
(see \cite[\S 3]{broshi} for the definition and properties of $\cM^\vee$).  By \cite[Thm. 3.2]{LLLM1} the Kisin varieties of both $(\sigma^f)^*\, \mathcal{M}$ and $ \mathcal{M}^\vee$ are trivial.  Since $(\sigma^f)^*\, \ovl{\fM}$ and $ \ovl{\fM}^\vee$ are {$\fS_{\bbF}$}-lattices in $(\sigma^f)^*\, \mathcal{M}$ and $ \mathcal{M}^\vee$, respectively, we conclude that the map $\overline{\iota} \defeq \iota|_{(\sigma^f)^*\, \overline{\fM}}$ factors through an isomorphism $(\sigma^f)^*\, \overline{\fM} \stackrel{\sim}{\longrightarrow} \overline{\fM}^\vee${, giving a polarization on $\overline{\fM}$}.

We now claim that if $\overline{\iota}_1$, $\overline{\iota}_2$ are polarizations on $\overline{\fM}$ which are compatible with the polarization $\overline{\alpha}$ on $\BC(\rhobar)|_{\Gamma_{K_{2,\infty}}}$ then $\overline{\iota}_1=\overline{\iota}_2$.
Since $T_{\textnormal{dd}}^*(\overline{\iota}_1)=T_{\textnormal{dd}}^*(\overline{\iota}_2)$ we deduce that $(\overline{\iota}_1 - \overline{\iota}_2)\otimes_{\fS_\bbF}\cO_{\cE,\bbF} = 0$ and hence $\textnormal{im}(\overline{\iota}_1 - \overline{\iota}_2)$ is a $u$-torsion $\fS_{\bbF}$-submodule of $\overline{\fM}^{\vee}$.  Since $\overline{\fM}^{\vee}$ is a projective $\fS_{\bbF}$-submodule we conclude that $\overline{\iota}_1 - \overline{\iota}_2 = 0$.
\end{proof}

We may now introduce the following definition:
\begin{df}
Let $\rhobar:\Gamma_K \longrightarrow {}^C\bU_2(\bbF)$ be a tamely ramified $L$-parameter such that $\widehat{\imath}\circ\rhobar = \overline{\varepsilon}$ and let $\tau'$ be a $2$-generic principal series tame type satisfying $(\tau')^{\varphi^{-f}} \cong \tau'^\vee$.  Assume that there exists $(\overline{\fM},\overline{\iota}) \in Y_{\pol}^{\mu,\tau'}(\bbF)$ together with an isomorphism $T_{\textnormal{dd}}^*(\overline{\fM}) \stackrel{\sim}{\longrightarrow} \BC(\rhobar)|_{\Gamma_{K_{2,\infty}}}$ compatible with the polarizations on both sides.  We define the \emph{shape of $\rhobar$ with respect to $\tau'$} to be the shape of $\overline{\fM}$, and denote it by $\tw(\rhobar,\tau')$.
\end{df}

Whenever we invoke the shape of an $L$-parameter with respect to a $2$-generic type $\tau'$ (with $\rhobar$ and $\tau'$ as above), we implicity assume that there exists a (necessarily unique) polarized Kisin module $(\overline{\fM},\overline{\iota}) \in Y_{\pol}^{\mu,\tau'}(\bbF)$ such that $T_{\textnormal{dd}}^*(\overline{\fM}) \stackrel{\sim}{\longrightarrow} \BC(\rhobar)|_{\Gamma_{K_{2,\infty}}}$ compatibly with the polarizations on both sides.

\subsubsection{}

In what follows, we fix a polarized Kisin module $(\overline{\fM}, \overline{\iota})\in Y^{\mu,\tau'}_{\pol}(\bbF)$ and an isomorphism $\overline{\delta}:T_{\textnormal{dd}}^*(\overline{\fM}) \stackrel{\sim}{\longrightarrow} \BC(\rhobar)|_{\Gamma_{K_{2,\infty}}}$, which is compatible with the polarizations on both sides (with $\rhobar$ and $\tau'$ as above).  (The existence of such a $(\overline{\fM}, \overline{\iota})$ is a necessary condition for the ring $R^{\tau'}_{\rhobar}$ to be non-zero, {since a non-zero morphism $x:R^{\tau'}_{\rhobar}\longrightarrow\cO$ gives rise to an element of $Y^{\mu,\tau'}_{\pol}(\cO)$ which reduces to $(\overline{\fM}, \overline{\iota})$ modulo $\varpi$, by the analogue of \cite[Theorem (0.1)]{KisinFcrys} with coefficients and descent data.)}

Let $R$ denote a local Artinian $\cO$-algebra with residue field $\bbF$.  We define
\begin{align*}
D^{\tau',\Box}_{\overline{\fM}, \BC(\rhobar)}(R) & \defeq \left\{ (\fM_R, {\jmath_R},  \rho'_R,\delta_R): 
\begin{array}{l}
\diamond~ (\fM_R, {\jmath_R}) \in Y^{\mu,\tau'}_{\overline{\fM}}(R) \\
\diamond~ \rho'_R \in {D}_{\BC(\rhobar)}^{\tau'}(R) \\ 
\diamond~ \delta_R: T^*_{\textnormal{dd}}(\fM_R) \stackrel{\sim}{\longrightarrow} \rho'_R|_{\Gamma_{K_{2,\infty}}} \textnormal{lifts $\overline{\delta}$}
\end{array} \right\} \\
\\
{D}^{\tau',\Box}_{\overline{\fM},\pol; \rhobar}(R) &\defeq \left\{ (\fM_R,\iota_R, {\jmath_R}, \rho_R,\delta_R) :  \begin{array}{l}
\diamond~ (\fM_R,\iota_R,{\jmath_R})\in Y^{\mu, \tau'}_{\overline{\fM},\pol}(R)\\
\diamond~ \rho_R \in {D}_{\rhobar}^{\tau'}(R)\\
\diamond~ \delta_R: T_{\textnormal{dd}}^*(\fM_R) \stackrel{\sim}{\longrightarrow} \BC(\rho_R)|_{\Gamma_{K_{2,\infty}}} ~\textnormal{lifts $\overline{\delta}$,}\\
\quad \textnormal{compatibly with the polarizations}
\end{array} \right\}
\end{align*}
The forgetful functor $(\fM_R, \iota_R) \longmapsto \fM_R$ along with the base change map $\rho_R \longmapsto \BC(\rho_R)$ induces a morphism $D_{\overline{\fM}, \pol; \rhobar}^{\tau',\Box} \longrightarrow D_{\overline{\fM},\BC(\rhobar)}^{\tau',\Box}$ which is compatible with $T^*_{\textnormal{dd}}$.

\begin{lem}
\label{lem:fib:prod}
Let $\rhobar$ and $\tau'$ be as above, so that in particular $\tau'$ is $2$-generic and satisfies $(\tau')^{\varphi^{-f}} \cong \tau'^\vee$.  Then the natural map $D^{\tau',\Box}_{\overline{\fM},\pol; \rhobar} \longrightarrow D^{\tau'}_{\rhobar}$ is an isomorphism.  
\end{lem}

\begin{proof}
Let $R$ be a local Artinian $\cO$-algebra with residue field $\bbF$, and let $\rho_R\in D_{\rhobar}^{\tau'}(R)$.  Recall that the data of $\rho_R$ is equivalent to the data of $(\BC(\rho_R), \alpha_R)$, with $\alpha_R$ a compatible polarization.  By \cite[Cor. 3.6]{LLLM1}, the representing rings $R^{\tau',\Box}_{\overline{\fM}, \BC(\rhobar)}$ and $R^{\tau'}_{\BC(\rhobar)}$ are isomorphic, and hence there exists a unique pair $(\fM_R, \delta_R)$, where $\fM_R\in Y^{\mu,\tau'}_{\overline{\fM}}(R)$ and $\delta_R :T_{\textnormal{dd}}^*(\fM_R) \stackrel{\sim}{\longrightarrow} \BC(\rho_R)|_{\Gamma_{K_{2,\infty}}}$ lifts $\overline{\delta}$.  It remains to construct a unique polarization on $\fM_R$ compatible with $\alpha_R$.  By the equivalence of categories between \'etale $\varphi$-modules and $\Gamma_{K_{2,\infty}}$-representations, the polarization $\alpha_R$ induces a polarization $\iota_R: (\sigma^f)^*\cM_R \stackrel{\sim}{\longrightarrow} \cM^\vee_R$, where $\cM_R$ denotes the \'etale $\varphi$-module associated to $\fM_R$.  The uniqueness of $\fM_R$ implies that $\iota_R$ carries $(\sigma^f)^*\fM_R$ to $\fM^\vee_R$.  Finally, the fact that $\iota_R$ is unique follows exactly as in the proof of Lemma \ref{lem:Kis:var:triv}.  
\end{proof}

\subsubsection{}

We now fix a gauge basis $\overline{\beta}$ on $(\overline{\fM},\overline{\iota})$.  For a local Artinian $\cO$-algebra $R$ with residue field $\bbF$, we define
$${D}^{\tau',\overline{\beta},\Box}_{\overline{\fM},\pol; \rhobar} (R) \defeq \left\{ (\fM_R,\iota_R, { \jmath_R }, \beta_R,\rho_R,\delta_R) : 
\begin{array}{l}
\diamond~(\fM_R,\iota_R, { \jmath_R }, \rho_R,\delta_R)\in {D}^{\tau',\Box}_{\overline{\fM},\pol; \rhobar}(R)\\
\diamond~\beta_R~ \textnormal{is a gauge basis for $(\fM_R,\iota_R)$ lifting $\overline{\beta}$}
\end{array}
\right\}.$$
We see by Proposition \ref{gaugeuniqueBC} that the forgetful map $D^{\tau',\overline{\beta},\Box}_{\overline{\fM},\pol; \rhobar}\longrightarrow D^{\tau',\Box}_{\overline{\fM},\pol; \rhobar}$ is a representable formal torus torsor of relative dimension $2f$.
We denote by $R_{\overline{\fM},\pol; \rhobar}^{\tau',\Box} \longrightarrow R^{\tau, \overline{\beta},\Box}_{\overline{\fM},\pol; \rhobar}$ the corresponding map of deformation rings.  It is a formally smooth morphism of relative dimension $2f$ between complete local Noetherian $\cO$-algebras.

Finally, we define the deformation problem
\begin{align*}
{D}^{\tau', \overline{\beta},\Box}_{\overline{\fM},\pol}(R) & \defeq \left\{ (\fM_R,\iota_R, { \jmath_R },  \beta_R,e_R) : 
\begin{array}{l}
\diamond~ (\fM_R,\iota_R, { \jmath_R }, \beta_R) \in D^{\tau', \overline{\beta}}_{\overline{\fM},\pol}(R) \\
\diamond~ e_R~ \textnormal{is a basis for}~ T_{\textnormal{dd}}^*(\fM_R) \\
\quad \textnormal{lifting the (pullback via $\overline{\delta}$ of the)}\\
\quad \textnormal{standard basis on $\BC(\rhobar)|_{\Gamma_{K_{2,\infty}}}$}
\end{array}
\right\}
\end{align*}
In particular, if $(\fM_R,\iota_R, { \jmath_R }, \beta_R,e_R) \in {D}^{\tau', \overline{\beta},\Box}_{\overline{\fM},\pol}(R)$, then $(T^*_{\textnormal{dd}}(\fM_R),e_R)$ is a framed deformation of $\BC(\rhobar)|_{\Gamma_{K_{2,\infty}}}$.  We let $R^{\tau', \overline{\beta},\Box}_{\overline{\fM},\pol}$ denote the deformation ring corresponding to the above deformation problem.

\subsubsection{}

The relationships between the various deformation problems are summarized in the following diagram, where ``f.s.'' stands for formally smooth. 
\begin{equation} \label{defdiagram}
\begin{tikzcd}
 \Spf R^{\tau'}_{\rhobar} & \Spf R_{\overline{\fM},\pol; \rhobar}^{\tau', \Box}  \ar[l, "\sim"'] & \Spf R^{\tau', \overline{\beta}, \Box}_{\overline{\fM}, \pol; \rhobar} \ar[l, "\textnormal{f.s.}"']   \ar[r]  & \Spf R_{\overline{\fM},\pol}^{\tau', \overline{\beta}, \Box} \ar[r, 
"\textnormal{f.s.}"] & \Spf R_{\overline{\fM},\pol}^{\tau', \overline{\beta}}
\end{tikzcd}
\end{equation}

The maps which are formally smooth correspond to forgetting either a gauge basis on the (polarized) Kisin module or a framing on the Galois representation.  The former is formally smooth of relative dimension $2f$ while the latter is formally smooth of relative dimension $4$.  The isomorphism follows from Lemma \ref{lem:fib:prod}.

Our next goal will be to show that the remaining map $\Spf R^{\tau', \overline{\beta}, \Box}_{\overline{\fM}, \pol; \rhobar} \longrightarrow \Spf R_{\overline{\fM},\pol}^{\tau', \overline{\beta}, \Box}$ is an isomorphism.  This will follow from some calculations with Galois cohomology.

\subsubsection{}

Given the tamely ramified $L$-parameter $\rhobar: \Gamma_K \longrightarrow {}^C\bU_2(\bbF)$ with $\widehat{\imath}\circ \rhobar = \overline{\varepsilon}$, we set $\textnormal{ad}^0(\rhobar) \defeq \fg\fl_2(\bbF)$. 
{It is a direct summand of the Lie algebra of ${}^C\bU_2$ endowed with the adjoint action of $\Gamma_K$ via $\rhobar$.
Explicitly the action of $\Gamma_K$ on the direct summand $\textnormal{ad}^0(\rhobar)$ is given as follows:}
$\Gamma_{K_2}$ acts by the adjoint action (via $\BC(\rhobar)$), and $\rhobar(\varphi^{-f}) = (A,1)\rtimes \varphi^{-f}$ acts by 
$$X \longmapsto -A\Phi_2X^{\top}\Phi_2^{-1}A^{-1}.$$

\begin{lem}
Suppose $\rhobar$ is $1$-generic.  Then the restriction map on cocycles
$$Z^1(\Gamma_{K},\textnormal{ad}^0(\rhobar)) \longrightarrow Z^1(\Gamma_{K_\infty}, \textnormal{ad}^0(\rhobar))$$
is injective.  
\end{lem}

\begin{proof}
Lemma \ref{galreps} implies that $\textnormal{ad}^0(\rhobar)|_{\Gamma_{K_2}}$ is a direct sum of four characters, and the condition of $1$-genericity implies that none are equal to the mod $p$ cyclotomic character.  Thus, $\textnormal{ad}^0(\rhobar)$ is cyclotomic free, in the terminology of \cite[Def. 3.8]{LLLM1}.  The result now follows from \emph{op. cit.}, Proposition 3.12.  
\end{proof}

\begin{prop}
\label{forget-galois}
Suppose $\rhobar$ is $1$-generic.  Then the natural map $\Spf R^{\tau', \overline{\beta}, \Box}_{\overline{\fM}, \pol; \rhobar} \longrightarrow \Spf R_{\overline{\fM},\pol}^{\tau', \overline{\beta}, \Box}$ is an isomorphism.
\end{prop}

\begin{proof}
By considering tangent spaces and using the above lemma, the map in question is a closed immersion (compare \cite[Prop. 5.11]{LLLM1}).  Therefore it suffices to prove it is surjective on $R$-points.  
{This is obtained following the argument of the proof of \cite[Theorem 5.12]{LLLM1}, noting that in our situation, the monodromy condition in \emph{op.~cit.}~ is empty and the $p$-adic Hodge type is $(1,0)$ in all embeddings.}
{(Alternatively, one can invoke \cite[Thm. 2.1.12]{CDM}: the cited theorem implies that if $(\fM_R,\iota_R, \jmath_R,  \beta_R,e_R) \in {D}^{\tau', \overline{\beta},\Box}_{\overline{\fM},\pol}(R)$, then we may extend the framed deformation $(T^*_{\textnormal{dd}}(\fM_R),e_R)$ of $\BC(\rhobar)|_{\Gamma_{K_{2,\infty}}}$ to a framed deformation of $\BC(\rhobar)$; the claim about functoriality in \emph{op.~cit.}~implies that the polarization $T^*_{\textnormal{dd}}(\iota_R)$ also extends.)}
\end{proof}

\subsubsection{}
\label{subsubsec:integrality}

By Theorem \ref{thm:dringBC}, Proposition \ref{forget-galois} and \eqref{defdiagram}, we finally conclude that:
\begin{equation}
\label{defringpres}
R^{\tau'}_{\rhobar}[\![S_1, \ldots, S_{2f}]\!] \cong R^{\tau', \overline{\beta}, \Box}_{\overline{\fM}, \pol; \rhobar} \cong R_{\overline{\fM},\pol}^{\tau',\overline{\beta}}[\![T_1,\ldots,T_4]\!] \cong \left(\widehat{\bigotimes}_{i \in \{0, \ldots, f - 1\}} R_{\tw_i}^{\textnormal{expl}}\right)[\![T_1,\ldots,T_4]\!]
\end{equation}
where $\tw = (\tw_i)_i = \tw(\rhobar,\tau')\in \tW'$ is the shape of $\rhobar$ with respect to $\tau'$.

\subsubsection{}
\label{subsub:main:galois}

The following corollary is the main result on the local Galois side.

\begin{corollary}
\label{cor:irr:cmpts}
Let $\rhobar: \Gamma_K \longrightarrow {}^C\bU_2(\bbF)$ be a $3$-generic tamely ramified $L$-parameter which satisfies $\widehat{\imath} \circ \rhobar = \overline{\varepsilon}$.  Let $\tau'$ denote a $3$-generic principal series tame type which satisfies $(\tau')^{\varphi^{-f}} \cong \tau'^\vee$, and let $\sigma(\tau')$ denote the tame type associated to $\tau'$ via Theorem \ref{ILLC}.  We view $\sigma(\tau')$ as a Deligne--Lusztig representation of $\bG(\zp)$ on which $\imath(\cO_K^\times)$ acts trivially.  
{Assume that there exists $(\overline{\fM},\overline{\iota}) \in Y_{\pol}^{\mu,\tau'}(\bbF)$ together with an isomorphism $T_{\textnormal{dd}}^*(\overline{\fM}) \stackrel{\sim}{\longrightarrow} \BC(\rhobar)|_{\Gamma_{K_{2,\infty}}}$ compatible with the polarizations on both sides.}

We then have
$$\big|\textnormal{W}^?(\rhobar) \cap \JH\big(\overline{\sigma(\tau')}\big)\big| = e(R_{\rhobar}^{\tau'}\otimes_{\cO}\bbF),$$
where $e(-)$ denotes the Hilbert--Samuel multiplicity.  
\end{corollary}

\begin{proof}
Let $(\overline{\fM},\overline{\iota}) \in Y^{\mu,\tau'}_{\pol}(\bbF)$ correspond to $\rhobar$, let $\overline{\beta}$ denote a gauge basis for $(\overline{\fM},\overline{\iota})$, and let $\tw = (\tw_i)_i = \tw(\rhobar,\tau') \in \tW'$ denote the shape of $\rhobar$ with respect to $\tau'$.  The isomorphism \eqref{defringpres} above implies that 
$$e(R_{\rhobar}^{\tau'}\otimes_{\cO}\bbF) = e(R_{\overline{\fM},\pol}^{\tau',\overline{\beta}}\otimes_{\cO}\bbF) = 2^{|\{0 \leq i \leq f - 1: \tw_i = \fw\}|},$$
where the last equality follows from Table \ref{Table3}.

By the $\bG\bL_2$-analog of the discussion in \cite[\S 5.2]{LLLM1}, we see that $R_{\BC(\rhobar)}^{\tau'}$ is a formally smooth modification of $R^{\tau',\overline{\beta}}_{\overline{\fM}}$, {where the latter ring represents the functor sending a local Artinian $\cO$-algebra $R$ with residue field $\bbF$ to the set of triples $(\fM_R, \jmath_R, \beta_R)$, where $(\fM_R, \jmath_R)\in Y^{\mu,\tau'}_{\overline{\fM}}(R)$ and $\beta_R$ is a gauge basis of $(\fM_R, \jmath_R)$ lifting $\overline{\beta}$.  Further, the structure of $R^{\tau', \overline{\beta}}_{\overline{\fM}}$ is obtained by removing the restriction ``$i \in \{0, \ldots, f - 1\}$'' in the right-hand side of Theorem \ref{thm:dringBC} (this is the $\bG\bL_2$-analog of \cite[Thm.\ 4.17]{LLLM1}).}  Thus, Lemma \ref{lem:BC:shape}\ref{it:1:lem:BC:shape} implies
\begin{eqnarray*}
e(R_{\BC(\rhobar)}^{\tau'}\otimes_{\cO}\bbF) & = & e(R^{\tau',\overline{\beta}}_{\overline{\fM}}\otimes_{\cO}\bbF) \\
 & = & 2^{|\{i: \tw_i = \fw\}|} \\
 & = & 2^{2|\{0 \leq i \leq f - 1: \tw_i = \fw\}|}\\
 & = & e(R_{\rhobar}^{\tau'}\otimes_{\cO}\bbF)^2.
\end{eqnarray*}
After unwinding definitions and conventions regarding duals and Hodge--Tate weights, \cite[Thm. A]{gee-kisin} gives
$$\big|\textnormal{W}^?(\BC(\rhobar)) \cap \JH\big(\overline{\BC(\sigma(\tau'))}\big)\big| = e(R_{\BC(\rhobar)}^{\tau'}\otimes_{\cO}\bbF).$$
Hence, it is enough to prove that
$$\big|\textnormal{W}^?(\BC(\rhobar)) \cap \JH\big(\overline{\BC(\sigma(\tau'))}\big) \big| = \big| \textnormal{W}^?(\rhobar) \cap \JH\big(\overline{\sigma(\tau')}\big)\big|^2.$$
This follows from Propositions \ref{prop:double:weights} {(applied to $\beta(V_\phi(\rhobar))$ and $\sigma(\tau')$)}, \ref{prop:weights:type} and \ref{prop:compatibilityBC}.
\end{proof}

\section{Global applications I}
\label{sec:global}

In this section we apply the results of Sections \ref{sec:RepThy} and \ref{sec:Lpar} in a global context.  Our main references will be \cite{CHT} and \cite{CEGGPS}; as such, we will be considering Galois representations valued in the group $\cG_2$.  (We will translate these results back to the group ${}^C\bU_2$ at the end of Section \ref{glob2}.)  After preliminaries on automorphic forms on unitary groups and their associated Galois representations (Theorem \ref{galrepheckealg}), we give the main result on weight elimination in Theorem \ref{thm:WE}, building on the compatibility of base change of tame types and $L$-parameters.

We caution the reader that some of the notation below differs from previous sections.

\subsection{Unitary groups}
\label{subsec:UGGlob}
\subsubsection{}  
Let $F$ be an imaginary CM field with maximal totally real subfield $F^+$.  We suppose:
\begin{itemize}
\item $F^+/\bbQ$ is unramified at $p$;
\item $F/F^+$ is unramified at all finite places; and
\item every place of $F^+$ above $p$ is inert in $F$.  
\end{itemize}
This implies that $[F^+:\bbQ]$ is even (cf. \cite[\S 3.1]{gee-kisin}), and there exists a reductive group $\bG_{/ \cO_{F^+}}$, which is a totally definite unitary group, quasi-split at all finite places. {More precisely we take}
$$\bG(R) = \{g\in \bG\bL_2(\cO_{F}\otimes_{\cO_{F^+}}R)\, :\, g^{(c\otimes 1)\top}g = 1_2 \},$$
where $R$ is an $\cO_{F^+}$-algebra, and where we write $c\in\Gal(F/F^+)$ for the complex conjugation.  

\emph{Note that this group is different from the group $\bG$ from Subsubsection \ref{sub:sub:unitary:loc}}.

The group $\bG$ is equipped with an isomorphism 
$$\iota: \bG\times_{\cO_{F^+}}\cO_F\stackrel{\sim}{\longrightarrow} {\bG\bL_2}_{/ \cO_{F}}$$ 
which satisfies $\iota\circ (1\otimes c) \circ \iota^{-1}(g) = g^{-c\top}$.  For all places $v$ of $F^+$ which split in $F$ as $v=w w^c$, we obtain an induced isomorphism 
$$\iota_w:\bG(\cO_{F^+_v})\stackrel{\sim}{\longrightarrow}\bG\bL_2(\cO_{F_w})$$ 
such that $\iota_w \circ \iota^{-1}_{w^{c}}(g) = g^{-c\top}$. 
If $v$ is a place of $F^+$ which is inert in $F$, then we have an isomorphism
$$\iota_v:\bG(\cO_{F^+_v}) \stackrel{\sim}{\longrightarrow} \bU_2(\cO_{F^+_v})\subseteq\bG\bL_2(\cO_{F_v}),$$
where $\bU_2$ is the quasi-split unitary group over $\cO_{F^+_v}$ defined Subsection \ref{unitarygps:local}.  This isomorphism is given by $g\longmapsto \sm{1}{b}{b}{-1}g\sm{1}{b}{b}{-1}^{-1}$, where $b\in \cO_{F_v}^\times$ is an element which satisfies $bb^c = -1$ and $b \not\in \cO_{F^+_v}^\times$.
{Finally, for an embedding $\kappa^+:F^+ \longhookrightarrow \bbR$, the group $\bG(F^+_{\kappa^+})$ is compact, and isomorphic to the compact unitary group $\bU_{2}(\bbR)$.  }

(We note that the running hypothesis in \cite{gee-kisin} that $v$ splits in $F$ for $v$ a place of $F^+$ above $p$ is irrelevant for the construction and the basic properties of the group $\bG$.)

\subsubsection{}

Set $F_p^+\defeq F^+\otimes_{\bbQ}\bbQ_p$ and $\cO_{F^+,p}\defeq \cO_{F^+}\otimes_\bbZ\bbZ_p$.  Recall that $E$ is our coefficient field, with ring of integers $\cO$, uniformizer $\varpi$, and residue field $\bbF$.  We assume $E$ is sufficiently large; in particular, we will assume that the image of every embedding $F \longhookrightarrow \overline{\bbQ}_p$ is contained in $E$.

We write $\Sigma_p^+$ (resp.~$\Sigma_p$) for the set of places of $F^+$ (resp.~$F$) lying above $p$.  Restriction from $F$ to $F^+$ gives a bijection between $\Sigma_p$ and $\Sigma_p^+$, and we will often identify these two sets.  Similarly, we let $I_p^+$ (resp.~$I_p$) denote the set of embeddings $\kappa^+: F^+ \longhookrightarrow E$ (resp.~$\kappa: F \longhookrightarrow E$).  We fix a subset $\widetilde{I}_p \subseteq I_p$ such that $I_p = \widetilde{I}_p \sqcup \widetilde{I}_p^c$.  Then restriction from $F$ to $F^+$ gives a bijection between $\widetilde{I}_p$ and $I_p^+$.  Further, composing $\kappa^+ \in I_p^+$ (resp.~$\kappa\in \widetilde{I}_p$) with the valuation on $E$ gives an element of $\Sigma_p^+$ (resp.~$\Sigma_p$), and we let $v(\kappa^+)$ (resp.~$v(\kappa)$) denote the place induced from the embedding $\kappa^+\in I_p^+$ (resp.~$\kappa\in \widetilde{I}_p$).  This gives the following diagram:

\begin{center}
\begin{equation}
\label{diag:emb}
\begin{tikzcd}[row sep = large, column sep = 8ex]
I_p  \ar[r, hookleftarrow]  & \widetilde{I}_p \ar[r, twoheadrightarrow, "\kappa \mapsto v(\kappa)"] \ar[d, leftrightarrow, "\textnormal{res}"] & \Sigma_p \ar[d, leftrightarrow, "\textnormal{res}"]\\
 & I_p^+ \ar[r, twoheadrightarrow, "\kappa^+ \mapsto v(\kappa^+)"] & \Sigma_p^+
\end{tikzcd}
\end{equation}
\end{center}

For a finite place $v$ of $F^+$ (resp.~$F$), we let $\bbF_v^+$ (resp.~$\bbF_v$) denote the residue field of $v$.  We have $\bG(\bbF_v^+)\cong \bU_2(\bbF_v^+)$ for all $v\in\Sigma_p^+$ by construction.

\subsection{Algebraic automorphic forms on unitary groups}
\label{sec:AAF}

\subsubsection{}

Let $K = \prod_{v}K_v$ be a compact open subgroup of $\bG(\bbA_{F^+}^{\infty})$.  We set 
$$K_p \defeq \prod_{v\in \Sigma_p^+}K_v,\qquad K^p \defeq \prod_{v\not\in \Sigma_p^+}K_v,$$ 
and if $k\in K$, we write $k_p$ for the projection of $k$ to $K_p$.  We say that the level $K$ is \emph{sufficiently small} if for all $t \in \bG(\bbA^{\infty}_{F^+})$, the finite group $t^{-1} \bG(F^+) t \cap K$ does not contain an element of order $p$.

\subsubsection{}
Let $K = \prod_v K_v \subseteq \bG(\bbA_{F^+}^{\infty,p})\times\bG(\cO_{F^+,p})$ be a compact open subgroup, and suppose $W$ is an $\cO$-module endowed with an action of $\bG(\cO_{F^+,p})$.  The space of algebraic automorphic forms on $\bG(\bbA_{F^+}^{\infty})$ of level $K$ and coefficients in $W$ is defined as the $\cO$-module
\begin{equation*}
S_\bG(K,W)\defeq \left\{f:\,\bG(F^{+})\backslash \bG(\bbA^{\infty}_{F^{+}}) \longrightarrow W\,:\, f(gk) = k_p^{-1}f(g)\,\,\forall g\in \bG(\bbA^{\infty}_{F^{+}}), k\in K\right\}.
\end{equation*}

Given a compact open subgroup $K$ as above, we have 
$$\bG(\bbA_{F^+}^\infty) = \bigsqcup_i \bG(F^+)t_iK$$
for some finite set $\{t_i\}_i$.  This induces an isomorphism of $\cO$-modules
\begin{eqnarray*}
S_{\bG}(K,W) & \stackrel{\sim}{\longrightarrow} & \bigoplus_i W^{K\cap t_i^{-1}\bG(F^+)t_i} 
\\
f & \longmapsto & (f(t_i))_i
\end{eqnarray*}
In particular we have inclusions $S_{\bG}(K,W)\subseteq S_{\bG}(K',W)$ for $K'\subseteq K$.  If we assume that $K$ is sufficiently small or $A$ is a flat $\cO$-algebra, we further have
\begin{equation}
\label{cmp:ext:sc}
S_{\bG}(K,W)\otimes_{\cO}A \cong S_{\bG}(K,W\otimes_{\cO}A).
\end{equation}

\subsubsection{}

Suppose that $J = \prod_v J_v \subseteq \bG(\bbA_{F^+}^{\infty,p})\times\bG(\cO_{F^+,p})$ is a compact subgroup.  We define
$$S_{\bG}(J,W) \defeq \varinjlim_{K\supseteq J} S_{\bG}(K,W),$$
where $K$ runs over compact open subgroups containing $J$, for which $K_p \subseteq \bG(\cO_{F^+,p})$.  If $g\in \bG(\bbA_{F^+}^\infty)$ is such that $g_p\in \bG(\cO_{F^+,p})$ then 
$$(g.f)(h) = g_p.f(hg)$$
defines an element $g.f$ of $S_{\bG}(gJg^{-1},W)$.  Hence, we obtain an action of $g$ on $S_{\bG}(J,W)$ as soon the relation $J \subseteq gJg^{-1}$ is satisfied.  In particular, if $J = \prod_v J_v \subseteq \bG(\bbA_{F^+}^{\infty,p})\times\bG(\cO_{F^+,p})$ is any compact subgroup, then $J$ acts on $S_{\bG}(\{1\},W)$, and we have
\begin{equation}
\label{eq:fixed:aut}
S_{\bG}(\{1\},W)^J = S_{\bG}(J,W).
\end{equation}

\subsubsection{}

Recall the map $I_p \longtwoheadrightarrow \Sigma_p$ defined by $\kappa \longmapsto v(\kappa)$.  This gives a bijection $I_p \stackrel{\sim}{\longrightarrow} \bigsqcup_{v\in \Sigma_p}\textnormal{Hom}(F_v,E)$ and we identify embeddings $F_v\longhookrightarrow E$ with elements in $I_p$ without further comment.
Let $v\in \Sigma_p$.  
We define $\widetilde{\textnormal{Hom}}(F_v,E)\subseteq \textnormal{Hom}(F_v,E)$ by the condition
$$
\widetilde{I}_p\stackrel{\sim}{\longrightarrow} \bigsqcup_{v\in\Sigma_p}\widetilde{\textnormal{Hom}}(F_v,E)
$$
where the map is given as restriction of the map $I_p \stackrel{\sim}{\longrightarrow} \bigsqcup_{v\in \Sigma_p}\textnormal{Hom}(F_v,E)$.  Note that $\kappa \longmapsto \kappa\circ c$ defines a non-trivial involution on $\textnormal{Hom}(F_v,E)$ and hence $|\widetilde{\textnormal{Hom}}(F_v,E)| = \frac{1}{2}|\textnormal{Hom}(F_v,E)|$.

\subsubsection{}

Let $\bbZ_+^2$ denote the set of all pairs of integers $(\lambda_1,\lambda_2)$ such that $\lambda_1 \geq \lambda_2$.  ({Thus, for $v\in \Sigma_p^+$, we may identify $(\bbZ_+^2)^{\textnormal{Hom}(F_v^+,E)}$ with $X_+(\textnormal{Res}_{\cO_{F^+_v}/\bbZ_p}(\bT_{\bU}))$, where $\bT_\bU$ denotes the torus of the group $\bU_2$ defined in Subsubsection \ref{tori} with $K = F^+_v$.  Note that the discussion in Subsection \ref{subsec:thegroupG} works equally well for the group $\bU_2$ and its restriction of scalars.})  Given $\lambda_v = (\lambda_\kappa)_\kappa \in (\bbZ_+^2)^{\widetilde{\textnormal{Hom}}(F_v,E)}$, we let $W_{\lambda_v}$ denote the free $\cO$-module
$$W_{\lambda_v} \defeq \bigotimes_{\kappa\in \widetilde{\textnormal{Hom}}(F_v,E)}\det{}^{\lambda_{\kappa,2}}\otimes_{\cO_{F_v}} \textnormal{Sym}^{\lambda_{\kappa,1} - \lambda_{\kappa,2}}(\cO_{F_v}^2)\otimes_{\cO_{F_v},\kappa}\cO,$$
which, by restriction and using the isomorphism $\iota_v$, has an action of $\bG(\cO_{F^+_v})$.  
Given an element $\lambda = (\lambda_{\kappa})_\kappa\in (\bbZ_+^2)^{\widetilde{I}_p} = \bigoplus_{v\in \Sigma_p}(\bbZ_+^2)^{\widetilde{\textnormal{Hom}}(F_v,E)}$, we set
$$W_{\lambda} \defeq \bigotimes_{v\in \Sigma_p} W_{\lambda_v},$$
which is a free $\cO$-module with an action of $\prod_{v\in \Sigma_p^+}\bG(\cO_{F^+_v}) = \bG(\cO_{F^+,p})$.

Since $F_v^+$ is unramified over $\bbQ_p$ for every $v \in \Sigma_p^+$, restriction and reduction mod $p$ give bijections 
$$\widetilde{\Hom}(F_v,E) \stackrel{\sim}{\longrightarrow} \Hom(F_v^+,E) \stackrel{\sim}{\longrightarrow} \Hom(\bbF_v^+, \bbF).$$  
For an element $\lambda = (\lambda_v)_{v\in \Sigma_p}\in (\bbZ_{+}^2)^{\widetilde{I}_p} = \bigoplus_{v\in \Sigma_p} (\bbZ_{+}^2)^{\widetilde{\Hom}(F_v,E)}$, we let $\overline{\lambda} = (\overline{\lambda}_v)_{v\in \Sigma_p}$ denote its image in $\bigoplus_{v\in \Sigma_p} (\bbZ_+^2)^{\Hom(\bbF_v^+,\bbF)}$.  Let $\bbZ_{+,p}^2$ denote the subset of $\bbZ_+^2$ consisting of elements $(\lambda_1,\lambda_2)$ satisfying $\lambda_1 - \lambda_2 \leq p - 1$. Then the image of $(\bbZ_{+,p}^2)^{\widetilde{I}_p}$ in $\bigoplus_{v\in \Sigma_p} (\bbZ_+^2)^{\Hom(\bbF_v^+,\bbF)}$ {gives rise to the irreducible mod $p$ representations of $\bG(\cO_{F^+,p})$, in a manner similar to Proposition \ref{serrewtparam}.}  {(More precisely, under the identification of $(\bbZ_+^2)^{\textnormal{Hom}(F_v^+,E)}$ with $X_+(\textnormal{Res}_{\cO_{F^+_v}/\bbZ_p}(\bT_{\bU}))$, the set $(\bbZ_{+,p}^2)^{\textnormal{Hom}(F_v^+,E)}$ is identified with $X_1(\textnormal{Res}_{\cO_{F^+_v}/\bbZ_p}(\bT_{\bU}))$.)}  In particular, if $\lambda = (\lambda_v)_{v\in \Sigma_p}\in (\bbZ_{+,p}^2)^{\widetilde{I}_p} = \bigoplus_{v\in \Sigma_p} (\bbZ_{+,p}^2)^{\widetilde{\Hom}(F_v,E)}$, we have
$$W_\lambda \otimes_{\cO} \bbF \cong \bigotimes_{v\in \Sigma_p} F(\overline{\lambda}_v)$$
as mod $p$ representations of $\bG(\cO_{F^+,p})$.

\subsubsection{}
\label{subsubssec:cl:aut:fms}

We now relate the spaces $S_{\bG}(K,W)$ to spaces of classical automorphic forms.

We let $\cA$ denote the space of automorphic forms on $\bG(\bbA_{F^+})$ (see, e.g., \cite[\S\S 1.5 - 1.8]{GSchwermer}).  Since $\bG$ is totally definite, $\cA$ decomposes as a $\bG(\bbA_{F^+})$-representation as
\begin{equation}
\label{eq:aut:dec}
\cA \cong \bigoplus_{\pi} m(\pi)\pi
\end{equation}
where $\pi$ runs through the isomorphism classes of irreducible admissible representations of $\bG(\bbA_{F^+})$ and $m(\pi)$ is the (finite) multiplicity of $\pi$ in $\cA$ (\cite[\S 2.2]{guerberoff}, \cite[\S 6.2.3]{BC}).

Fix an isomorphism $\imath:\overline{E} \stackrel{\sim}{\longrightarrow} \bbC$.  This gives an identification
$$\imath_*: (\bbZ^2_+)^{\widetilde{I}_p} \stackrel{\sim}{\longrightarrow} (\bbZ^2_+)^{\textnormal{Hom}(F^+,\bbR)}$$
defined by $(\imath_*\lambda)_\kappa = \lambda_{\widetilde{\imath^{-1}\circ\kappa}}$ for $\kappa:F^+\longhookrightarrow \bbR$ (here $\widetilde{\imath^{-1}\circ\kappa}$ denotes the unique element of $\widetilde{I}_p$ lying over $\imath^{-1}\circ\kappa\in I_p^+$).

The set $(\bbZ^2_+)^{\textnormal{Hom}(F^+,\bbR)}$ parametrizes irreducible complex representations of $\bG(F^+_\infty)$; given $\mu\in (\bbZ^2_+)^{\textnormal{Hom}(F^+,\bbR)}$, we let $\mathsf{W}_\mu$ denote the associated irreducible complex $\bG(F^+_\infty)$-representation.

For $\lambda\in (\bbZ^2_+)^{\widetilde{I}_p}$, the space $W_\lambda\otimes_{\cO,\imath}\bbC$ is a complex representation of $\bG(F^+_p)$.  We let
$$\theta:W_\lambda \otimes_{\cO,\imath} \bbC \stackrel{\sim}{\longrightarrow} \mathsf{W}_{\imath_*\lambda}$$
denote a $\bG(F^+)$-equivariant isomorphism.

\subsubsection{}
\label{sub:sub:autom:forms}

From now onwards we let $\sigma^\circ = \bigotimes_{v\in \Sigma_p^+}\sigma_v^\circ$ denote a smooth $\bG(\cO_{F^+,p})$-representation on a finite free $\cO$-module such that $\sigma^\circ \otimes_{\cO}E$ is a tame $\bG(\cO_{F^+,p})$-type.  (By abuse of language, we say that $\sigma^\circ$ is a tame $\bG(\cO_{F^+,p})$-type over $\cO$.)

Fix $\lambda \in (\bbZ_+^2)^{\widetilde{I}_p}$.  By letting $\bG(F^+_\infty)$ act trivially on the second tensor factor of $\mathsf{W}_{\imath_*\lambda}^\vee \otimes_{\bbC} (\sigma^\circ\otimes_{\cO,\imath}\bbC)^\vee$ we define an isomorphism 
\begin{equation}\label{autcomp}
S_{\bG}\big(\{1\},(W_\lambda\otimes_{\cO,\imath}\bbC)\otimes_{\bbC}(\sigma^\circ\otimes_{\cO,\imath}\bbC)\big) \stackrel{\sim}{\longrightarrow} \Hom_{\bG(F^+_\infty)}\big(\mathsf{W}_{\imath_*\lambda}^\vee \otimes_{\bbC} (\sigma^\circ\otimes_{\cO,\imath}\bbC)^\vee, \cA\big)
\end{equation}
as follows.
Let $f:\bG(F^+)\backslash \bG(\bbA_{F^+}^\infty) \longrightarrow (W_\lambda\otimes_{\cO,\imath}\bbC)\otimes_{\bbC}(\sigma^\circ\otimes_{\cO,\imath}\bbC)$ be an element of the left hand side.  
We send this element to a homomorphism $\widetilde{f}:\mathsf{W}_{\imath_*\lambda}^\vee \otimes_{\bbC} (\sigma^\circ\otimes_{\cO,\imath}\bbC)^\vee \longrightarrow \cA$ defined by
$$\widetilde{f}(w^\vee)(g) = w^\vee\left((\xi_{\infty}(g_\infty^{-1})\otimes 1) \circ(\theta \otimes 1) \circ (\xi_p(g_p)\otimes 1). f(g^{\infty})\right),$$
where $g\in \bG(\bbA_{F^+})$, $w^\vee\in (\mathsf{W}_{\imath_*\lambda}\otimes_{\bbC} (\sigma^\circ \otimes_{\cO,\imath}\bbC))^\vee \cong \mathsf{W}_{\imath_*\lambda}^\vee \otimes_{\bbC} (\sigma^\circ\otimes_{\cO,\imath}\bbC)^\vee$, $\xi_p$ denotes the action of $\bG(F^+_p)$ on $W_\lambda\otimes_{\cO,\imath}\bbC$, and $\xi_\infty$ denotes the action of $\bG(F^+_\infty)$ on $\mathsf{W}_{\imath_*\lambda}$.  One easily checks that this isomorphism is well defined and $\bG(\bbA_{F^+}^{\infty,p})\times \bG(\cO_{F^+,p})$-equivariant.   
Therefore if $J = \prod_v J_v \subseteq \bG(\bbA_{F^+}^{\infty,p})\times \bG(\cO_{F^+,p})$ is a compact subgroup we have 
\begin{eqnarray}
S_{\bG}\big(J,W_\lambda\otimes_{\cO}\sigma^\circ\big) \otimes_{\cO,\imath} \bbC & \stackrel{\eqref{cmp:ext:sc}}{\cong} & S_{\bG}\big(J,(W_\lambda\otimes_{\cO,\imath}\bbC)\otimes_{\bbC}(\sigma^\circ\otimes_{\cO,\imath}\bbC)\big)\nonumber\\
& \stackrel{\eqref{eq:fixed:aut}}{\cong} & S_{\bG}\big(\{1\},(W_\lambda\otimes_{\cO,\imath}\bbC)\otimes_{\bbC}(\sigma^\circ\otimes_{\cO,\imath}\bbC)\big)^J\nonumber
\\
& \stackrel{\eqref{autcomp}}{\cong} & \Hom_{\bG(F^+_\infty)}\big(\mathsf{W}_{\imath_*\lambda}^\vee \otimes_{\bbC} (\sigma^\circ\otimes_{\cO,\imath}\bbC)^\vee, \cA\big)^J \nonumber
\\
 & \cong & \Hom_{\bG(F^+_\infty)\times J}\big(\mathsf{W}_{\imath_*\lambda}^\vee \otimes_{\bbC} (\sigma^\circ\otimes_{\cO,\imath}\bbC)^\vee, \cA\big)\nonumber\\
& \stackrel{\eqref{eq:aut:dec}}{\cong} & \bigoplus_\pi m(\pi) \Hom_{\bG(F^+_\infty)\times J}\big(\mathsf{W}_{\imath_*\lambda}^\vee \otimes_{\bbC} (\sigma^\circ\otimes_{\cO,\imath}\bbC)^\vee, \pi\big)\nonumber\\
 & \cong & \bigoplus_{\pi_\infty \cong  \mathsf{W}_{\imath_*\lambda}^\vee} m(\pi) \Hom_{J_p}\big((\sigma^\circ\otimes_{\cO,\imath}\bbC)^\vee, \pi_p\big) \otimes_{\bbC} (\pi^{\infty,p})^{J^p}.
 \label{lem:decompose}
\end{eqnarray}
In particular, this implies that $S_{\bG}(\bG(\cO_{F^+,p}), W_\lambda\otimes_{\cO}\sigma^\circ)\otimes_{\cO}\overline{E}$ is a semisimple admissible $\bG(\bbA_{F^+}^{\infty,p})$-representation.

\subsection{Galois representations associated to automorphic representations and Hecke algebras}
\label{subsec:Galois-Automorphic}

\subsubsection{}

We define $(\bbZ_+^2)_0^{I_p}$ to be the subset of $(\bbZ_+^2)^{I_p}$ consisting of all $\lambda = (\lambda_\kappa)_\kappa$ for which
$$\lambda_{\kappa,i} = -\lambda_{\kappa\circ c,3 - i}$$
for $i = 1,2$.  Note that the restriction map induces a bijection 
$$(\bbZ_+^2)_0^{I_p} \stackrel{\sim}{\longrightarrow} (\bbZ_+^2)^{\widetilde{I}_p}.$$

We use the following notation in the theorem below.  Throughout, we fix an isomorphism $\imath: \overline{E} \stackrel{\sim}{\longrightarrow} \bbC$, and recall that $\textnormal{rec}_{\overline{E}}$ denotes the Local Langlands correspondence over $\overline{E}$.  We define $|\det|^{-1/2}$ to be the $\overline{E}$-valued character whose composition with $\imath$ is the square root of $|\det|^{-1}$ which takes positive values.

\begin{theo}\label{auttogal}
Fix $\lambda \in (\bbZ_+^2)^{\widetilde{I}_p}$, and for every $v\in \Sigma_p^+$, let $\tau_v'$ denote a tame inertial type of $I_{F_v}$ which factors as in Definition \ref{inertialtype} and satisfies $(\tau'_v)^{\varphi^{-[F_v^+:\bbQ_p]}} \cong \tau_v'^\vee$.  {
Let $\sigma \defeq \bigotimes_{v\in \Sigma_p^+}\sigma(\tau_v')$ and let $\Xi$ be an irreducible $\bG(\bbA_{F^+}^{\infty,p})$-subrepresentation of $S_{\bG}(\bG(\cO_{F^+,p}), (W_\lambda\otimes_{\cO}E)\otimes_E\sigma^{\vee})\otimes_{E}\overline{E}$.
}  
Then there exists a cuspidal automorphic representation $\pi$ of $\bG(\bbA_{F^+})$ such that $\pi_v\cong \Xi_v\otimes_{\overline{E},\imath}\bbC$ for all finite places $v\notin \Sigma_p^+$, $\pi_\infty \cong \mathsf{W}_{\imath_*\lambda}^\vee$, and $\pi_p|_{\bG(\cO_{F^+,p})}$ contains {$\sigma\otimes_{E,\imath}\bbC$}.  Furthermore, there exists a unique continuous semisimple representation 
$$r_{\imath}(\pi):\Gamma_F \longrightarrow \bG\bL_2(\overline{E})$$
satisfying the following properties: 
\begin{enumerate}
\item We have an isomorphism
$$r_{\imath}(\pi)^c \cong r_{\imath}(\pi)^\vee \otimes \varepsilon^{-1}.$$
\item If $v$ is a finite place of $F^+$ which splits as $v = ww^c$ in $F$, then
$$\textnormal{WD}\big(r_{\imath}(\pi)|_{\Gamma_{F_w}}\big)^{\textnormal{F-ss}} \cong \textnormal{rec}_{\overline{E}}\big((\Xi_v \circ \iota_w^{-1})\otimes |\det|^{-1/2}\big).$$
\item If $v\not\in\Sigma_p^+$ is a finite place of $F^+$ which is inert in $F$, then
$$\textnormal{WD}\big(r_{\imath}(\pi)|_{\Gamma_{F_v}}\big)^{\textnormal{F-ss}} \cong \textnormal{rec}_{\overline{E}}\big(\BC_{F_v/F^+_v}(\Xi_v)\otimes |\det|^{-1/2}\big),$$
where $\BC_{F_v/F^+_v}$ denotes the stable local base change.  
\item Let $v\in \Sigma_p^+$.  Then $r_{\imath}(\pi)$ is potentially crystalline at $v$ \emph{(}viewed as a place of $F$\emph{)}, and we have
$$\textnormal{WD}(r_{\imath}(\pi)|_{\Gamma_{F_v}})|_{I_{F_v}} \cong \tau'_v.$$
If $\kappa\in I_p$ satisfies $v(\kappa) = v$, then
$$\textnormal{HT}_\kappa(r_{\imath}(\pi)|_{\Gamma_{F_v}}) = \{\lambda_{\kappa,1} + 1,~ \lambda_{\kappa,2}\}$$
\emph{(}where we view $\lambda$ as an element of $(\bbZ_+^2)_0^{I_p}$ via the bijection preceding the theorem\emph{)}.  In particular, $r_{\imath}(\pi)|_{\Gamma_{F_v}}$ is Hodge--Tate regular.  
\end{enumerate}
\end{theo}

\begin{proof}
Firstly, we note that the existence of the representation $\pi$ follows from Subsubsection \ref{sub:sub:autom:forms} (specifically, equation \eqref{lem:decompose}).  Additionally, the set of primes of $F$ which are split over a place of $F^+$ has Dirichlet density 1.  Therefore, if we have two semisimple continuous Galois representations satisfying (ii), they must be isomorphic.

Let $\bG^*$ denote the quasi-split unitary group in two variables over $F^+$, defined as in \cite[\S 1.9]{rogawski}.  {There exists a Jacquet--Langlands transfer} from $L$-packets on $\bG(\bbA_{F^+})$ to $L$-packets on $\bG^*(\bbA_{F^+})$, which induces isomorphisms at all finite places of the constituents of the $L$-packets.  {({In order to see this, we may appeal to any of the following methods}: (1) noting that $\bG^{\textnormal{der}} \cong \bS\bL_1(D)$ and $\bG^{*,\textnormal{der}} \cong \bS\bL_2$ (where $D$ denotes the quaternion algebra over $F^+$ which is ramified exactly at the infinite places of $F^+$), we proceed in a similar fashion as \cite[\S 7, p. 781]{labesselanglands}; (2) we may embed $\bG$ and $\bG^*$ into their respective similitude groups, which are isomorphic to $\bG\bL_1(D)\times^{{\bG_m}} \textnormal{Res}_{F/F^+}\bG_m$ and $\bG\bL_2 \times^{{\bG_m}} \textnormal{Res}_{F/F^+}\bG_m$, and apply the results of \cite{labesseschwermer} along with the classical Jacquet--Langlands correspondence between $\bG\bL_1(D)$ and $\bG\bL_2$;
{(3) use \cite[Thm. 1.7.1]{KMSW}, which nevertheless is conditional on forthcoming work of the cited authors.})}

Let $\BC_{F/F^+}$ denote the global stable base change map (cf. \cite[\S 11.5]{rogawski}), and put $\Pi \defeq \BC_{F/F^+}(\textnormal{JL}([\pi]))$, where $[\pi]$ denotes the $L$-packet containing $\pi$.  Then $\Pi$ is an automorphic representation of $\bG\bL_2(\bbA_F)$ which enjoys the following properties:
\begin{itemize}
\item $\Pi$ is conjugate self-dual;
\item $\Pi_\infty$ is  cohomological of weight $\imath_*\lambda$ (viewed as an element of $(\bbZ^2_+)_0^{\Hom(F,\bbC)}$);
\item If $w$ is a place of $F$ which is split over a place $v$ of $F^+$, then
$$\Pi_w \cong \BC_{F_w/F^+_v}(\pi_v) = \pi_v\circ \iota_w^{-1}$$
where $\BC_{F_w/F^+_v}$ denotes the local base change (cf. \cite[\S 2.4]{guerberoff}); 
\item If $v$ is a place of $F$ lying over an inert place $v$ of $F^+$, then
$$\Pi_v \cong \BC_{F_v/F^+_v}(\pi_v),$$
where $\BC_{F_v/F^+_v}$ denotes the local base change (described explicitly in \cite[Prop. 11.4.1]{rogawski}, and in further detail in \cite[Cor. 3.6 and Thm. 4.4]{blasco});
\item If $v\in \Sigma_p$, then we have an injection
$$\sigma'(\tau'_v)\otimes_{E,\imath}\bbC \longhookrightarrow \Pi_v|_{\bG\bL_2(\cO_{F_v})}.$$
Hence, if $\tau_v'$ is a principal series tame type, $\Pi_v$ is a principal series representation.
\end{itemize}

The construction of $r_{\imath}(\pi)$ now follows just as in \cite[Thm. 2.3]{guerberoff}, appealing to \cite[Thm. 11.5.1]{rogawski} instead of \cite[Cor. 5.3]{labesse} in order to control what happens above $p$.  All the properties listed follow from \cite[Thm. 0.1]{guerberoff}, \cite[Thm. 1.1]{caraiani:LGC1}, \cite[Thm. 1.1]{caraiani:LGC2}, and Theorem \ref{ILLC}.  
\end{proof}

\subsubsection{} 
\label{heckealggalois}

Fix a sufficiently small compact open subgroup $K = \prod_v K_v \subseteq \bG(\bbA_{F^+}^\infty)$, {and let $T$ denote a finite set of finite places of $F^+$,} which contains all {inert} places $v$ for which $K_v$ {is not hyperspecial and all split places $v$ for which $K_v \neq \bG(\cO_{F^+_v})$}.  We define the abstract Hecke algebra $\bbT^T$ to be the commutative polynomial $\cO$-algebra generated by formal variables $T^{(i)}_w$ for $i = 1,2,$ and $w$ a place of $F$ split over a place of $F^+$ such that $w|_{F^+}\not\in T$.

Fix $\lambda\in (\bbZ^2_+)^{\widetilde{I}_p}$ and let $\tau' \defeq \{\tau_v'\}_{v\in \Sigma_p}$ {and $\sigma^{\vee,\circ}$ {denote a $\bG(\cO_{F^+,p})$-stable $\cO$-lattice in the dual of $\sigma = \bigotimes_{v\in \Sigma_p^+} \sigma(\tau_v')$}}.  Given $K$ as above, {with $K_v \subseteq \bG(\cO_{F^+_v})$ for all $v \in \Sigma_p^+$}, the Hecke operator $T_w^{(i)}$ acts on the space $S_{\bG}(K, W_\lambda\otimes_{\cO}\sigma^{{\vee,\circ}})$ via the characteristic function of double coset
$$K_v \iota_w^{-1}\begin{pmatrix} \varpi_{w} 1_i &  \cr  & 1_{2-i} \end{pmatrix}K_v \cdot K^v$$
(here $\varpi_w$ denotes a choice of an uniformizer of $F_w$, and $v = w|_{F^+}$).  The image of $\bbT^T$ in $\End_{\cO}(S_\bG(K,W_\lambda\otimes_{\cO}\sigma^{{\vee,\circ}}))$ will be denoted $\bbT^T_{\lambda,\tau'}(K)$.  The algebra $\bbT^T_{\lambda,\tau'}(K)$ is reduced, finite free over $\cO$, and thus a semi-local ring.  Furthermore, note that we have $T^{(i)}_{w^c}=T^{(2-i)}_w(T^{(2)}_w)^{-1}$ in $\bbT^T_{\lambda,\tau'}(K)$.

Recall that we have
\begin{equation}
\label{aut:dec:spc}
S_{\bG}(K, W_\lambda\otimes_{\cO}\sigma^{{\vee,\circ}})\otimes_{\cO}\overline{E} \cong \bigoplus_\Xi M_\Xi \otimes \Xi^{K^p},
\end{equation}
where the direct sum runs over all irreducible constituents $\Xi$ of $S_{\bG}(\bG(\cO_{F^+,p}), W_\lambda\otimes_{\cO}\sigma^{{\vee,\circ}})\otimes_{\cO}\overline{E}$ for which $\Xi^{K^p}\neq 0$, and where $M_\Xi$ is a multiplicity space.  The Hecke algebra $\bbT^T_{\lambda,\tau'}(K)$ acts on each $\Xi^{K^p}$ by scalars, and we obtain a homomorphism
$$\lambda_\Xi:\bbT^T_{\lambda,\tau'}(K) \longrightarrow \overline{E}.$$
The ideal $\ker(\lambda_\Xi)$ is a minimal prime ideal, and every minimal prime of $\bbT^T_{\lambda,\tau'}(K)$ arises in this way.

Fix a maximal ideal $\fm \subseteq \bbT^T_{\lambda,\tau'}(K)$.  Then we have
$$S_{\bG}(K,W_\lambda\otimes_{\cO}\sigma^{{\vee,\circ}})_{\fm} \otimes_{\cO}\overline{E} \neq 0,$$
and this localization annihilates all the direct summands of \eqref{aut:dec:spc} for which $\ker(\lambda_\Xi) \not\subseteq \fm$.  Let $\fp\subseteq \fm$ denote a minimal prime ideal, corresponding to an irreducible constituent $\Xi$ of \eqref{aut:dec:spc}.  We choose an invariant lattice in $r_{\imath}(\pi)$ (for $\pi$ associated to $\Xi$ as in Theorem \ref{auttogal}), reduce modulo $p$, and semisimplify to obtain a representation $\rbar_\fm$; by the density argument of Theorem \ref{auttogal} this is independent of the choice of $\fp$ and $\Xi$.

\begin{theo}
\label{galrepheckealg}
Fix $\lambda\in (\bbZ^2_+)^{\widetilde{I}_p}$ and let $\tau' = \{\tau_v'\}_{v\in \Sigma_p}$ be as in Theorem \ref{auttogal}.  Suppose that $\fm$ is a maximal ideal of $\bbT^T_{\lambda,\tau'}(K)$ such that that the residue field $\bbT^T_{\lambda,\tau'}(K)/\fm$ is equal to $\bbF$.  Suppose further that $\rbar_\fm$ is absolutely irreducible.  Then $\rbar_\fm$ has an extension to a continuous homomorphism
$$\rbar_\fm:\Gamma_{F^+} \longrightarrow \cG_2(\bbF).$$
Choose such an extension.  There exists a continuous lifting
$$r_\fm:\Gamma_{F^+} \longrightarrow \cG_2\big(\bbT^T_{\lambda,\tau'}(K)_{\fm}\big)$$
satisfying the following properties.  Note that properties \emph{(i)} and \emph{(iii)} characterize ${\BC'(r_\fm)}$ up to isomorphism.
\begin{enumerate}
\item The representation $r_\fm$ is unramified at all but finitely many places.
\item We have $\nu \circ r_\fm = \varepsilon^{-1}$.
\item If $v\not\in T$ is a finite place of $F^+$ which splits as $v = ww^c$ in $F$, then $r_\fm$ is unramified at $w$ and $\BC'(r_{\fm})(\textnormal{Frob}_w)$ has characteristic polynomial
$$X^2 - T^{(1)}_wX + \bN(w)T^{(2)}_w.$$
\item If $v\not\in \Sigma_p^+$ is an inert place, then $r_\fm$ is unramified at $v$.
\item Given $v\in \Sigma_p^+$ and a homomorphism $x:\bbT^T_{\lambda,\tau'}(K)_{\fm} \longrightarrow \overline{E}$, the representation $x\circ r_{\fm}|_{\Gamma_{F^+_v}}$ is potentially crystalline, and we have
$$\textnormal{WD}\big(x\circ r_{\fm}|_{\Gamma_{F^+_v}}\big)|_{I_{F_v^+}} \cong \tau'_v \oplus \mathbf{1}_{I_{F_v^+}}.$$
If $\kappa\in I_p^+$ satisfies $v(\kappa) = v$, then 
$$\textnormal{HT}_\kappa\big(\BC'(x\circ r_{\fm})|_{\Gamma_{F_v}}\big) = \{\lambda_{\kappa,1} + 1,~ \lambda_{\kappa,2}\}.$$
\end{enumerate}

\end{theo}

\begin{proof}
This follows exactly as in \cite[Prop. 3.4.4]{CHT}, using Theorem \ref{auttogal}.  The fact that $\nu \circ r_\fm = \varepsilon^{-1}$ in (ii) (instead of $\varepsilon^{-1}\delta_{F/F^+}^{\mu_\fm}$) follows from the main result of \cite{BC11}.  
\end{proof}

\subsubsection{}

We recall one more well known lemma on the space of algebraic automorphic forms.

\begin{lem}
\label{lem:non-zero}
Let $K = \prod_v K_v \subseteq \bG(\bbA_{F^+}^\infty)$ be a sufficiently small compact open subgroup as above, and let $W$ be a finite, $p$-torsion free $\cO$-module endowed with an action of $\bG(\cO_{F^+,p})$.  Fix a maximal ideal $\fm$ of $\bbT^T$.  Then
$$S_\bG(K,W\otimes_{\cO}\bbF)_{\fm} \neq 0 \Longleftrightarrow S_\bG(K,W\otimes_{\cO}E)_{\fm} \neq 0.$$
\end{lem}

\begin{proof}
This is standard (see, for example, the proof of \cite[Lem. 3.4.1]{CHT}).  More precisely, the fact that $K$ is sufficiently small gives the isomorphism \eqref{cmp:ext:sc}, and implies $S_\bG(K,W)_{\fm}$ is $p$-torsion free.  We therefore obtain
$$S_\bG(K,W\otimes_{\cO}\bbF)_{\fm} \cong S_\bG(K,W)_{\fm}\otimes_{\cO}\bbF \neq 0 \Longleftrightarrow S_\bG(K,W\otimes_{\cO}E)_{\fm} \cong S_\bG(K,W)_{\fm}\otimes_{\cO}E \neq 0.$$
\end{proof}

\subsection{Weight elimination}
\label{sub:WE}

\subsubsection{}

\begin{df}
A \emph{Serre weight for $\bG$} is an isomorphism class of smooth, absolutely irreducible representations of $\prod_{v\in \Sigma_p^+} \bG(\bbF_v^+)$ over $\bbF$, inflated to $\bG(\cO_{F^+,p})$.  If $v\in \Sigma_p^+$, a \emph{Serre weight at $v$} is an isomorphism class of smooth, absolutely irreducible representations of $\bG(\bbF_v^+)$ over $\bbF$, inflated to $\bG(\cO_{F^+_v})$.  
\end{df}

Note that any Serre weight $V$ for $\bG$ is of the form $V\cong \bigotimes_{v\in\Sigma_p^+} V_v$ where $V_v$ are Serre weights at $v$.

\begin{df}
\label{maxidealgalois}
Let $\rbar:\Gamma_{F^+} \longrightarrow \cG_2(\bbF)$ be a continuous representation such that 
\begin{itemize}
\item $\nu\circ\rbar = \overline{\varepsilon}^{-1}$;
\item $\rbar^{-1}(\bG\bL_2(\bbF)\times\bG_m(\bbF)) = \Gamma_F$;
\item $\BC'(\rbar):\Gamma_F \longrightarrow \bG\bL_2(\bbF)$ is absolutely irreducible.
\end{itemize}
Let $K = \prod_v K_v\subseteq \bG(\bbA_{F^+}^\infty)$ be a compact open subgroup, $T$ a finite set of places as in Subsubsection \ref{heckealggalois}, and suppose $\rbar$ is unramified at all finite places $v$ of $F^+$ which split in $F$ and for which $v\not\in T$.  We define a maximal ideal $\fm_{\rbar}$ of $\bbT^T$ by
$$\fm_{\rbar} \defeq \left\langle \varpi,~ T^{(1)}_w - \textnormal{tr}\big(\BC'(\rbar)(\textnormal{Frob}_w)\big),~T^{(2)}_w - \bN(w)^{-1}\det\big(\BC'(\rbar)(\textnormal{Frob}_w)\big)\right\rangle_{w|_{F^+}\not\in T}$$
where $w|_{F^+} = v\not\in T$ splits as $v = ww^c$ in $F$.  
\end{df}

\begin{df}
\label{definition modularity}
Let $\rbar$ be as in Definition \ref{maxidealgalois}, and let $V$ be a Serre weight for $\bG$.  
{
Let $K = \prod_{v} K_v \subseteq \bG(\bbA^{\infty}_{F^+})$ be a sufficiently small compact open subgroup with $K_v$ hyperspecial for $v$ inert in $F$ and $K_v = \bG(\cO_{F^+_v})$ for $v\in \Sigma_p^+$, and let $T$ be a finite subset as in Subsubsection \ref{heckealggalois} such that $\rbar$ is unramified at each split place not in $T$.
We say that $\rbar$ is \emph{modular of weight $V$ and level $K$} (or that $V$ is a \emph{Serre weight of $\rbar$ at level $K$}) if 
$$S_{\bG}(K,V^\vee)_{\mathfrak{m}_{\rbar}} \neq 0.$$
We say that $\rbar$ is \emph{modular of weight $V$} (or that $V$ is a \emph{Serre weight of $\rbar$}) if $\rbar$ is modular of weight $V$ and level $K$, for some sufficiently small compact open subgroup $K \subseteq \bG(\bbA^{\infty}_{F^+})$ as above.
}
We denote by $\textnormal{W}_{\textnormal{mod}}(\rbar)$ for the set of all Serre weights of $\rbar$.  We say that $\rbar$ is \emph{modular} if $\textnormal{W}_{\textnormal{mod}}(\rbar) \neq \emptyset$.  
\end{df}

\subsubsection{} 

Fix $\rbar$ as in Definition \ref{maxidealgalois}, and for $v\in \Sigma^+_p$ define $\rhobar_v \defeq \rbar|_{\Gamma_{F^+_v}}$.  By Proposition \ref{prop:weights:type}, we have a set of conjectural Serre weights $\textnormal{W}^?(\rhobar_v)$ at $v$ for every $v\in \Sigma^+_p$.  (Here we use the isomorphism ${}^C\bU_2 \cong \cG_2$ of Subsection \ref{isomsect}.  Moreover, the condition $\nu \circ \rbar = \overline{\varepsilon}^{-1}$ implies that the $\tld{\bU}_2(\bbF_v^+)$-representations appearing in $\textnormal{W}^?(\rhobar_v)$ descend to $\bU_2(\bbF_v^+) \cong \bG(\bbF_v^+)${, see for instance Proposition \ref{prop:compatibilityBC}}.)  Thus, we can attach to $\rbar$ a set $\textnormal{W}^?(\rbar)$ of predicted Serre weights for $\bG$ defined as
$$\textnormal{W}^?(\rbar) \defeq \left\{\bigotimes_{v\in\Sigma_p^+}V_v : V_v\in \textnormal{W}^?(\rhobar_v)~ \textnormal{for all}~ v\in\Sigma_p^+\right\}.$$
In a similar fashion we define the set $\textnormal{W}^?(\BC(\rbar))$ of conjectural weights attached to $\BC(\rbar)$.  (Note that, under the isomorphism in Subsection \ref{isomsect}, we have $\BC(\rbar) \cong \BC'(\rbar)\otimes \overline{\varepsilon}$.)

If $\sigma = \bigotimes_{v\in \Sigma_p^+} \sigma_v$ is a tame $\bG(\cO_{F^+,p})$-type, we define the base change of $\sigma$ as 
$$\BC(\sigma) \defeq \bigotimes_{v\in\Sigma_p^+}\BC_v(\sigma_v),$$
where $\BC_v$ denotes the local base change of types.

\begin{theo}
\label{thm:WE}
Let $\rbar:\Gamma_{F^+} \longrightarrow \cG_2(\bbF)$ be a continuous representation such that 
\begin{itemize}
\item $\nu\circ\rbar = \overline{\varepsilon}^{-1}$;
\item $\rbar^{-1}(\bG\bL_2(\bbF)\times\bG_m(\bbF)) = \Gamma_F$;
\item $\BC'(\rbar):\Gamma_F \longrightarrow \bG\bL_2(\bbF)$ is absolutely irreducible;
\item $\rbar|_{\Gamma_{F^+_v}}$ is tamely ramified and $3$-generic for every $v\in \Sigma_p^+$.
\end{itemize}
Then
$$\textnormal{W}_{\textnormal{mod}}(\rbar) \subseteq \textnormal{W}^?(\rbar).$$
\end{theo}

\begin{proof}
{Assume $\textnormal{W}_{\textnormal{mod}}(\rbar)\neq \emptyset$, otherwise there is nothing to prove.}  Let $V\in \textnormal{W}_{\textnormal{mod}}(\rbar)$, and assume by contradiction that $V\not\in \textnormal{W}^?(\rbar)$.  By Proposition \ref{prop:weights:type} and Lemma \ref{emptyint}, there exists a tame $\bU_2(\cO_{F^+,p})$-type $\sigma = \bigotimes_{v\in \Sigma_p^+}\sigma_v$ such that
\begin{enumerate}
\item $V\in \JH(\overline{\sigma})$; \label{item:we:1}
\item $\JH(\overline{\sigma}) \cap \textnormal{W}^?(\rbar) = \emptyset$. \label{item:we:2}
\end{enumerate}
(Note that if $\rhobar_v$ is $n$-generic, so is $\beta( V_{\phi}(\rhobar_v))$.)
We define $\tau_v'$ to be the tame principal series type such that $\sigma_v \cong \sigma(\tau_v')$ (so that, in particular, $(\tau_v')^{\varphi^{-[\bbF_v^+:\bbF_p]}} \cong \tau_v'^\vee$).

By definition of modularity, there exists a sufficiently small compact open subgroup $K = \prod_v K_v$ such that {$K_v$ is hyperspecial for $v$ inert in $F$ and $K_v = \bG(\cO_{F^+_v})$ for $v\in \Sigma_p^+$}, and a finite set of places $T$ such that 
$$S_{\bG}(K,V^\vee)_{\fm_{\rbar}} \neq 0.$$
Since $K$ is sufficiently small, the functor of algebraic automorphic forms is exact, so item \ref{item:we:1} implies
{$$S_{\bG}(K,\overline{\sigma^{{\vee,\circ}}})_{\fm_{\rbar}} \neq 0,$$}
and Lemma \ref{lem:non-zero} gives
$$S_{\bG}(K, \sigma^{{\vee,\circ}}\otimes_{\cO}E)_{\fm_{\rbar}} \neq 0.$$

By the discussion preceding Theorem \ref{galrepheckealg} and upon choosing an isomorphism $\imath: \overline{E} \stackrel{\sim}{\longrightarrow} \bbC$, there exists a cuspidal automorphic representation $\pi$ of $\bG(\bbA_{F^+})$ such that
\begin{itemize}
\item $\pi_\infty$ is the trivial representation of $\bG(F^+_\infty)$; 
\item $\textnormal{Hom}_{K_p}(\sigma\otimes_{E,\imath}\bbC,\pi_p) \neq 0$; 
\item for each place $v$ of $F^+$ which is split in $F$ and not contained in $T$, the local constituent $\pi_v$ is an unramified principal series with Satake parameters determined by a minimal prime of $\bbT^T_{0,\tau'}(K)_{\fm_{\rbar}}$ via $\imath$.
\end{itemize}
As in the proof of Theorem \ref{auttogal}, we obtain a continuous representation
$$r_{\imath}(\pi):\Gamma_F \longrightarrow \bG\bL_2(\overline{E})$$
such that
\begin{itemize}
\item $r_{\imath}(\pi)$ lifts $\BC'(\rbar)$;
\item for each $v\in \Sigma_p$, $r_{\imath}(\pi)|_{\Gamma_{F_v}}$ is potentially crystalline, and 
$$\textnormal{WD}(r_{\imath}(\pi)|_{\Gamma_{F_v}})|_{I_{F_v}} \cong \tau_v';$$
\item for each $\kappa\in I_p$, we have 
$$\textnormal{HT}_\kappa(r_{\imath}(\pi)|_{\Gamma_{F_{v(\kappa)}}}) = \{1,~ 0\}.$$
\end{itemize}
Consequently, the representation $r_{\imath}(\pi)\otimes\varepsilon$ has the following properties:
\begin{itemize}
\item $r_{\imath}(\pi)\otimes\varepsilon$ lifts $\BC'(\rbar)\otimes\overline{\varepsilon} \cong \BC(\rbar)$;
\item for each $v\in \Sigma_p$, $(r_{\imath}(\pi)\otimes\varepsilon)|_{\Gamma_{F_v}}$ is potentially crystalline, and 
$$\textnormal{WD}\big((r_{\imath}(\pi)\otimes\varepsilon)|_{\Gamma_{F_v}}\big)|_{I_{F_v}} \cong \tau_v';$$
\item for each $\kappa\in I_p$, we have 
$$\textnormal{HT}_\kappa\big((r_{\imath}(\pi)\otimes\varepsilon)|_{\Gamma_{F_{v(\kappa)}}}\big) = \{0,~ -1\}.$$
\end{itemize}

By the above, we see that $\BC(\rhobar_v)$ has a potentially Barsotti--Tate lift of type $\tau_v'$ for every $v\in \Sigma_p$, namely $(r_{\imath}(\pi)\otimes\varepsilon)|_{\Gamma_{F_v}}$ (with notation as in \cite[Def. 2.3]{gee-hmf}).  Therefore, Proposition 3.12 of \emph{op. cit.} implies that
$$\JH\big(\overline{\BC(\sigma(\tau_v'))}\big) \cap \textnormal{W}^?\big(\BC(\rhobar_v)\big) \neq \emptyset$$
for all $v\in \Sigma_p$.  By Propositions \ref{prop:double:weights}, \ref{prop:weights:type} and \ref{prop:compatibilityBC} we obtain
$$\JH\big(\overline{\sigma(\tau_v')}\big) \cap \textnormal{W}^?(\rhobar_v) \neq \emptyset$$
for all $v\in \Sigma_p^+$.  However, this contradicts item \ref{item:we:2}, which concludes the proof.  
\end{proof}

\section{Global Applications II}
\label{glob2}

In this section we use patching techniques to prove the existence of Serre weights for $L$-parameters, using the results on local deformation theory obtained in Section \ref{sec:Loc:Def}. We first adapt the patching construction of \cite{CEGGPS} to the case of unitary groups which are not split at places above $p$, and state the necessary properties in Proposition \ref{patchingprops}.  This allows us to deduce the main results on weight existence in Theorem \ref{thm:WExt} and automorphy lifting in Theorem \ref{thm:mod:lift-body}.

\subsection{Setup}
\label{sub:setup}
\subsubsection{}

\label{subsub:setup}
Suppose that $\rbar: \Gamma_{F^+}\longrightarrow \cG_2(\bbF)$ is a fixed Galois representation such that 
\begin{itemize}
\item $\rbar^{-1}(\bG\bL_2(\bbF)\times\bG_m(\bbF)) = \Gamma_{F}$;
\item $\nu\circ \rbar = \overline{\varepsilon}^{-1}$;
\item $\rbar$ is modular;
\item $\rbar$ is unramified outside $\Sigma_p^+$;
{\item $\rbar|_{\Gamma_{F^+_v}}$ is tamely ramified and $4$-generic for all $v\in \Sigma_p^+$;}
\item $\overline{F}^{\ker(\textnormal{ad}^0(\rbar))}$ does not contain $F(\zeta_p)$; and
\item $\BC'(\rbar)(\Gamma_F) \supseteq \bG\bL_2(\bbF')$ for some subfield $\bbF' \subseteq \bbF$ with $|\bbF'| > 6$.
\end{itemize}
The last condition implies that $\BC'(\rbar)(\Gamma_{F(\zeta_p)})$ is adequate (cf. \cite[Prop. 6.5]{BLGG13}), and that $\BC'(\rbar)$ is absolutely irreducible.  Furthermore, the argument in \cite{CEGGPS} shows that this condition also guarantees the existence of a place $v_1$ of $F^+$ such that
\begin{enumerate}[(a)]
\item
\label{it:aux:1} 
$v_1$ splits as $\widetilde{v}_1\widetilde{v}_1^c$ in $F$;
\item 
\label{it:aux:2}
$v_1$ does not split completely in $F(\zeta_p)$;
\item 
\label{it:aux:3}
$\BC'(\rbar)(\textnormal{Frob}_{\widetilde{v}_1})$ has distinct $\bbF$-rational eigenvalues, whose ratio is not equal to $\bN(v_1)^{\pm 1}$.  
\end{enumerate}

\subsubsection{}
Let $\lambda \in (\bbZ^2_+)^{\widetilde{I}_p}$ and for every $v\in \Sigma_{{p}}^+$, let $\tau_v'$ denote a tame inertial type which satisfies $(\tau'_v)^{\varphi^{-[\bbF^+_v:\bbF_p]}} \cong \tau_v'^\vee$.  Set $T\defeq \Sigma_p^+ \cup \{v_1\}$ and $\tT \defeq \Sigma_p \cup \{\widetilde{v}_1\}$.  We consider a slight generalization of the global deformation problems of \cite[\S 2.3]{CHT}:
\begin{eqnarray*} 
\cS & \defeq & \left(F/F^+, T, \tT,  \cO, \rbar, \varepsilon^{-1}, \{\widetilde{R}^{\Box}_{v}\}_{v \in \Sigma_p^+} \cup \{\widetilde{R}^{\Box}_{v_1}\} \right)\\
\cS_{\lambda,\tau'} & \defeq & \left(F/F^+, T, \tT, \cO, \rbar, \varepsilon^{-1}, \{R^{\Box,\lambda_v,\tau_v'}_{v}\}_{v \in \Sigma_p^+} \cup \{ \widetilde{R}^{\Box}_{v_1}\} \right).
\end{eqnarray*}
The difference here is that we allow places in $T$ to be inert in $F$.  In this notation, $\widetilde{R}^{\Box}_{v}$ denotes the maximal reduced and $p$-torsion free quotient of the universal framed deformation ring {parametrizing lifts $\rho$ of} $\rbar|_{\Gamma_{F_v^+}}$ {such that $\nu\circ \rho = \varepsilon^{-1}$}.  Further, the ring $R^{\Box,\lambda_v,\tau_v'}_{v}$ denotes the unique quotient of $\widetilde{R}^{\Box}_{v}$ with the property that if $B$ is a finite local $E$-algebra, then $x: \widetilde{R}^{\Box}_v \longrightarrow B$ factors through $R^{\Box,\lambda_v,\tau_v'}_{v}$ if and only if the corresponding representation $r_x: \Gamma_{F_v^+} \longrightarrow \cG_2(B)$ is potentially crystalline, and satisfies $\nu\circ r_x = \varepsilon^{-1}$, $\textnormal{HT}_\kappa(\BC'(r_x)) = \{\lambda_{\kappa,1} + 1,~ \lambda_{\kappa,2}\}$, and $\textnormal{WD}({\BC'}(r_x))|_{I_{F_v}} \cong \tau_v'$.  (Again, the existence of such a quotient follows from \cite[\S\S 3.2 -- 3.3]{bellovin-gee}.)  In particular, if $\lambda = 0$, then by applying the isomorphism of Subsection \ref{isomsect}, we obtain an isomorphism of deformation rings $R^{\Box,0,\tau_v'}_{v} \cong R^{\tau_v'}_{\rhobar_v}$, where the second ring is as in Subsubsection \ref{localgalrepdefrings}.

We note also that $\widetilde{R}^{\square}_{v_1}$ is formally smooth over $\cO$ of relative dimension 4, and all of the corresponding Galois representations lifting $\rbar|_{\Gamma_{F^+_{v_1}}}$ are unramified (see \cite[Lem. 2.5]{CEGGPS}).

We let $R^{\univ}_{\cS}$ be the complete local Noetherian $\cO$-algebra representing the functor of deformations of $\rbar$ of type $\cS$, and let $R_{\cS}^{\square_T}$ denote the $\cO$-algebra representing $T$-framed deformations of $\rbar$ of type $\cS$.  (We define a framing at places in $\Sigma_p^+$ just as in \cite[Def. 2.2.1]{CHT}, i.e., as an element of $1_2 + \textnormal{Mat}_{2\times 2}(\fm_R) \subseteq \ker(\cG_2(R) \longtwoheadrightarrow \cG_2(\bbF))$.)  We have similar notation for the deformation problem $\cS_{\lambda,\tau'}$.

\subsubsection{}
Set 
$$\cT \defeq \cO[\![X_{v,i,j}:v\in T, 1 \leq i,j \leq 2]\!].$$  
Choose a lift $r_{\cS}^{\textnormal{univ}}$ representing the universal deformation of type $\cS$, and form the tuple
$$\left(r^{\textnormal{univ}}_{\cS},~ \left\{ \begin{pmatrix}1 + X_{v,1,1} & X_{v,1,2} \\ X_{v,2,1} & 1 + X_{v,2,2}\end{pmatrix}\right\}_{v\in T}\right).$$
This gives a representative of the universal $T$-framed deformation of type $\cS$, and we obtain
$$R_{\cS}^{\square_T} \cong R_{\cS}^{\textnormal{univ}}\widehat{\otimes}_{\cO}\cT \cong R_{\cS}^{\textnormal{univ}}[\![X_{v,i,j}]\!]$$
(and similarly for $\cS_{\lambda,\tau'}$).

We set
\begin{eqnarray*} 
R^{\loc}  & \defeq & \left( \widehat{\bigotimes}_{v \in \Sigma_p^+} \widetilde{R}^{\Box}_{v} \right) \widehat{\otimes} \widetilde{R}^{\Box}_{v_1}, \\
R^{\loc}_{\lambda,\tau'} & \defeq & \left( \widehat{\bigotimes}_{v \in\Sigma_p^+} R^{\Box,\lambda_v,\tau_v'}_{v}\right) \widehat{\otimes}  \widetilde{R}^{\Box}_{v_1},
\end{eqnarray*}
where all completed tensor products are taken over $\cO$.

\begin{prop}
\label{prop:dim:BG}
{Assume that $R^{\Box,\lambda_v,\tau_v'}_{v}$ has a non-zero $\cO$-point for all $v\in\Sigma_p^+$}.  Then $R^{\textnormal{loc}}_{\lambda,\tau'}[1/p]$ is {regular.  If moreover $(\lambda_v,\tau'_v)=(0,\tau'_v)$ with $\tau'_v$ being $2$-generic and $\rbar|_{\Gamma_{F^+_v}}$ being $1$-generic and semisimple, then $R^{\textnormal{loc}}_{0,\tau'}[1/p]$ is formally smooth, and $R^{\loc}_{0,\tau'}$ is equidimensional of dimension $1 + 4|T| + [F^+:\bbQ]$.}
\end{prop}

\begin{proof}
{The fact that $R^{\textnormal{loc}}_{\lambda,\tau'}[1/p]$ is regular follows from \cite[Thm. 3.3.7]{bellovin-gee}, formal smoothness of $\widetilde{R}^{\square}_{v_1}$, and \cite[Cor. A.2]{CEGGPS}.  When $\lambda = 0$, formal smoothness of $R^{\textnormal{loc}}_{\lambda,\tau'}[1/p]$ follows from  the results of Subsubsection \ref{subsubsec:integrality}, formal smoothness of $\widetilde{R}^{\square}_{v_1}$, and \cite[Lem. (3.4.12)]{kisin-annals}.  The claim about dimensions then follows from \cite[Thm. 3.3.7]{bellovin-gee}, the fact that $\widetilde{R}^{\square}_{v_1}$ is of relative dimension 4 over $\cO$, and \cite[Lem. 3.3]{BLGHT}.  }
\end{proof}

\subsubsection{}
\label{subsec:level:lowering}
We now relate the above constructions to spaces of automorphic forms.  We fix a compact open subgroup $K_m = \prod_{v}K_{m,v} \subseteq \bG(\bbA_{F^+}^\infty)$ satisfying the following properties:
\begin{itemize}
\item if $v$ is a place of $F^+$ which is inert in $F$ and $v\not\in \Sigma_p^+$, then $K_{m,v}$ {is {a hyperspecial} subgroup of $\bG(F^+_v)$};
\item if $v$ is a place of $F^+$ which is split in $F$ and $v \neq v_1$, then $K_{m,v} = \bG(\cO_{F^+_v})$;
\item if $v\in \Sigma_p^+$, then $K_{m,v} = \ker(\bG(\cO_{F^+_v}) \longtwoheadrightarrow \bG(\cO_{F^+_v}/\varpi_{v}^m))$;
\item if $v = v_1$ and $\widetilde{v}_1$ is the fixed place of $F$ above $v_1$, then $K_{m,v_1}$ is the preimage under $\iota_{\widetilde{v}_1}$ of the upper-triangular Iwahori subgroup of $\bG\bL_2(\cO_{F_{\widetilde{v}_1}})$.
\end{itemize}
These assumptions guarantee that $K_m$ is sufficiently small.  We define $K \defeq K_0$.

{Before proceeding, we will need the following level-lowering result.}

{\begin{prop}
\label{prop:level:lowering}
Suppose $\rbar$ satisfies the hypotheses from the beginning of Subsection \ref{sub:setup}, so that in particular $\rbar$ is modular, unramified outside $p$, and $\rbar|_{\Gamma_{F^+_v}}$ is tamely ramified and $4$-generic for all $v\in \Sigma_p^+$.  Then $\rbar$ is modular of level $K_0$.  
\end{prop}}

\begin{proof}

Suppose $\rbar$ is modular of weight $V = \bigotimes_{v\in \Sigma_p^+} V_v$.  Thus, there exists a finite set of finite places $T'$ and a sufficiently small level $K' = \prod_v K'_v \subseteq \bG(\bbA_{F^+}^\infty)$ (with $K'_v$  {hyperspecial for all inert $v$ and $K'_v = \bG(\cO_{F^+_v})$ for $v\in \Sigma_p^+$ and for split $v \not\in T'$}) satisfying
$$S_{\bG}(K',V^\vee)_{\fm_{\rbar}} \neq 0.$$
For each $v\in \Sigma_p^+$, we choose a principal series tame type $\tau_v'$ such that $(\tau'_v)^{\varphi^{-[\bbF^+_v:\bbF_p]}} \cong \tau_v'^\vee$ and such that $V \in \JH(\overline{\sigma})$, where $\sigma := \bigotimes_{v\in \Sigma_p^+} \sigma(\tau'_v)$.  By the genericity hypotheses and Theorem \ref{thm:WE}, $V_v$ is 3-deep for every $v \in \Sigma_p^+$, and consequently $\tau_v'$ is 2-generic.  Since $K'$ is sufficiently small, Lemma \ref{lem:non-zero} implies
$$S_{\bG}(K',\sigma^\vee)_{\fm_{\rbar}} \neq 0.$$
As in the proof of Theorem \ref{thm:WE}, there exists an automorphic representation $\pi$ of $\bG(\bbA_{F^+})$, and (after choosing an isomorphism $\imath: \overline{E} \stackrel{\sim}{\longrightarrow} \bbC$) an associated Galois representation 
$$r_\imath(\pi): \Gamma_F \longrightarrow \bG\bL_2(\overline{E})$$
which lifts $\BC'(\rbar) \otimes_{\bbF} \overline{\bbF}_p$.

Let $\Pi$ denote the automorphic representation of $\bG\bL_2(\bbA_F)$ obtained from $\pi$ by base change (as in the proof of Theorem \ref{auttogal}), and let $\Sigma_{\textnormal{ram}}^+$ denote the set of prime-to-$p$ places of $F^+$ at which $\pi$ is ramified (note that every place of $\Sigma_{\textnormal{ram}}^+$ is split in $F$, and if $\Pi$ is ramified at some place $w$, then $w|_{F^+} \in \Sigma_{\textnormal{ram}}^+$).  Adjusting the place $v_1$ if necessary, we may assume $v_1 \not\in \Sigma_{\textnormal{ram}}^+$.  We choose a totally real extension $L^+$ of $F^+$ such that the following conditions hold:
\begin{itemize}
\item $4$ divides $[L^+:\bbQ]$; 
\item $L^+/F^+$ is Galois and solvable;
\item $L \defeq L^+F$ is linearly disjoint from $\overline{F}^{\ker(\rbar)}(\zeta_p)$ over $F$;
\item $p$ is unramified in $L$;
\item $v_1$ splits completely in $L$;
\item if $w$ is a place of $L^+$ lying over a place in $\Sigma_{\textnormal{ram}}^+$, then $\bN(w) \equiv 1~(\textnormal{mod}~p)$;
\item if $\Pi_L$ denotes the base change of $\Pi$ to an automorphic representations of $\bG\bL_2(\bbA_L)$ and $w$ is a place of $L$ lying over a place in $\Sigma_{\textnormal{ram}}^+$, then $\Pi_{L,w}^{\textnormal{Iw}_w} \neq 0$, where $\textnormal{Iw}_w \subseteq \bG\bL_2(\cO_{L_w})$ denotes the standard upper-triangular Iwahori subgroup.
\end{itemize}
{(Note that $L/L^+$ is everywhere unramified.)}
We use the following notation in what follows: if $\widetilde{L}/\widetilde{F}$ is a finite extension of number fields and $\widetilde{T}$ is a finite set of finite places of $\widetilde{F}$, we let $\BC_{\widetilde{L}/\widetilde{F}}(\widetilde{T})$ (or $\BC(\widetilde{T})$ when the context is clear) denote the set of places of $\widetilde{L}$ lying above $\widetilde{T}$.

Let $\pi_{L^+}$ denote a descent of $\Pi_L$ to an automorphic representation of $\bG(\bbA_{L^+})$.  {We analyze the local behavior of $\pi_{L^+}$:}
\begin{enumerate}[(a)]
\item If $w$ is a place of $L^+$ which splits as $\widetilde{w}\widetilde{w}^c$ in $L$, then $\pi_{L^+,w} \cong \Pi_{L, \widetilde{w}} \circ \iota_{\widetilde{w}}$, where $\iota_{\widetilde{w}}$ is an isomorphism $\bG(L^+_w) \stackrel{\sim}{\longrightarrow} \bG\bL_2(L_{\widetilde{w}})$ which identifies groups of integral points.  In particular, if $\Pi_{L, \widetilde{w}}$ is unramified, so is $\pi_{L^+,w}$.  
\item 
\label{it:switch}
{Next, we consider the situation above $p$.  Define 
\begin{eqnarray*}
\sigma_{L^+} & \defeq & \bigotimes_{\substack{w \in \BC_{L^+/F^+}(\Sigma_p^+) \\ v = w|_{F^+}}} \sigma(\tau'_v|_{I_{L^+_w}}), \\
\sigma'_L \defeq \BC(\sigma_{L^+}) & = & \bigotimes_{\substack{w \in \BC_{L^+/F^+}(\Sigma_p^+) \\ v = w|_{F^+}}} \sigma'(\tau'_v|_{I_{L^+_w}}), \\
\end{eqnarray*}
which are representations of $\bG(\cO_{L^+,p})$ and $\bG\bL_2(\cO_{L,p})$, respectively (see Subsection \ref{ILLC-sect} and Definition \ref{inertialtype} for the definitions of $\sigma(\tau'_v|_{I_{L^+_w}})$, $\sigma'(\tau'_v|_{I_{L^+_w}})$).  We claim that we can choose $\pi_{L^+}$ so that we have a $\bG(\cO_{L^+,p})$-equivariant injection $\sigma_{L^+} \otimes_{E,\imath} \bbC \longhookrightarrow \pi_{L^+,p}$.  Indeed, note that by construction we have a $\bG\bL_2(\cO_{L,p})$-equivariant injection $\sigma_L' \otimes_{E, \imath} \bbC \longhookrightarrow \Pi_{L,p}$, which implies that $\Pi_{L,w}$ is a tamely ramified principal series representation for every place $w$ of $L$ above $p$ (this uses the genericity hypothesis).  By the explicit description of the local base change map given in \cite[\S~11.4]{rogawski} and \cite{blasco}, we see that if $w$ is a place of $L^+$ above $p$, then $\pi_{L^+,w}$ is either a tamely ramified principal series, or a supercuspidal representation contained in a local $L$-packet of size 2.  In the first case, we have a $\bG(\cO_{L^+_w})$-equivariant injection $\sigma_{L^+,w} \otimes_{E, \imath} \bbC \longhookrightarrow \pi_{L^+,w}$.  In the second case, it may happen that the supercuspidal representation $\pi_{L^+,w}$ has no invariants under the principal congruence subgroup of $\bG(\cO_{L^+,w})$, and therefore does not admit a $\bG(\cO_{L^+,w})$-equivariant injection $\sigma_{L^+,w} \otimes_{E, \imath} \bbC \longhookrightarrow \pi_{L^+,w}$; however, if we let $\pi_{L^+,w}^\ddagger$ denote the other element of the local $L$-packet containing $\pi_{L^+,w}$, then $\pi_{L^+,w}^\ddagger$ \emph{will} admit such an injection (see, e.g., the explicit description of depth 0 $L$-packets in \cite[\S~3.1]{adlerlansky}).  Let us define $\pi_{L^+}^\ddagger := \pi_{L^+,w}^\ddagger \otimes \bigotimes_{w' \neq w}' \pi_{L^+,w'}$, which lies in the same global $L$-packet as $\pi_{L^+}$.  Since the Galois representation associated to $\pi_{L^+}$ via Theorem \ref{auttogal} is irreducible, $\pi_{L^+}$ defines a stable $L$-packet, and in particular $\pi_{L^+}^\ddagger$ will be automorphic and cuspidal (this uses \cite[Prop. 11.2.1(a), Thm. 11.5.1]{rogawski}).  Therefore, by replacing $\pi_{L^+}$ by $\pi_{L^+}^\ddagger$ (for several $w\in \BC_{L^+/F^+}(\Sigma_p^+)$ if necessary), we can guarantee that we have a $\bG(\cO_{L^+,p})$-equivariant injection $\sigma_{L^+} \otimes_{E,\imath} \bbC \longhookrightarrow \pi_{L^+,p}$.}
\item Finally, suppose that $w$ is a place of $L^+$ which is inert in $L$ and such that $\Pi_{L,w}$ is unramified {(in particular, this means that $w \not\in \BC_{L^+/F^+}(\Sigma_p^+)$)}.  By the explicit description of local base change found in \cite[\S~ 11.4]{rogawski} and \cite{blasco}, we see that $\pi_{L^+,w}$ is unramified relative to a {hyperspecial subgroup} of $\bG(L^+_w)$, which is equal to $\bG(\cO_{L^+_w})$ for all but finitely many inert primes $w$.  
\end{enumerate}
Thus, we define $K_{L^+} = \prod_{w} K_{L^+,w} \subseteq \bG(\bbA_{L^+}^\infty)$ to be the compact open subgroup satisfying the following conditions:
\begin{itemize}
\item if $w \in \BC_{L^+/F^+}(\Sigma_{\textnormal{ram}}^+ \cup \{v_1\})$, then $K_{L^+,w}$ is the preimage under $\iota_{\widetilde{w}}$ of $\textnormal{Iw}_{\widetilde{w}}$, where $\widetilde{w}$ is a fixed choice of place of $L$ lying over $w$;
\item if $w$ is a place of $L^+$ which is split in $L$ and $w \not\in \BC_{L^+/F^+}(\Sigma_{\textnormal{ram}}^+ \cup \{v_1\})$, then $K_{L^+,w} = \bG(\cO_{L^+_w})$;
\item if $w\in \BC_{L^+/F^+}(\Sigma_p^+)$, then $K_{L^+,w} = \bG(\cO_{L^+_w})$;
\item  if $w$ is a place of $L^+$ which is inert in $L$ and $w\not\in \BC_{L^+/F^+}(\Sigma_p^+)$, then $K_{L^+,w}$ denotes a {hyperspecial subgroup} of $\bG(L^+_w)$ relative to which $\pi_{L^+,w}$ is unramified, chosen to be equal to $\bG(\cO_{L^+_w})$ for all but finitely many such $w$.
\end{itemize}
Note that $K_{L^+}$ is sufficiently small.  The representation $\pi_{L^+}$ then contributes to the space 
$$S_{\bG_{\cO_{L^+}}}(K_{L^+}, \sigma_{L^+}^\vee)_{\fm_{\rbar|_{\Gamma_{L^+}}}},$$ 
where $\fm_{\rbar|_{\Gamma_{L^+}}}$ is the maximal ideal of $\bbT^{\BC(T'')}$ defined as in Definition \ref{maxidealgalois}.  Here $T'' \defeq \Sigma_p^+ \cup \Sigma_{\textnormal{ram}}^+ \cup \{v_1\}$.

For every place $w \in \BC_{L^+/F^+}(\Sigma_{\textnormal{ram}}^+)$, we fix two distinct tame characters $\psi_{w,1}, \psi_{w,2}:\cO_{L^+_w}^\times \longrightarrow \cO^\times$ of $p$-power order, and define $\psi_w :K_{L^+,w} \longrightarrow \cO^\times$ by
$$\psi_w\left(\iota_{\widetilde{w}}^{-1}\begin{pmatrix}a & b \\ c & d\end{pmatrix}\right) = \psi_{w,1}(a)\psi_{w,2}(d),$$
where $\sm{a}{b}{c}{d} \in \textnormal{Iw}_{\widetilde{w}}$.  (Note that such a choice of characters is possible by the choice of $L^+$.)
We define $\psi \defeq \bigotimes_{w\in \BC(\Sigma_{\textnormal{ram}}^+)} \psi_w$, so that the reduction mod $\varpi$ of $\psi$ is the trivial character of $\prod_{w \in \BC(\Sigma_{\textnormal{ram}}^+)} K_{L^+,w}$.  Let 
$$S_{\bG_{\cO_{L^+}}}(K_{L^+}, \psi \otimes_{E} \sigma_{L^+}^\vee)$$ 
denote the space of algebraic automorphic forms with nebentypus $\psi$ above $\Sigma_{\textnormal{ram}}^+$ (defined as in Subsection \ref{sec:AAF}, except that the component $\prod_{w \in \BC(\Sigma_{\textnormal{ram}}^+)} K_{L^+,w}$ acts by $\psi$).

We claim $S_{\bG_{\cO_{L^+}}}(K_{L^+}, \psi \otimes_{E} \sigma_{L^+}^\vee)_{\fm_{\rbar|_{\Gamma_{L^+}}}} \neq 0$.  Indeed, after possibly replacing $E$ by a finite extension, the representation $\pi_{L^+}$ gives a morphism 
$$\theta:\bbT^{\BC(T'')}_{0,\BC(\tau')}(K_{L^+}) \longrightarrow \cO,$$ 
(where $\BC(\tau')$ denotes the collection $\{\tau'_v|_{I_{L_w}}\}_{w|_{F} = v \in \Sigma_p}$) and by reduction modulo $\varpi$ we obtain 
$$\theta \otimes_{\cO}\bbF: \bbT^{\BC(T'')}_{0,\BC(\tau')}(K_{L^+})\otimes_{\cO}\bbF \longrightarrow \bbF.$$  
Let $\bbT^{\BC(T'')}_{0,\BC(\tau')}(K_{L^+}, \bbF)$ denote the image of the universal Hecke algebra $\bbT^{\BC(T'')}$ in 
$$\textnormal{End}_{\cO}\left(S_{\bG_{\cO_{L^+}}}(K_{L^+}, \overline{\sigma_{L^+}^{{\vee,\circ}}})\right),$$
where $\sigma_{L^+}^{{\vee,\circ}}$ denotes a choice of $K_{L^+,p}$-stable $\cO$-lattice in {$\sigma_{L^+}^\vee$}.  Since the kernel of $\bbT^{\BC(T'')}_{0,\BC(\tau')}(K_{L^+})\otimes_{\cO}\bbF \longtwoheadrightarrow \bbT^{\BC(T'')}_{0,\BC(\tau')}(K_{L^+}, \bbF)$ is nilpotent, $\theta \otimes \bbF$ factors through a map 
$$\overline{\theta}: \bbT^{\BC(T'')}_{0,\BC(\tau')}(K_{L^+}, \bbF) \longrightarrow \bbF.$$  
Now let $\bbT^{\BC(T'')}_{0, \BC(\tau')}(K_{L^+},\psi)$ denote the image of the universal Hecke algebra $\bbT^{\BC(T'')}$ in 
$$\textnormal{End}_{\cO}\left(S_{\bG_{\cO_{L^+}}}(K_{L^+}, \psi \otimes_{\cO} \sigma_{L^+}^{{\vee,\circ}})\right),$$
and define $\bbT^{\BC(T'')}_{0,\BC(\tau')}(K_{L^+}, \psi, \bbF)$ analogously.  Since $\psi$ is of $p$-power order, we have $\bbT^{\BC(T'')}_{0,\BC(\tau')}(K_{L^+}, \psi, \bbF) = \bbT^{\BC(T'')}_{0,\BC(\tau')}(K_{L^+}, \bbF)$, and by pulling $\overline{\theta}$ back we get
$$\theta': \bbT^{\BC(T'')}_{0,\BC(\tau')}(K_{L^+}, \psi) \longtwoheadrightarrow \bbT^{\BC(T'')}_{0,\BC(\tau')}(K_{L^+}, \psi) \otimes_{\cO} \bbF \longtwoheadrightarrow \bbT^{\BC(T'')}_{0,\BC(\tau')}(K_{L^+}, \psi,\bbF) \stackrel{\overline{\theta}}{\longrightarrow} \bbF.$$
We view $\ker(\theta')$ as a prime ideal lying over the ideal $(\varpi)$ relative to the finite flat extension $\cO \longrightarrow \bbT^{\BC(T'')}_{0,\BC(\tau')}(K_{L^+}, \psi)$.  By the going-down theorem, there exists a prime $\fp \subseteq \ker(\theta')$ lying over $(0) \subseteq \cO$.

The minimal prime $\fp$ constructed above corresponds to an automorphic representation $\pi'_{L^+}$ contributing to $S_{\bG_{\cO_{L^+}}}(K_{L^+}, \psi \otimes_{E} \sigma_{L^+}^\vee)_{\fm_{\rbar|_{\Gamma_{L^+}}}}$.  Let $\Pi_L'$ denote the base change of $\pi_{L^+}'$ to $\bG\bL_2(\bbA_L)$.  Then for every place $w$ of $L^+$ which is inert in $L$ and for which $w\not\in \BC_{L^+/F^+}(\Sigma_p^+)$, the representation $\Pi'_{L,w}$ is unramified.  For every place $w$ of $L^+$ which splits as $w = \widetilde{w}\widetilde{w}^c$ in $L$ and for which $w\not\in \BC_{L^+/F^+}(\Sigma_{\textnormal{ram}}^+ \cup \{v_1\})$, the representation $\Pi'_{L,\widetilde{w}} = \pi'_{L^+,w} \circ \iota_{\widetilde{w}}^{-1}$ is unramified.  Finally, if $w$ splits in $L$ and $w \in \BC_{L^+/F^+}(\Sigma_{\textnormal{ram}}^+ )$, then $(\Pi'_{L,\widetilde{w}})^{\textnormal{Iw}_{\widetilde{w}},\psi_w \circ \iota_{\widetilde{w}}^{-1}} \neq 0$.  By choice of the characters $\psi_w$, the latter condition implies that $\Pi'_{L,\widetilde{w}}$ must be a principal series representation.

Associated to $\pi_{L^+}'$ (or $\Pi_L'$) we have a Galois representation
$$r_{\imath}(\pi_{L^+}'): \Gamma_L \longrightarrow \bG\bL_2(\overline{E})$$
lifting $\BC'(\rbar)|_{\Gamma_L} \otimes_{\bbF}\overline{\bbF}_p$.  By the discussion in the previous paragraph and local/global compatibility, the representation $r_\imath(\pi'_{L^+})$ is unramified outside $\BC_{L/F^+}(\Sigma_p^+ \cup \Sigma_{\textnormal{ram}}^+)$ (recall that all deformations at places above $v_1$ are unramified), and tamely ramified at $\BC_{L/F^+}(\Sigma_{\textnormal{ram}}^+)$.  Further, if $w\in \BC_{L/F^+}(\Sigma_p^+)$, then $r_{\imath}(\pi_{L^+}')|_{\Gamma_{L_w}}$ is potentially crystalline with (parallel) Hodge--Tate weights $\{1,0\}$ and inertial type $\tau'_v|_{I_{L_w}}$ (where $v = w|_F$).

We now choose another totally real extension $M^+/L^+$ satisfying the first five conditions imposed on $L^+$ above, along with the following {further} condition:  letting $\Pi'_M$ denote the base change of $\Pi'_L$ to $\bG\bL_2(\bbA_M)$ (where $M \defeq M^+F$), we have $(\Pi'_{M,w})^{\bG\bL_2(\cO_{M_w})} \neq 0$ for every $w \in \BC_{M/F^+}(\Sigma_{\textnormal{ram}}^+)$.  Thus, if we let 
$$r_{\imath}(\Pi'_M): \Gamma_M \longrightarrow \bG\bL_2(\overline{E})$$
denote the Galois representation associated to $\Pi'_M$, then we see that $r_{\imath}(\Pi'_M)$ is a lift of $\BC'(\rbar)|_{\Gamma_M} \otimes_{\bbF} \overline{\bbF}_p$.  Moreover, $r_{\imath}(\Pi'_M)$ is unramified outside $\BC_{M/F^+}(\Sigma_p^+)$ and if $w\in \BC_{M/F^+}(\Sigma_p^+)$, then $r_{\imath}(\Pi_M')|_{\Gamma_{M_w}}$ is potentially crystalline with (parallel) Hodge--Tate weights $\{1,0\}$ and inertial type $\tau'_v|_{I_{M_w}}$ (where $v = w|_F$).

Recall that we have defined a deformation problem $\cS_{0,\tau'}$.  We define $\cS_M$ to be the ``base changed'' deformation problem, so that
\begin{flushleft}
$\cS_M \defeq \bigg(M/M^+,~ \BC_{M^+/F^+}(\Sigma_p^+ \cup \{v_1\}),~ \BC_{M/F}(\Sigma_p \cup \{\widetilde{v_1}\}),~ \cO,~ \rbar|_{\Gamma_{M^+}},~ \varepsilon^{-1}, $
\end{flushleft}
\begin{flushright}
$  \{R^{\Box, 0, \tau'_v|_{I_{M_w}}}_w\}_{w \in \BC_{M^+/F^+}(\Sigma_p^+ )} \cup \{\widetilde{R}^{\Box}_w\}_{w\in \BC_{M^+/F^+}( \{v_1\})}\bigg).$
\end{flushright}
Thus, we see that the extension of $r_{\imath}(\Pi_M')$ to $\Gamma_{M^+}$ corresponds to an $\overline{E}$-point of $R_{\cS_M}^{\textnormal{univ}}$.  A variant of the patching construction of \cite[Thm. 3.4]{guerberoff} with potentially Barsotti--Tate deformation rings (see also the argument which follows in subsequent sections) shows that $(R_{\cS_M}^{\textnormal{univ}})^{\textnormal{red}}$ is isomorphic to an appropriate localized Hecke algebra, and consequently $R_{\cS_M}^{\textnormal{univ}}$ is finite over $\cO$.  Just as in the proof of \cite[Lem. 1.2.3(1)]{BLGGT}, we have that $R_{\cS_{0,\tau'}}^{\textnormal{univ}}$ is finite over $R_{\cS_M}^{\textnormal{univ}}$.  Combining these facts with the dimension calculation in \cite[Cor. 2.3.5]{CHT}, we see that $R_{\cS_{0,\tau'}}^{\textnormal{univ}}$ is a finite $\cO$-module of positive rank.

Now, let $r: \Gamma_{F^+} \longrightarrow \cG_2(\overline{E})$ correspond to an $\overline{E}$-point of $R_{\cS_{0,\tau'}}^{\textnormal{univ}}$, so that in particular $r$ is a lift of $\rbar \otimes_{\bbF} \overline{\bbF}_p$ which is unramified outside of $\Sigma_p^+$.  The restriction $r|_{\Gamma_{M^+}}$ corresponds to an $\overline{E}$-point of $R_{\cS_M}^{\textnormal{univ}}$, which necessarily factors through the reduced ring.  Thus, $\BC'(r)|_{\Gamma_{M}}$ is automorphic, and by {\cite[Lem. 2.2.2]{BLGGT} applied successively to $M/L$ and $L/F$, we conclude that} $\BC'(r)$ is also automorphic {(more appropriately, the pair $(\BC'(r), \varepsilon^{-1})$ is automorphic in the sense of \cite[\S~2.1]{BLGGT}).  Therefore, using \cite[Thm. 11.5.1]{rogawski} and the Jacquet--Langlands correspondence ({the latter in the  ``opposite direction'' as compared to the proof of Theorem \ref{auttogal}}), we can find an automorphic representation $\pi''$ of $\bG(\bbA_{F^+})$ which contributes to the space $S_{\bG}(K_0, \sigma^\vee)_{\fm_{\rbar}}$ (perhaps after replacing $\pi''$ by another element in its $L$-packet, as in item \ref{it:switch} above).}  This implies that $S_{\bG}(K_0, \sigma^\vee)_{\fm_{\rbar}} \neq 0$, and by Lemma \ref{lem:non-zero} we have $S_{\bG}(K_0, V'^\vee)_{\fm_{\rbar}} \neq 0$ for some $V' \in \JH(\overline{\sigma})$.
\end{proof}

{Recall that we have defined a maximal ideal $\fm \defeq \fm_{\rbar} \subseteq \bbT^{{T}}$ associated to $\rbar$ (Definition \ref{maxidealgalois}).  Proposition \ref{prop:level:lowering} shows that $S_{\bG}(K_m,\sigma^{{\vee,\circ}})_{\fm} \neq 0$ for $m \geq 1$, where $\sigma^{{\vee,\circ}}$ is a $\bG(\cO_{F^+,p})$-stable $\cO$-lattice in {the dual of} a tame type.  Since the $p$-component of $K_m$ acts trivially on $\sigma^{{\vee,\circ}}$ for $m \geq 1$, we have $S_{\bG}(K_m,\sigma^{{\vee,\circ}})_{\fm} \cong S_{\bG}(K_m, \cO)_{\fm} \otimes_{\cO} \sigma^{{\vee,\circ}}$.  Thus, the image of $\fm$ in $\bbT^{T}_{0,\mathbf{1}}(K_m)$ (which will be denoted by the same symbol $\fm$) is a maximal ideal.}  By Theorem \ref{galrepheckealg}, we have a continuous lift of $\rbar$ given by
$$r_{\fm}\otimes\cO/\varpi^r: \Gamma_{F^+} \longrightarrow \cG_2\left(\bbT^{{T}}_{0,\mathbf{1}}(K_m)_{\fm}\otimes_{\cO} \cO/\varpi^r\right)$$
which is of type $\cS$.  Therefore, we obtain a surjection
\begin{equation}\label{deftohecke}
R^{\textnormal{univ}}_{\cS} \longtwoheadrightarrow \bbT^{{T}}_{0, \mathbf{1}}(K_m)_{\fm} \otimes_{\cO} \cO/\varpi^r.
\end{equation}
In particular, $S_{\bG}(K_m,\cO/\varpi^r)_{\fm}$ is a finite $R_{\cS}^{\textnormal{univ}}$-module.

\subsection{Auxiliary primes}\label{auxprim}

\subsubsection{}

Let $q$ denote the maximum of $[F^+:\bbQ]$ and $\dim_{\bbF}H^1_{\cL^\perp,T}(\Gamma_{F^+,T}, \textnormal{ad}^0(\rbar)(1))$ (defined as in \cite[\S 2.3]{CHT}{; note that the latter cohomology group is the usual $H^1(\Gamma_{F^+,T}, \textnormal{ad}^0(\rbar)(1))$ since ``$S=T$'' in the notation of \emph{op.~cit.}}).  The proof {of \cite[Prop. 2.5.9]{CHT} (see also \cite[Prop. 4.4]{thorne})} remains valid, and thus for each $N\geq 1$ we can find a tuple $(Q_N, \widetilde{Q}_N, \{\overline{\psi}_v, \overline{\psi}_v'\}_{v\in Q_N})$ such that 
\begin{itemize} 
\item $Q_N$ is a finite set of places of $F^+$ which split in $F$;
\item $|Q_N| = q$;
\item $Q_N$ is disjoint from $T$;
\item $\widetilde{Q}_N$ consists of a single place $\widetilde{v}$ of $F$ above each place $v$ of $Q_N$;
\item if $v \in Q_N$ then $\bN(v) \equiv 1~(\textnormal{mod}~p^N)$; and
\item if $v\in Q_N$, then $\BC'(\rbar)|_{\Gamma_{F_{\widetilde{v}}}} \cong \overline{\psi}_v \oplus \overline{\psi}_v'$ where $\overline{\psi}_v$ and $\overline{\psi}_v'$ are distinct unramified characters.
\end{itemize}

For $v\in Q_N$, We let $R^{\overline{\psi}_v}_v$ denote the quotient of $\widetilde{R}^{\square}_v$ corresponding to lifts $\Gamma_{F_{\widetilde{v}}} \longrightarrow \bG\bL_2(R)$ of $\BC'(\rbar)|_{\Gamma_{F_{\widetilde{v}}}}$ which are $1_2 + \textnormal{Mat}_{2\times 2}(\fm_R)$-conjugate to a lift of the form $\psi \oplus \psi'$, where $\psi$ lifts $\overline{\psi}_{\widetilde{v}}$, $\psi'$ lifts $\overline{\psi}_{\widetilde{v}}'$, and $\psi'$ is unramified.  We then obtain a deformation problem
$$\cS_{Q_N} \defeq \left(F/F^+, T \cup Q_N, \widetilde{T}\cup \widetilde{Q}_N, \cO, \rbar, \varepsilon^{-1}, \{\widetilde{R}^{\Box}_{v}\}_{v \in \Sigma_p^+} \cup \{ \widetilde{R}^{\Box}_{v_1}\} \cup \{R^{\overline{\psi}_v}_v\}_{v\in Q_N}\right).$$
From this data we obtain the associated universal (resp.~$T$-framed) deformation ring $R_{\cS_{Q_N}}^{\textnormal{univ}}$ (resp.~$R_{\cS_{Q_N}}^{\square_T}$) of type $\cS_{Q_N}$.  By \cite[Prop. 4.4]{thorne}, the ring $R_{\cS_{Q_N}}^{\square_T}$ can be topologically generated over $R^{\textnormal{loc}}$ by $q - [F^+:\bbQ]$ elements.

\subsubsection{}\label{parahoric}
We now shrink the subgroup $K$ at places in $Q_N$.  For $w$ a place of $F$, denote by $K_0(w)$ and $K_1(w)$ the following subgroups: 
\begin{eqnarray*}
K_0(w) & \defeq & \textnormal{Iw}_{w} ~ = ~ \{g\in \bG\bL_2(\cO_{F_w}): g\equiv \sm{*}{*}{0}{*}~(\textnormal{mod}~w)\}\\
K_1(w) & \defeq & \ker\left(K_0(w) \longrightarrow \bbF_w^\times(p)\right)
\end{eqnarray*}
where $\bbF_w^\times(p)$ denotes the maximal $p$-power order quotient of $\bbF_w^\times$, and the map in question sends $\sm{a}{b}{c}{d}$ to the image of $d~(\textnormal{mod}~w)$ in $\bbF_w^\times(p)$.  For $i = 0,1$, define
$$K_i(Q_N)_m \defeq K^{Q_N}_m\cdot \prod_{v\in Q_N}\iota_{\widetilde{v}}^{-1}(K_i(\widetilde{v})).$$

\subsubsection{}\label{levelchange}

Let $\bbT^{{T}\cup Q_N} \subseteq \bbT^{{T}}$ denote the universal Hecke algebra away from ${T} \cup Q_N$, and define $\fm_{Q_N} \defeq \fm_{\rbar} \cap \bbT^{{T} \cup Q_N}$.  As in \cite[\S 2.6]{CEGGPS} (which is based on \cite[Prop. 5.9]{thorne}), we have a projection operator 
$$\textnormal{pr} \in \textnormal{End}_{\cO}\left(S_{\bG}\big(K_i(Q_N)_m, \cO/\varpi^r\big)_{\fm_{Q_N}}\right)$$
for $i = 0,1$.  This operator induces an isomorphism
$$\textnormal{pr}: S_{\bG}\big(K_m, \cO/\varpi^r\big)_{\fm} \stackrel{\sim}{\longrightarrow} \textnormal{pr}\left(S_{\bG}\big(K_0(Q_N)_m, \cO/\varpi^r\big)_{\fm_{Q_N}}\right),$$
which commutes with the action of $\bG(\cO_{F^+,p})$.

\subsubsection{}\label{coinvts}
Define 
$$\Gamma_m \defeq \prod_{v\in \Sigma_p^+} \bG(\cO_{F^+_v}/\varpi_v^m)$$
and
$$\Delta_{Q_N}  \defeq K_0(Q_N)_m/K_1(Q_N)_m \cong \prod_{v\in Q_N} K_0(\widetilde{v})/K_1(\widetilde{v}),$$
a finite $p$-group.

The space $\textnormal{pr}(S_{\bG}(K_1(Q_N)_m, \cO/\varpi^r)_{\fm_{Q_N}})$ has commuting actions of $\Gamma_m$ and $\Delta_{Q_N}$, under which it becomes a projective $(\cO/\varpi^r)[\Delta_{Q_N}][\Gamma_m]$-module (this follows from \cite[Lem. 3.3.1]{CHT}).  From this, we obtain $\Gamma_m$-equivariant isomorphisms
\begin{eqnarray*}
\textnormal{pr}\left(S_{\bG}\big(K_1(Q_N)_m, \cO/\varpi^r\big)_{\fm_{Q_N}}\right)^{\Delta_{Q_N}} & \cong & \textnormal{pr}\left(S_{\bG}\big(K_0(Q_N)_m, \cO/\varpi^r\big)_{\fm_{Q_N}}\right)\\
 & \cong &  S_{\bG}\big(K_m, \cO/\varpi^r\big)_{\fm}
 \end{eqnarray*}
where the last isomorphism follows from the previous subsection.

\subsubsection{}
Let {$\bbT^{{T}\cup Q_N}_{0,\mathbf{1}}(K_i(Q_N)_m,\cO/\varpi^r)_{\fm_{Q_N}}$} denote the image of $\bbT^{{T}\cup Q_N}$ in $\End_{\cO}(\textnormal{pr}(S_{\bG}(K_i(Q_N)_m, \cO/\varpi^r)_{\fm_{Q_N}}))$.  
{(This is the mod $\varpi^r$ reduction of the image of $\bbT^{{T}\cup Q_N}_{0,\mathbf{1}}(K_i(Q_N)_m)$ in  $\End_{\cO}(\textnormal{pr}(S_{\bG}(K_i(Q_N)_m, \cO)_{\fm_{Q_N}})).)$}
We let
$$r_{\fm_{Q_N}}^{\textnormal{pr}} :
\Gamma_{F^+} \longrightarrow \cG_2\left({\bbT^{{T}\cup Q_N}_{0,\mathbf{1}}(K_1(Q_N)_m,\cO/\varpi^r)_{\fm_{Q_N}}}\right)$$ 
denote the Galois representation {obtained by pushing forward the representation of Theorem \ref{galrepheckealg} to $\bbT^{{T}\cup Q_N}_{0,\mathbf{1}}(K_i(Q_N)_m,\cO/\varpi^r)_{\fm_{Q_N}}$}.  
Using the construction of $r_{\fm_{Q_N}}$, the local/global compatibility statements of Theorem \ref{auttogal}, and the properties of the auxiliary primes {(along with \cite[Prop. 5.12]{thorne})}, we see that $r_{\fm_{Q_N}}^{\textnormal{pr}}$
is a lift of type $\cS_{Q_N}$.  In particular, $\textnormal{pr}(S_{\bG}(K_i(Q_N)_m,\cO/\varpi^r)_{\fm_{Q_N}})$ is a finite $R_{\cS_{Q_N}}^{\textnormal{univ}}$-module.

\subsubsection{}
We identify the group $\Delta_{Q_N}$ with the image of $\prod_{v\in Q_N} I_{F_{\widetilde{v}}}$ in the maximal abelian $p$-power order quotient of $\prod_{v\in Q_N} \Gamma_{F_{\widetilde{v}}}$.  This gives rise to a homomorphism $\Delta_{Q_N} \longrightarrow R_{\cS_{Q_N}}^{\textnormal{univ},\times}$ as follows: let $r_{\cS_{Q_N}}^{\textnormal{univ}}$ denote any choice of universal deformation, and consider the map
$$\prod_{v\in Q_N}{I}_{F_{\widetilde{v}}} \stackrel{r_{\cS_{Q_N}}^{\textnormal{univ}}}{\longrightarrow} \prod_{v\in Q_N}\bG\bL_2(R_{\cS_{Q_N}}^{\textnormal{univ}}) \stackrel{\det}{\longrightarrow} R_{\cS_{Q_N}}^{\textnormal{univ},\times}.$$
Thus, we obtain morphisms $\cO[\Delta_{Q_N}] \longrightarrow R_{\cS_{Q_N}}^{\textnormal{univ}} \longrightarrow R_{\cS_{Q_N}}^{\square_{T}}$.  This gives an induced $\cO[\Delta_{Q_N}]$-module structure on $\textnormal{pr}(S_{\bG}(K_1(Q_N)_m, \cO/\varpi^r)_{\fm_{Q_N}})$, which agrees with the action of $\Delta_{Q_N}$ via diamond operators.  These morphisms also lead to natural isomorphisms 
$$R_{\cS_{Q_N}}^{\textnormal{univ}}/\fa_{Q_N} \cong R_{\cS}^{\textnormal{univ}} \qquad\textnormal{and}\qquad R_{\cS_{Q_N}}^{\square_T}/\fa_{Q_N} \cong R_{\cS}^{\square_T},$$
where $\fa_{Q_N}$ denotes the augmentation ideal of $\cO[\Delta_{Q_N}]$ 
({{cf. \cite[\S 4.3.7]{gee-kisin}}}).

\subsubsection{}
For each $N$, we choose a lift $r^{\textnormal{univ}}_{\cS_{Q_N}}$ representing the universal deformation of type $\cS_{Q_N}$, with $r^{\textnormal{univ}}_{\cS_{Q_N}}~(\textnormal{mod}~\fa_{Q_N}) = r^{\textnormal{univ}}_{\cS}$.  The choice of $r^{\textnormal{univ}}_{\cS_{Q_N}}$ gives an isomorphism $R^{\square_T}_{\cS_{Q_N}} \cong R^{\textnormal{univ}}_{\cS_{Q_N}}\widehat{\otimes}_{\cO}\cT$, which reduces modulo $\fa_{Q_N}$ to the isomorphism $R^{\square_T}_{\cS} \cong R^{\textnormal{univ}}_{\cS}\widehat{\otimes}_{\cO}\cT$.

\subsection{Patching}
\label{sub:patching}

\subsubsection{}
\label{patching-rings}
Let $q$ be as in Section \ref{auxprim}, and define
\begin{eqnarray*}
\Delta_\infty & \defeq & \bbZ_p^q\\
S_\infty & \defeq & \cT[\![\Delta_\infty]\!] \cong \cO[\![z_1, \ldots, z_{4|T|}, y_1, \ldots, y_q]\!]\\
R_\infty & \defeq & R^{\textnormal{loc}}[\![x_1,\ldots, x_{q - [F^+:\bbQ]}]\!]\\
R_{\lambda,\tau',\infty} & \defeq & R^{\textnormal{loc}}_{\lambda,\tau'}[\![x_1,\ldots, x_{q - [F^+:\bbQ]}]\!] 
\end{eqnarray*}
We let $\fa \subseteq S_\infty$ denote the augmentation ideal of $S_\infty$.  For each $N\geq 1$, fix a surjection $\Delta_\infty\longtwoheadrightarrow \Delta_{Q_N}$; passing to completed group algebras, we get a surjective map $S_\infty = \cT[\![\Delta_\infty]\!] \longtwoheadrightarrow \cT[\Delta_{Q_N}]$.  We view $R^{\square_T}_{\cS_{Q_N}}$ as an $S_\infty$-algebra via $S_\infty \longtwoheadrightarrow \cT[\Delta_{Q_N}] \longrightarrow R^{\square_T}_{\cS_{Q_N}}$, which gives $R^{\square_T}_{\cS_{Q_N}}/\fa \cong R_{\cS}^{\textnormal{univ}}$.

Recall that $R^{\square_T}_{\cS_{Q_N}}$ can be topologically generated over $R^{\textnormal{loc}}$ by $q - [F^+:\bbQ]$ variables.  Therefore, we can chooose a surjection of $R^{\textnormal{loc}}$-algebras
$$R_\infty \longtwoheadrightarrow R^{\square_T}_{\cS_{Q_N}}.$$

\subsubsection{}

We may now proceed exactly as in \cite[\S 2.8]{CEGGPS} and patch together (certain quotients of) the spaces
\begin{equation}
\label{patchfinitelevel}
\textnormal{pr}\left(S_{\bG}\big(K_i(Q_N)_{2N},\cO/\varpi^N \big)_{\fm_{Q_N}}\right)^\vee\otimes_{R_{\cS_{Q_N}}^{\textnormal{univ}}} R_{\cS_{Q_N}}^{\square_T},
\end{equation}
where $\vee$ denotes {here} the Pontryagin dual.  (In our setup, we are omitting the Hecke operators at $v_1$, and we ignore the maps $\alpha_N$ of \emph{op. cit.}.)  Thus, we obtain a profinite topological $S_\infty[\![\bG(\cO_{F^+,p})]\!]$-module $M_\infty$, with a commuting action of $R_\infty$.  Furthermore, $M_\infty$ enjoys the following properties:
\begin{itemize}
\item $S_\infty$ acts on \eqref{patchfinitelevel} via the map $S_\infty \longtwoheadrightarrow \cT[\Delta_{Q_N}] \longrightarrow R^{\square_T}_{\cS_{Q_N}}$ of Subsubsection \ref{patching-rings}, and this action factors through an $\cO$-algebra morphism $S_\infty \longrightarrow R_\infty$.  Since the image of $S_\infty$ in $\textnormal{End}_{S_\infty}(M_\infty)$ is closed, this implies we may factor the action of $S_\infty$ on $M_\infty$ through an $\cO$-algebra morphism $S_\infty \longrightarrow R_\infty$.
\item The argument at the bottom of p. 29 of \cite{CEGGPS} implies that $M_\infty$ is a finite $S_\infty[\![\bG(\cO_{F^+,p})]\!]$-module, and thus it is a finite $R_\infty[\![\bG(\cO_{F^+,p})]\!]$-module.
\item As in \cite[2.10 Prop.]{CEGGPS}, $M_\infty$ is projective over $S_\infty[\![\bG(\cO_{F^+,p})]\!]$.
\end{itemize}

\subsubsection{}

Using the patched module $M_\infty$, we define a patching functor $M_\infty(-)$ from the category of finitely generated $\cO$-modules with an action of $\bG(\cO_{F^+,p})$ to the category of $R_\infty$-modules by
$$M_\infty(W) \defeq \Hom^{\textnormal{cts}}_{\bG(\cO_{F^+,p})}\big(W, M_\infty^\vee\big)^\vee.$$
By projectivity of $M_\infty$ {(in the category of pseudocompact $\cO[\![\bG(\cO_{F^+,p})]\!]$--modules)}, the functor $M_\infty(-)$ is exact.  Moreover, if $W$ is $p$-torsion free, then we have 
$$M_\infty(W) \cong \textnormal{Hom}^{\textnormal{cts}}_{\bG(\cO_{F^+,p})}\big(W,M_\infty^{\textnormal{d}}\big)^\textnormal{d},$$
{where ``$\textnormal{d}$'' denotes the Schikhof dual (cf.~\cite[\S 1.8]{CEGGPS} for the definition)}.

\begin{prop}\hfill
\label{patchingprops}
\begin{enumerate}
\item \label{patch:item1} We have a $\bG(\cO_{F^+,p})$-equivariant isomorphism
$$M_\infty/\fa \cong \left(\varprojlim_{n}S_{\bG}(K^p,\cO/\varpi^n)_{\fm}\right)^{\textnormal{d}},$$
which is compatible with the action of $\widehat{\bigotimes}_{v\in \Sigma_p^+} \widetilde{R}^{\square}_v$ on both sides \emph{(}the action on the right-hand side is given by the maps 
$$\left. \widehat{\bigotimes}_{v\in \Sigma_p^+} \widetilde{R}^{\square}_v \longrightarrow R_{\cS}^{\textnormal{univ}} \longrightarrow \varprojlim_m \bbT{^T_{0,\mathbf{1}}}(K_m)_{\fm} \right).$$

\item \label{patch:item2} If $W$ is a free $\cO$-module of finite type \emph{(}resp.~a free $\bbF$-module of finite type\emph{)} with a continuous $\bG(\cO_{F^+,p})$-action, then $M_\infty(W)$ is a free $S_\infty$-module of finite type \emph{(}resp.~a free $S_\infty\otimes_{\cO}\bbF$-module of finite type\emph{)}.

\item\label{it:patch-type} {Let $\lambda\in (\bbZ^2_+)^{\widetilde{I}_p}$ and let $\tau' = \{\tau'_v\}_{v\in \Sigma_p^+}$ be a collection of tame inertial types satisfying $(\tau_v')^{\varphi^{-[\bbF_v^+:\bbF_p]}} \cong \tau_v'^\vee$.  Set $\sigma \defeq \bigotimes_{v\in \Sigma_p^+} \sigma(\tau'_v)$ and $W \defeq W_\lambda^{{\textnormal{d}}}\otimes_{\cO}\sigma^{\circ}$.}  Then we have an isomorphism
$$M_\infty(W)/\fa \cong S_{\bG}\big(K,W^{\textnormal{d}}\big)_{\fm}^{\textnormal{d}},$$
compatible with the surjection $R_\infty/\fa \longtwoheadrightarrow R_{\cS}^{\textnormal{univ}}$
{\emph{(}note that $R^{\textnormal{univ}}_{\cS}$ acts on the right hand side via
$R^{\textnormal{univ}}_{\cS} \longtwoheadrightarrow \bbT^{T}_{\lambda, \tau'}(K)_{\fm}$ \emph{)}.}

\item\label{patch-weight} Suppose $V$ is a Serre weight.  Then we have an isomorphism
$$M_\infty(V)/\fa \cong S_{\bG}\big(K,V^\vee\big)_{\fm}^\vee,$$
compatible with the surjection $R_\infty/\fa \longtwoheadrightarrow R_{\cS}^{\textnormal{univ}}$
{\emph{(}and the action on the right hand side is obtained as in item \ref{it:patch-type}\emph{)}.}

\item\label{it:MCM1} Let $\lambda\in (\bbZ^2_+)^{\widetilde{I}_p}$ and let $\tau' = \{\tau'_v\}_{v\in \Sigma_p^+}$ be a collection of tame inertial types satisfying $(\tau_v')^{\varphi^{-[\bbF_v^+:\bbF_p]}} \cong \tau_v'^\vee$.  Set $\sigma \defeq \bigotimes_{v\in \Sigma_p^+} \sigma(\tau'_v)$.  Then the $R_\infty$-action on $M_\infty(W_\lambda^{{\textnormal{d}}}\otimes_{\cO}\sigma^\circ)$ factors through $R_{\lambda,\tau',\infty}$.  Further, if $M_\infty(W_\lambda^{{\textnormal{d}}}\otimes_{\cO}\sigma^{\circ}) \neq 0$, then it is maximal Cohen-Macaulay over $R_{\lambda,\tau',\infty}$, and the support of $M_\infty(W_\lambda^{{\textnormal{d}}} \otimes_{\cO} \sigma^\circ)$ is a union of components of $\textnormal{Spec} R_{\lambda,\tau',\infty}$.

Let $\overline{R}_{\lambda,\tau',\infty}$ denote the quotient of $R_{\lambda,\tau',\infty}$ which acts faithfully on $M_\infty(W_\lambda^{{\textnormal{d}}} \otimes_{\cO} \sigma^\circ)$.  Then $M_\infty(W_\lambda^{{\textnormal{d}}}\otimes_{\cO}\sigma^{\circ})[1/p]$ is locally free of rank 2 over $\overline{R}_{\lambda,\tau',\infty}[1/p]$.  

\item\label{it:MCM2} Let $V$ be a Serre weight with highest weight $\lambda\in (\bbZ^2_{+,p})^{\widetilde{I}_p} \subseteq (\bbZ^2_+)^{\widetilde{I}_p}$.  Then $M_\infty(V) \neq 0$ if and only if $\rbar$ is modular of weight $V$.  In this case, the $R_\infty$-action on $M_\infty(V)$ factors through $R_{\lambda,\mathbf{1},\infty}\otimes_{\cO}\bbF$ and $M_\infty(V)$ is maximal Cohen-Macaulay over $R_{\lambda,\mathbf{1},\infty}\otimes_{\cO}\bbF$.  
\end{enumerate}
\end{prop}

\begin{proof}
\noindent (i) This follows from the patching construction (see \cite[\S 2.8]{CEGGPS}).  The argument of 2.11 Corollary of \emph{op. cit.} shows $\widehat{\bigotimes}_{v\in \Sigma_p^+} \widetilde{R}^{\square}_v$-equivariance.  

\vspace{5pt}

\noindent (ii) The module $M_\infty$ is a finite projective $S_\infty[\![\bG(\cO_{F^+,p})]\!]$-module.  If $W$ is $p$-torsion free, then the proof of \cite[4.18 Lem.]{CEGGPS} implies that $M_\infty(W)$ is a finite free $S_\infty$-module.  The same argument applies when $W$ is a free $\bbF$-module of finite type.  

\vspace{5pt}

\noindent (iii) Using part \ref{patch:item1}, we have
\begin{eqnarray*}
 \big(M_\infty(W)/\fa\big)^{\textnormal{d}} & \cong & \textnormal{Hom}^{\textnormal{cts}}_{\bG(\cO_{F^+,p})}\big(W,M_\infty^{\textnormal{d}}\big)[\fa]\\
  & \cong & \textnormal{Hom}^{\textnormal{cts}}_{\bG(\cO_{F^+,p})}\big(W,(M_\infty/\fa)^{\textnormal{d}}\big) \\
  & \cong & S_{\bG}(K,W^{\textnormal{d}})_{\fm}.
\end{eqnarray*}
{The statement about the action of the deformation ring follows in a manner analogous to the proof of \cite[Thm. 5.2.1(iii)]{HLM}.}

\vspace{5pt}

\noindent (iv) This follows by applying the previous point to a free $\cO$-module whose reduction mod $p$ is $V$, and reducing mod $p$.

\vspace{5pt}

\noindent (v) {The claims regarding the $R_\infty$-action, the support of the module $M_\infty(W_\lambda^{{\textnormal{d}}} \otimes_{\cO} \sigma^\circ)$, and its local freeness follow exactly as the proof of \cite[4.18 Lem.]{CEGGPS}, using \cite[Thm. 3.3.7]{bellovin-gee} instead of \cite[Thm. 3.3.8]{kisinPSS}.}

{In order to calculate the precise value of the rank, we proceed as follows.  We first claim that every irreducible component of $\overline{R}_{\lambda,\tau',\infty}[1/p]$ has nonempty intersection with the locus $\fa = 0$.  Indeed, since $\overline{R}_{\lambda,\tau',\infty}$ acts faithfully on $M_\infty(W_\lambda^{{\textnormal{d}}} \otimes_{\cO} \sigma^\circ)$, the localized ring $\overline{R}_{\lambda, \tau',\infty}[1/p]$ acts faithfully on $M_\infty(W_\lambda^{{\textnormal{d}}} \otimes_{\cO} \sigma^\circ)[1/p]$.  The latter module is free of finite rank over $S_\infty[1/p]$ (by item \ref{patch:item2}), so the ring $S_\infty[1/p]$ acts faithfully on it, and the ring $\overline{R}_{\lambda,\tau',\infty}[1/p]$ injects into a matrix ring over $S_\infty[1/p]$.  As the action of $S_\infty[1/p]$ factors through the action of $\overline{R}_{\lambda,\tau',\infty}[1/p]$, we conclude that we have an injection $S_\infty[1/p] \longhookrightarrow \overline{R}_{\lambda,\tau',\infty}[1/p]$, and that $\overline{R}_{\lambda,\tau',\infty}[1/p]$ is finite as an $S_\infty[1/p]$-module. }

{Now let $\fq$ denote a minimal prime of $\overline{R}_{\lambda,\tau',\infty}[1/p]$, and consider the composite map $\gamma: S_\infty[1/p] \longhookrightarrow \overline{R}_{\lambda,\tau',\infty}[1/p] \longtwoheadrightarrow \overline{R}_{\lambda,\tau',\infty}[1/p]/\fq$.  Using the fact that $\gamma$ is a finite map between integral domains of the same Krull dimension (the latter because $\overline{R}_{\lambda,\tau',\infty}[1/p]$ is equidimensional; see \cite[Thm. 3.3.7]{bellovin-gee} and note that $\textnormal{Spec} \overline{R}_{\lambda,\tau',\infty}[1/p]$ is a union of irreducible components of $\textnormal{Spec} R_{\lambda,\tau',\infty}[1/p]$), standard commutative algebra arguments imply that $\gamma$ must be injective (indeed, one sees that $\textnormal{ker}(\gamma)$ is a prime ideal of height 0).  Thus, by applying the Lying Over Theorem to the integral extension $\gamma: S_\infty[1/p] \longhookrightarrow \overline{R}_{\lambda,\tau',\infty}[1/p]/\fq$, we see that there exists a prime ideal lying over the augmentation ideal $\fa$.  This verifies the claim.  }

{Since the rank of $M_\infty(W_\lambda^{{\textnormal{d}}} \otimes_{\cO} \sigma^\circ)[1/p]$ is constant on the irreducible components of $\overline{R}_{\lambda,\tau',\infty}[1/p]$, the paragraphs above imply that it suffices to compute the rank at prime ideals $\fp$ containing $\fa$.  In particular, we may compute the rank after modding out by $\fa$.  Since $M_\infty(W_\lambda^{{\textnormal{d}}} \otimes_{\cO} \sigma^\circ)[1/p]/\fa$ is locally free of positive rank over $\overline{R}_{\lambda,\tau',\infty}[1/p]/\fa$, the localized ring $(\overline{R}_{\lambda,\tau',\infty}[1/p]/\fa)_\fp$ acts faithfully on $(M_\infty(W_\lambda^{{\textnormal{d}}} \otimes_{\cO} \sigma^\circ)[1/p]/\fa)_\fp$, and since this action factors through $(\bbT_{\lambda,\tau'}^T(K)_{\fm}[1/p])_\fp$ we obtain an isomorphism
$$\big(\overline{R}_{\lambda,\tau',\infty}[1/p]/\fa\big)_\fp \stackrel{\sim}{\longrightarrow} \big(\bbT_{\lambda,\tau'}^T(K)_{\fm}[1/p]\big)_\fp.$$
(The surjectivity of this map follows from item \ref{it:patch-type}.)  It therefore suffices to compute the rank of $M_\infty(W_\lambda^{{\textnormal{d}}} \otimes_{\cO} \sigma^\circ)[1/p]/\fa \cong S_\bG(K, W_\lambda \otimes_{\cO} \sigma^{{\vee,\circ}})_{\fm}^{\textnormal{d}}[1/p]$ as a module over $\bbT_{\lambda,\tau'}^T(K)_{\fm}[1/p]$.  Finally, since $\bbT_{\lambda,\tau'}^T(K)_{\fm}[1/p]$ is a product of fields (being a reduced Artinian $E$-algebra), this is equivalent to computing the rank of the linear dual $S_\bG(K, W_\lambda \otimes_{\cO} \sigma^{{\vee,\circ}})_{\fm}[1/p]$ over $\bbT_{\lambda,\tau'}^T(K)_{\fm}[1/p]$.  }

{Up to enlarging $E$ if necessary, we can assume that all prime ideals of $\bbT_{\lambda,\tau'}^T(K)_{\fm}[1/p]$ have residue field $E$.
Hence a prime ideal}
$\fp$ of $\bbT_{\lambda,\tau'}^T(K)_{\fm}[1/p]$ corresponds to a Hecke eigensystem $\lambda_\fp: \bbT_{\lambda,\tau'}^T(K)_{\fm}[1/p] \longrightarrow \overline{E}$, and therefore we obtain
$$\left(S_\bG(K, W_\lambda \otimes_{\cO} \sigma^{{\vee,\circ}})_{\fm}[1/p]\right)_{\fp} \otimes_{E,\imath}\bbC \cong \bigoplus_{\substack{\pi_\infty \cong  \mathsf{W}_{\imath_*\lambda}^\vee \\ \lambda_\pi = \lambda_\fp}} m(\pi) \Hom_{\bG(\cO_{F^+,p})}(\sigma^\circ\otimes_{\cO} \bbC, \pi_p ) \otimes_{\bbC} (\pi^{\infty,p})^{K^p},$$
where $\lambda_\pi: \bbT_{\lambda,\tau'}^T(K)_{\fm}[1/p] \longrightarrow \bbC \stackrel{\imath^{-1}}{\longrightarrow} \overline{E}$ denotes the Hecke eigensystem corresponding to $\pi$.  Since the base change map is injective on $L$-packets, strong multiplicity one for $\bG\bL_2$ implies that there is at most one $L$-packet contributing to the direct sum above.  Further, the base change map is determined by local base changes of local $L$-packets.  We have that the condition $\pi_v^{K_v} \neq 0$ for $v \not\in \Sigma_p^+$ inert in $F$ determines a unique member of the local $L$-packet at $v$, and the condition $\Hom_{\bG(\cO_{F^+,p})}(\sigma^\circ\otimes_{\cO} \bbC, \pi_p) \neq 0$ along with the multiplicity one property of Theorem \ref{ILLC} also determine a unique member of the local $L$-packet at $v\in \Sigma_p^+$.  Therefore, there is {exactly} one automorphic representation $\pi$ contributing to the direct sum above.  For this $\pi$, we know that $r_\imath(\pi)$ is irreducible (being a lift of $\BC'(\rbar)$), and therefore the base change of $\pi$ to $\bG\bL_2(\bbA_F)$ is cuspidal.  This implies that the $L$-packet $\JL([\pi])$ on $\bG^*(\bbA_{F^+})$ is stable, and therefore $m(\pi^*) = 1$ for any $\pi^*\in \JL([\pi])$ by \cite[Thm. 11.5.1(c)]{rogawski}.  Using an analog of the relation ``$n(\pi) = n(\tau)\prod_v c(\pi_v)$'' in \cite[p. 781]{labesselanglands}, we obtain $m(\pi) = 1$ (see also \cite[Thm. 1.7.1]{KMSW}).  To conclude, we note that $\dim_\bbC \Hom_{\bG(\cO_{F^+,p})}(\sigma^\circ\otimes_{\cO} \bbC, \pi_p) = 1$ by Theorem \ref{ILLC} and $\dim_\bbC( (\pi^{\infty,p})^{K^p}) = 2$ by \cite[Lem. 1.6(2)]{Taylor06} (since we have omitted Hecke operators at $v_1$).

\vspace{5pt}

\noindent (vi) Let $V$ be a Serre weight.  By point (iv) and Nakayama's lemma, $M_\infty(V) \neq 0$ if and only if $\rbar$ is modular of weight $V$ {and level $K$.  Therefore, in order to conclude it suffices to show that if $\rbar$ is modular of weight $V$, then $\rbar$ is modular of weight $V$ and level $K$.  This follows from Proposition \ref{prop:level:lowering} : in that proof, if we choose $\sigma$ so that $\JH(\overline{\sigma}) \cap \textnormal{W}^{?}(\rbar) = \{V\}$, then Theorem \ref{thm:WE} and exactness of the functor of algebraic automorphic forms guarantees that the $V'$ appearing at the end of the proof is equal to $V$.}

The claim about $M_\infty(V)$ being maximal Cohen--Macaulay follows exactly as in the previous point.
\end{proof}

\subsection{Weight Existence}
\label{subsec:WExt}

\begin{theo}
\label{thm:WExt}
Let $\rbar:\Gamma_{F^+}\longrightarrow \cG_2(\bbF)$ be a continuous representation such that
\begin{itemize}
	\item $\nu\circ \rbar = \overline{\varepsilon}^{-1}$;
	\item $\rbar^{-1}(\bG\bL_2(\bbF)\times\bG_m(\bbF)) = \Gamma_{F}$;
	\item $\BC'(\rbar)(\Gamma_F) \supseteq \bG\bL_2(\bbF')$ for some subfield $\bbF' \subseteq \bbF$ with $|\bbF'| > 6$;
	\item $\rbar$ is modular;
	\item $\rbar|_{\Gamma_{F^+_v}}$ is tamely ramified and $4$-generic for all $v\in \Sigma_p^+$;
	\item $\rbar$ is unramified outside $\Sigma_p^+$;
	\item $\overline{F}^{\ker(\textnormal{ad}^0(\rbar))}$ does not contain $F(\zeta_p)$.
\end{itemize}
Then
$$\textnormal{W}^?(\rbar) \subseteq \textnormal{W}_{\textnormal{mod}}(\rbar).$$
\end{theo}

\begin{proof}
Let $V \in \textnormal{W}^?(\rbar)$ and $V' \in \textnormal{W}_{\textnormal{mod}}(\rbar)$.  We will prove that $e(M_\infty(V)) = 2$ by induction on $d \defeq \dgr{V}{V'} = \sum_{v\in \Sigma_p^+}\dgr{V_v}{V'_v}$.  
{(We write $e(M_\infty(V))$ to denote $d!$ times the coefficient of degree $d$ of the Hilbert--Samuel polynomial of $M_\infty(V)$ as a module over $R_\infty/\textnormal{Ann}_{R_\infty}(M_\infty(V))$, where $d$ denotes the Krull dimension of $R_\infty/\textnormal{Ann}_{R_\infty}(M_\infty(V))$.)}

By Lemma \ref{lem:comb:typ1} there exists a tame $\bU_2(\cO_{F^+,p})$-type $\sigma = \bigotimes_{v\in \Sigma_p^+} \sigma_v$ such that:
\begin{enumerate}
	\item $V, V'\in \JH(\overline{\sigma})$;
	\item $|\JH(\overline{\sigma})\cap \textnormal{W}^?(\rbar)| = 2^d$;
	\item\label{it:str:low} for any $V''\in \JH(\overline{\sigma})\cap \textnormal{W}^?(\rbar)$ satisfying $V''\neq V$ one has $\dgr{V''}{V'} < \dgr{V}{V'}$.
\end{enumerate}
We define $\tau_v'$ to be the tame principal series type such that $\sigma_v \cong \sigma(\tau_v')$.  We note that in this case, we have isomorphisms 
$$\left(\tR^{\Box}_{v_1}\widehat{\otimes}\widehat{\bigotimes}_{v\in \Sigma_p^+}R_{\rhobar_v}^{\tau_v'}\right)[\![x_1, \ldots, x_{q - [F^+:\bbQ]}]\!] \cong R_{0,\tau',\infty} = \overline{R}_{0,\tau',\infty}.$$
(The last equality follows from Proposition \ref{patchingprops}\ref{it:MCM1} and the fact that each $R_{\rhobar_v}^{\tau_v'}$ is integral, cf. Table \ref{Table3}.

We thus have
\begin{eqnarray*}
2(2^d - 1) + e(M_\infty(V)) & = & \sum_{V''\in \JH(\overline{\sigma})\cap \textnormal{W}^?(\rbar)} e(M_\infty(V''))\\
 & = & e(M_\infty(\overline{\sigma^\circ}))\\
 & = & 2e\Bigg(\Big(\widehat{\bigotimes}_{v\in \Sigma_p^+}R_{\rhobar_v}^{\tau_v'}\Big)\otimes_{\cO}\bbF\Bigg)\\
 & = & 2\cdot 2^d.
\end{eqnarray*}
The first equality follows from the inductive hypothesis and item \ref{it:str:low}.  For the second, we note that $M_\infty(-)$ is exact, and if $V''$ is a Serre weight such that $V''\not\in \textnormal{W}^?(\rbar)$, then Theorem \ref{thm:WE} and Proposition \ref{patchingprops}\ref{it:MCM2} imply $M_\infty(V'') = 0$.  For the third we use Proposition \ref{patchingprops}\ref{it:MCM1} above, and the fourth follows by Corollary \ref{cor:irr:cmpts}.  Hence, we obtain $e(M_\infty(V)) = 2$, and in particular $M_\infty(V) \neq 0$.  Thus $V\in \textnormal{W}_{\textnormal{mod}}(\rbar)$ by Proposition \ref{patchingprops}\ref{it:MCM2}.  
\end{proof}

Combining Theorems \ref{thm:WE} and \ref{thm:WExt}, along with the isomorphism in Subsection \ref{isomsect}, we obtain the following.

\begin{corollary}
\label{cor:SWC}
Let $\rbar:\Gamma_{F^+}\longrightarrow {}^C\bU_2(\bbF)$ be a continuous $L$-parameter such that
\begin{itemize}
	\item $\widehat{\imath}\circ \rbar = \overline{\varepsilon}$;
	\item $\rbar^{-1}(\bG\bL_2(\bbF)\times\bG_m(\bbF)) = \Gamma_{F}$;
	\item $\BC(\rbar)(\Gamma_F) \supseteq \bG\bL_2(\bbF')$ for some subfield $\bbF' \subseteq \bbF$ with $|\bbF'| > 6$;
	\item $\rbar$ is modular;
	\item $\rbar|_{\Gamma_{F^+_v}}$ is tamely ramified and $4$-generic for all $v\in \Sigma_p^+$;
	\item $\rbar$ is unramified outside $\Sigma_p^+$;
	\item $\overline{F}^{\ker(\textnormal{ad}^0(\rbar))}$ does not contain $F(\zeta_p)$.
\end{itemize}
Then
$$\textnormal{W}^?(\rbar) = \textnormal{W}_{\textnormal{mod}}(\rbar).$$
\end{corollary}

\subsection{Automorphy lifting}
\label{autlift}

We now discuss our other main global application.

\begin{defn}
Let $F/F^+$ and $\bG$ be as in Subsection \ref{subsec:UGGlob}, and suppose $r': \Gamma_F \longrightarrow \bG\bL_2({E})$ is a continuous Galois representation.  We say $r'$ is \emph{automorphic} if there exists a cuspidal automorphic representation $\pi$ of $\bG(\bbA_{F^+})$ such that $r'{\otimes_E\ovl{E}} \cong r_{\imath}(\pi)$, where $r_{\imath}(\pi)$ is as in Theorem \ref{auttogal}.
\end{defn}

\begin{theo}
\label{thm:mod:lift-body}
Let $F/F^+$ and $\bG$ be as in Subsection \ref{subsec:UGGlob}.  Let $r': \Gamma_F \longrightarrow \bG\bL_2(\cO)$ be a Galois representation and let $\rbar': \Gamma_F \longrightarrow \bG\bL_2(\bbF)$ denote the associated residual representation.
Assume that
\begin{itemize}
	\item $r'$ is unramified at all but finitely many places;
	\item we have $r'^c \cong r'^\vee \otimes \varepsilon^{-1}$;
	\item for all $\kappa \in \widetilde{I}_p$, the local representation $r'|_{\Gamma_{F_{v(\kappa)}}}$ is potentially crystalline, with $\textnormal{HT}_{\kappa}(r'|_{\Gamma_{F_{v(\kappa)}}}) = \{1,~ 0\}$ and $4$-generic tame inertial type $\tau'_{v(\kappa)}$;
	\item $\rbar'$ is unramified outside $\Sigma_p$;
	\item for all $v\in \Sigma_p$, the local representation $\rbar'|_{\Gamma_{F_v}}$ is tamely ramified and $4$-generic;
	\item $\rbar' \cong \overline{r_{\imath}(\pi)}$ where $\pi$ is a cuspidal automorphic representation of $\bG(\bbA_{F^+})$ with $\pi_\infty$ trivial and such that for all $v\in \Sigma_p^+$, $\pi_v|_{\bG(\cO_{F^+_v})}$ contains the tame $\bG(\cO_{F^+_v}) \cong \bU_2(\cO_{F^+_v})$-representation associated to $\tau'_v$ by the inertial Local Langlands correspondence \emph{(}cf.~Theorems \ref{ILLC} and \ref{auttogal}\emph{)};
	\item $\overline{F}^{\ker(\textnormal{ad}(\rbar'))}$ does not contain $F(\zeta_p)$; 
	\item $\rbar'(\Gamma_F) \supseteq \bG\bL_2(\bbF')$ for some subfield $\bbF' \subseteq \bbF$ with $|\bbF'| > 6$.
\end{itemize}
Then
$r'{\otimes_{\cO}E}$ is automorphic.

\end{theo}

\begin{proof}  
We outline the proof, which is based on \cite[\S\S~ 4, 5]{Taylor08}.

We begin with several reductions.  Let $\Sigma_{\textnormal{ram}}$ denote the set of prime-to-$p$ places of $F$ at which $r'$ is ramified, and $\Sigma_{\textnormal{ram}}^+$ the set of places of $F^+$ which are the restriction to $F^+$ of places in $\Sigma_{\textnormal{ram}}$.  For every $v\in \Sigma_{\textnormal{ram}}^+$, we let $\widetilde{v}$ denote a fixed choice of place of $F$ lying above $v$.  We moreover fix a finite place $v_1$ of $F^+$ satisfying the hypotheses \ref{it:aux:1},\ref{it:aux:2} and \ref{it:aux:3} of Subsubsection \ref{subsub:setup}.  By \cite[Lem.~2.2.2]{BLGGT}, we may replace $r'$ by $r'|_{\Gamma_L}$ (for $L = L^+F$ and an appropriately chosen $L^+$ furnished by \cite[Lem. 4.1.2]{CHT}) and assume without loss of generality that the following conditions are satisfied: for every $w\in \Sigma_{\textnormal{ram}}$, we have
\begin{itemize}
\item $w$ is split over $w|_{F^+}$;
\item $\bN(w) \equiv 1~(\textnormal{mod}~p)$;
\item $\rbar'|_{\Gamma_{F_w}}$ is trivial;
\item the representation $r'|_{\Gamma_{F_{w}}}$ is, up to an unramified twist, a nonsplit extension of the trivial character by the cyclotomic character.  
\end{itemize}
By the proof of Proposition \ref{prop:level:lowering} we can assume further that $r_\imath(\pi)$ is unramified outside $\Sigma_p$ (in particular it is unramified outside $\Sigma_p \cup \{\widetilde{v}_1,  \tv_1^c\} \cup \Sigma_{\textnormal{ram}}$).

We first discuss Galois representations.  Enlarging $\cO$ if necessary, we view $r_\imath(\pi)$ as being valued in $\bG\bL_2(\cO)$.  Let $\widetilde{r}':\Gamma_{F^+} \longrightarrow \cG_2(\cO)$ and $\widetilde{r}_{\imath}(\pi): \Gamma_{F^+} \longrightarrow \cG_2(\cO)$ denote respectively the extensions of $r'$ and $r_{\imath}(\pi)$ to $\Gamma_{F^+}$ which satisfy $\nu\circ \widetilde{r}' = \varepsilon^{-1} = \nu \circ \widetilde{r}_{\imath}(\pi)$. We also let $\overline{\widetilde{r}'}$ denote the reduction mod $\varpi$ of $\widetilde{r}'$ or $\widetilde{r}_{\imath}(\pi)$.

For $v \in \Sigma_{\textnormal{ram}}^+$, we let $\chi_{v,1}, \chi_{v,2}: \Gamma_{F_{\tv}} \longrightarrow 1 + \varpi \cO$ denote two distinct continuous characters.  We denote by $R_v^{(\chi_{v,1},\chi_{v,2})}$ denote the quotient of $R^{\Box}_{v}$ parametrizing lifts $\rho$ of $\overline{\widetilde{r}'}|_{\Gamma_{F^+_v}}$ which satisfy $\nu \circ \rho = \varepsilon^{-1}$ and 
$$\textnormal{char}_{\BC'(\rho)(\gamma)}(X) = (X - \chi_{v,1}(\gamma))(X - \chi_{v,2}(\gamma))$$
for all $\gamma \in I_{F_{\tv}}$.  We define $R_v^{(1,1)}$ similarly, with the characters $\chi_{v,1}, \chi_{v,2}$ replaced by the trivial character.  (Note that these quotients exist and are non-zero, by the discussion in \cite[\S~ 3]{Taylor08}.)  Since the characters $\chi_{v,1}, \chi_{v,2}$ are trivial modulo $\varpi$, we have $R_v^{(\chi_{v,1},\chi_{v,2})}/\varpi \cong R_v^{(1,1)}/\varpi$.

We now consider two global deformation problem $\cS_{\Sigma_{\textnormal{ram}}, \tau'}$ and $\cS_{\Sigma_{\textnormal{ram}}, \tau'}'$  given by
\begin{flushleft}
$\cS_{\Sigma_{\textnormal{ram}},\tau'} \defeq \left(F/F^+,~ \Sigma_p^+ \cup \{v_1\}\cup \Sigma^+_{\textnormal{ram}},~\Sigma_p \cup \{\tld{v}_1\}\cup \tld{\Sigma}^+_{\textnormal{ram}} ,~ \cO,~ \overline{\widetilde{r}'},~ \varepsilon^{-1},\right.$ 
\end{flushleft}
\begin{flushright}
$\left.\{R^{\Box,0,\tau_v'}_{v}\}_{v \in \Sigma_p^+} \cup \{ \widetilde{R}^{\Box}_{v_1}\} \cup \{R^{(1,1)}_{v}\}_{v\in \Sigma^+_{\textnormal{ram}}}\right),$
\end{flushright}
\begin{flushleft}
$\cS_{\Sigma_{\textnormal{ram}},\tau'}' \defeq \left(F/F^+,~ \Sigma_p^+ \cup \{v_1\}\cup \Sigma^+_{\textnormal{ram}},~\Sigma_p \cup \{\tld{v}_1\}\cup \tld{\Sigma}^+_{\textnormal{ram}} ,~ \cO,~ \overline{\widetilde{r}'},~ \varepsilon^{-1},\right.$ 
\end{flushleft}
\begin{flushright}
$\left.\{R^{\Box,0,\tau_v'}_{v}\}_{v \in \Sigma_p^+} \cup \{ \widetilde{R}^{\Box}_{v_1}\} \cup \{R^{(\chi_{v,1}, \chi_{v,2})}_{v}\}_{v\in \Sigma^+_{\textnormal{ram}}}\right).$
\end{flushright}
We let $R_{\cS_{\Sigma_{\textnormal{ram}},\tau'}}^{\textnormal{univ}}$ (resp., $R_{\cS_{\Sigma_{\textnormal{ram}},\tau'}'}^{\textnormal{univ}}$) denote the complete local Noetherian $\cO$-algebra representing the functor of deformations of $\overline{\widetilde{r}'}$ of type $\cS_{\Sigma_{\textnormal{ram}},\tau'}$ (resp., of type $\cS_{\Sigma_{\textnormal{ram}},\tau'}'$).  We note that by the conditions at $\Sigma_{\textnormal{ram}}^+$, we have 
$$R_{\cS_{\Sigma_{\textnormal{ram}},\tau'}}^{\textnormal{univ}}/ \varpi \cong R_{\cS_{\Sigma_{\textnormal{ram}},\tau'}'}^{\textnormal{univ}}/\varpi.$$  
By the assumptions on $r'$ and $r_{\imath}(\pi)$, both $\widetilde{r}'$ and $\widetilde{r}_{\imath}(\pi)$ are deformations of $\overline{\widetilde{r}'}$ of type $\cS_{\Sigma_{\textnormal{ram}},\tau'}$, and therefore the $\ker(\cG_2(\cO) \longtwoheadrightarrow \cG_2(\bbF))$-conjugacy classes of $\widetilde{r}'$ and $\widetilde{r}_{\imath}(\pi)$ give rise to morphisms $\zeta': R_{\cS_{\Sigma_{\textnormal{ram}},\tau'}}^{\textnormal{univ}} \longrightarrow \cO$ and $\zeta_\pi: R_{\cS_{\Sigma_{\textnormal{ram}},\tau'}}^{\textnormal{univ}} \longrightarrow \cO$, respectively.

Next, we construct the spaces of algebraic automorphic forms that we will patch.  {Recall from Subsubsection \ref{subsec:level:lowering} the compact open $K_0\subseteq \bG(\bbA_{F^+}^\infty)$.  Let $K' = \prod_v K'_v \subseteq K_0$} denote the compact open subgroup satisfying the following conditions:
\begin{itemize}
\item if $v$ is a place of $F^+$ which is inert in $F$ and $v \not\in \Sigma_p^+$, then $K'_v$ is {a hyperspecial} subgroup of $\bG(F^+_v)$;
\item if $v$ is a place of $F^+$ which is split in $F$ and $v \not\in \{v_1\} \cup \Sigma_{\textnormal{ram}}^+$, then $K'_v = \bG(\cO_{F^+_v})$;
\item if $v \in \Sigma_p^+$, then $K'_v = \bG(\cO_{F^+_v})$;
\item if $v \in \{v_1\} \cup \Sigma_{\textnormal{ram}}^+$, then $K'_v$ is the preimage under $\iota_{\tv}$ of $\textnormal{Iw}_{\tv}$, the upper-triangular Iwahori subgroup of $\bG\bL_2(\cO_{F_{\tv}})$. 
\end{itemize}
With these choices, we have that $K'$ is sufficiently small.  Let $\sigma \defeq \bigotimes_{v\in \Sigma_p^+} \sigma(\tau'_v)$ denote the tame type associated to the collection $\tau' = \{\tau'_v\}_{v\in \Sigma_p^+}$ by Theorem \ref{ILLC}, and let $\sigma^\circ$ (resp., $\sigma^{\vee,\circ}$) denote a fixed choice $\bG(\cO_{F^+,p})$-stable $\cO$-lattice in $\sigma$ (resp., $\sigma^\vee$).  For $v \in \Sigma_{\textnormal{ram}}^+$, we also define a character $\chi_v:K'_v \longrightarrow \cO^\times$ by
$$\chi_v\left(\iota_{\tv}^{-1}\begin{pmatrix} a & b \\ c & d\end{pmatrix}\right) = (\chi_{v,1} \circ \textnormal{Art}_{F_\tv})(a)(\chi_{v,2} \circ \textnormal{Art}_{F_\tv})(d),$$
where $\sm{a}{b}{c}{d} \in \textnormal{Iw}_{\tv}$.  We then set $\chi \defeq \bigotimes_{v\in \Sigma_{\textnormal{ram}}^+}\chi_v$.  We will examine the spaces of algebraic automorphic forms
$$S_{\bG}(K', \sigma^{\vee, \circ}) \qquad \textnormal{and} \qquad S_{\bG}(K', \sigma^{\vee,\circ}\otimes_\cO \chi^{-1})$$
(the latter is defined as in Subsection \ref{sec:AAF}, except that the component $\prod_{v\in \Sigma_{\textnormal{ram}}^+}K'_v$ acts by $\chi$).

Let $\fm_{\overline{r}'} \subseteq \bbT^{\Sigma_p^+ \cup \{v_1\} \cup \Sigma_{\textnormal{ram}}^+}$ denote the maximal ideal associated to $\overline{r}'$ as in Definition \ref{maxidealgalois}.  Since the representation $\pi$ contributes to the space $S_{\bG}(K_0, \sigma^{\vee,\circ})_{\fm_{\overline{r}'}}$ {(recall that $K_0\supseteq K'$ with $K_0$ defined in Subsubsection \ref{subsec:level:lowering})}, we obtain $S_{\bG}(K', \sigma^{\vee,\circ})_{\fm_{\overline{r}'}} \neq 0$.  Using the fact that $\chi$ is congruent to the trivial character modulo $\varpi$, repeated application of Lemma \ref{lem:non-zero} gives
\begin{eqnarray*}
S_{\bG}(K', \sigma^{\vee,\circ})_{\fm_{\overline{r}'}} \neq 0 & \Longleftrightarrow & S_{\bG}(K', \sigma^{\vee,\circ} \otimes_{\cO} \bbF)_{\fm_{\overline{r}'}} \neq 0 \\
 & \Longleftrightarrow & S_{\bG}(K', (\sigma^{\vee,\circ} \otimes_\cO \chi^{-1}) \otimes_{\cO} \bbF)_{\fm_{\overline{r}'}} \neq 0 \\
 & \Longleftrightarrow & S_{\bG}(K', \sigma^{\vee,\circ} \otimes_\cO \chi^{-1} )_{\fm_{\overline{r}'}} \neq 0.
\end{eqnarray*}

We now outline the patching argument which uses the above spaces.  Let 
$$R^{\loc}_{0,\tau',\Sigma_{\textnormal{ram}}}\defeq \left( \widehat{\bigotimes}_{v \in\Sigma_p^+} R^{\Box,0,\tau_v'}_{v}\right) \widehat{\otimes}  \widetilde{R}^{\Box}_{v_1}  \widehat{\otimes}   \left( \widehat{\bigotimes}_{v \in \Sigma^+_{\textnormal{ram}}} R^{(1,1)}_{v}\right),$$
$$R^{\loc,'}_{0,\tau',\Sigma_{\textnormal{ram}}}\defeq \left( \widehat{\bigotimes}_{v \in\Sigma_p^+} R^{\Box,0,\tau_v'}_{v}\right) \widehat{\otimes}  \widetilde{R}^{\Box}_{v_1}  \widehat{\otimes}   \left( \widehat{\bigotimes}_{v \in \Sigma^+_{\textnormal{ram}}} R^{(\chi_{v,1}, \chi_{v,2})}_{v}\right).$$
A variant of the patching construction in Subsections \ref{auxprim} and \ref{sub:patching} with $\Sigma_p^+ \cup \{v_1\}$ replaced by $\Sigma_p^+ \cup \{v_1\} \cup \Sigma_{\textnormal{ram}}^+$ provides us with the following data (see \cite[\S~4]{Taylor08}):
\begin{enumerate} 
\item
\label{it:1:last}
A ring $R_{0,\tau',\Sigma_{\textnormal{ram}},\infty}$ which a formal power series ring in $q - [F^+:\bbQ]$ variables over $R^{\loc}_{0,\tau',\Sigma_{\textnormal{ram}}}$, together with a surjection $R_{0,\tau',\Sigma_{\textnormal{ram}},\infty} \longtwoheadrightarrow R_{\cS_{\Sigma_{\textnormal{ram}},\tau'}}^{\textnormal{univ}}$; 
\item 
\label{it:2:last}
an $R_{0,\tau',\Sigma_{\textnormal{ram}},\infty}$-module $M_\infty(\sigma^\circ)$ whose support is a union of irreducible components of $\textnormal{Spec} R_{0,\tau',\Sigma_{\textnormal{ram}},\infty}$; 
\item 
\label{it:3:last}
the mod $\mathfrak{a}$ reduction of $M_\infty(\sigma^\circ)$ is isomorphic to $S_{\bG}(K',(\sigma^{\circ})^{\textnormal{d}})_{\fm_{\rbar'}}^{\textnormal{d}}$, compatibly with the morphism $R_{0,\tau',\Sigma_{\textnormal{ram}},\infty}/\fa \longtwoheadrightarrow R_{\cS_{\Sigma_{\textnormal{ram}},\tau'}}^{\textnormal{univ}} \longtwoheadrightarrow \bbT_{0,\tau'}^{\Sigma_p^+ \cup \{v_1\} \cup \Sigma_{\textnormal{ram}}^+}(K')_{\fm_{\rbar'}}$; 
\item 
\label{it:4:last}
we have analogous ``primed'' versions of the above constructions corresponding to the deformation problem $\cS'_{\Sigma_{\textnormal{ram}},\tau'}$ (e.g., $R'_{0,\tau',\Sigma_{\textnormal{ram}},\infty}$, $R'_{0,\tau',\Sigma_{\textnormal{ram}},\infty} \longtwoheadrightarrow R_{\cS_{\Sigma_{\textnormal{ram}},\tau'}'}^{\textnormal{univ}}$, $M'_\infty(\sigma^\circ)/\fa \cong S_{\bG}(K',(\sigma^\circ\otimes_\cO \chi)^{\textnormal{d}})_{\fm_{\rbar'}}^{\textnormal{d}}$, etc.).  Furthermore, the primed data may be chosen so that it is compatible with the previous data modulo $\varpi$ (e.g., under the isomorphism $M_\infty(\sigma^\circ)/\varpi \cong M'_\infty(\sigma^\circ)/\varpi$, the action of $R_{0,\tau',\Sigma_{\textnormal{ram}},\infty}/\varpi$ on the left-hand side intertwines with the action of $R'_{0,\tau',\Sigma_{\textnormal{ram}},\infty}/\varpi$ on the right-hand side). 
\end{enumerate}

By the primed version of item \ref{it:2:last} and irreducibility of $\Spec R'_{0,\tau',\Sigma_{\textnormal{ram}},\infty}$, we conclude that
$$\textnormal{Supp}_{R'_{0,\tau',\Sigma_{\textnormal{ram}},\infty}}\left(M_\infty'(\sigma^\circ)\right) = \textnormal{Spec} R'_{0,\tau',\Sigma_{\textnormal{ram}},\infty}.$$  
(To see that $\Spec R'_{0,\tau',\Sigma_{\textnormal{ram}},\infty}$ is irreducible, we use the primed version of item \ref{it:1:last} and \cite[Lem. 3.3(5)]{BLGHT}, and observe that each of the local deformation rings comprising $R^{\loc,'}_{0,\tau',\Sigma_{\textnormal{ram}}}$ has an irreducible spectrum:  for $v\in \Sigma_p^+$, this follows from Table \ref{Table3} and equation \eqref{defringpres}; for $v = v_1$, this follows from \cite[Lem.~2.5]{CEGGPS}; for $v \in \Sigma_{\textnormal{ram}}^+$, this follows from \cite[Prop. 3.1(1)]{Taylor08}.)  In particular, we get 
$$\textnormal{Supp}_{R'_{0,\tau',\Sigma_{\textnormal{ram}},\infty}/\varpi}\left(M_\infty'(\sigma^\circ)/\varpi\right) = \textnormal{Spec} R'_{0,\tau',\Sigma_{\textnormal{ram}},\infty}/\varpi.$$
By item \ref{it:4:last}, we obtain the analogous statement for the deformation problem $\cS_{\Sigma_{\textnormal{ram}}, \tau'}$:
\begin{equation}
\label{full-support-mod-pi}
\textnormal{Supp}_{R_{0,\tau',\Sigma_{\textnormal{ram}},\infty}/\varpi}\left(M_\infty(\sigma^\circ)/\varpi\right) = \textnormal{Spec} R_{0,\tau',\Sigma_{\textnormal{ram}},\infty}/\varpi.
\end{equation}

Likewise, item \ref{it:2:last} implies that $\textnormal{Supp}_{R_{0,\tau',\Sigma_{\textnormal{ram}},\infty}}(M_\infty(\sigma^\circ))$ is a union of irreducible components of $\textnormal{Spec} R_{0,\tau',\Sigma_{\textnormal{ram}},\infty}$.  Since the irreducible components of $R_{0,\tau',\Sigma_{\textnormal{ram}},\infty}/\varpi$ are in bijection with the irreducible components of $R_{0,\tau',\Sigma_{\textnormal{ram}},\infty}$ by \cite[Prop. 3.1(3)]{Taylor08}, equation \eqref{full-support-mod-pi} implies
$$\textnormal{Supp}_{R_{0,\tau',\Sigma_{\textnormal{ram}},\infty}}\left(M_\infty(\sigma^\circ)\right) = \textnormal{Spec} R_{0,\tau',\Sigma_{\textnormal{ram}},\infty}.$$
Consequently, we get 
$$\textnormal{Supp}_{R_{0,\tau',\Sigma_{\textnormal{ram}},\infty}/\fa}\left(M_\infty(\sigma^\circ)/\fa\right) = \textnormal{Spec} R_{0,\tau',\Sigma_{\textnormal{ram}},\infty}/\fa,$$
which implies by item \ref{it:3:last} that
$$\textnormal{Supp}_{R_{\cS_{\Sigma_{\textnormal{ram}},\tau'}}^{\textnormal{univ}}} \left(S_{\bG}(K',(\sigma^{\circ})^{\textnormal{d}})_{\fm_{\rbar'}}^{\textnormal{d}} \right) = \textnormal{Spec} R_{\cS_{\Sigma_{\textnormal{ram}},\tau'}}^{\textnormal{univ}}.$$
Since $S_{\bG}(K',(\sigma^{\circ})^{\textnormal{d}})_{\fm_{\rbar'}}^{\textnormal{d}}$ is a faithful $\bbT_{0,\tau'}^{\Sigma_p^+ \cup \{v_1\} \cup \Sigma_{\textnormal{ram}}^+}(K')_{\fm_{\rbar'}}$-module and the latter ring is reduced, the surjection 
$$R_{\cS_{\Sigma_{\textnormal{ram}},\tau'}}^{\textnormal{univ}} \longtwoheadrightarrow \bbT_{0,\tau'}^{\Sigma_p^+ \cup \{v_1\} \cup \Sigma_{\textnormal{ram}}^+}(K')_{\fm_{\rbar'}}$$ 
induces an isomorphism
$$\Big(R_{\cS_{\Sigma_{\textnormal{ram}},\tau'}}^{\textnormal{univ}}\Big)^{\textnormal{red}} \stackrel{\sim}{\longrightarrow} \bbT_{0,\tau'}^{\Sigma_p^+ \cup \{v_1\} \cup \Sigma_{\textnormal{ram}}^+}(K')_{\fm_{\rbar'}}.$$
Thus, the homomorphism $\zeta': R_{\cS_{\Sigma_{\textnormal{ram}},\tau'}}^{\textnormal{univ}} \longrightarrow \cO$ factors through $(R_{\cS_{\Sigma_{\textnormal{ram}},\tau'}}^{\textnormal{univ}})^{\textnormal{red}} \cong \bbT_{0,\tau'}^{\Sigma_p^+ \cup \{v_1\} \cup \Sigma_{\textnormal{ram}}^+}(K')_{\fm_{\rbar'}}$, which implies that $r'$ is automorphic.  
\end{proof}

\bibliographystyle{amsalpha} 
\bibliography{Biblio}

\end{document}